\documentclass[11pt]{article}
\usepackage{amsthm,amsmath,amssymb,bbm}
\usepackage{color}
\def\bbone{{\mathbbm 1}}

\theoremstyle{plain}
\newtheorem{thm}{Theorem}[section]

\newtheorem{cor}{Corollary}[section]
\newtheorem{prop}{Proposition}[section]
\newtheorem{lem}{Lemma}[section]
\theoremstyle{definition}
\newtheorem{rem}{Remark}[section]
\newtheorem{defi}{Definition}[section]
\allowdisplaybreaks[1]

\begin{document}
\title{Some Notes on Quantitative Generalized CLTs with Self-Decomposable Limiting Laws by Spectral Methods}

\author{Benjamin Arras\thanks{Univ. Lille, CNRS, UMR 8524 - Laboratoire Paul Painlev\'e, F-59000 Lille, France; benjamin.arras@univ-lille.fr
\newline\indent Keywords: self-decomposability, Stein's method, stable laws, Dirichlet forms, Wasserstein distances, non-local partial differential equations, Poincar\'e-type inequality.
\newline\indent MSC 2010: 60E07; 60F05; 31C25; 26D10.}}

\maketitle

\vspace{\fill}
\begin{abstract}
 In these notes, we obtain new stability estimates for centered non-degenerate selfdecomposable probability measures on $\mathbb{R}^d$ with finite second moment and for non-degenerate symmetric $\alpha$-stable probability measures on $\mathbb{R}^d$ with $\alpha \in [1,2)$. These new results are refinements of the corresponding ones available in the literature. The proofs are based on Stein's method for self-decomposable laws, recently developed in a series of papers, and on closed forms techniques together with a new ingredient:~weighted Poincar\'e-type inequalities. As applications, rates of convergence in Wasserstein-type distances are computed for several instances of the generalized central limit theorems (CLTs).~In particular, a $n^{1-2/\alpha}$-rate is obtained in $1$-Wasserstein distance when the target law is a non-degenerate symmetric $\alpha$-stable one with $\alpha \in (1,2)$.~Finally, the non-degenerate symmetric Cauchy case is studied at length from a spectral point of view. At last, in this Cauchy situation, a $n^{-1}$-rate of convergence is obtained when the initial law is a certain instance of layered stable distributions.
\end{abstract}
\vspace{\fill}

\section{Introduction}
\subsection{Overview and main results}
\noindent
These notes are concerned with application of Stein's method in order to prove quantitative versions of generalized central limit theorems when the target law is not a Gaussian probability measure.~Originating in the work of C.~Stein (see \cite{Stein_72,Stein_86}), Stein's method is a collection of techniques which allows to assess the distance between two probability measures on $\mathbb{R}^d$, $d \geq 1$, and, in particular, to compute explicit rates of convergence for weak limit theorems of probability theory.~Extensions of the method outside the Gaussian world have been introduced so far and several non-equivalent general approaches are now available.~In this regard, let us cite \cite{MRRS_23} for a multivariate generalization of the density approach, \cite{FSX_19,GDVM_19} for recent developments of the method when the target law is the invariant measure of certain diffusions and \cite{AH18_1,AH19_2,AH20_3} where the method has been set up for some infinitely divisible (ID) probability measures on $\mathbb{R}^d$. For good introductions on the method and its generalizations, let us refer the reader to \cite{Chen_Goldstein_Shao_11,Ross_11,Cha_14}. Moreover, over the recent years, the method has proven to be extremely fruitful when it is combined with other fields of mathematics (see, e.g., \cite{Cha_09,NP_12,Ledoux_Nourdin_Peccati_GAFA15,CFP_19}). 

In particular, the methodology developed in this manuscript is an extension of the Stein kernel approach to stability estimates for functional inequalities and to quantitative approximation results for target probability measures which are not Gaussian measures. Implicitly initiated in \cite{Stein_86,CaPaUtev_AOP94}, further developed in \cite{Cha_09} and in \cite{NP_12}, the Stein kernel approach has been an active line of research in the last years starting with \cite{Ledoux_Nourdin_Peccati_GAFA15,CFP_19,F_AOP19} and very recently put forward in \cite{FK_AAP21,FK_AIHP24,MS_PTRF24}. Let us recall its basic principles when the reference measure is the standard Gaussian measure (denoted by $\gamma$ in the sequel) on $\mathbb{R}^d$. Let $\mu$ be a centered probability measure on $\mathbb{R}^d$ with finite second moment such that its covariance matrix is the identity and let $\tau_\mu$ be a mapping from $\mathbb{R}^d$ to the set of $d \times d$ matrices with real entries (denoted by $M_{d \times d}(\mathbb{R})$ in the sequel). $\tau_\mu$ is a (Gaussian-)Stein kernel for $\mu$ if, for all sufficiently smooth $\mathbb{R}^d$-valued function $f$ defined on $\mathbb{R}^d$, 
\begin{align}\label{eq:Stein_kernel_Gaussian}
\int_{\mathbb{R}^d} \langle x ; f(x) \rangle \mu(dx) = \int_{\mathbb{R}^d} \langle \tau_\mu(x) ; \nabla(f)(x) \rangle_{HS} \mu(dx), 
\end{align}
where $\langle \cdot ; \cdot \rangle$ stands for the Euclidean inner product on $\mathbb{R}^d$ and where $\langle \cdot;\cdot\rangle_{HS}$ stands for the Hilbert-Schmidt inner product between elements of $M_{d\times d}(\mathbb{R})$. Thanks to the characterization of $\gamma$ by Stein's lemma, $\tau_{\mu}$ is equal to the identity matrix if and only if $\mu$ is equal to $\gamma$. Then, the Stein kernel approach is based on the following heuristic: the proximity of $\tau_\mu$ to the identity matrix must imply, in a precise sense, the proximity of $\mu$ to $\gamma$. Indeed, this idea has been made quantitative and has taken, for example, the following form (see \cite[Proposition $3.1$]{Ledoux_Nourdin_Peccati_GAFA15}): 
\begin{align}\label{ineq:W2_SteinDiscrepancy_ineq}
W_2\left(\mu, \gamma\right) \leq  \left(\int_{\mathbb{R}^d} \|\tau_\mu(x) - I_d\|^2_{HS}\mu(dx)\right)^{\frac{1}{2}},
\end{align}
where $W_2(\mu, \gamma)$ denotes the classical $2$-Wasserstein distance between $\mu$ and $\gamma$ (see \cite[Definition $6.1$]{Villani_OT09}), where $\|\cdot\|_{HS}$ is the Hilbert-Schmidt norm and where $I_d$ is the identity matrix of $M_{d \times d}(\mathbb{R})$.~Using identity \eqref{eq:Stein_kernel_Gaussian} combined with the linear structure of the central limit theorem and inequality \eqref{ineq:W2_SteinDiscrepancy_ineq}, it is then possible to obtain sharp quantitative approximation results with explicit dimension dependence.~This strategy has been implemented, for example, in \cite{CFP_19,F_AOP19,F_SM21,FK_AAP21,AH22_5,FK_AIHP24,MS_PTRF24} and has provided rates of convergence in the multivariate central limit theorem under various assumptions on the initial law $\mu$. Indeed, a non-trivial task in dimension $d \geq 2$ is to find relevant sufficient conditions on the measure $\mu$ which ensure the existence of a Stein kernel $\tau_\mu$ together with a control of its moments: under the assumption of a finite Poincar\'e constant for $\mu$ and a closability condition, \cite{CFP_19,AH22_5} prove the existence of $\tau_\mu$ satisfying \eqref{eq:Stein_kernel_Gaussian} with a control on its second moment using techniques from the calculus of variations and from Dirichlet form theory (see, e.g., \cite[Proposition $4.1$]{AH22_5}); using tools coming from optimal transport and partial differential equations, the author of \cite{F_AOP19} has built a Stein kernel verifying \eqref{eq:Stein_kernel_Gaussian}, taking values in the set of symmetric positive definite real matrices and with an explicit dimension-dependent upper bound on its $p$-th moment, with $p \geq 2$, under log-concavity and non-degeneracy conditions on $\mu$ (see, \cite[Theorem $2.3$ and Proposition $3.2$]{F_AOP19}); finally, in \cite{MS_PTRF24}, the authors have built Stein kernel, in the sense of equation \eqref{eq:Stein_kernel_Gaussian}, for $\mu$ which is the image measure of an isotropic log-concave compactly supported probability measure by a continuously differentiable mapping with bounded partial derivatives using the F\"ollmer process as well as stochastic and Malliavin calculus on the Wiener space. 

When the target probability measure is not Gaussian, the Stein kernel approach has been developed in some specific situations and has lead to new functional inequalities and to new stability estimates (see, e.g., \cite{Ledoux_Nourdin_Peccati_GAFA15,AS_SPA17,AH19_2,AH20_3,AH22_5,Cheng_Thalmaier_Wang_jfa23}). In particular, \cite[Theorem $2.1$]{AH22_5} furnishes a general and abstract device in the setting of closed forms to build
Stein kernels without any constraints on the reference probability measure. This is \textit{the starting point} of these notes in order to develop the Stein kernel approach to stability estimates and to quantitative approximation results when the target probability measure is either self-decomposable (SD) with finite second moment or symmetric non-degenerate $\alpha$-stable with $\alpha \in [1,2)$.

Indeed, the set of SD probability measures on $\mathbb{R}^d$ is a subset of the set of infinitely divisible ones and comprises the Gaussian and the stable laws. Originally introduced by P. L\'evy in \cite{Levy_book54}, this set plays a prominent role in probability theory since these probability measures appear as weak limits in the generalized central limit theorems (see, e.g., \cite[Theorems $9.3$ and $15.3$]{Sato_15}).~As such, they have been studied by many authors (\cite{Loeve_book78,Bon_92,Steutel_Harn_04,Sato_15}). In particular, these probability measures are known to satisfy a Poincar\'e-type inequality with non-local quadratic forms (\cite{Chen_85,Chen_Lou_87,Houdre_al_98,Rockner_Wang_03}).

The stability estimates which are presented next rely on refinements of the Stein's method results obtained in \cite{AH19_2,AH20_3} as well as on closed forms techniques already used to some extent in \cite{AH20_3,AH22_5} together with a new spectral inequality: a weighted Poincar\'e-type inequality (see Definition \ref{def:wnl_Poincare_ineq} below).~This spectral inequality allows to introduce a weight function from which one can assess the distance to the target SD probability measure. Throughout the paper, an $\mathbb{R}^d$-valued random vector $X$, with characteristic function $\varphi_X$, is self-decomposable if, for all $\gamma \in (0,1)$, the function $\varphi_{X,\gamma}$ defined, for all $\xi \in \mathbb{R}^d$, by
\begin{align}\label{eq:def_SD_fc}
\varphi_{X,\gamma}(\xi) : = \dfrac{\varphi_X(\xi)}{\varphi_X(\gamma \xi)}, 
\end{align}
is a characteristic function.~Self-decomposable random vectors are infinitely divisible and their L\'evy measures admit the following polar decomposition:  
\begin{align}\label{eq:polar_decomposition_SD}
\nu(du) =\bbone_{\mathbb{S}^{d-1}}(x)\bbone_{(0,+\infty)}(r) \frac{k_x(r)dr}{r} \sigma(dx),
\end{align}
where $\mathbb{S}^{d-1}$ is the Euclidean unit sphere, where $\sigma$ is a finite Borel positive measure on $\mathbb{S}^{d-1}$ and where $k_x(r)$ is a function which is non-negative, non-increasing in $r$, ($k_x(r_1)\leq k_x(r_2)$, for $0<r_2\leq r_1$) and measurable in $x$. $\sigma$ is called the \textit{spherical part} of $\nu$.~Without loss of generality, $k_x(r)$ is assumed to be right-continuous in $r\in (0,+\infty)$, to admit a left-limit at each $r\in (0,+\infty)$ and to be such that $\int_0^{+\infty} (1\wedge r^2) k_x(r)dr/r$ is finite and independent of $x$ (see \cite[Theorem $15.10.$ and Remark $15.12.$]{Sato_15}). In the sequel, $\mu_1 << \mu_2$ means that the positive Borel measure $\mu_1$ is absolutely continuous with respect to the positive Borel measure $\mu_2$, $\ast$ denotes the convolution product between measures or functions, $\mathcal{C}^1_b(\mathbb{R}^d, \mathbb{R}^d)$ is the set of $\mathbb{R}^d$-valued functions which are bounded and continuously differentiable on $\mathbb{R}^d$ with uniformly bounded first order partial derivatives and $\mathcal{B}(\mathbb{R}^d)$ is the Borel sigma-field on $\mathbb{R}^d$.~A probability measure $\mu$ on $\mathbb{R}^d$ is degenerate if there exist $a \in \mathbb{R}^d$ and a proper linear subspace $E$ of $\mathbb{R}^d$ such that the support of $\mu$ is contained in the set $a+E$. Otherwise, it is called non-degenerate. Finally, the integral probability metric $d_{W_2}$ considered in this first stability estimate is the smooth 2-Wasserstein distance as defined in \eqref{eq:def_smooth_wasserstein_dist} with $r=2$.

\begin{thm}\label{thm:stability_estimate_second_moment}
Let $d \geq 1$ be an integer and let $\mu$ be a non-degenerate centered SD probability measure on $\mathbb{R}^d$ with finite second moment and with L\'evy measure $\nu$ given by 
\begin{align}\label{eq:polar_decomposition_stability}
\nu(du) =\bbone_{\mathbb{S}^{d-1}}(x)\bbone_{(0,+\infty)}(r) \frac{k(r)dr}{r} \sigma(dx),
\end{align}
where $k$ is positive on $(0,+\infty)$ and $\sigma$ is a finite positive Borel measure on $\mathbb{S}^{d-1}$. Let $X$ be a centered random vector of $\mathbb{R}^d$ with law $\mu_X$ and such that $\mathbb{E} \|X\|^2 <+\infty$. 
Let $\omega_X$ be a positive Borel function defined on $(0,+\infty)$ such that 
\begin{itemize}
\item $\int_{\mathbb{R}^d} \|u\|^2 \omega_X(\|u\|) \nu(du) <+ \infty$;
\item $\mu_X \ast \omega_X \nu << \mu_X$;
\item  For all $f \in \mathcal{C}_b^1(\mathbb{R}^d,\mathbb{R}^d)$ with $\int_{\mathbb{R}^d} f(x) \mu_X(dx) = 0$, 
\begin{align}\label{ineq:Poinc_type_weights}
\int_{\mathbb{R}^d} \| f(x)\|^2 \mu_X(dx) \leq \int_{\mathbb{R}^d} \int_{\mathbb{R}^d} \| f(x+u) - f(x) \|^2 \omega_X(\|u\|) \nu(du) \mu_X(dx).
\end{align}
\end{itemize}
Then, 
\begin{align}\label{eq:stability_estimate_second_moment}
d_{W_2}(\mu_X , \mu) & \leq \frac{1}{2}  \left(\int_{\mathbb{R}^d} \| u \|^2 \omega_X(\| u \|) \nu(du)\right)^{\frac{1}{2}}  \sqrt{ \int_{\mathbb{R}^d} \|u\|^2 \omega_X(\|u\|) \nu(du) - \int_{\mathbb{R}^d} \|x\|^2 \mu_X(dx) } \nonumber \\
& \quad\quad+ \frac{1}{2} \int_{\mathbb{R}^d} \|u\|^2 \left| \omega_X(\|u\|)-1\right| \nu(du). 
\end{align}
\end{thm}
\noindent
The proof of Theorem \ref{thm:stability_estimate_second_moment} is postponed to Section \ref{sec:stability_estimate_finite_second_moment}.~At this point, let us compare Inequality \eqref{eq:stability_estimate_second_moment} with previous stability estimates for Poincar\'e inequalities which have appeared in the literature.~In \cite[Theorem $1$]{Utev_SMJ89}, the author obtained the following stability estimate for the Gaussian Poincar\'e inequality in dimension $1$: for all centered probability measure $\mu$ on $\mathbb{R}$ with $\int_{\mathbb{R}} |x|^2 \mu(dx)= 1$ and with Poincar\'e constant $0<C_p(\mu)<+\infty$, 
\begin{align}\label{ineq:Utev_stability_estimate}
\underset{A \in \mathcal{B}(\mathbb{R})}{\sup} \left| \mu(A) - \gamma_1(A) \right| \leq 3 \sqrt{C_p(\mu)-1}, 
\end{align}
where $\gamma_1$ is the standard Gaussian measure on $\mathbb{R}$. Moreover, the authors of \cite[Theorem $4.1$]{CFP_19} obtained the following multidimensional extension of \eqref{ineq:Utev_stability_estimate}: for all centered probability measure $\mu$ on $\mathbb{R}^d$ with covariance matrix the identity matrix and with Poincar\'e constant $0<C_p(\mu)<+\infty$, 
\begin{align}\label{ineq:CFP_stability_estimate}
W_2\left(\mu,\gamma\right) \leq \sqrt{d \left(C_p(\mu) - 1\right)}. 
\end{align}  
Finally, \cite[Theorem $4.2$]{AH22_5} provides the following anisotropic extension of \eqref{ineq:CFP_stability_estimate}: for all centered probability measure $\mu$ on $\mathbb{R}^d$ with non-degenerate covariance matrix $\Sigma$ and with anisotropic Poincar\'e constant $0<U_{\Sigma, \mu}<+\infty$ (see \cite[Section $4$]{AH22_5}), 
\begin{align}\label{ineq:AH_stability_estimate}
W_1\left(\mu, \gamma_{\Sigma}\right) \leq \|\Sigma^{- \frac{1}{2}}\|_{op} \|\Sigma\|_{HS} \sqrt{U_{\Sigma, \mu} - 1},
\end{align}
where $W_1$ is the $1$-Wasserstein distance as defined in \eqref{eq:def_W1} below, where $\|\cdot\|_{op}$ is the operator norm for matrices and where $\gamma_{\Sigma}$ is the centered Gaussian probability measure on $\mathbb{R}^d$ with covariance matrix $\Sigma$.~Clearly, Inequality \eqref{eq:stability_estimate_second_moment} is an extension of the previous stability estimates taking into account the fact that the probability measure $\mu_X$ is not necessarily self-decomposable and does not necessarily have second moments which match those of $\mu$. 

The next set of probability measures which are of interest in these notes is the set of $\alpha$-stable probability measures on $\mathbb{R}^d$, with $\alpha \in (0,2)$. An ID $\mathbb{R}^d$-valued random vector $X$, with characteristic function $\varphi_X$, is called a stable random vector if, for all $a>0$,  there exist $b>0$ and $c \in \mathbb{R}^d$ such that, for all $\xi \in \mathbb{R}^d$, 
\begin{align}\label{eq:def_stable_rv}
\varphi_X(\xi)^a = \varphi_X(b \xi) e^{i \langle c ; \xi \rangle}. 
\end{align}
It is called strictly stable if $c = 0$. Stable random vectors of $\mathbb{R}^d$ are infinitely divisible and their L\'evy measures verify, for all $c>0$ and all $B \in \mathcal{B}(\mathbb{R}^d)$, 
\begin{align}\label{eq:def_scale_invariance}
c^{-\alpha}T_c(\nu)(B) = \nu(B),
\end{align}
for some $\alpha \in (0,2)$ which is named the stability index and where $T_c(\nu)(\cdot) = \nu(\cdot/c)$. Thanks to \cite[Theorem 14.3.]{Sato_15}, $\alpha$-stable random vectors of $\mathbb{R}^d$ are SD and their L\'evy measures admit the following polar decomposition: 
\begin{align}\label{eq:polar_decomposition_stable}
\nu(du) =\bbone_{\mathbb{S}^{d-1}}(x)\bbone_{(0,+\infty)}(r) \frac{dr}{r^{\alpha+1}} \sigma(dx).
\end{align}
Among the class of $\alpha$-stable probability measures, a particular emphasis will be put on the \textit{non-degenerate symmetric} ones. A probability measure $\mu$ on $\mathbb{R}^d$ is called symmetric if $\mu(B) = \mu(-B)$, for all $B \in \mathcal{B}(\mathbb{R}^d)$.~Thanks to \cite[Theorem 14.13]{Sato_15}, a non-trivial symmetric $\alpha$-stable probability measure $\mu_\alpha$ on $\mathbb{R}^d$ is characterized by its Fourier transform which is given, for all $\xi \in \mathbb{R}^d$, by 
\begin{align}\label{eq:fourier_symmetric_stable}
\widehat{\mu_\alpha}(\xi) : = \int_{\mathbb{R}^d} e^{i \langle  \xi ; x\rangle} \mu_\alpha(dx) =  \exp\left( - \int_{\mathbb{S}^{d-1}} |\langle \xi ; y\rangle|^\alpha \lambda(dy)\right),
\end{align}
where $\lambda$ is a symmetric finite non-zero Borel positive measure on $\mathbb{S}^{d-1}$ which is proportional to the spherical part of the L\'evy measure of $\mu_\alpha$.~$\lambda$ is called the \textit{spectral measure} of $\mu_\alpha$, is uniquely determined by $\mu_\alpha$ and encodes the anisotropy of $\mu_\alpha$ when it is not proportional to the surface measure of $\mathbb{S}^{d-1}$.~Also, a \textit{necessary and sufficient condition} which ensures the non-degeneracy of $\mu_\alpha$ is 
\begin{align}\label{eq:ND_CNS}
\underset{e \in \mathbb{S}^{d-1}}{\inf} \int_{\mathbb{S}^{d-1}} \left| \langle e ; y \rangle \right|^\alpha \lambda(dy) >0.
\end{align}
Regarding the non-degenerate symmetric $\alpha$-stable probability measures on $\mathbb{R}^d$ with $d \geq 1$ and $\alpha \in (1,2)$, the stability estimate reads as follows.  

\begin{thm}\label{thm:stability_estimate_NDS_stable}
Let $d \geq 1$ be an integer, let $\alpha \in (1,2)$, let $\mu_\alpha$ be the non-degenerate symmetric $\alpha$-stable probability measure on $\mathbb{R}^d$ which Fourier transform is given by \eqref{eq:fourier_symmetric_stable} and let $\nu_\alpha$ be the associated L\'evy measure. For all $R>0$, let $g_{R}$ be the function defined by $g_{R}(x) = x \exp\left( - \|x\|^2/R^2\right)$, for all $x \in \mathbb{R}^d$.~Let $\mu_X$ be a symmetric probability measure on $\mathbb{R}^d$ such that there exists $\beta \in (1,\alpha)$ with $\int_{\mathbb{R}^d} \|x\|^\beta \mu_X(dx) <+\infty$. Let $\omega_X$ be a positive Borel function defined on $(0,+\infty)$ such that
\begin{itemize}
\item $\mu_X \ast \omega_X \nu_\alpha << \mu_X$;
\item $\int_{\| u \|\leq 1} \|u\|^2 \omega_X(\|u\|) \nu_\alpha(du)< +\infty$ and $\int_{\| u \|\geq 1} \|u\| \omega_X(\|u\|) \nu_\alpha(du) <+\infty$;
\item For all $f \in \mathcal{C}_b^1(\mathbb{R}^d ,\mathbb{R}^d)$ such that $\int_{\mathbb{R}^d} f(x) \mu_X(dx) = 0$, 
\begin{align*}
\int_{\mathbb{R}^d} \|f(x)\|^2 \mu_X(dx) \leq \int_{\mathbb{R}^d} \int_{\mathbb{R}^d} \|f(x+u)-f(x)\|^2 \omega_X(\|u\|) \nu_\alpha(du) \mu_X(dx); 
\end{align*}
\item $\underset{R \longrightarrow +\infty}{\limsup} \left(\mathcal{E}_{\nu_\alpha , \omega_X}(g_R , g_R) - \langle g_R ; g_R \rangle_{L^2(\mu_X)} \right) \leq \delta$, where $\mathcal{E}_{\nu_\alpha , \omega_X}$ is given, for all $f_1, f_2 \in \mathcal{C}_b^1(\mathbb{R}^d , \mathbb{R}^d)$, by 
\begin{align}\label{eq:weighted_form_NDS_gen}
\mathcal{E}_{\nu_\alpha, \omega_X}(f_1 , f_2) = \int_{\mathbb{R}^d} \int_{\mathbb{R}^d} \langle f_1(x+u) - f_1(x)  ; f_2(x+u) - f_2(x) \rangle \omega_X(\|u\|) \nu_\alpha(du) \mu_X(dx),
\end{align}
and where $\delta \geq 0$ depends on $\alpha$, $d$, $\omega_X$ and $\mu_X$.
\end{itemize}
Then, 
\begin{align}\label{ineq:stability_estimate_W1_stable_NDS_gen}
W_1\left(\mu_X , \mu_\alpha \right) & \leq \left(C_{\alpha,d}^2 \int_{\|u\| \leq 1} \|u\|^2 \omega_X(\|u\|)\nu_\alpha(du) + 4 \int_{\|u\| \geq 1} \omega_X(\|u\|) \nu_\alpha(du)\right)^{\frac{1}{2}}\sqrt{\delta} \nonumber \\
&\quad\quad + 2 \int_{\|u\|\geq 1} \|u \| \left| \omega_X( \|u\| )-1 \right|\nu_\alpha(du) \nonumber \\
& \quad\quad + C_{\alpha,d} \int_{\|u\|\leq 1} \|u \|^2 \left| \omega_X( \|u\| )-1 \right| \nu_\alpha(du),
\end{align}
for some positive constant $C_{\alpha,d}$ which depends only on $\alpha$ and on $d$.
\end{thm}
\noindent
The proof of inequality \eqref{ineq:stability_estimate_W1_stable_NDS_gen} is postponed to Section \ref{sec:stability_alpha_stable_pm_12}.~It relies on the generator approach to Stein's method for non-degenerate symmetric $\alpha$-stable probability measure, with $\alpha \in (1,2)$, and on a smooth truncation argument in order to use \cite[Theorem $2.1$]{AH22_5}. At the core of the proof lies a ``regularization phenomenon" enjoyed by the solution to the Stein equation: the supremum norm of the Hessian matrix of the solution is bounded by a constant depending on the data while the second member of the Stein equation is controlled in Lipschitz seminorm (see Equation \eqref{ineq:regularity_estimate_NDS_stable} in Proposition \ref{prop:Stein_Method_Stable_NDS}). This global regularization phenomenon is linked to this finite dimensional setting and is reminiscent of the non-degenerate Gaussian case: in \cite[Lemma $4.3$]{AH22_5}, supremum norms of the Hessian matrix of the solution to the Gaussian Stein equation are controlled by the operator norm of $\Sigma^{-1/2}$, $\Sigma$ being the non-degenerate covariance matrix. The ellipticity condition \eqref{eq:ND_CNS} plays the role of $\operatorname{det}(\Sigma) \ne 0$ and is the only requirement to observe this phenomenon in the symmetric $\alpha$-stable case with $\alpha \in (1,2)$.~This fact has already been noticed in the PDE literature when studying, e.g., sharp regularity estimates for the solution to the Poisson-type equation with a symmetric $\alpha$-stable L\'evy operator (see, e.g., \cite[Assumptions $(1.2)$ and Theorem $1.1$]{ROS_16}).
 
 Moreover, no restriction on the ``anisotropy" of the target non-degenerate symmetric $\alpha$-stable probability measure is imposed which contrasts \textit{significantly} with the Stein's method results available in the literature up until now (see \cite[Theorem $3$, (i)]{Chen_Nourdin_Xu_Yang_23} and Remark \ref{rem:Stein_Equation_Solution_NDS} below for a detailed discussion).~In particular, this allows us to reach second order Stein's factors for the whole set of non-degenerate symmetric $\alpha$-stable probability measures, with $\alpha \in (1,2)$, without using pointwise estimates on the associated density and its derivatives (see \cite[Lemma $2$]{Chen_Nourdin_Xu_Yang_23}). Finally, in infinite dimension, one might hope to extend these stability results using the smooth $2$-Wasserstein distance as shown by the inequalities \eqref{eq:stability_estimate_second_moment} and \eqref{ineq:stability_result}.~This will be investigated elsewhere.
  
Next, let us present the stability estimate for the non-degenerate symmetric $1$-stable (Cauchy) probability measures on $\mathbb{R}^d$.~In this critical situation, we do not develop an extension of the Stein kernel approach and do not use \cite[Theorem $2.1$]{AH22_5} to derive the corresponding stability estimate.~Instead, assuming that the compared probability measure $\mu_X$ is also infinitely divisible with an appropriate L\'evy measure, the proof of Theorem \ref{thm:stability_result_sym_Cauchy_measure} relies on a version of Stein's lemma for infinitely divisible probability measures on $\mathbb{R}^d$ together with a comparison of the integro-differential parts of the non-local Stein operators along the solution to the Cauchy Stein equation.~Indeed, as discussed in Remark \ref{rem:comments_previous_theorem}, (iii), a natural quadratic form to consider in order to apply \cite[Theorem $2.1$]{AH22_5} is given by \eqref{eq:form_central_difference} below.~Nevertheless, when the reference measure is the rotationally invariant $1$-stable probability measure and the weight function is the constant function identically equal to $1$, we do not know if a Poincar\'e-type inequality holds true with this non-local form. Such Poincar\'e-type inequalities are important to infer quantitative ergodic properties of the associated semigroups (see \cite[Theorem $2.1$, Assumption $(2.29)$]{AH22_5} and \cite[Chapter $4$]{BGL_book14}). The proof of Theorem \ref{thm:stability_result_sym_Cauchy_measure} below is postponed to Subsection \ref{sec:SE_NDS_Cauchy}.  

\begin{thm}\label{thm:stability_result_sym_Cauchy_measure}
Let $d \geq 1$ be an integer, let $\mu_1$ be a non-degenerate symmetric $1$-stable probability measure on $\mathbb{R}^d$ with associated L\'evy measure $\nu_1$. Let $\omega_X$ be a positive Borel function defined on $(0,+\infty)$ such that
\begin{align}\label{eq:comp_cond_weight_cauchy}
\int_{\|u\|\leq 1} \| u \|^2 \omega_X(\|u\|) \nu_1(du) < +\infty , \quad \int_{\|u\|\geq 1} \omega_X(\|u\|) \nu_1(du) <+\infty,
\end{align}
such that, for all $r \in (0, +\infty)$ and all $y \in \mathbb{S}^{d-1}$,
\begin{align}\label{eq:cond_semi_cv}
\int_{\mathbb{R}^d} \sin\left( r \langle u ; y \rangle\right) \langle y ; u \rangle \omega_X(\|u\|) \nu_1(du) <+\infty, 
\end{align}
and such that, for all $r \in (0, +\infty)$ and all $y \in \mathbb{S}^{d-1}$, 
\begin{align}\label{eq:cond_A_X}
\underset{R \geq 1}{\sup} \left| \int_{\frac{1}{R} \leq \|u\| \leq R} \sin\left( r \langle u ; y \rangle\right) \langle y ; u \rangle \omega_X(\|u\|) \nu_1(du) \right| \leq \Phi_X(r), 
\end{align}
where $\Phi_X$ is a positive valued function defined on $(0,+\infty)$ with at most polynomial growth when $r$ tends to $+\infty$.~Let $X$ be a random vector of $\mathbb{R}^d$ with law $\mu_X$ such that $\mu_X$ is infinitely divisible with L\'evy measure $\nu_X = \omega_X \nu_1$ and parameter $b_X = 0$ and such that 
\begin{align}\label{eq:moment_condition}
\int_{\mathbb{R}^d} \|x\|^{\varepsilon} \mu_X(dx) <+\infty,
\end{align}
for some $\varepsilon \in (1/2,1)$. Then, 
\begin{align}\label{ineq:stability_result}
d_{W_2}(\mu_X , \mu_1) & \leq  \int_{\|u\| \leq 1} \|u\|^2 \left| \omega_X(\| u \|) - 1 \right| \nu_1(du) + M_1 \int_{\|u\| \geq 1} \left| \omega_X(\|u\|) - 1 \right| \nu_1(du) \nonumber \\
& \quad\quad + \int_{\|x\| \geq 1} \left(\int_{1 \leq \|u\| \leq \|x\|} \|u\| \left| \omega_X(\|u\|) - 1 \right| \nu_1(du)\right) \mu_X(dx)  \nonumber \\
& \quad\quad + M_2\int_{\|x\| \geq 1} \left(1+ \|x\|\right) \left(\int_{\|u\| \geq \|x\|} \left| \omega_X(\|u\|) - 1 \right| \nu_1(du)\right) \mu_X(dx),
\end{align} 
for some $M_1, M_2 >0$. 
\end{thm}
\noindent
As direct applications of the previous stability results, several quantitative instances of generalized central limit theorems with $\alpha$-stable target laws are put forward. In particular, the computed rates of convergence are sharp and match, in the non-degenerate symmetric $\alpha$-stable case with $\alpha \in (1,2)$, the best known rates. Moreover, Subsection \ref{subsec:literature_review} below discusses at length the related literature and provides a detailed comparison between Theorem \ref{thm:all_dimensions_NDS} and the corresponding results contained in \cite{Chen_Nourdin_Xu_Yang_23}. Finally, in the non-degenerate symmetric Cauchy case, a $n^{-1}$-rate of convergence is obtained when the initial law is a certain instance of layered stable distributions (see, e.g., \cite{HK_07,RS_10}).

\begin{thm}\label{thm:all_dimensions_NDS}
Let $\alpha \in (1,2)$, let $d \geq 1$ be an integer and let $\mu_\alpha$ be a non-degenerate symmetric $\alpha$-stable probability measure on $\mathbb{R}^d$ with associated L\'evy measure $\nu_\alpha$ which spherical component is denoted by $\sigma$.~Let $\mu$ be a symmetric SD probability measure on $\mathbb{R}^d$ with L\'evy measure $\nu$ such that  
\begin{align}
\nu(du) =\bbone_{(0,+\infty)}(r)\bbone_{\mathbb{S}^{d-1}}(\theta) \frac{k(r)}{r}dr \sigma(d\theta),
\end{align}
where the function $k$ is continuous on $(0,+\infty)$, is such that $\underset{z \in (0,+\infty)}{\sup} (z^\alpha k(z)) <+\infty$ and is such that 
\begin{align}\label{eq:second_order_expansion_NDS}
k(r) = \frac{1}{r^\alpha} +  \frac{c_{\alpha,\beta}}{r^\beta}, \quad r \longrightarrow +\infty,
\end{align}
for some $c_{\alpha,\beta} \in \mathbb{R}^*$ and $\beta >2$.~Let $(Z_k)_{k \geq 1}$ be a sequence of independent and identically distributed (i.i.d.) random vectors of $\mathbb{R}^d$ with law $\mu$ and let $(S^{\alpha}_n)_{n \geq 1}$ be the sequence of random vectors of $\mathbb{R}^d$ defined, for all $n \geq 1$, by 
\begin{align}\label{eq:Sn_alpha_meta}
S_n^{\alpha} = \dfrac{1}{n^{\frac{1}{\alpha}}} \sum_{k=1}^n Z_k.
\end{align}
For all $n \geq 1$, let $\mu^\alpha_n$ be the law of $S^\alpha_n$. Then, for all $n \geq 1$ big enough, 
\begin{align}\label{eq:explicit_rate_convergence_NDA_w1_alldim} 
W_1(\mu^\alpha_n , \mu_\alpha) & \leq \frac{C_{\alpha,\beta,d}(k)}{n^{\frac{2}{\alpha}-1}},
\end{align}
for some $C_{\alpha,\beta,d}(k)>0$ depending only on $\alpha$, $k$, $\beta$ and on $d$. 
\end{thm}
\noindent
The conditions in Theorem \ref{thm:all_dimensions_NDS} on the function $k$ are tested in Section \ref{sec:stability_alpha_stable_pm_12} in dimension $1$ on two Pareto-type distributions as initial law.~Indeed, the one-sided Pareto and the double Pareto distributions are self-decomposable on $\mathbb{R}_+$ and on $\mathbb{R}$ respectively. Moreover, they belong to the extended generalized gamma convolutions (EGGC) class (see \cite{Th_77,Bon_92,Steutel_Harn_04}). In particular, the second order asymptotic expansion \eqref{eq:second_order_expansion_NDS} is natural in this setting and follows from a refined analysis of the behavior at $0^+$ of the Lebesgue densities of the associated (extended) Th\"orin measures.~Moreover, the next two results go beyond Theorem \ref{thm:all_dimensions_NDS} since the $k$-functions of the initial laws do not satisfy the condition $\sup_{z \in (0,+\infty)} (z^\alpha k(z)) <+\infty$. Indeed, these initial laws are two simple examples of layered stable distributions whose $k$-functions exhibit two different behaviors at $0^+$ and at $+\infty$ depending on an inner exponent $\beta$ and on an outer exponent $\alpha$. The proofs of Theorem \ref{thm:quantitative_approximation_Ls} and Theorem \ref{thm:quantitative_approximation_CauchyLs} are postponed to Section \ref{sec:stability_alpha_stable_pm_12} and Subsection \ref{sec:SE_NDS_Cauchy} respectively.  

\begin{thm}\label{thm:quantitative_approximation_Ls}
Let $d \geq 1$ be an integer, let $\alpha \in (1,2)$ and let $\mu_\alpha$ be a non-degenerate symmetric $\alpha$-stable probability measure on $\mathbb{R}^d$ with associated L\'evy measure $\nu_\alpha$ which spherical component is denoted by $\sigma$. Let $\beta \in (\alpha,2)$ and let $\mu^L_{\alpha,\beta}$ be the probability measure on $\mathbb{R}^d$ which Fourier transform is given, for all $\xi \in \mathbb{R}^d$, by
\begin{align*}
\widehat{\mu^L_{\alpha,\beta}}(\xi) = \exp \left( \int_{\mathbb{R}^d} \left(e^{i \langle \xi ; u \rangle} - 1 - i \langle \xi ;u \rangle\right) \nu^L_{\alpha,\beta}(du)\right),
\end{align*}
with,
\begin{align}\label{eq:Levymeasure_L} 
\nu^L_{\alpha,\beta}(du) =\bbone_{(0,+\infty)}(r)\bbone_{\mathbb{S}^{d-1}}(y) \left(\frac{1}{r^{\beta}}\bbone_{(0,1]}(r) + \frac{1}{r^\alpha}\bbone_{(1,+\infty)}(r)\right) \frac{dr}{r} \sigma(dy).
\end{align}
Let $(Z_k)_{k \geq 1}$ be a sequence of i.i.d. random vectors of $\mathbb{R}^d$ such that $Z_1 \sim \mu^L_{\alpha,\beta}$ and let $(S_n)_{n \geq 1}$ be the sequence of random vectors defined, for all $n \geq 1$, by
\begin{align*}
S_n  = \frac{1}{n^{\frac{1}{\alpha}}} \sum_{k = 1}^n Z_k. 
\end{align*}
Let us denote by $\mu_n^{\alpha,\beta}$ the law of $S_n$, for all $n \geq 1$.~Then, for all $n \geq 1$ big enough, 
\begin{align}\label{ineq:rate_W1_LayeredSimple}
W_1(\mu_n^{\alpha, \beta} , \mu_\alpha) \leq \frac{C_{\alpha, \beta, d}}{n^{\frac{2}{\alpha}-1}},
\end{align}
for some $C_{\alpha,\beta,d}>0$ depending on $\alpha$, on $\beta$ and on $d$. 
\end{thm}

\begin{thm}\label{thm:quantitative_approximation_CauchyLs}
Let $d \geq 1$ be an integer and let $\mu_1$ be a non-degenerate symmetric $1$-stable probability measure on $\mathbb{R}^d$ with associated L\'evy measure $\nu_1$ which spherical component is denoted by $\sigma$.~Let $\beta \in (1,2)$ and let $\mu^L_{1,\beta}$ be the probability measure on $\mathbb{R}^d$ which Fourier transform is given, for all $\xi \in \mathbb{R}^d$, by
\begin{align}\label{eq:cf_layered_cauchy}
\widehat{\mu_{1,\beta}^L}(\xi) = \exp \left(\int_{\mathbb{R}^d} \left(e^{i \langle u ; \xi \rangle} - 1 - i \langle u ; \xi \rangle\bbone_{\|u\| \leq 1}\right) \nu_{1,\beta}^{L}(du)\right),
\end{align}
with, 
\begin{align*}
\nu_{1, \beta}^L(du) =\bbone_{(0,+\infty)}(r)\bbone_{\mathbb{S}^{d-1}}(y) \left( \frac{1}{r^{\beta}}\bbone_{(0,1]}(r) + \frac{1}{r}\bbone_{(1,+\infty)}(r)\right) \dfrac{dr}{r} \sigma(dy).
\end{align*}
Let $(Z_k)_{k \geq 1}$ be a sequence of i.i.d. random vectors of $\mathbb{R}^d$ such that $Z_1 \sim \mu_{1, \beta}^L$ and let $(S_n)_{n \geq 1}$ be the sequence of random vectors of $\mathbb{R}^d$ defined, for all $n \geq 1$, by 
\begin{align}\label{def:Sn_Cauchy_case}
S_n = \frac{1}{n} \sum_{k=1}^n Z_k. 
\end{align}
Let us denote by $\mu_{1,\beta}^{L,n}$ the law of $S_n$, for all $n \geq 1$.
Then, for all $n \geq 2$,
\begin{align}\label{ineq:rate_convergence_cauchy_case}
d_{W_2}(\mu_{1,\beta}^{L,n} , \mu_1) \leq \frac{C_{\beta,d}}{n},
\end{align}
for some positive constant $C_{\beta,d}$ depending on $\beta$ and on $d$ only.
\end{thm}

\subsection{Literature review on quantitative stable approximations}\label{subsec:literature_review}
\noindent
In dimension $1$, there have been, recently, several works dealing with Stein's method for stable approximation.~First, \cite[Theorems $2.1$ and $2.6$]{Lihu_Xu_19} provide quantitative $1$-Wasserstein bounds for the generalized central limit theorem when the limit distribution is a symmetric $\alpha$-stable one with $\alpha \in (1,2)$.~For several examples, accurate rates of convergence are computed explicitly; in particular, when the initial law is a ``symmetric Pareto" distribution with Lebesgue density given, for all $x \in \mathbb{R}$, by  
\begin{align*}
g_\alpha(x) = \left\{
    \begin{array}{ll}
       0  & \mbox{if}\, |x| \leq 1 ,\\
        \frac{\alpha}{2|x|^{\alpha+1}} & \mbox{if}\, |x|>1,
    \end{array}
\right.
\end{align*} 
a rate of order $n^{-(2/\alpha-1)}$ is obtained (see \cite[Example $1$]{Lihu_Xu_19}).~Note that this Pareto distribution is not infinitely divisible (see \cite[Remark $4.6$]{AH18_1} for the details). The proof relies on Stein's method for symmetric $\alpha$-stable law with $\alpha \in (1,2)$ combined with a generalization of the K-function approach (see, e.g., \cite[Chapters $2.3$ and $2.4$]{Chen_Goldstein_Shao_11}). This approach has been generalized to one-dimensional self-decomposable distributions with finite first moment in \cite[Chapters $5$ and $6$]{AH18_1}. Let us mention as well the works \cite{Chen_Nourdin_Xu_21,Chen_Nourdin_Xu_Yang_Zhang_22} where the cases of non-symmetric $\alpha$-stable distributions with $\alpha \in (1,2)$ (\cite[Theorems $1.4$ and $1.6$]{Chen_Nourdin_Xu_21}) and non-integrable $\alpha$-stable distributions with $\alpha \in (0,1]$ (\cite[Theorem 4]{Chen_Nourdin_Xu_Yang_Zhang_22}) are investigated by means of Stein's method and the leave-one-out approach. Finally, still in dimension one and for $\alpha \in (0,2)$, the authors of \cite{Chen_Xu_19} use an adaptation of the Lindeberg principle together with Taylor-like arguments to obtain rates of convergence in smooth Wasserstein distance (see \cite[Theorem $1.4$ and Theorem $1.7$]{Chen_Xu_19}).

In dimension $d \geq 2$, there are not that many results in the literature dealing with quantitative stable approximations.~\cite[Theorem $3.2$ and Theorem $3.3$]{Davydov_Nagaev_02} provide rates of convergence in total variation distance and in $\|\cdot\|_\infty$-distance (at the level of densities) for initial laws which are $\nu$-Paretian (see \cite[Definition $3.1$]{Davydov_Nagaev_02}). In particular, a rate of convergence of order $n^{-\beta_1}$ is obtained in \cite[Theorem $3.2$]{Davydov_Nagaev_02} with $\beta_1 = \min\left(\alpha, 2-\alpha\right)/(\alpha+d)$. It is conjectured in there that a rate of order $n^{-(2/\alpha-1)}$, when $\alpha \in (1,2)$, (independent of the dimension) should hold. In \cite{Chen_Nourdin_Xu_Yang_23}, extending to the multivariate setting the methodology developed in \cite{Lihu_Xu_19,Chen_Nourdin_Xu_21,Chen_Nourdin_Xu_Yang_Zhang_22}, the authors obtained $1$-Wasserstein bound when the target distribution is a strictly $\alpha$-stable probability measure on $\mathbb{R}^d$ with $\alpha \in (1,2)$ and with Fourier transform given, for all $\xi \in \mathbb{R}^d$, by
\begin{align}\label{eq:characteristic_function_CNXY23}
\widehat{\mu_\alpha}(\xi) = \exp\left( \int_{\mathbb{R}^d} \left(e^{i \langle u ; \xi \rangle}-1- i \langle u ; \xi \rangle\right) \nu_\alpha(du)\right),
\end{align}
where $\nu_\alpha$ is the associated L\'evy measure with polar decomposition,
\begin{align*}
\nu_\alpha(du) =\bbone_{(0,+\infty)}(r)\bbone_{\mathbb{S}^{d-1}}(y) \frac{dr}{r^{1+\alpha}} \sigma(dy).
\end{align*}
The L\'evy measure is assumed to be symmetric and to be a $\gamma$-measure with $\gamma \in [1,d]$ such that $\gamma>d-\alpha$ (see \cite[Lemma $2$]{Chen_Nourdin_Xu_Yang_23}).~In particular, note that, in \cite[Lemma $2$]{Chen_Nourdin_Xu_Yang_23}, the Lebesgue density of $\mu_\alpha$ is assumed to exist and so the associated spectral measure satisfies the ellipticity condition \eqref{eq:ND_CNS}.~Under the previous assumptions and for initial laws which are $\nu$-Paretian, a rate of order $n^{-(2/\alpha-1)}$ is obtained in \cite[Theorem $16$]{Chen_Nourdin_Xu_Yang_23}. At this point, let us stress the main differences between Theorem \ref{thm:all_dimensions_NDS} and \cite[Theorem $16$]{Chen_Nourdin_Xu_Yang_23}. As already mentioned in the introduction (see also Remark \ref{rem:Stein_Equation_Solution_NDS}), the whole set of non-degenerate symmetric $\alpha$-stable probability measures with $\alpha \in (1,2)$ and with $d \geq 1$ is attained by Theorem \ref{thm:all_dimensions_NDS} whereas \cite[Theorem $16$]{Chen_Nourdin_Xu_Yang_23} is restricted to non-degenerate symmetric $\alpha$-stable probability measures whose L\'evy measures are $\gamma$-measures with $\gamma \in (d-\alpha,d]$. The assumptions regarding the initial laws are of different nature: in Theorem \ref{thm:all_dimensions_NDS}, the Fourier transform of the initial law $\mu$ is linked to the Fourier transform of the target non-degenerate symmetric $\alpha$-stable probability measure, with $\alpha \in (1,2)$.~Conversely, \cite[Theorem $16$]{Chen_Nourdin_Xu_Yang_23} imposes a rigid spatial structure on the density of the initial law: it has to be $\nu$-Paretian in the sense of \cite[Definition 2]{Chen_Nourdin_Xu_Yang_23}.

Next, let us discuss briefly the non-degenerate symmetric Cauchy ($\alpha = 1$) case in dimension $d \geq 2$. As already mentioned previously, in \cite[Theorem $3.2$]{Davydov_Nagaev_02}, the rate $n^{-\beta_1}$, with  $\beta_1  = 1/(d+1)$, is obtained in total variation distance for initial laws which are $\nu$-Paretian. It is conjectured in there that a rate of order $n^{-1}$ (independent of the dimension) should hold. We answer positively this conjecture in the smooth $2$-Wasserstein distance and for an initial law which is a specific instance of layered stable distributions.~In \cite[Theorem $16$]{Chen_Nourdin_Xu_Yang_23}, a rate of order $(\log(n))^2/n$ is obtained in a different Wasserstein-type distance and for initial laws which are $\nu$-Paretian.~Still regarding the Cauchy case and after a first version of this manuscript was posted on the arXiv, \cite[Theorems 1.1 and 1.2]{LXY_23} provide a rate of order $n^{-1}$ in total variation distance under mild assumptions on the initial laws and on the spherical component of the target L\'evy measure (see \cite[Assumptions I and II]{LXY_23}). In particular, the spectral measure of the target L\'evy measure must be either (i) absolutely continuous with respect to the surface measure on $\mathbb{S}^{d-1}$ with a density uniformly lower and upper bounded or (ii) a symmetric $\gamma-1$-measure with $\gamma \in (d-1, d] \cap [1, +\infty[$. 

Finally, let us briefly discuss previous results regarding quantitative stable approximation.~In dimension one, the author of \cite{Hall_81} obtained Kolmogorov distance bounds when the initial law belongs to the domain of normal attraction of an $\alpha$-stable probability measure on $\mathbb{R}$ with $\alpha \in (1,2)$ (see \cite[Theorem $1$]{Hall_81}).~In particular, necessary and sufficient conditions on the remainders in the tail sum and in the tail difference ensure a rate of convergence of order $n^{-(2/\alpha-1)}$ (see, \cite[Corollary $1$]{Hall_81}). Still in dimension one and for the Kolmogorov distance, the authors of \cite{Juo_Paulauskas_98} considered for the initial law perturbations of the one-sided and of the two-sided Pareto distributions.~Depending on the order of the perturbation, different rates of convergence are observed.~Regarding the $1$-Wasserstein distance, \cite[Theorem $1.2$]{Johnson_Samworth_05} ensures a rate of order $n^{-(1/\alpha-1/\beta)}$ when $\alpha \in (1,2)$ in the $\beta$-Wasserstein distance, for some $\beta \in (\alpha, 2]$, when the initial law belongs to the strong domain of normal attraction of a non-symmetric $\alpha$-stable probability measure on $\mathbb{R}$ (see, \cite[Definition $5.2$]{Johnson_Samworth_05}).

\subsection{Side results}
\noindent
Before ending this introduction, let us focus on one of the conditions of the stability estimates of Theorem \ref{thm:stability_estimate_second_moment} and Theorem \ref{thm:stability_estimate_NDS_stable}; namely, $\mu_X \ast \omega_X\nu << \mu_X$, where $\mu_X$ is a probability measure, $\omega_X$ is the weight function and $\nu$ is the target L\'evy measure.~This condition implies that the bilinear form defined, for all $f_1,f_2 \in \mathcal{C}^1_b(\mathbb{R}^d, \mathbb{R}^d)$, by 
\begin{align*}
\mathcal{E}_{\nu,\omega_X}(f_1,f_2) = \int_{\mathbb{R}^d} \int_{\mathbb{R}^d} \langle f_1(x+u) - f_1(x) ; f_2(x+u) - f_2(x) \rangle \omega_X(\|u\|)\nu(du) \mu_X(dx),
\end{align*}
with domain $\mathcal{C}^1_b(\mathbb{R}^d, \mathbb{R}^d)$ is closable when it is considered as a bilinear non-negative definite symmetric form on $L^2(\mu_X, \mathbb{R}^d)$, the set of equivalence classes of $\mathbb{R}^d$-valued functions which are Borel measurable and square-integrable with respect to $\mu_X$.~As discussed in Remark \ref{rem:applications_forms}, (ii), under this assumption, there are several closed extensions to $(\mathcal{E}_{\nu,\omega_X},\mathcal{C}^1_b(\mathbb{R}^d, \mathbb{R}^d))$ which do not coincide \textit{a priori} in this general setting. Nevertheless, in Proposition \ref{prop:markov_uniqueness_rot_inv} of the Appendix section, we prove that the smallest closed extension and another classical Dirichlet form do coincide when $\mu_X = \mu_{\alpha}^{\operatorname{rot}}$, $\omega_X = 1$ and $\nu = \nu_\alpha^{\operatorname{rot}}$ as defined in \eqref{eq:characteristic_function_rot_inv}. In a diffusive setting, this property is known as Markov uniqueness (see, e.g., \cite[Chapter $3$, Section $3.3$]{FOT_10}).

Related to this question is the one discussed in Remark \ref{rem:discussion_condition_constant_weight_function}, (iv), regarding semigroups naturally associated with the probability measures $\mu_\alpha$, with $\alpha \in (1,2)$. In particular, Theorem \ref{thm:thesame1} of the Appendix section ensures that the semigroup generated by the smallest closed extension of $(\mathcal{E}_{\nu_\alpha, \mu_\alpha},\mathcal{C}^1_b(\mathbb{R}^d))$ (defined in \eqref{eq:form_nua_mua}) and the ``carr\'e de Mehler" semigroup recently put forward in \cite{AH20_4} and recalled in \eqref{eq:carre_mehler_sg} coincide when the logarithmic derivative of the positive Lebesgue density of $\mu_\alpha$ is uniformly bounded on $\mathbb{R}^d$. Interestingly, in the rotationally invariant situation, the generator of the ``carr\'e de Mehler" semigroup coincides with the one associated with the closure of $(\mathcal{E}_{\nu^{\operatorname{rot}}_\alpha, \mu^{\operatorname{rot}}_\alpha},\mathcal{C}^1_b(\mathbb{R}^d))$ and with the one associated with the form $\left(\mathcal{E}_{\nu^{\operatorname{rot}}_\alpha, \mu^{\operatorname{rot}}_\alpha} , D\left(\mathcal{E}_{\nu^{\operatorname{rot}}_\alpha, \mu^{\operatorname{rot}}_\alpha}\right)\right)$ defined by 
\begin{align*}
D\left(\mathcal{E}_{\nu^{\operatorname{rot}}_\alpha, \mu^{\operatorname{rot}}_\alpha}\right) = \{f \in L^2 \left(\mu_\alpha^{\operatorname{rot}}\right): \, \mathcal{E}_{\nu^{\operatorname{rot}}_\alpha, \mu^{\operatorname{rot}}_\alpha}(f,f)<+\infty \}. 
\end{align*} 
When the reference measure is the standard Gaussian measure $\gamma$, the Ornstein-Ulhenbeck operator $\mathcal{L}^\gamma:=- \langle x ; \nabla \rangle + \Delta$ is essentially self-adjoint. 

Finally, in subsection \ref{sec:L2_Spec_Prop}, a spectral analysis is performed regarding the non-local operators associated with the non-degenerate symmetric $1$-stable probability measures on $\mathbb{R}^d$, with a particular emphasis on the rotationally invariant case. This subsection starts with a semigroup proof of the Poincar\'e-type inequality for non-degenerate symmetric Cauchy measures (see Proposition \ref{prop:Poinc_type_Cauchy}). Then, the subsection focuses on the Ornstein-Uhlenbeck and on the ``carr\'e de Mehler" generators when the reference Cauchy measure is rotation invariant. In particular, Proposition \ref{prop:aps_ou_stable} and Proposition \ref{prop:approximate_point_spectrum} provide asymptotic analysis of spectral quantities related to these generators along the sequence of smooth truncations defined by $g_R(x) = x \exp\left(- \|x\|^2/R^2\right)$, for all $x \in \mathbb{R}^d$ and all $R>0$. Proposition \ref{lem:dual_Cauchy_SG} and Lemma \ref{lem:action_square_root_laplacian} furnish exact formulas which are central to the aforementioned asymptotic analysis and which can be of independent interest.  

\subsection{Organization}
\noindent
Let us further describe the content of these notes.~In the next section, we introduce the notations and the definitions used throughout the manuscript and we recall and prove some preliminary results regarding Stein's method for self-decomposable probability measures on $\mathbb{R}^d$ and weighted Poincar\'e-type inequalities.~In Section \ref{sec:stability_estimate_finite_second_moment}, the stability estimates for the non-degenerate centered self-decomposable probability measures with finite second moment are considered while Section \ref{sec:stability_alpha_stable_pm_12} deals with the corresponding results for the non-degenerate symmetric $\alpha$-stable probability measures with $\alpha \in (1,2)$. Section \ref{sec:spectral_cauchy} is divided into two sub-sections: in sub-section \ref{sec:L2_Spec_Prop}, we investigate the $L^2$-spectral properties of several non-local operators associated with the non-degenerate symmetric Cauchy probability measures, with a special focus on the rotationally invariant one, while in sub-section \ref{sec:SE_NDS_Cauchy}, we prove the corresponding stability estimate together with its direct applications. Finally, the manuscript ends with an appendix section gathering technical results used throughout these notes.

\section{Notations and Preliminaries}\label{sec:notations_preliminaries}
\noindent
Let us now introduce the notations and some preliminary results which will be used throughout the text. For all integer $d \geq 1$, let $\| \cdot \|$ be the Euclidean norm on $\mathbb{R}^d$. Let $\mathcal{S}(\mathbb{R}^d)$ be the Schwartz space of functions which are infinitely differentiable on $\mathbb{R}^d$ and such that, for all $\alpha \geq 0$ and all $\beta \in \mathbb{N}^d$, (the subset of vectors of $\mathbb{R}^d$ with non-negative integer coordinates), 
\begin{align}\label{eq:semi_norms_Schwartz}
\|f\|_{\alpha,\beta} := \underset{x \in \mathbb{R}^d}{\sup} \left| \left(1+\|x\|\right)^{\alpha} D^\beta(f)(x)\right| < + \infty,
\end{align}
where $D^\beta = \partial_{x_1}^{\beta_1} \dots  \partial_{x_d}^{\beta_d} $ is the derivative operator of order $|\beta| = \beta_1 + \dots + \beta_d$ and where $ \partial_{x_k}^{\beta_k}$, with $k \in \{1, \dots, d\}$, is the partial derivative operator of order $\beta_k$ in the coordinate $x_k$.~Let $\mathcal{C}_c^{\infty}(\mathbb{R}^d)$ be the set of functions which are infinitely differentiable on $\mathbb{R}^d$ with compact support. In the sequel, $\mathcal{F}$ denotes the Fourier transform operator which is an isomorphism on $\mathcal{S}(\mathbb{R}^d)$ and which is defined, for all $f \in \mathcal{S}(\mathbb{R}^d)$ and all $\xi \in \mathbb{R}^d$, by
\begin{align}\label{eq:def_Fourier_transform}
\mathcal{F}(f)(\xi) = \int_{\mathbb{R}^d} f(x) e^{- i \langle \xi ; x \rangle} dx. 
\end{align} 
The inverse Fourier transform is given, for all $f \in \mathcal{S}(\mathbb{R}^d)$ and all $x \in \mathbb{R}^d$, by
\begin{align}\label{eq:defi_inverse_Fourier_transform}
f(x)  = \frac{1}{(2\pi)^d} \int_{\mathbb{R}^d} \mathcal{F}(f)(\xi) e^{i \langle \xi ;  x \rangle} d\xi. 
\end{align}
For a Borel probability measure $\mu$ on $\mathbb{R}^d$ and for all $p \in [1, +\infty)$, let $L^p(\mu)$ be the Banach space of equivalence classes of functions which are real-valued, $\mathcal{B}(\mathbb{R}^d)$-measurable and such that 
\begin{align}\label{eq:def_norm_Lp_pfinite}
\|f\|_{L^p(\mu)} : = \left(\int_{\mathbb{R}^d} \left| f(x)\right|^p \mu(dx) \right)^{\frac{1}{p}} < + \infty. 
\end{align}
For a linear operator $T$ between two Banach spaces $\left(\mathcal{X} , \|\cdot\|_{\mathcal{X}}\right)$ and $\left( \mathcal{Y} , \|\cdot\|_{\mathcal{Y}}\right)$, the operator norm is denoted by $\| \cdot \|_{\mathcal{X} \rightarrow \mathcal{Y}}$ and is defined by 
\begin{align}\label{eq:def_operatornorm}
\|  T \|_{\mathcal{X} \rightarrow \mathcal{Y}} : = \underset{x \in \mathcal{X},\, \|x\|_{\mathcal{X}} \ne 0}{\sup} \dfrac{\left\|  T(x) \right\|_{\mathcal{Y}}}{\|x\|_{\mathcal{X}}}.
\end{align}
Similarly, for any $m$-multilinear form $F$ defined on $(\mathbb{R}^d)^m$ and real-valued, the operator norm is defined by
\begin{align}\label{eq:def_operatornom_rform}
\|F\|_{\operatorname{op}} := \sup \{|F(u_1 , \dots, u_m)| : \, \|u_i\| = 1, \, i \in \{1, \dots, m\}\}. 
\end{align}

In the sequel, we are interested in a very specific set of probability measures on $\mathbb{R}^d$, namely, the infinitely divisible ones. A random vector $X$ from a probability space $\left(\Omega, \mathcal{F}, \mathbb{P}\right)$ to $(\mathbb{R}^d , \mathcal{B}(\mathbb{R}^d))$ is said to be infinitely divisible with triplet $(b , Q, \nu)$ if its characteristic function $\varphi_X$ admits the following representation:  for all $\xi \in \mathbb{R}^d$,
\begin{align}\label{eq:cf_inf_div}
\varphi_X(\xi) = \exp \left( i \langle \xi ; b \rangle - \frac{\langle Q\xi ; \xi \rangle}{2} + \int_{\mathbb{R}^d} \left(e^{i \langle u ; \xi \rangle} - 1 - i \langle u ; \xi \rangle\bbone_{\|u\| \leq 1}\right) \nu(du) \right), 
\end{align}
where $b$ is a vector of $\mathbb{R}^d$, $Q$ is a non-negative definite symmetric matrix of size $d \times d$ with real entries and $\nu$ is a Borel positive measure on $\mathbb{R}^d$ such that 
\begin{align}\label{eq:defi_Levy_measure}
\nu \left(\{0\}\right) = 0 ,\quad \int_{\mathbb{R}^d} \left(1 \wedge \|u\|^2\right) \nu(du) < +\infty. 
\end{align}
The representation \eqref{eq:cf_inf_div} is mainly the one to be used with the (unique) generating triplet
$(b,Q,\nu)$.~However, other types of representations are also possible 
and two of them are presented
next.  First, if $\nu$ is such that $\int_{\|u\|\le 1} \|u\|\nu (du)<+\infty$,
then \eqref{eq:cf_inf_div} becomes
\begin{equation}\label{eq:cf_inf_div2}
\varphi_X(\xi)=\exp\left(i \langle b_0;\xi \rangle-\frac{1}{2}\langle \xi;Q \xi \rangle +\int_{\mathbb{R}^d}\left(e^{i \langle \xi; u\rangle}-1\right)\nu(du)
\right),
\end{equation}
where $b_0 = b-\int_{\|u\|\le 1}u\nu (du)$ is called the {\it drift} of $X$. This representation is expressed as $X\sim ID(b_0,Q,\nu)_0$.  Second, if $\nu$ is such that $\int_{\|u\|>1} \|u\|\nu(du)<+\infty$, then \eqref{eq:cf_inf_div} becomes
\begin{equation}\label{eq:cf_inf_div3}
\varphi_X(\xi)=\exp\left(i \langle b_1;\xi \rangle-\frac{1}{2}\langle \xi;Q \xi \rangle +\int_{\mathbb{R}^d}\left(e^{i \langle \xi; u\rangle}-1-i\langle \xi;u \rangle\right)\nu(du)
\right),
\end{equation}
where $b_1=b+\int_{\|u\|>1} u\nu (du)$ is called the {\it center} of $X$. In turn, this last representation is now written as $X\sim ID(b_1,Q,\nu)_1$. In fact, $b_1=\mathbb{E} X$  and, for any
$p>0$, $\mathbb{E} \|X\|^p<+\infty$ is equivalent to $\int_{\|u\|>1}\|u\|^p\nu (du)
<+\infty$.  Also, for any $r>0$, $\mathbb{E} e^{r\|X\|}<+\infty$ is equivalent to $$\int_{\|u\|>1} e^{r\|u\|}\nu (du)<+\infty.$$ 
In the sequel, we assume that the Gaussian part of the ID random vectors under consideration is null, namely, $Q = 0$.

Let us illustrate briefly the subset of non-degenerate symmetric $\alpha$-stable probability measures on $\mathbb{R}^d$ for which our methodology applies when $\alpha \in (1,2)$. Let $\mu_\alpha^{\operatorname{rot}}$ be the probability measure defined through its Fourier transform, for all $\xi \in \mathbb{R}^d$, by 
\begin{align}\label{eq:characteristic_function_rot_inv}
\widehat{\mu^{\operatorname{rot}}_\alpha} \left( \xi \right) = \exp \left( - \frac{\|\xi\|^\alpha}{2}\right). 
\end{align}
Clearly, $\mu_\alpha^{\operatorname{rot}}$ is a non-degenerate symmetric $\alpha$-stable probability measure on $\mathbb{R}^d$ with spectral measure proportional to the spherical part of the Lebesgue measure on $\mathbb{S}^{d-1}$. Another interesting example of a non-degenerate symmetric $\alpha$-stable probability measure on $\mathbb{R}^d$ is the product measure $\mu_{\alpha,d}$ defined through its Fourier transform, for all $\xi \in \mathbb{R}^d$, by 
\begin{align}\label{eq:characteristic_function_ind}
\widehat{\mu_{\alpha,d}} \left( \xi \right) = \exp \left( - \|\xi\|^\alpha_\alpha\right),  
\end{align}
where $\|\xi\|^\alpha_\alpha = \sum_{k} |\xi_k|^\alpha$. $\mu_{\alpha,d}$ is the law of an $\alpha$-stable random vector of $\mathbb{R}^d$ with independent coordinates which marginals are distributed according to a one-dimensional symmetric $\alpha$-stable probability measure on $\mathbb{R}$. 

The class of SD probability measures on $\mathbb{R}^d$ is naturally connected with limit theorems. More precisely, thanks to \cite[Theorem 15.3]{Sato_15}, if $(Z_n)_{n \geq 1}$ is a sequence of independent random vectors of $\mathbb{R}^d$, if $(b_n)_{n \geq 1}$ is a sequence of positive reals and if $(c_n)_{n \geq 1}$ is a sequence of deterministic vectors of $\mathbb{R}^d$ such that the sequence $(S_n)_{n \geq 1}$ defined, for all $n \geq 1$, by
\begin{align}\label{eq:def_sum}
S_n = b_n \sum_{k=1}^n Z_k + c_n,
\end{align}
converges in law to a probability measure $\mu$ on $\mathbb{R}^d$ and such that $\{b_nZ_k:\, k \in \{1, \dots, n\},\, n \in \{1, \dots, \}\}$ is a null array, then $\mu$ is SD. Conversely, for all SD probability measure $\mu$ on $\mathbb{R}^d$, it is possible to find $(Z_n)_{n \geq 1}$, $(b_n)_{n \geq 1}$ and $(c_n)_{n \geq 1}$ as previously such that $S_n \overset{\mathcal{L}}{\longrightarrow} \mu$, as $n$ tends to $+\infty$ (here and in the sequel, $ \overset{\mathcal{L}}{\longrightarrow}$ denotes convergence in law). Moreover, if $\mu$ is non-trivial, then $b_n \rightarrow 0$ and $b_{n+1} /b_n \rightarrow 1$, as $n$ tends to $+\infty$ (see \cite[Lemma 15.4.]{Sato_15}).~In the next sections, we will be interested in the following simple limit theorem: let $\mu$ be a non-trivial SD probability measure on $\mathbb{R}^d$ and let $\hat{\mu}$ be its Fourier transform. Let $(Z_k)_{k \geq 1}$ be a sequence of independent random vectors of $\mathbb{R}^d$ with characteristic functions $\varphi_k$ given, for all $\xi \in \mathbb{R}^d$ and all $k \geq 1$, by 
\begin{align}\label{eq:def_cf_k}
\varphi_k(\xi) = \dfrac{\hat{\mu}((k+1)\xi)}{\hat{\mu}(k \xi)}.  
\end{align} 
Let $(S_n)_{n \geq 1}$ be the sequence of random vectors defined, for all $n \geq 1$, by 
\begin{align}\label{eq:def_Sn_simple}
S_n = \frac{1}{n} \sum_{k=1}^n Z_k.
\end{align}
Then, standard Fourier analysis ensures that $(S_n)_{n \geq 1}$ converges in law to $\mu$. 

One of the main objectives of these notes is to provide quantitative versions of the previous limit theorems by means of Stein's method combined with spectral methods.~Stein's method for SD and for $\alpha$-stable probability measures on $\mathbb{R}^d$, $d\geq 1$, has been developed recently (see \cite{AH18_1,AH19_2,AH20_3,AH22_5}). Let us recall the metrics of convergence investigated in \cite{AH18_1,AH19_2,AH20_3,AH22_5} as well as the main results regarding the Stein equation for SD and for $\alpha$-stable probability measures and the regularity estimates of the associated solution.~The $1$-Wasserstein distance between two probability measures on $\mathbb{R}^d$, $\mu_1$ and $\mu_2$, with finite first moment, is defined by
\begin{align}\label{eq:def_W1}
W_1(\mu_1, \mu_2) : = \inf \int_{\mathbb{R}^d} \|x-y\| \pi(dx,dy), 
\end{align}
where the infimum is taken over the set of probability measures $\pi$ on $\mathbb{R}^{2d}$ such that the first $d$-dimensional marginal is $\mu_1$ and the second one is $\mu_2$. By Kantorovich-Rubinstein duality theorem, the $1$-Wasserstein distance admits the following representation: for all $\mu_1$, $\mu_2$ probability measures on $\mathbb{R}^d$ with finite first moment,   
\begin{align}\label{eq:KR_duality}
W_1(\mu_1, \mu_2) = \underset{h \in \operatorname{Lip}_1}{\sup} \left| \int_{\mathbb{R}^d} h(x) \mu_1(dx) - \int_{\mathbb{R}^d} h(x) \mu_2(dx) \right|, 
\end{align}
where $\operatorname{Lip}_1$ is the set of Lipschitz functions on $\mathbb{R}^d$ with Lipschitz constant not greater than $1$. Now, thanks to \cite[Lemma $5.2$]{AH22_5}, the following representation formula holds true: for all $\mu_1$, $\mu_2$ probability measures on $\mathbb{R}^d$ with finite first moment, 
\begin{align}\label{eq:reduction_formula}
W_1(\mu_1, \mu_2) = \underset{h \in \mathcal{C}_c^{\infty}(\mathbb{R}^d), \, |h|_{\operatorname{Lip}}\leq 1}{\sup} \left| \int_{\mathbb{R}^d} h(x) \mu_1(dx) - \int_{\mathbb{R}^d} h(x) \mu_2(dx) \right|,
\end{align}
where $ |h|_{\operatorname{Lip}}$ is given by 
\begin{align}\label{eq:Lip_seminorm}
 |h|_{\operatorname{Lip}} = \underset{x,y \in \mathbb{R}^d,\, x\ne y}{\sup} \dfrac{|h(x) - h(y)|}{\|x - y\|}. 
\end{align}
Next, for all $r \geq 1$, let $\nabla^r$ be the $r$-th derivative seen as a $r$-multilinear form when acting on functions. For all $h$ from $\mathbb{R}^d$ to $\mathbb{R}$ which is $r$-times continuously differentiable on $\mathbb{R}^d$, let $M_\ell(h)$ be defined, for all $1\leq \ell \leq r$, by 
\begin{align}\label{eq:def_iso}
M_\ell(h) = \underset{x \in \mathbb{R}^d}{\sup} \left\| \nabla^\ell(h)(x)\right\|_{\operatorname{op}} = \underset{x\ne y}{\sup} \dfrac{\|\nabla^{\ell-1}(h)(x) - \nabla^{\ell-1}(h)(y)\|_{\operatorname{op}}}{\|x - y\|}. 
\end{align}
Moreover, set $M_0(h) = \sup_{x \in \mathbb{R}^d} |h(x)|$, for all $h$ bounded continuous function on $\mathbb{R}^d$. Now, for all $r \geq 0$, let $\mathcal{H}_r$ be the space of real-valued bounded functions defined on $\mathbb{R}^d$ which are $r$-times continuously differentiable on $\mathbb{R}^d$ and such that, 
\begin{align*}
\underset{0\leq \ell \leq r}{\max} M_\ell(h) \leq 1. 
\end{align*} 
Finally, for all $r \geq 1$, the smooth $r$-Wasserstein distance between two probability measures $\mu_1$ and $\mu_2$ is defined by 
\begin{align}\label{eq:def_smooth_wasserstein_dist}
d_{W_r}(\mu_1 , \mu_2) := \underset{h \in \mathcal{H}_r}{\sup} \left| \int_{\mathbb{R}^d} h(x) \mu_1(dx) - \int_{\mathbb{R}^d} h(x) \mu_2(dx) \right|. 
\end{align}
As for the $1$-Wasserstein distance, the smooth Wasserstein distances admit a reduction formula (see \cite[Lemma A.2 of the Appendix]{AH19_2}): Namely, for all $r \geq 1$ and all $\mu_1$, $\mu_2$ probability measures on $\mathbb{R}^d$, 
\begin{align}\label{eq:reduction_representation_smooth_wasserstein}
d_{W_r}(\mu_1 , \mu_2) := \underset{h \in \mathcal{H}_r \cap \mathcal{C}_c^{\infty}(\mathbb{R}^d)}{\sup} \left| \int_{\mathbb{R}^d} h(x) \mu_1(dx) - \int_{\mathbb{R}^d} h(x) \mu_2(dx) \right|. 
\end{align}
The following set of inequalities holds true: for all $r \geq 1$ and all $\mu_1$, $\mu_2$ probability measures on $\mathbb{R}^d$ with finite first moment,  
\begin{align}\label{ineq:order_metric}
d_{W_r}(\mu_1, \mu_2) \leq d_{W_1}(\mu_1, \mu_2) \leq W_1(\mu_1 , \mu_2). 
\end{align} 
At this point, we are ready to recall some of the main results contained in \cite{AH19_2} regarding Stein's method for SD and for $\alpha$-stable probability measures with finite first moment in a multivariate setting. Let us start with the general SD situation (see \cite[Proposition $3.5$ and Proposition $3.6$]{AH19_2}).

\begin{prop}\label{prop:Stein_Method_SD_finite_first_moment}
Let $X$ be a non-degenerate self-decomposable random vector in $\mathbb{R}^d$ without Gaussian component, with law $\mu_X$, with characteristic function $\varphi_X$, with L\'evy measure $\nu$ and such that $\mathbb{E} \|X\|<\infty$. Moreover, let the function $k_x$ given by \eqref{eq:polar_decomposition_SD} satisfies, for all $a,b \in (0,+\infty)$ with $a<b$, 
\begin{align}\label{eq:cond_kfunction}
\sup_{x\in \mathbb{S}^{d-1}}\sup_{r\in (a,b)}k_x(r)<+\infty.
\end{align}
Let $h\in \mathcal{H}_2\cap \mathcal{C}^{\infty}_c(\mathbb{R}^d)$ and let $(P^\nu_t)_{t\geq 0}$ be the semigroup of operators defined, for all $t \geq 0$ and all $x \in \mathbb{R}^d$, by
\begin{align}\label{eq:def_SG_SD}
P^\nu_t(h)(x) = \int_{\mathbb{R}^d} h\left(xe^{-t} + y\right)\mu_t(dy), \quad \widehat{\mu_t}(\xi) = \dfrac{\varphi_X(\xi)}{\varphi_X(e^{-t} \xi)}, \quad \xi \in \mathbb{R}^d. 
\end{align} 
Let $f_h$ be the function defined, for all $x \in \mathbb{R}^d$, by
\begin{align}\label{eq:solution_stein_equation_SD}
f_h(x) = - \int_0^{+\infty} (P_t^{\nu}(h)(x) - \mathbb{E} h (X)) dt.
\end{align}
Then, $f_h$ is a strong solution to the non-local partial differential equation: for all $x \in \mathbb{R}^d$
\begin{align}\label{eq:stein_equation_SD_prop}
\langle \mathbb{E} X-x;\nabla (f_h)(x) \rangle+\int_{\mathbb{R}^d}\langle \nabla(f_h)(x+u)-\nabla(f_h)(x) ;u\rangle \nu(du)=h(x)-\mathbb{E} h(X).
\end{align}
Moreover, $f_h$ is twice continuously differentiable on $\mathbb{R}^d$ and is such that
\begin{align}\label{ineq:regularity_estimate_SD}
M_1(f_h)\leq 1,\quad\quad M_2(f_h)\leq \frac{1}{2}.
\end{align}
\end{prop} 
\noindent
Regarding the $\alpha$-stable probability measures on $\mathbb{R}^d$, with $\alpha \in (1,2)$, we focus on the non-degenerate and symmetric ones in the sequel.~Moreover, a well-known integrated gradient estimate (see \cite[Theorem $2.1$]{KS_13}) reminiscent of the non-degenerate Gaussian case (see \cite[Lemma $4.3$]{AH22_5})~allows to consider $h \in \mathcal{C}_c^{\infty}(\mathbb{R}^d)$ with $|h|_{\operatorname{Lip}} \leq 1$. 

\begin{prop}\label{prop:Stein_Method_Stable_NDS}
Let $d \geq 1$ be an integer, let $\alpha \in (1,2)$ and let $X_\alpha$ be a non-degenerate symmetric $\alpha$-stable random vector of $\mathbb{R}^d$ without Gaussian component, with law $\mu_\alpha$ and with L\'evy measure $\nu_\alpha$.~Let $h \in \mathcal{C}^{\infty}_c(\mathbb{R}^d)$ be such that $\left| h\right|_{\operatorname{Lip}} \leq 1$ and let $(P^{\nu_\alpha}_t)_{t \geq 0}$ be the semigroup of operators defined, for all $t \geq 0$ and all $x \in \mathbb{R}^d$, by 
\begin{align}\label{eq:def_SG_stable}
P_t^{\nu_\alpha}(h)(x) = \int_{\mathbb{R}^d} h \left(xe^{-t} + \left(1-e^{- \alpha t}\right)^{\frac{1}{\alpha}} y\right) \mu_\alpha(dy).
\end{align}
Let $f_h$ be given, for all $x \in \mathbb{R}^d$, by 
\begin{align}\label{eq:solution_stein_equation_NDS_stable}
f_h(x) = - \int_0^{+\infty} (P_t^{\nu_\alpha}(h)(x) - \mathbb{E} h (X_\alpha)) dt.
\end{align}
Then, $f_h$ is well-defined, twice continuously differentiable on $\mathbb{R}^d$ and is a strong solution to the non-local partial differential equation: for all $x \in \mathbb{R}^d$, 
\begin{align}\label{eq:stein_equation_Stable}
- \langle x;\nabla (f_h)(x) \rangle+\int_{\mathbb{R}^d}\langle \nabla(f_h)(x+u)-\nabla(f_h)(x) ;u\rangle \nu_\alpha(du)=h(x)-\mathbb{E} h(X_\alpha).
\end{align}
Finally, 
\begin{align}\label{ineq:regularity_estimate_NDS_stable}
M_1(f_h) \leq 1, \quad M_2(f_h) \leq C_{\alpha,d},
\end{align}
for some $C_{\alpha,d}>0$ depending on $\alpha$ and on $d$ only.
\end{prop}

\begin{proof}
The proof is a combination and an adaptation of \cite[Proposition $3.5$ and Proposition $3.6$]{AH19_2} using the fact that $h \in \mathcal{C}^{\infty}_c(\mathbb{R}^d)$ with $\left| h\right|_{\operatorname{Lip}} \leq 1$ together with an integration by parts and with 
\begin{align}\label{eq:integrated_gradient_estimate}
\int_{\mathbb{R}^d} \|\nabla(p_\alpha)(x)\| dx < + \infty,
\end{align}
where $p_\alpha$ is the Lebesgue density of the probability measure $\mu_\alpha$.
\end{proof}

\begin{rem}\label{rem:Stein_Equation_Solution_NDS}
Let us compare Proposition \ref{prop:Stein_Method_Stable_NDS} with the corresponding results available in the literature.~In \cite[Theorem $3$, (i)]{Chen_Nourdin_Xu_Yang_23}, regularity estimates for the solution to the Stein equation \eqref{eq:stein_equation_Stable} are obtained when the target law is a multivariate strictly $\alpha$-stable probability measure, with $\alpha \in (1,2)$, on $\mathbb{R}^d$ which satisfies the following set of assumptions (see \cite[Lemma 2]{Chen_Nourdin_Xu_Yang_23}): 
\begin{itemize}
\item $\nu_\alpha$ associated with $\mu_\alpha$ is a $\gamma$-measure with $\gamma \in (d-\alpha,d]$. Namely, there exists $c>0$ such that, for all $x \in \mathbb{S}^{d-1}$ and all $r \in (0,1/2)$, 
\begin{align*}
\nu_\alpha \left(\mathcal{B}(x,r)\right) \leq cr^\gamma, 
\end{align*}
where $\mathcal{B}(x,r)$ is the Euclidean ball centered at $x$ and with radius $r$; 
\item $\nu_\alpha$ is symmetric, i.e., for all $B \in \mathcal{B}(\mathbb{R}^d)$, $\nu_\alpha(-B) = \nu_\alpha(B)$;
\item $\mu_\alpha$ is absolutely continuous with respect to the $d$-dimensional Lebesgue measure.  
\end{itemize}
Under this set of assumptions, \cite[Theorem $3$, (i)]{Chen_Nourdin_Xu_Yang_23} ensures that $f_h$, with $\left| h\right|_{\operatorname{Lip}} \leq 1$, is such that 
\begin{align*}
M_1(f_h) \leq C_1 \quad M_2(f_h) \leq C_2, 
\end{align*}
for some positive constants $C_1,C_2$ depending on $\alpha$, on $d$ and on the spherical component of $\nu_\alpha$. Proposition \ref{prop:Stein_Method_Stable_NDS} extends the first order and the second order regularity estimates of \cite[Theorem $3$, (i)]{Chen_Nourdin_Xu_Yang_23} to the whole set of non-degenerate symmetric $\alpha$-stable probability measures, with $\alpha \in (1,2)$. Note, in particular, that the product probability measure $\mu_{\alpha,d}$ characterized by \eqref{eq:characteristic_function_ind} is not covered by \cite[Lemma $2$ and Theorem $3$, (i)]{Chen_Nourdin_Xu_Yang_23} whereas it is an interesting example of a non-degenerate symmetric $\alpha$-stable probability measure. 
\end{rem}
\noindent
Now, let us recall and prove results regarding Stein's method for the non-degenerate symmetric $1$-stable (Cauchy) probability measures on $\mathbb{R}^d$. For this purpose, let $\mathcal{A}_1$ be the non-local operator, defined for all $f \in \mathcal{S}(\mathbb{R}^d)$ and all $x \in \mathbb{R}^d$, by 
\begin{align}\label{eq:1_stable_symmetric_diff}
\mathcal{A}_1(f)(x) = \int_{\mathbb{R}^d} \left(f(x+u) - f(x) - \langle u ; \nabla(f)(x)\rangle\bbone_{\|u\|\leq 1}\right)\nu_1(du),
\end{align}
where $\nu_1$ is the L\'evy measure of a non-degenerate symmetric $1$-stable probability measure on $\mathbb{R}^d$. The next proposition is partly contained in \cite[Theorem $4.2$]{AH20_3}.

\begin{prop}\label{prop:stein_solution_cauchy}
Let $d \geq 1$ be an integer and let $\mu_1$ be a non-degenerate symmetric $1$-stable probability measure on $\mathbb{R}^d$ with L\'evy measure $\nu_1$. Let $h \in \mathcal{H}_2 \cap \mathcal{C}_c^{\infty}(\mathbb{R}^d)$ and let $(P_t^{\nu_1})_{t \geq 0}$ be the semigroup of operators defined, for all $t \geq 0$ and all $x \in \mathbb{R}^d$, by
\begin{align}\label{eq:1_stable_symmetric_sg}
P_t^{\nu_1}(h)(x) = \int_{\mathbb{R}^d} h \left(e^{-t}x + \left(1- e^{-t}\right)y\right)\mu_1(dy). 
\end{align}
Let $f_h$ be given, for all $x \in \mathbb{R}^d$, by
\begin{align}\label{eq:1_stable_Stein_sol}
f_h(x) = - \int_0^{+\infty} \left(P_t^{\nu_1}(h)(x) - \mathbb{E} h(X_1)\right) dt, \quad X_1 \sim \mu_1.
\end{align}
Then, $f_h$ is well-defined, infinitely differentiable on $\mathbb{R}^d$ and is a strong solution to the non-local partial differential equation: for all $x \in \mathbb{R}^d$, 
\begin{align}\label{eq:1_stable_Stein_equation}
- \langle x ; \nabla(f_h)(x) \rangle + \mathcal{A}_1(f_h)(x)= h(x) - \mathbb{E} h(X_1),
\end{align}
where $\mathcal{A}_1$ is defined by \eqref{eq:1_stable_symmetric_diff}. Moreover, 
\begin{align}\label{eq:1_stable_Linfinity_estimate_1}
M_1(f_h) \leq 1, \quad M_2(f_h) \leq \frac{1}{2}. 
\end{align}
Finally, for all $x \in \mathbb{R}^d$ and all $u \in \mathbb{R}^d$ such that $\| u \|\geq 1$, 
\begin{align}\label{eq:1_stable_logarithmic_increment}
\left| f_h(x+u) - f_h(x) \right| \leq 2 \left(1+ \ln\left(\|u\|\right)\right),
\end{align}
and, for all $x \in \mathbb{R}^d$, 
\begin{align}\label{eq:1_stable_Linfinity_estimate_2}
\left| \mathcal{A}_1(f_h)(x) \right| \leq C_2 , \quad \left| \langle x ; \nabla(f_h)(x) \rangle \right| \leq C_3,
\end{align}
for some $C_2 , C_3 >0$ not depending on $x$.
\end{prop}

\begin{proof}
The existence of $f_h$, \eqref{eq:1_stable_Stein_equation} and \eqref{eq:1_stable_Linfinity_estimate_1} follow from \cite[Theorem $4.2$]{AH20_3}. So, it remains to prove \eqref{eq:1_stable_logarithmic_increment} and \eqref{eq:1_stable_Linfinity_estimate_2}. Then, for all $x \in \mathbb{R}^d$, all $u \in \mathbb{R}^d$ such that $\|u\| \geq 1$ and all $t \geq 0$, 
\begin{align*}
\left|P^{\nu_1}_t(h)(x+u) - P^{\nu_1}_t(h)(x)\right| & =\bigg| \int_{\mathbb{R}^d} \mu_1(dy) \bigg(h \left(e^{-t}(x+u) + (1-e^{-t}) y\right) \\
& \quad\quad- h \left(xe^{-t} + (1-e^{-t})y\right)\bigg) \bigg|  , \\
& \leq 2 \int_{\mathbb{R}^d} 1 \wedge \| e^{-t} u \| \mu_1(dy) = 2( 1 \wedge e^{-t} \|u\|), 
 \end{align*}
since, for $h \in \mathcal{H}_2$, for all $x \in \mathbb{R}^d$ and all $u \in \mathbb{R}^d$, 
\begin{align*}
\left| h \left(x+u\right) - h \left(x\right) \right| \leq 2( 1\wedge \|u\|).
\end{align*}
Thus, for all $x \in \mathbb{R}^d$ and all $u \in \mathbb{R}^d$ such that $\|u\| \geq 1$, 
\begin{align*}
\left| f_h(x+u) - f_h(x) \right| & \leq \int_0^{+\infty} \left| P^{\nu_1}_t(h)(x+u) - P^{\nu_1}_t(h)(x) \right| dt  , \\
& \leq 2 \int_0^{+\infty} 1 \wedge e^{-t} \|u\| dt =2 \left( \ln(\|u\|) + \int_{\ln\left(\|u\|\right)}^{+\infty} e^{-t} \| u \| dt \right) =2 (1 + \ln(\|u\|)),
\end{align*}
and so we are done with \eqref{eq:1_stable_logarithmic_increment}. Now, thanks to \eqref{eq:1_stable_logarithmic_increment} and to \eqref{eq:1_stable_Linfinity_estimate_1}, for all $x \in \mathbb{R}^d$, 
\begin{align*}
\left| \mathcal{A}_1(f_h)(x) \right| & = \left| \int_{\mathbb{R}^d} \left(f_h(x+u) - f_h(x) - \langle u ; \nabla(f_h)(x) \rangle\bbone_{\|u\| \leq 1}\right) \nu_1(du) \right| , \\
& \leq \int_{\| u \| \geq 1} \left| f_h(x+u) -f_h(x) \right| \nu_1(du) \\
&\quad\quad + \int_{\| u \| \leq 1} \left| f_h(x+u) - f_h(x) - \langle u ; \nabla(f_h)(x) \rangle \right| \nu_1(du) , \\
& \leq 2 \int_{\|u\|\geq 1} \left(1+ \ln(\|u\|)\right) \nu_1(du) + \frac{1}{4} \int_{\|u\|\leq 1} \|u\|^2 \nu_1(du).
\end{align*}
Finally, since $f_h$ is a strong solution to the Stein equation \eqref{eq:1_stable_Stein_equation}, for all $x \in \mathbb{R}^d$, 
\begin{align*}
\left| \langle x ; \nabla(f_h)(x) \rangle \right| & \leq \left|  \mathcal{A}_1(f_h)(x) \right| + \left| h(x) - \mathbb{E} h(X_1)\right|, \\
& \leq 2+ C_2, 
\end{align*}
and so we are done.  
\end{proof}

\begin{rem}\label{rem:literature_1stablemlt_stein}
Proposition \ref{prop:stein_solution_cauchy} should be compared with \cite[Theorem $3$ (ii)]{Chen_Nourdin_Xu_Yang_23} where regularity estimates for the solution to the Stein equation are obtained when $\alpha = 1$ and when the test function $h$ belongs to the space of functions from $\mathbb{R}^d$ to $\mathbb{R}$ such that, for all $x,y \in \mathbb{R}^d$,
\begin{align*}
\left| h(x) - h(y) \right| \leq \|x-y\| \wedge \|x-y\|^\beta,
\end{align*}
for some $0<\beta<1$.~Note, in particular, that $\nu_1$ in there must verify the assumptions of \cite[Lemma $2$]{Chen_Nourdin_Xu_Yang_23}.
\end{rem}
\noindent
Based on the previous proposition, let us present alternative representations for $\mathcal{A}_1(f_h)(x)$, for all $x \in \mathbb{R}^d$. This lemma is important to investigate stability estimates for the non-degenerate symmetric $1$-stable probability measure on $\mathbb{R}^d$.

\begin{lem}\label{lem:alternative_representations}
Let $\mathcal{A}_1$ be the non-local linear operator defined by \eqref{eq:1_stable_symmetric_diff}. Let $h \in \mathcal{H}_2 \cap \mathcal{C}_c^{\infty}(\mathbb{R}^d)$ and let $f_h$ be defined by \eqref{eq:1_stable_Stein_sol}. Then, for all $x \in \mathbb{R}^d$, 
\begin{align}\label{eq:alternative_representations}
\mathcal{A}_1(f_h)(x) & = \int_{\mathbb{R}^d} \langle u ; \nabla(f_h)(x+u) - \nabla(f_h)(x)\bbone_{\|u\| \leq 1}\rangle \nu_1(du), \nonumber \\
& = \frac{1}{2} \int_{\mathbb{R}^d} \langle u ; \nabla(f_h)(x+u) - \nabla(f_h)(x-u) \rangle \nu_1(du).
\end{align}
\end{lem}

\begin{proof}
The proof of these representations follows from an integration by parts in the radial variable which needs to be justified because of the singular nature of $\nu_1(du)$.~Thanks to the polar decomposition of the L\'evy measure $\nu_1$, for all $x \in \mathbb{R}^d$, 
\begin{align*}
\mathcal{A}_1(f_h)(x) = \int_{(0,+\infty) \times \mathbb{S}^{d-1}} \left(f_h(x+ry) - f(x) - \langle ry ; \nabla(f_h)(x)\bbone_{r \leq 1} \rangle\right) \frac{dr}{r^2} \sigma(dy),
\end{align*}
where $\sigma$ is the spherical part of $\nu_1$. Let us fix $y \in \mathbb{S}^{d-1}$ and $x \in \mathbb{R}^d$. Then, 
\begin{align*}
\int_0^{+\infty} \left(f_h(x+ry) - f(x) - \langle ry ; \nabla(f_h)(x)\bbone_{r \leq 1} \rangle\right) \frac{dr}{r^2} & = \int_0^1  \big(f_h(x+ry) - f(x) \\
&\quad\quad - \langle ry ; \nabla(f_h)(x) \rangle\big) \frac{dr}{r^2}  \\ 
&\quad\quad + \int_1^{+\infty} \left(f_h(x+ry) - f_h(x)\right) \frac{dr}{r^2}. 
\end{align*}
Now, recall that, for all $r\in (0, +\infty)$, 
\begin{align*}
\frac{d}{dr}\left(-r^{-1}\right) = \frac{1}{r^2} , \quad \frac{d}{dr} \left( f_h(x+ry) - f_h(x) \right) = \langle y ; \nabla(f_h)(x+ry) \rangle.
\end{align*}
Moreover, for all $r \in [1, +\infty)$, 
\begin{align*}
\left| \langle y ; \nabla(f_h)(x+ry) \right| \leq \frac{1}{r} \left| \langle x + ry ; \nabla(f_h)(x+ry) \rangle - \langle x ; \nabla(f_h)(x+ry) \rangle \right| \leq \frac{1}{r} \left( C_3 + \|x\| \right). 
\end{align*}
Finally, thanks to \eqref{eq:1_stable_logarithmic_increment},
\begin{align*}
\underset{r \rightarrow +\infty}{\lim} \frac{1}{r} \left(f_h(x+ry) - f_h(x)\right) = 0. 
\end{align*}
By a first integration by parts, 
\begin{align*}
\int_1^{+\infty} \left(f_h(x+ry) - f_h(x)\right) \frac{dr}{r^2} = \int_1^{+\infty} \langle y ; \nabla(f_h)(x+ry) \rangle \frac{dr}{r} + f_h(x+y) - f_h(x). 
\end{align*} 
Similarly, since $M_2(f_h) \leq 1/2$, 
\begin{align*}
\int_0^1 \left(f_h(x+ry) - f_h(x) - \langle ry ; \nabla(f_h)(x) \rangle\right) \frac{dr}{r^2} & = \int_0^1 \langle y ; \nabla(f_h)(x+ry) - \nabla(f_h)(x) \rangle \frac{dr}{r} \\
&\quad\quad - \left(f_h(x+y) - f_h(x) - \langle y ; \nabla(f_h)(x)\rangle\right). 
\end{align*}
Thus,
\begin{align*}
\int_0^{+\infty} \left(f_h(x+ry) - f_h(x) - \langle ry ; \nabla(f_h)(x)\rangle\bbone_{r \leq 1}\right) \frac{dr}{r^2} & = \int_0^{+\infty} \langle y ; \nabla(f_h)(x+ry) \\
&\quad\quad - \nabla(f_h)(x)\bbone_{r\leq 1}\rangle \frac{dr}{r} \\
&\quad\quad + \langle y ; \nabla(f_h)(x)\rangle. 
\end{align*}
Integrating with respect to $\sigma$ over $\mathbb{S}^{d-1}$ and using the fact that $\int_{\mathbb{S}^{d-1}} y \sigma(dy) = 0$, 
\begin{align*}
\mathcal{A}_1(f_h)(x) = \int_{\mathbb{R}^d} \langle u ; \nabla(f_h)(x+u) - \nabla(f_h)(x)\bbone_{\|u\|\leq 1} \rangle \nu_1(du). 
\end{align*}
The second representation follows from a similar analysis and from the fact that, for all $x \in \mathbb{R}^d$, 
\begin{align}\label{eq:alternative_representation_3}
\mathcal{A}_1(f_h)(x) = \frac{1}{2} \int_{\mathbb{R}^d} \left(f_h(x+u) + f_h(x-u) - 2f_h(x)\right) \nu_1(du).
\end{align}
This concludes the proof of the lemma. 
\end{proof}
\noindent
Finally, to end this preliminary section, let us introduce the main new ingredient which allows us to obtain quantitative versions of the limit theorems mentioned before: the weighted Poincar\'e-type inequalities. 

\begin{defi}\label{def:wnl_Poincare_ineq} 
Let $\mu$ be a probability measure on $\mathbb{R}^d$, let $\nu$ be a L\'evy measure on $\mathbb{R}^d$ and let $\omega$ be a positive Borel function defined on $(0,+\infty)$ such that 
\begin{align}\label{eq:compatibility_condition}
\int_{\|u\| \leq 1} \|u\|^2 \omega(\|u\|) \nu(du) < + \infty, \quad \int_{\|u\| \geq 1} \omega(\|u\|) \nu(du) < + \infty. 
\end{align}
Then, $\mu$ satisfies a weighted Poincar\'e-type inequality with weight function $\omega$ and L\'evy measure $\nu$, if, for all $f \in \mathcal{C}^1_b(\mathbb{R}^d, \mathbb{R}^d)$ such that $\int_{\mathbb{R}^d} f(x) \mu(dx) = 0$, 
\begin{align}\label{ineq:wnl_Poincare_ineq}
\int_{\mathbb{R}^d} \left\| f(x)\right\|^2 \mu(dx) \leq \int_{\mathbb{R}^d} \int_{\mathbb{R}^d} \left \| f(x+u) -f(x) \right\|^2 \omega\left(\|u\|\right) \nu(du) \mu(dx). 
\end{align}
\end{defi}

\begin{rem}\label{rem:examples}
(i) Thanks to \cite[Theorem $4.1$]{Chen_85} or to \cite[Corollary $2$]{Houdre_al_98}, ID probability measures on $\mathbb{R}^d$ verify weighted Poincar\'e-type inequalities with weight function identically equal to $1$.~More precisely, if $\mu$ is an ID probability measure on $\mathbb{R}^d$ with L\'evy measure $\nu$, then, for all $f \in \mathcal{C}^1_b(\mathbb{R}^d)$ such that $\int_{\mathbb{R}^d} f(x) \mu(dx) = 0$, 
\begin{align}\label{ineq:Poincar_type_ID_pm}
\int_{\mathbb{R}^d} |f(x)|^2 \mu(dx) \leq \int_{\mathbb{R}^d}\int_{\mathbb{R}^d} |f(x+u)-f(x)|^2 \nu(du) \mu(dx).
\end{align}
(ii) Let $\mu$ be a SD probability measure on $\mathbb{R}^d$ with parameters $(b, 0 ,\nu)$, with finite first moment and let $\hat{\mu}$ be its Fourier transform. Let us assume that $\mu$ is centered. Let $(Z_k)_{k \geq 1}$ be a sequence of independent random vectors of $\mathbb{R}^d$ with characteristic functions $(\varphi_k)_{k \geq 1}$ given by \eqref{eq:def_cf_k} and let $(S_n)_{n \geq 1}$ be defined by \eqref{eq:def_Sn_simple}. Then, by standard computations, the characteristic function of $S_n$, $n \geq 1$, is given, for all $\xi \in \mathbb{R}^d$, by 
\begin{align}\label{eq:fc_Sn}
\varphi_{S_n}(\xi) = \mathbb{E} e^{i \langle \xi ; S_n \rangle} = \exp\left(\int_{\mathbb{R}^d} \left(e^{i \langle u ; \xi \rangle}-1 - i \langle u ; \xi \rangle\right) \tilde{\nu}_n(du) \right), 
\end{align}
where,
\begin{align}\label{eq:Levy_measure_Sn}
\tilde{\nu}_n(du) =\bbone_{\mathbb{S}^{d-1}}(x)\bbone_{(0,+\infty)}(r) \left(k_x\left(\frac{n r}{n+1}\right) - k_x\left(nr\right)\right)\frac{dr}{r} \sigma(dx),
\end{align}
with $\sigma$ and $k_x$ given by \eqref{eq:polar_decomposition_SD}. Then, it is clear that $S_n$ is an ID random vector of $\mathbb{R}^d$ with L\'evy measure $\tilde{\nu}_n$, for all $n \geq 1$. By point (i), for all $n \geq 1$ and all $f \in \mathcal{C}_b^1(\mathbb{R}^d)$ such that $\mathbb{E} f(S_n) = 0$, 
\begin{align}\label{ineq:Poincare_inequality_simple_case}
\mathbb{E} f(S_n)^2 \leq \mathbb{E} \int_{\mathbb{R}^d} |f(S_n+u) - f(S_n)|^2 \tilde{\nu}_n(du). 
\end{align}
Moreover, if $k_x$ is independent of the $x$ variable and positive on $(0, +\infty)$, then, for all $n \geq 1$, 
\begin{align}\label{eq:weigthed_representation}
\tilde{\nu}_{n}(du) =\bbone_{\mathbb{S}^{d-1}}(x)\bbone_{(0,+\infty)}(r) \dfrac{k\left(\frac{n r}{n+1}\right) - k\left(nr\right)}{k(r)}\frac{k(r)dr}{r} \sigma(dx) = \omega_n(\|u\|)\nu(du),
\end{align}
where $\omega_n(r) = (k(nr/(n+1)) - k(nr))/k(r)$, for all $r \in (0,+\infty)$. Thus, the law of $S_n$ verifies a weighted Poincar\'e-type inequality with weight function $\omega_n$ and L\'evy measure $\nu$. 
\end{rem}
\noindent
The next lemma investigates the stability of the weighted Poincar\'e-type inequalities with respect to the convolution operation.~Recall that for two Borel positive measures $\mu_1$ and $\mu_2$ on $\mathbb{R}^d$, the convolution measure, denoted by $\mu_1 \ast \mu_2$, is the image measure of the product measure $\mu_1 \otimes \mu_2$ under the mapping $T$ from $\mathbb{R}^d \times \mathbb{R}^d$ to $\mathbb{R}^d$ defined, for all $(x,y) \in \mathbb{R}^{2d}$, by $T(x,y) = x+y$. Namely, $\mu_1 \ast \mu_2 = (\mu_1 \otimes \mu_2) \circ T^{-1}$. 

\begin{lem}\label{lem:convolution_weighted_poincare_type_ineq}
Let $\nu$ be a L\'evy measure on $\mathbb{R}^d$.~Let $\mu_1$ and $\mu_2$ be two probability measures on $\mathbb{R}^d$. Let $\omega_1$ and $\omega_2$ be two positive Borel functions defined on $(0,+\infty)$ such that $\mu_1$ and $\mu_2$ satisfy weighted Poincar\'e-type inequalities with respective weight functions $\omega_1$ and $\omega_2$.~Then, $\mu_1 \ast \mu_2$ verifies a weighted Poincar\'e-type inequality with weight function $\omega_1 + \omega_2$. Namely, for all $f \in \mathcal{C}_b^1(\mathbb{R}^d , \mathbb{R}^d)$ such that $\int_{\mathbb{R}^d} f(x) (\mu_1 \ast \mu_2)(dx) = 0$, 
\begin{align}\label{ineq:poincare_type_convolution}
\int_{\mathbb{R}^d} \|f(x)\|^2 (\mu_1 \ast \mu_2)(dx) & \leq \int_{\mathbb{R}^d} \int_{\mathbb{R}^d} \left\| f(x+u)-f(x)\right\|^2 \bigg(\omega_1(\|u\|) \nonumber \\
&\quad\quad+ \omega_2(\|u\|) \bigg) \nu(du) (\mu_1 \ast \mu_2)(dx).
\end{align}
\end{lem}

\begin{proof}
The proof is very similar to the corresponding property for the classical Poincar\'e inequality (see \cite{BU_84}) and is based on the fact that the first order difference operator commutes with the translation operator. In the sequel, let us denote by $\tau_z$ the translation operator in the direction $z \in \mathbb{R}^d$ defined, for all $f\in \mathcal{C}_b^1(\mathbb{R}^d)$, by 
\begin{align*}
\tau_z(f)(\cdot) = f(\cdot+z).  
\end{align*}
Next, let $z \in \mathbb{R}^d$ and let $f \in \mathcal{C}_b^1(\mathbb{R}^d)$ be such that $\int_{\mathbb{R}^d} f(x)(\mu_1 \ast \mu_2)(dx) = 0$. Then, since $\mu_1$ verifies a weighted Poincar\'e-type inequality with weight function $\omega_1$, 
\begin{align*}
\int_{\mathbb{R}^d} \left| \tau_z(f)(x) - \int_{\mathbb{R}^d} \tau_z(f)(x) \mu_1(dx) \right|^2 \mu_1(dx) & \leq \int_{\mathbb{R}^d} \int_{\mathbb{R}^d} \left| \tau_z(f)(x+u)-\tau_z(f)(x)\right|^2 \\
&\quad\quad \times \omega_1(\|u\|) \nu(du) \mu_1(dx). 
\end{align*}
Developing the square gives,  
\begin{align*}
\int_{\mathbb{R}^d} \left| \tau_z(f)(x)\right|^2 \mu_1(dx) & \leq \int_{\mathbb{R}^d} \int_{\mathbb{R}^d} \left| \tau_z(f)(x+u)-\tau_z(f)(x)\right|^2 \omega_1(\|u\|) \nu(du) \mu_1(dx) \\
& \quad\quad + \left| \int_{\mathbb{R}^d} \tau_z(f)(x) \mu_1(dx) \right|^2.
\end{align*}
Now, integrating the previous inequality in the $z$ variable with respect to $\mu_2$, using the weighted Poincar\'e-type inequality for $\mu_2$ and  using Jensen's inequality, 
\begin{align*}
\int_{\mathbb{R}^d} \left| f(x+z)\right|^2 \mu_1(dx) \mu_2(dz) & \leq \int_{\mathbb{R}^{2d}} \int_{\mathbb{R}^d} \left| \tau_z(f)(x+u)-\tau_z(f)(x)\right|^2 \omega_1(\|u\|) \nu(du) \mu_1(dx) \mu_2(dz) \\
& \quad\quad + \int_{\mathbb{R}^d} \left| \int_{\mathbb{R}^d} \tau_z(f)(x) \mu_1(dx) \right|^2 \mu_2(dz) , \\
& \leq \int_{\mathbb{R}^{2d}} \int_{\mathbb{R}^d} \left| \tau_z(f)(x+u)-\tau_z(f)(x)\right|^2 \omega_1(\|u\|) \nu(du) \mu_1(dx) \mu_2(dz) \\
& \quad + \int_{\mathbb{R}^{d}} \int_{\mathbb{R}^d} \left|  \int_{\mathbb{R}^d} \tau_{z+u}(f)(x) \mu_1(dx)-  \int_{\mathbb{R}^d} \tau_z(f)(x) \mu_1(dx) \right|^2 \\
&\quad\quad \times \omega_2(\|u\|) \nu(du) \mu_2(dz) , \\
& \leq \int_{\mathbb{R}^{2d}} \int_{\mathbb{R}^d} \left| \tau_z(f)(x+u)-\tau_z(f)(x)\right|^2 \omega_1(\|u\|) \nu(du) \mu_1(dx) \mu_2(dz) \\
& + \int_{\mathbb{R}^{2d}} \int_{\mathbb{R}^d} \left| \tau_z(f)(x+u)-\tau_z(f)(x)\right|^2 \omega_2(\|u\|) \nu(du) \mu_2(dz) \mu_1(dx).
\end{align*}
The definition of the convolution operation concludes the proof of the lemma.
\end{proof}
\noindent 
As an application of the previous lemma and to end this preliminaries section, let us discuss how weighted Poincar\'e-type inequalities can fit in the framework of the generalized central limit theorems discussed in \eqref{eq:def_sum}. Let $\nu$ be a L\'evy measure on $\mathbb{R}^d$, let $(Z_k)_{k \geq 1}$ be a sequence of independent random vectors of $\mathbb{R}^d$ with laws $(\mu_k)_{k \geq 1}$ and let $(b_n)_{n \geq 1}$ be a sequence of positive reals such that 
\begin{align}\label{eq:assumption_bn}
b_n \longrightarrow 0 , \quad \frac{b_{n+1}}{b_n} \longrightarrow 1,
\end{align}
as $n$ tends to $+\infty$. Let $(S_n)_{n \geq 1}$ be the sequence of random vectors of $\mathbb{R}^d$ defined, for all $n \geq 1$, by
\begin{align}\label{eq:def_Sn_general}
S_n = b_n \sum_{k =1}^n Z_k.  
\end{align}
Let us assume that there exists a sequence of positive Borel functions $(\omega_k)_{k \geq 1}$ defined on $(0,+\infty)$ verifying \eqref{eq:compatibility_condition} with respect to $\nu$, such that, for all $k \geq 1$ and all $f \in \mathcal{C}^1_b(\mathbb{R}^d)$ with $\mathbb{E} f(Z_k) = 0$, 
\begin{align}\label{ineq:assumption_Poincare}
\mathbb{E} f(Z_k)^2 \leq \mathbb{E} \int_{\mathbb{R}^d} |f(Z_k+u) -f(Z_k)|^2 \omega_k(\|u\|) \nu(du). 
\end{align}
Now, recall that, for all $c>0$, $T_c$ denotes the scaling operator defined on functions in the following way: for all $f$ $\mathcal{B}(\mathbb{R}^d)$-measurable real-valued function defined on $\mathbb{R}^d$ and all $x \in \mathbb{R}^d$, 
\begin{align}\label{eq:defi_scaling_operator}
T_c(f)(x) = f(M_c(x)) = f(cx),
\end{align}
where $M_c$ is the multiplication operator by the positive real $c$.~Based on the previous notations, the law of $S_n$, denoted by $\tilde{\mu}_n$, is then given, for all $n \geq 1$, by
\begin{align}\label{eq:decomposition_law_S_n}
\tilde{\mu}_n = \left(\mu_1 \ast \mu_2 \ast \cdots \ast \mu_n\right) \circ T_{b_n}^{-1}.
\end{align}
Based on Lemma \ref{lem:convolution_weighted_poincare_type_ineq} and on the previous decomposition for $\tilde{\mu}_n$, for all $f \in \mathcal{C}_b^1(\mathbb{R}^d)$ with $\mathbb{E} f(S_n) = 0$, 
\begin{align*}
\int_{\mathbb{R}^d} f(x)^2 \tilde{\mu}_n(dx) & \leq \int_{\mathbb{R}^d} \int_{\mathbb{R}^d} \left| \Delta_u(T_{b_n}(f))(x)\right|^2 \left(\sum_{k = 1}^n \omega_k(\| u \|) \right) \nu(du)( \mu_1 \ast \cdots \ast \mu_n)(dx). 
\end{align*}
At this point, let us assume that $\nu$ is the L\'evy measure of a SD probability measure on $\mathbb{R}^d$ with polar decomposition given by 
\begin{align}\label{eq:more_structure_levy}
\nu(du) =\bbone_{\mathbb{S}^{d-1}}(x)\bbone_{(0,+\infty)}(r) \frac{k(r)dr}{r} \sigma(dx),
\end{align}
where $k$ is positive on $(0,+\infty)$. Now, for all $n \geq 1$, 
\begin{align*}
\int_{\mathbb{R}^d} \int_{\mathbb{R}^d} & \left| \Delta_u(T_{b_n}(f))(x)\right|^2 \left(\sum_{\ell = 1}^n \omega_\ell(\| u \|) \right) \nu(du)( \mu_1 \ast \cdots \ast \mu_n)(dx)  \\
& = \int_{\mathbb{R}^d} \int_{\mathbb{R}^d} \left| f( x +  u) - f( x) \right|^2 \left(\sum_{\ell = 1}^n \omega_\ell\left(\frac{\|u\|}{b_n}\right) \right) \frac{k\left(\frac{\|u\|}{b_n}\right)}{k(\|u\|)} \nu(du) \tilde{\mu}_n(dx). 
\end{align*}
Then, for all $n \geq 1$, $\tilde{\mu}_n$ verifies a weighted Poincar\'e-type inequality with weight function $\tilde{\omega}_n$ given, for all $n \geq 1$ and all $r \in (0,+\infty)$, by
\begin{align}\label{eq:weight_function_gene_CLT}
\tilde{\omega}_{n}(r) = \left(\sum_{j = 1}^n \omega_j\left(\frac{r}{b_n}\right) \right) \frac{k\left(\frac{r}{b_n}\right)}{k(r)},
\end{align}
and with L\'evy measure $\nu$.

\section{Stability Estimate for SD probability measures with finite second moment}\label{sec:stability_estimate_finite_second_moment}
\noindent
Now, let us prove Theorem \ref{thm:stability_estimate_second_moment} using the Stein's method results developed in \cite{AH19_2} (and recalled in the previous section) and using the weighted Poincar\'e-type inequality.\\
\\
\textit{Proof of Theorem \ref{thm:stability_estimate_second_moment}}. To start, let $\mathcal{E}_{\nu,\omega_X}$ be the bilinear symmetric non-negative definite form defined, for all $f,g \in \mathcal{S}(\mathbb{R}^d, \mathbb{R}^d)$, by
\begin{align}\label{eq:weighted_form_E}
\mathcal{E}_{\nu,\omega_X}(f,g) = \int_{\mathbb{R}^d} \int_{\mathbb{R}^d} \langle f(x+u) - f(x) ; g(x+u) - g(x)\rangle \omega_X(\|u\|)\nu(du) \mu_X(dx).
\end{align} 
Using the abstract results \cite[Theorem $5.10$]{AH20_3} or \cite[Theorem 2.1]{AH22_5}, there exists $\tau_{\mu_X}$ an element of $\mathcal{D}(\mathcal{E}_{\nu, \omega_X})$ such that for all $f \in \mathcal{D}(\mathcal{E}_{\nu, \omega_X})$, 
\begin{align}\label{eq;stein_kernel_identity}
\mathcal{E}_{\nu, \omega_X}(\tau_{\mu_X} , f) = \langle x , f \rangle_{L^2(\mu_X)},
\end{align}
where $\langle \cdot ; \cdot\rangle_{L^2(\mu_X)}$ is the classical scalar product on $L^2(\mu_X)$ and $\mathcal{D}(\mathcal{E}_{\nu, \omega_X})$ is the $L^2(\mu_X)$-domain of the form $\mathcal{E}_{\nu, \omega_X}$ (actually its closure since the condition $\mu_X \ast \omega_X \nu << \mu_X$ ensures that the form is closable). Moreover, by Cauchy-Schwarz inequality,
\begin{align*}
\mathcal{E}_{\nu, \omega_X}(\tau_{\mu_X} , \tau_{\mu_X}) & =  \langle x , \tau_{\mu_X} \rangle_{L^2(\mu_X)} , \\
& \leq \left(\int_{\mathbb{R}^d} \| x \|^2 \mu_X(dx) \right)^{\frac{1}{2}} \left(\int_{\mathbb{R}^d} \|\tau_{\mu_X}(x)\|^2 \mu_X(dx)\right)^{\frac{1}{2}} , \\
& \leq \int_{\mathbb{R}^d} \|x\|^2 \mu_X(dx). 
\end{align*}
Now, let $h \in \mathcal{H}_2 \cap \mathcal{C}_c^{\infty}(\mathbb{R}^d)$. Then, by Proposition \ref{prop:Stein_Method_SD_finite_first_moment}, for all $x \in \mathbb{R}^d$, 
\begin{align}\label{eq:Stein_equation_SD}
- \langle x ; \nabla(f_h)(x) \rangle + \int_{\mathbb{R}^d} \langle \nabla(f_h)(x+u) - \nabla(f_h)(x) ; u \rangle \nu(du) = h(x) - \mathbb{E} h(Z),
\end{align}
where $Z \sim \mu$. Integrating with respect to $\mu_X$ and using \eqref{eq;stein_kernel_identity}, 
\begin{align*}
\mathbb{E} h(X) - \mathbb{E} h(Z) & = \mathbb{E} \left( - \langle X ; \nabla(f_h)(X) \rangle + \int_{\mathbb{R}^d} \langle \nabla(f_h)(X+u) - \nabla(f_h)(X) ; u \rangle \nu(du)\right) , \\
& = \mathbb{E} \bigg( - \int_{\mathbb{R}^d} \langle \nabla(f_h)(X+u) - \nabla(f_h)(X) ; \tau_{\mu_X}(X+u)- \tau_{\mu_X}(X) \rangle \omega_X(\|u\|) \nu(du) \\
& \quad+ \int_{\mathbb{R}^d} \langle \nabla(f_h)(X+u) - \nabla(f_h)(X) ; u \rangle \nu(du)\bigg). 
\end{align*}
Next,
\begin{align}\label{eq:two_error_terms}
\mathbb{E} h(X) - \mathbb{E} h(Z) & = \mathbb{E} \bigg( - \int_{\mathbb{R}^d} \langle \nabla(f_h)(X+u) - \nabla(f_h)(X) ; \tau_{\mu_X}(X+u)- \tau_{\mu_X}(X) - u \rangle \nonumber \\
&\quad \times \omega_X(\|u\|) \nu(du) \nonumber \\
&\quad + \int_{\mathbb{R}^d} \langle \nabla(f_h)(X+u) - \nabla(f_h)(X) ; u \rangle( 1-\omega_X(\|u\|)) \nu(du)\bigg). 
\end{align}
Let us start to bound the second term using the $L^\infty$-bounds on the solution to the Stein's equation \eqref{eq:Stein_equation_SD}. Then,
\begin{align*}
\left| \mathbb{E} \int_{\mathbb{R}^d} \langle \nabla(f_h)(X+u) - \nabla(f_h)(X) ; u \rangle( 1-\omega_X(\|u\|)) \nu(du) \right| & \leq \frac{1}{2}\int_{\mathbb{R}^d} \|u\|^2 \left|\omega_X(\|u\|) - 1\right|\nu(du).
\end{align*}
For the first term, using Cauchy-Schwarz inequality twice, 
\begin{align*}
| (I) | & \leq \frac{1}{2} \mathbb{E} \int_{\mathbb{R}^d} \| u \| \left\| \tau_{\mu_X}(X+u)- \tau_{\mu_X}(X) - u \right\| \omega_X(\|u\|) \nu(du) , \\
& \leq \frac{1}{2}  \left(\int_{\mathbb{R}^d} \| u \|^2 \omega_X(\| u \|) \nu(du)\right)^{\frac{1}{2}} \left(\mathbb{E} \int_{\mathbb{R}^d} \left\| \tau_{\mu_X}(X+u)- \tau_{\mu_X}(X) - u \right\|^2 \omega_X(\|u\|) \nu(du)  \right)^{\frac{1}{2}} , \\
& \leq \frac{1}{2} \left(\int_{\mathbb{R}^d} \| u \|^2 \omega_X(\| u \|) \nu(du)\right)^{\frac{1}{2}} S_{\omega_X} \left(\mu_X | \mu\right),
\end{align*}
where $S_{\omega_X} \left(\mu_X | \mu\right)$ is the weighted Stein discrepancy between $\mu_X$ and $\mu$ with weight $\omega_X$.~Next, developing the square 
\begin{align*}
\mathbb{E} \int_{\mathbb{R}^d} \left\| \tau_{\mu_X}(X+u)- \tau_{\mu_X}(X) - u \right\|^2 & \omega_X(\|u\|) \nu(du)  = \mathcal{E}_{\nu,\omega_X}(\tau_{\mu_X},\tau_{\mu_X}) \\
&+ \int_{\mathbb{R}^d} \|u\|^2 \omega_X(\|u\|) \nu(du)  \\
& - 2 \mathbb{E} \int_{\mathbb{R}^d} \langle \tau_{\mu_X}(X+u) - \tau_{\mu_X}(X) ; u \rangle \omega_X( \|u\| ) \nu(du) , \\
& \leq \int_{\mathbb{R}^d} \|x\|^2 \mu_X(dx) + \int_{\mathbb{R}^d} \|u\|^2 \omega_X(\|u\|) \nu(du) \\
&\quad\quad - 2 \int_{\mathbb{R}^d} \|x\|^2 \mu_X(dx).
\end{align*}
Thus, 
\begin{align*}
S_{\omega_X}(\mu_X | \mu) \leq \sqrt{ \int_{\mathbb{R}^d} \|u\|^2 \omega_X(\|u\|) \nu(du) - \int_{\mathbb{R}^d} \|x\|^2 \mu_X(dx) }. 
\end{align*}
Putting everything together, we arrive at: for all $h \in \mathcal{H}_2 \cap \mathcal{C}_c^{\infty}(\mathbb{R}^d)$, 
\begin{align*}
|\mathbb{E} h(X) - \mathbb{E} h(Z) | & \leq \frac{1}{2}  \left(\int_{\mathbb{R}^d} \| u \|^2 \omega_X(\| u \|) \nu(du)\right)^{\frac{1}{2}} \\
&\quad \quad \times  \sqrt{ \int_{\mathbb{R}^d} \|u\|^2 \omega_X(\|u\|) \nu(du) - \int_{\mathbb{R}^d} \|x\|^2 \mu_X(dx) } \\
& \quad\quad+ \frac{1}{2} \int_{\mathbb{R}^d} \|u\|^2 |\omega_X(\|u\|)-1| \nu(du). 
\end{align*}
This concludes the proof of the theorem. $\square$

\begin{rem}\label{rem:applications_forms}
(i) A direct application of Theorem \ref{thm:stability_estimate_second_moment} provides a stability estimate when the law of $X$ satisfies a weighted Poincar\'e-type inequality with constant weight function $C$ and with L\'evy measure $\nu$. Indeed, if for all $f \in \mathcal{C}^1_b(\mathbb{R}^d, \mathbb{R}^d)$ such that $\mathbb{E} f(X) = 0$, 
\begin{align}\label{ineq:weighted_Poincare_inequality_constant_weight}
\int_{\mathbb{R}^d} \|f(x)\|^2 \mu_X(dx) \leq C \int_{\mathbb{R}^d}\int_{\mathbb{R}^d} \|f(x+u) - f(x)\|^2 \nu(du) \mu_X(dx),
\end{align} 
then, provided that $\mu_X \ast \nu << \mu_X$, 
\begin{align}\label{ineq:stability_estimate_constant_function_1}
d_{W_2}(\mu_X , \mu) & \leq \frac{1}{2} \sqrt{C}  \left(\int_{\mathbb{R}^d} \| u \|^2 \nu(du)\right)^{\frac{1}{2}}  \sqrt{ C \int_{\mathbb{R}^d} \|u\|^2\nu(du) - \int_{\mathbb{R}^d} \|x\|^2 \mu_X(dx) } \nonumber \\
& \quad\quad+ \frac{1}{2} \left( \int_{\mathbb{R}^d} \|u\|^2\nu(du) \right) |C-1|.
\end{align}
In particular, if $\int_{\mathbb{R}^d} \|u\|^2 \nu(du) = \int_{\mathbb{R}^d} \|x\|^2 \mu_X(dx)$, then 
\begin{align}\label{ineq:stability_estimate_constant_function_2}
d_{W_2}(\mu_X , \mu) & \leq \frac{1}{2} \sqrt{C}  \left(\int_{\mathbb{R}^d} \| u \|^2 \nu(du)\right) \sqrt{ C  - 1 } \nonumber \\
& \quad\quad+ \frac{1}{2} \left( \int_{\mathbb{R}^d} \|u\|^2\nu(du) \right) |C-1|.
\end{align}
(ii) Let $\mu_X$ be a non-degenerate probability measure on $\mathbb{R}^d$ and let $\nu$ be a non-degenerate L\'evy measure on $\mathbb{R}^d$.~Let us assume that $\mu_X \ast \nu << \mu_X$ and let us consider the bilinear symmetric non-negative definite form defined, for all $f,g \in \mathcal{C}_b^1(\mathbb{R}^d, \mathbb{R}^d)$, by 
\begin{align*}
\mathcal{E}_{\nu,1}(f,g) : = \int_{\mathbb{R}^d} \int_{\mathbb{R}^d} \langle f(x+u) - f(x) ; g(x+u) - g(x)\rangle \nu(du) \mu_X(dx).  
\end{align*}
It is a symmetric form on $L^2(\mu_X, \mathbb{R}^d)$ in the sense of \cite[Chapter $1$]{FOT_10}. There are at least two ways to build Dirichlet forms which are closed extensions of $\left(\mathcal{E}_{\nu,1}, \mathcal{C}_b^1(\mathbb{R}^d, \mathbb{R}^d)\right)$. First, consider the following subset of $L^2(\mu_X, \mathbb{R}^d)$,
\begin{align}\label{eq:L2_maximal_domain}
D \left(\mathcal{E}_{\nu,1}\right) : = \{f \in L^2(\mu_X, \mathbb{R}^d) : \, \mathcal{E}_{\nu,1}(f,f)<+\infty\}. 
\end{align}
Since $\mu_X \ast \nu << \mu_X$, $\left(\mathcal{E}_{\nu,1}, D \left(\mathcal{E}_{\nu,1}\right)\right)$ is a Dirichlet form in the sense of \cite[Chapter $1$]{FOT_10}.~Moreover, the assumption $\mu_X \ast \nu << \mu_X$ ensures that $\left(\mathcal{E}_{\nu,1}, \mathcal{C}_b^1(\mathbb{R}^d, \mathbb{R}^d)\right)$ is closable and so one can consider its closure which is denoted by $\left(\overline{\mathcal{E}_{\nu,1}}, \mathcal{D}(\overline{\mathcal{E}_{\nu, 1}})\right)$. Finally, it is clear that
\begin{align}\label{eq:rep_domain_closure}
\mathcal{D}(\overline{\mathcal{E}_{\nu, 1}}) = \overline{\mathcal{C}_b^1(\mathbb{R}^d,\mathbb{R}^d)}^{\| \cdot \|_{L^2+\mathcal{E}_{\nu,1}}},
\end{align}
where, for all $f \in D \left(\mathcal{E}_{\nu,1}\right)$, 
\begin{align}\label{eq:norm_nl}
\| f \|^2_{L^2+\mathcal{E}_{\nu,1}} = \|f\|^2_{L^2(\mu_X, \mathbb{R}^d)} + \mathcal{E}_{\nu,1}(f,f). 
\end{align}
By standard approximation arguments, $\left(\overline{\mathcal{E}_{\nu,1}}, \mathcal{D}(\overline{\mathcal{E}_{\nu, 1}})\right)$ is a regular Dirichlet form on $L^2(\mu_X, \mathbb{R}^d)$. In general, $\left(\overline{\mathcal{E}_{\nu,1}}, \mathcal{D}(\overline{\mathcal{E}_{\nu, 1}})\right) \subset \left(\mathcal{E}_{\nu,1}, D \left(\mathcal{E}_{\nu,1}\right)\right)$ and it is not known if equality holds (see, e.g., \cite{U_07,SU_12} for related results on $L^2(\mathbb{R}^d,dx)$ and see Lemmas \ref{lem:reduction_cc} and \ref{lem:reduction_csb} and Proposition \ref{prop:markov_uniqueness_rot_inv} of the Appendix section for the case $\mu_X= \mu_\alpha^{\operatorname{rot}}$ and $\nu= \nu_\alpha^{\operatorname{rot}}$, with $\alpha \in (0,2)$).~This question will be investigated elsewhere in full generality.\\
(iii) Note that in the proof of Theorem \ref{thm:stability_estimate_second_moment}, the following two facts have been used: for all $f \in \mathcal{D}(\overline{\mathcal{E}_{\nu, \omega_X}}),
$
\begin{align}\label{eq:nullity}
\overline{\mathcal{E}_{\nu, \omega_X}}(f, 1) = 0,
\end{align}
and, for all $f \in L^2(\mu_X , \mathbb{R}^d)$ and all $t \geq 0$, 
\begin{align}\label{eq:mass_conservation}
\int_{\mathbb{R}^d} P_t(f)(x) \mu_X(dx) = \int_{\mathbb{R}^d} f(x)\mu_X(dx),
\end{align}
where $(P_t)_{t \geq 0}$ is the $\mathcal{C}_0$-semigroup of linear symmetric contractions generated by \\$\left(\overline{\mathcal{E}_{\nu,\omega_X}}, \mathcal{D}(\overline{\mathcal{E}_{\nu, \omega_X}})\right)$. \eqref{eq:mass_conservation} follows from an integration by parts combined with \eqref{eq:nullity}.
\end{rem}

Now, let us apply the previous result to obtain explicit rates of convergence for the generalized central limit theorems already put forward in Section \ref{sec:notations_preliminaries}.~For this purpose, let us recall the framework of the canonical example. Let $(Z_k)_{k \geq 1}$ be a sequence of independent random vectors of $\mathbb{R}^d$ which sequence of characteristic functions $\left(\varphi_k\right)_{k \geq 1}$ is given, for all $k \geq 1$ and all $ \xi \in \mathbb{R}^d$, by
\begin{align*}
\varphi_k(\xi) = \dfrac{\hat{\mu}((k+1) \xi)}{\hat{\mu}(k \xi)}, 
\end{align*}
where $\hat{\mu}$ is the Fourier transform of $\mu$. Let $(S_n)_{n \geq 1}$ be the sequence of random vectors of $\mathbb{R}^d$, defined, for all $n \geq 1$, by
\begin{align*}
S_n = \frac{1}{n} \sum_{k =1}^n Z_k. 
\end{align*}
Then, thanks to Remark \ref{rem:examples}, (ii), $S_n$, $n \geq 1$, is infinitely divisible with characteristic function given by \eqref{eq:fc_Sn} and L\'evy measure $\tilde{\nu}_n$ defined by \eqref{eq:Levy_measure_Sn}.~Thanks to \cite[Lemma $4.1$]{Chen_85}, $\tilde{\mu}_n \ast \tilde{\nu}_n << \tilde{\mu}_n$, for all $n \geq 1$, where $\tilde{\mu}_n$ denotes the law of $S_n$. As already underlined in Remark \ref{rem:examples}, Point $(ii)$, the law of $S_n$, $n \geq 1$, verifies a weighted Poincar\'e-type inequality with weight function $\omega_n$ given, for all $r \in (0, +\infty)$, by  
\begin{align}\label{eq:omega_n_simple_case}
\omega_n(r) = \frac{1}{k(r)} \left(k\left(\dfrac{nr}{n+1}\right) - k \left(nr\right)\right),
\end{align}
and with L\'evy measure $\nu$ (see \eqref{ineq:Poincare_inequality_simple_case}). Finally, since $S_n$ is infinitely divisible, for all $n \geq 1$, and since $\mu$ has finite second moment, 
\begin{align}\label{eq:second_moment_simple_case}
\int_{\mathbb{R}^d} \|u\|^2 \omega_n \left(\|u\|\right) \nu(du) = \int_{\mathbb{R}^d} \|u\|^2 \tilde{\nu}_n(du) = \mathbb{E} \|S_n\|^2 <+\infty. 
\end{align}
Thus, the following corollary of Theorem \ref{thm:stability_estimate_second_moment} holds true.

\begin{cor}\label{cor:simple_case_finite_second_moment}
Let $d \geq 1$ be an integer, let $(S_n)_{n \geq 1}$ be the sequence of random vectors of $\mathbb{R}^d$ with law $\tilde{\mu}_n$, $n \geq 1$. Then, for all $n \geq 1$, 
\begin{align}\label{ineq:rates_convergence_simple_case_second_moment}
d_{W_2}(\tilde{\mu}_n , \mu) \leq \frac{1}{n} \left(1+ \frac{1}{n}\right) \left(\int_{\mathbb{R}^d} \|u\|^2 \nu(du)\right).
\end{align}
\end{cor} 

\begin{proof}
Taking into account \eqref{eq:second_moment_simple_case} and applying Theorem \ref{thm:stability_estimate_second_moment}, for all $n \geq 1$, 
\begin{align*}
d_{W_2}(\tilde{\mu}_n , \mu) \leq \frac{1}{2} \int_{\mathbb{R}^d} \|u\|^2 |\omega_n(\|u\|)-1| \nu(du),
\end{align*}
where $\omega_n$ is given by \eqref{eq:omega_n_simple_case}. Now, using spherical coordinates and standard changes of variables, for all $n \geq1$, 
\begin{align*}
d_{W_2}(\tilde{\mu}_n , \mu) & \leq  \frac{1}{2} \int_{\mathbb{R}^d} \|u\|^2 |\omega_n(\|u\|)-1| \nu(du) , \\
& \leq \frac{1}{2} \int_{\mathbb{R}^d} \|u\|^2 \left|\dfrac{k\left(\frac{n \|u\|}{n+1}\right)-k(n \|u\|)}{k(\|u\|)} -1\right| \nu(du) , \\
& \leq \frac{1}{2} \sigma\left(\mathbb{S}^{d-1}\right) \left(\int_0^{+\infty} \left|k\left(\frac{n r}{n+1}\right)-k(nr)-k(r)\right| rdr \right) , \\
& \leq \frac{1}{2} \sigma\left(\mathbb{S}^{d-1}\right) \left(\int_0^{+\infty} rk(rn) dr + \int_0^{+\infty} r \left(k\left(\frac{n r}{n+1}\right) - k(r)\right) dr \right) , \\
& \leq \frac{1}{2} \left( \frac{1}{n^2} + \left(1 + \frac{1}{n}\right)^2 - 1\right) \left(\int_{\mathbb{R}^d} \|u\|^2 \nu(du) \right).
\end{align*}
This concludes the proof of the corollary. 
\end{proof}
\noindent
To end this section, let us consider the more general example discussed at the end of Section \ref{sec:notations_preliminaries} after the proof of Lemma \ref{lem:convolution_weighted_poincare_type_ineq}.~For this purpose, let us adopt the notations and the assumptions of this paragraph.~Moreover, based on the previously discussed example, let us make the following assumption: for all $k \geq 1$, 
\begin{align}\label{eq:cond_absolute_continuity_Zk}
\mu_k \ast \omega_k \nu << \mu_k, 
\end{align}
where $\mu_k$ is the law of the random vector $Z_k$, for all $k \geq 1$. First, let us prove the following technical lemma.

\begin{lem}\label{lem:propa_absolute_continuity}
Let $d \geq 1$ be an integer, Let $\nu$ be a L\'evy measure on $\mathbb{R}^d$ and let $(Z_k)_{k \geq 1}$ be a sequence of independent random vectors of $\mathbb{R}^d$ such that the associated sequence of laws $(\mu_k)_{k \geq 1}$ verifies \eqref{eq:cond_absolute_continuity_Zk} for a sequence of positive Borel functions $(\omega_k)_{k \geq 1}$ defined on $(0,+\infty)$ satisfying \eqref{eq:compatibility_condition} with respect to $\nu$. Then, for all $n \geq 1$, 
\begin{align}\label{eq:absolute_continuity_convolutions}
\left(\mu_1 \ast \dots \ast \mu_n\right) \ast \left(\sum_{j=1}^n \omega_j\right) \nu << \left(\mu_1 \ast \dots \ast \mu_n\right). 
\end{align}
\end{lem}

\begin{proof}
The case $n =2$ follows from standard observations. The general case follows by induction in $n$. 
\end{proof}
\noindent
Then, let us conclude this section with the following general quantitative bound.

\begin{cor}\label{cor:general_case_second_moment}
Let $d \geq 1$ be an integer and let $(b_n)_{n \geq 1}$ be a sequence of positive reals such that 
\begin{align*}
b_n \longrightarrow 0, \quad \frac{b_{n+1}}{b_n} \longrightarrow 1,
\end{align*}
as $n$ tends to $+\infty$. Let $(Z_k)_{k \geq 1}$ be a sequence of independent centered random vectors of $\mathbb{R}^d$ with finite second moment and with laws $(\mu_k)_{k \geq 1}$ such that there exists a sequence of positive Borel functions $(\omega_k)_{k \geq 1}$ defined on $(0,+\infty)$ such that, for all $k \geq 1$, 
\begin{itemize}
\item $\mu_k \ast \omega_k \nu << \mu_k$;
\item $\omega_k$ verifies \eqref{eq:compatibility_condition} with respect to $\nu$ and, for all $f \in \mathcal{C}_b^1(\mathbb{R}^d , \mathbb{R}^d)$ such that $\mathbb{E} f(Z_k) = 0$, 
\begin{align}\label{ineq:weighted_Poincare_type_ineq}
\mathbb{E} \|f(Z_k)\|^2 \leq \int_{\mathbb{R}^d} \int_{\mathbb{R}^d} \|f(z+u)-f(z)\|^2 \omega_k(\|u\|) \nu(du) \mu_k(dz). 
\end{align}
\end{itemize}
Let $(S_n)_{n \geq 1}$ be the sequence of random vectors of $\mathbb{R}^d$ defined, for all $n \geq 1$, by
\begin{align*}
S_n = b_n \sum_{k=1}^n Z_k,
\end{align*}
and with law $\tilde{\mu}_n$. Then, for all $n \geq 1$, 
\begin{align*}
d_{W_2}(\tilde{\mu}_n , \mu) & \leq \frac{1}{2}  \left(\int_{\mathbb{R}^d} \| u \|^2 \tilde{\omega}_n(\| u \|) \nu(du)\right)^{\frac{1}{2}}  \sqrt{ \int_{\mathbb{R}^d} \|u\|^2 \tilde{\omega}_n(\|u\|) \nu(du) - \int_{\mathbb{R}^d} \|x\|^2 \tilde{\mu}_n(dx) } \nonumber \\
& \quad\quad+ \frac{1}{2} \int_{\mathbb{R}^d} \|u\|^2 \left| \tilde{\omega}_n(\|u\|)-1\right| \nu(du),
\end{align*}
where $\tilde{\omega}_n$ is defined by \eqref{eq:weight_function_gene_CLT}. In particular, if $\mathbb{E} \|Z_k\|^2 = \int_{\mathbb{R}^d} \|u\|^2 \omega_k(\|u\|) \nu(du)$, for all $k \geq 1$, then for all $n \geq 1$, 
\begin{align*}
d_{W_2}(\tilde{\mu}_n , \mu) & \leq \frac{1}{2}\int_{\mathbb{R}^d} \|u\|^2 \left| \tilde{\omega}_n(\|u\|)-1\right| \nu(du).
\end{align*}
\end{cor} 

\begin{proof}
The first bound is a direct consequence of the stability estimate of Theorem \ref{thm:stability_estimate_second_moment} together with the fact that the law of $S_n$ satisfies a weighted Poincar\'e-type inequality with weight function $\tilde{\omega}_n$ and L\'evy measure $\nu$, for all $n \geq 1$. For the second bound, under the assumption that $\mathbb{E} \|Z_k\|^2 = \int_{\mathbb{R}^d} \|u\|^2 \omega_k(\|u\|) \nu(du)$, for $k \geq 1$, 
\begin{align*}
\int_{\mathbb{R}^d} \|u\|^2 \tilde{\omega}_{n}\left(\|u\|\right) \nu(du) & = \int_{\mathbb{R}^d} \|u\|^2 \left(\sum_{\ell = 1}^n \omega_\ell\left(\frac{\|u\|}{b_n}\right)\right) \dfrac{k\left(\frac{\|u\|}{b_n}\right)}{k(\|u\|)} \nu(du) , \\
& = \int_{(0, +\infty) \times \mathbb{S}^{d-1}} r^2 \left(\sum_{\ell = 1}^n \omega_\ell \left(\frac{r}{b_n}\right)\right) \dfrac{k \left(\frac{r}{b_n}\right)}{k(r)} \dfrac{k(r) dr}{r} \sigma(dy) , \\
& = b_n^2 \sigma\left(\mathbb{S}^{d-1}\right) \int_0^{+\infty} r \left(\sum_{\ell = 1}^n \omega_\ell \left(r\right)\right)k(r)dr, \\
& = b_n^2 \sum_{\ell = 1}^{n} \int_{\mathbb{R}^d} \|u\|^2 \omega_\ell\left(\|u\|\right) \nu(du) , \\
& =b_n^2 \sum_{\ell = 1}^n \mathbb{E} \|Z_\ell\|^2 = \mathbb{E} \|S_n\|^2.  
\end{align*} 
This concludes the proof of the corollary. 
\end{proof}

\section{Stability Estimate for non-degenerate Symmetric $\alpha$-stable Probability Measures with $\alpha \in (1,2)$.}\label{sec:stability_alpha_stable_pm_12}
\noindent
In this section, we prove Theorems \ref{thm:stability_estimate_NDS_stable}, \ref{thm:all_dimensions_NDS} and \ref{thm:quantitative_approximation_Ls}.~The main idea in \cite{AH20_3} to obtain stability estimates, in the rotationally invariant case, is to use a smooth truncation procedure to build a sequence of ``approximate" eigenvectors: let $g$ and $g_{R}$ be the functions defined, for all $R>0$ and all $x \in \mathbb{R}^d$, by 
\begin{align}\label{eq:eigen_approximate_eigenvector}
g(x) = x , \quad g_{R}(x) = x \exp\left( - \frac{\|x\|^2}{R^2}\right). 
\end{align}
Next, let us prove an analogue of \cite[Lemma $5.6$ and Lemma $5.7$]{AH20_3}.
\begin{prop}\label{prop:Weyl_sequence_general_case}
Let $d \geq 1$ be an integer, let $\alpha \in (1,2)$, let $\mu_\alpha$ be the non-degenerate symmetric $\alpha$-stable probability measure on $\mathbb{R}^d$ which Fourier transform is given by \eqref{eq:fourier_symmetric_stable} and let $\nu_\alpha$ be the associated L\'evy measure on $\mathbb{R}^d$. For $R>0$, let $g_{R}$ be the function defined by \eqref{eq:eigen_approximate_eigenvector}. Then,
\begin{align}\label{eq:Weyl_sequence}
\underset{R \rightarrow +\infty}{\lim} \mathcal{E}^{\nu_\alpha} \left(g_R ; g_R\right) - \langle g_R ; g_R \rangle_{L^2(\mu_\alpha)} = 0,
\end{align}
where $\mathcal{E}^{\nu_\alpha}$ is defined, for all $f_1, f_2 \in \mathcal{C}_b^1(\mathbb{R}^d,\mathbb{R}^d)$, by
\begin{align}\label{eq:form_ND_symmetric_case}
\mathcal{E}^{\nu_\alpha}(f_1 , f_2) = \int_{\mathbb{R}^d} \int_{\mathbb{R}^d} \langle f_1(x+u)-f_1(x) ; f_2(x+u)-f_2(x) \rangle \nu_\alpha(du) \mu_\alpha(dx). 
\end{align}
\end{prop}
\begin{proof}
The proof relies on Fourier methods as in \cite[Lemma 5.6 and Lemma 5.7]{AH20_3}~but with one integration by parts only and with a careful analysis of the leading terms.~First, by Fourier inversion, for all $R>0$ and all $j \in \{1, \dots, d\}$, 
\begin{align*}
\mathbb{E} g_{R,j}(X_\alpha)^2 = \frac{1}{(2\pi)^{2d}} \int_{\mathbb{R}^{2d}} \mathcal{F}(g_{R,j})(\xi_1) \mathcal{F}(g_{R,j})(\xi_2) \widehat{\mu_\alpha} \left(\xi_1 + \xi_2\right) d\xi_1 d\xi_2, 
\end{align*}
where $X_\alpha \sim \mu_\alpha$ and where $g_{R,j}(x) = x_j \exp\left(- \|x\|^2/R^2\right)$, for all $x \in \mathbb{R}^d$ and all $j \in \{1, \dots, d\}$.~Now, classical Fourier analysis ensures, for all $R>0$ and all $\xi_1 \in \mathbb{R}^d$, that
\begin{align}\label{eq:formula_gaussian}
\mathcal{F}(g_{R,j})(\xi_1) &  = \int_{\mathbb{R}^d} x_j \exp\left( - \frac{\|x\|^2}{R^2}\right) \exp(- i  \langle x ; \xi_1 \rangle) dx , \nonumber \\
& = i \partial_{\xi_{1,j}} \left( \int_{\mathbb{R}^d} \exp\left( - \frac{\|x\|^2}{R^2}\right) \exp(- i  \langle x ; \xi_1 \rangle) dx\right)  , \nonumber \\
& = i \pi^{\frac{d}{2}} R^d \partial_{\xi_{1,j}} \left( \exp\left( -\frac{R^2}{4} \langle \xi_1 , \xi_1 \rangle \right)\right),
\end{align}
and similarly for $\mathcal{F}(g_{R,j})(\xi_2)$. Thus, for all $R>0$, 
\begin{align*}
\mathbb{E} g_{R,j}(X_\alpha)^2 & =   -\frac{\pi^d R^{2d}}{(2\pi)^{2d}} \int_{\mathbb{R}^{2d}} \partial_{\xi_{1,j}} \left( \exp\left( -\frac{R^2}{4} \langle \xi_1 , \xi_1 \rangle \right)\right)\partial_{\xi_{2,j}} \left( \exp\left( -\frac{R^2}{4} \langle \xi_2 , \xi_2 \rangle \right)\right) \\
& \quad\quad \times \widehat{\mu_\alpha}\left(\xi_1 + \xi_2\right) d\xi_1 d\xi_2 , \\
& = \frac{\pi^d R^{2d}}{(2\pi)^{2d}} \int_{\mathbb{R}^{2d}} \partial_{\xi_{1,j}} \left( \exp\left( -\frac{R^2}{4} \|\xi_1\|^2 \right)\right) \exp\left( -\frac{R^2}{4} \|\xi_2\|^2 \right) \\
&\quad\quad \times \partial_{\xi_{2,j}}\left(\widehat{\mu_\alpha}\left(\xi_1 + \xi_2\right)\right) d\xi_1 d\xi_2 , \\
& = - \frac{\pi^d R^{2d}}{(2\pi)^{2d}}  \int_{\mathbb{R}^{2d}}\frac{R^2}{2} \xi_{1, j} \exp\left( -\frac{R^2}{4} \|\xi_1\|^2 \right) \exp\left( -\frac{R^2}{4} \|\xi_2\|^2 \right) \\
&\quad\quad \times \tau_{j}(\xi_1+ \xi_2) \widehat{\mu_{\alpha}}(\xi_1+ \xi_2) d\xi_1 d\xi_2 , \\
&= - \frac{\pi^d R}{(2\pi)^{2d}} \int_{\mathbb{R}^{2d}} \frac{\xi_{1, j}}{2} \exp\left( -\frac{\|\xi_1\|^2}{4} \right) \exp\left( -\frac{\|\xi_2\|^2}{4} \right) \\
&\quad\quad \times  \tau_{j}\left(\frac{\xi_1+ \xi_2}{R}\right) \widehat{\mu_{\alpha}}\left(\frac{\xi_1+ \xi_2}{R}\right) d\xi_1 d\xi_2,
\end{align*}
with $\tau_j(\xi_1 + \xi_2) = \int_{\mathbb{R}^d} iu_j \left(e^{i \langle u ; \xi_1 +\xi_2 \rangle}-1\right)\nu_\alpha(du)$. Thanks to scale invariance, for all $R>0$,
\begin{align*}
\mathbb{E} g_{R,j}(X_\alpha)^2 & =  - \frac{\pi^d R^{2-\alpha}}{(2\pi)^{2d}} \int_{\mathbb{R}^{2d}} \frac{\xi_{1, j}}{2} \exp\left( -\frac{\|\xi_1\|^2}{4} \right) \exp\left( -\frac{\|\xi_2\|^2}{4} \right) \\
& \quad\quad \times  \tau_{j}\left(\xi_1+ \xi_2\right) \hat{\mu}_{\alpha}\left(\frac{\xi_1+ \xi_2}{R}\right) d\xi_1 d\xi_2. 
\end{align*}
Now, for the second term, thanks to a similar reasoning, for all $R>0$, 
\begin{equation}\label{eq:analysis_form_stable}
E^{\nu_\alpha} \left(g_{R, j } ; g_{R , j} \right) := \frac{1}{(2\pi)^{2d}} \int_{\mathbb{R}^{2d}} \mathcal{F}(g_{R,j})(\xi_1) \mathcal{F}(g_{R,j})(\xi_2) \widehat{\mu_\alpha}(\xi_1 + \xi_2) m_{\alpha}(\xi_1; \xi_2) d\xi_1 d\xi_2, 
\end{equation}
where $m_\alpha (\xi_1 , \xi_2) = \int_{\mathbb{R}^d} \nu_\alpha(du) \left(e^{ i \langle u ; \xi_1 \rangle}-1\right) \left(e^{ i \langle u ; \xi_2 \rangle}-1\right)$. Then, as previously, for all $R>0$, 
\begin{align*}
E^{\nu_\alpha} \left(g_{R, j } ; g_{R , j} \right) & = -\frac{\pi^{d} R^{2d}}{(2 \pi)^{2d}} \int_{\mathbb{R}^{2d}}  \partial_{\xi_{1,j}} \left( \exp\left( -\frac{R^2}{4} \|\xi_1\|^2 \right)\right)\partial_{\xi_{2,j}} \left( \exp\left( -\frac{R^2}{4} \|\xi_2\|^2 \right)\right)  \\
& \quad \quad \widehat{\mu_\alpha} \left(\xi_1 + \xi_2\right) m_{\alpha}(\xi_1; \xi_2) d\xi_1 d\xi_2 , \\
&= -\frac{\pi^{d} R^{2d}}{(2 \pi)^{2d}} \int_{\mathbb{R}^{2d}} \frac{R^2}{2} \xi_{1,j} \exp\left( -\frac{R^2}{4} \|\xi_1\|^2 \right) \exp\left( -\frac{R^2}{4} \|\xi_2\|^2 \right) \\
& \quad\quad \partial_{\xi_{2,j}}\left( \widehat{\mu_\alpha} \left(\xi_1 + \xi_2\right) m_{\alpha}(\xi_1; \xi_2) \right)d\xi_1 d\xi_2, \\
& = - \frac{\pi^{d} R^{2d}}{(2 \pi)^{2d}} \int_{\mathbb{R}^{2d}} \frac{R^2}{2} \xi_{1,j} \exp\left( -\frac{R^2}{4} \|\xi_1\|^2 \right) \exp\left( -\frac{R^2}{4} \|\xi_2\|^2 \right) \\
& \quad\quad \times \bigg(\tau_j(\xi_1+\xi_2) \widehat{\mu_\alpha}\left(\xi_1 + \xi_2\right) m_{\alpha}(\xi_1; \xi_2)  \\
& \quad\quad +\widehat{\mu_\alpha}\left(\xi_1 + \xi_2\right) \int_{\mathbb{R}^d} i u_j e^{i \langle u ; \xi_2 \rangle }\left(e^{i \langle u ; \xi_1 \rangle}-1\right) \nu_\alpha(du) \bigg) d\xi_1 d\xi_2 \\
& = - \frac{\pi^{d} R}{(2 \pi)^{2d}} \int_{\mathbb{R}^{2d}} \frac{\xi_{1,j}}{2} \exp\left( -\frac{\|\xi_1\|^2}{4} \right) \exp\left( -\frac{\|\xi_2\|^2}{4} \right) \\
& \quad\quad \times \bigg(\tau_j\left(\frac{\xi_1+\xi_2}{R}\right) \widehat{\mu_\alpha}\left(\frac{\xi_1 + \xi_2}{R}\right) m_{\alpha}\left(\frac{\xi_1}{R}; \frac{\xi_2}{R}\right)  \\
& \quad\quad +\widehat{\mu_\alpha}\left(\frac{\xi_1 + \xi_2}{R}\right) \int_{\mathbb{R}^d} i u_j e^{i \langle u ; \frac{\xi_2}{R} \rangle }\left(e^{i \langle u ; \frac{\xi_1}{R} \rangle}-1\right) \nu_\alpha(du) \bigg) d\xi_1 d\xi_2 , \\
& = - \frac{\pi^{d} R}{(2 \pi)^{2d}} \int_{\mathbb{R}^{2d}} \frac{\xi_{1,j}}{2} \exp\left( -\frac{\|\xi_1\|^2}{4} \right) \exp\left( -\frac{\|\xi_2\|^2}{4} \right) \\
& \quad\quad \times \bigg(R^{1- 2\alpha}\tau_j\left(\xi_1+\xi_2\right) \widehat{\mu_\alpha}\left(\frac{\xi_1 + \xi_2}{R}\right) m_{\alpha}\left(\xi_1; \xi_2\right)  \\
& \quad\quad +R^{1 - \alpha}\widehat{\mu_\alpha}\left(\frac{\xi_1 + \xi_2}{R}\right) \int_{\mathbb{R}^d} i u_j e^{i \langle u ; \xi_2 \rangle }\left(e^{i \langle u ; \xi_1 \rangle}-1\right) \nu_\alpha(du) \bigg) d\xi_1 d\xi_2.
\end{align*}
Thanks to the previous computations, one sees that $\mathcal{E}^{\nu_\alpha}(g_{R} , g_{R})$ is the sum of two terms: a term of order $R^{2-2\alpha}$ which goes to $0$ as $R$ tends to $+\infty$ and a term of order $R^{2 - \alpha}$ similar to the one appearing in the formula for $\mathbb{E} g_{R,j}(X_\alpha)^2$. Finally, for all $\xi_1 , \xi_2 \in \mathbb{R}^{d}$, 
\begin{align*}
 \int_{\mathbb{R}^d} i u_j e^{i \langle u ; \xi_2 \rangle }\left(e^{i \langle u ; \xi_1 \rangle}-1\right) \nu_\alpha(du)  - \tau_j(\xi_1 + \xi_2) & = \int_{\mathbb{R}^d} i u_j \bigg(e^{i \langle u ; \xi_1 + \xi_2 \rangle}- e^{i \langle u ; \xi_2 \rangle } \\
 &\quad\quad -e^{i \langle u ; \xi_1 + \xi_2 \rangle}+1 \bigg) \nu_\alpha(du) , \\
 & = - \int_{\mathbb{R}^d} i u_j \left(e^{i \langle u ; \xi_2 \rangle}-1\right)\nu_\alpha(du) , \\
 & = -\tau_j(\xi_2).
\end{align*}
Then, for the difference of the two leading terms,
{\allowdisplaybreaks
\begin{align*}
\frac{\pi^{d} R^{2-\alpha}}{(2 \pi)^{2d}} & \int_{\mathbb{R}^{2d}} \frac{\xi_{1,j}}{2} \exp\left( -\frac{\|\xi_1\|^2}{4} \right) \exp\left( -\frac{\|\xi_2\|^2}{4} \right)\widehat{\mu_\alpha}\left(\frac{\xi_1 + \xi_2}{R}\right) \tau_j(\xi_2) d\xi_1 d\xi_2 , \\
& = - \frac{\pi^{d} R^{2-\alpha}}{(2 \pi)^{2d}} \int_{\mathbb{R}^{2d}} \partial_{\xi_{1,j}} \left( \exp\left( -\frac{\|\xi_1\|^2}{4}\right) \right)\exp\left( -\frac{\|\xi_2\|^2}{4} \right) \\
&\quad\quad \times \widehat{\mu_\alpha}\left(\frac{\xi_1 + \xi_2}{R}\right) \tau_j(\xi_2) d\xi_1 d\xi_2 , \\
& =  \frac{\pi^{d} R^{2-\alpha}}{(2 \pi)^{2d}} \int_{\mathbb{R}^{2d}} \exp\left( -\frac{\|\xi_1\|^2}{4} \right) \exp\left( -\frac{\|\xi_2\|^2}{4} \right) \\
&\quad\quad \times \partial_{\xi_{1,j}} \left(\widehat{\mu_\alpha}\left(\frac{\xi_1 + \xi_2}{R}\right)\right) \tau_j(\xi_2) d\xi_1 d\xi_2 , \\
& =  \frac{\pi^{d} R^{2-\alpha}}{(2 \pi)^{2d}} \int_{\mathbb{R}^{2d}} \exp\left( -\frac{\|\xi_1\|^2}{4} \right) \exp\left( -\frac{\|\xi_2\|^2}{4} \right) \frac{1}{R} \tau_j\left(\frac{\xi_1+\xi_2}{R}\right) \\
&\quad\quad \times \widehat{\mu_\alpha}\left(\frac{\xi_1 + \xi_2}{R}\right)\tau_j(\xi_2) d\xi_1 d\xi_2 , \\
& =  \frac{\pi^{d} R^{2-2\alpha}}{(2 \pi)^{2d}} \int_{\mathbb{R}^{2d}} \exp\left( -\frac{\|\xi_1\|^2}{4} \right) \exp\left( -\frac{\|\xi_2\|^2}{4} \right)\tau_j\left(\xi_1 + \xi_2\right) \\
&\quad\quad \times \widehat{\mu_\alpha}\left(\frac{\xi_1 + \xi_2}{R}\right)\tau_j(\xi_2) d\xi_1 d\xi_2,
\end{align*}
}
which clearly tends to $0$ as $R$ tends to $+\infty$. This concludes the proof of the proposition. 
\end{proof}
\noindent
Then, thanks to the previous observation and thanks to Proposition \ref{prop:Stein_Method_Stable_NDS}, we are ready to prove Theorem \ref{thm:stability_estimate_NDS_stable}.\\
\\
\textit{Proof of Theorem \ref{thm:stability_estimate_NDS_stable}}.
The proof is very similar to the one of \cite[Theorem $5.15$]{AH20_3}.~Let $h \in \mathcal{C}_c^{\infty}\left(\mathbb{R}^d\right)$ be such that $|h|_{\operatorname{Lip}} \leq 1$.  First, by Proposition \ref{prop:Stein_Method_Stable_NDS}, 
\begin{align*}
\mathbb{E} h(X) - h(X_\alpha) = \mathbb{E} \left( - \langle X ; \nabla(f_h)(X)\rangle + \int_{\mathbb{R}^d} \langle \nabla(f_h)(X+u) - \nabla(f_h)(X)  ; u\rangle \nu_\alpha(du) \right),
\end{align*}
where $X \sim \mu_X$ and where $X_\alpha \sim \mu_\alpha$. Now, the main idea is to introduce error terms which are easy to control. First, note that, for all $R>0$, 
\begin{align*}
\left| \mathbb{E} \left( - \langle X ; \nabla(f_h)(X)\rangle \right) + \mathbb{E} \left( \langle g_R(X) ; \nabla(f_h)(X) \rangle \right)\right| & \leq M_1(f_h) \mathbb{E} \| X - g_R(X) \| , \\
& \leq \mathbb{E} \| X - g_R(X) \|,
\end{align*}
which clearly tends to $0$ as $R$ tends to $+\infty$. So, we are left to bound the following error term: for all $R>0$, 
\begin{align*}
E_{R,1} : = \left|\mathbb{E} \left( - \langle g_R(X) ; \nabla(f_h)(X) \rangle + \int_{\mathbb{R}^d} \langle \nabla(f_h)(X+u) - \nabla(f_h)(X) ; u \rangle \nu_\alpha(du) \right) \right|.
\end{align*}
Now, for all $R>0$, $g_R \in \mathcal{S}(\mathbb{R}^d, \mathbb{R}^d)$ and $\int_{\mathbb{R}^d} g_R(x) \mu_X(dx) = 0$, so that, 
\begin{align*}
\mathbb{E} \langle g_R(X) ; \nabla(f_h)(X) \rangle & = \mathcal{E}_{\nu_\alpha, \omega_X} \left( G_{0^+}(g_R); \nabla(f_h)\right) , \\
& = \int_{\mathbb{R}^d} \int_{\mathbb{R}^d} \langle \Delta_u(G_{0^+}(g_R))(x) ; \Delta_u(\nabla(f_h))(x) \rangle \omega_X(\|u\|) \nu_\alpha(du) \mu_X(dx),
\end{align*}
where $\Delta_u(f)(x) = f(x+u) - f(x)$, where $\mathcal{E}_{\nu_\alpha, \omega_X}$ is defined by \eqref{eq:weighted_form_NDS_gen}, where
\begin{align*}
G_{0^+}(g_R) = \int_0^{+\infty} P_t(g_R) dt,
\end{align*}
and where $(P_t)_{t \geq 0}$ is the symmetric contraction semigroup generated by the smallest closed extension of $\left(\mathcal{E}_{\nu_\alpha, \omega_X}, \mathcal{C}_b^1(\mathbb{R}^d , \mathbb{R}^d) \right)$. Thus, for all $R>0$, 
\begin{align*}
E_{R,1} & = \bigg| \mathbb{E} \bigg( -\int_{\mathbb{R}^d} \langle \Delta_u(G_{0^+}(g_R))(X) ; \Delta_u(\nabla(f_h))(X) \rangle \omega_X(\|u\|) \nu_\alpha(du) \\
&\quad\quad   + \int_{\mathbb{R}^d} \langle \Delta_u(\nabla(f_h))(X) ; u \rangle \nu_\alpha(du) \bigg) \bigg|. 
\end{align*}
Next, for all $R>0$,
\begin{align*}
E_{R,1} & \leq \left| \mathbb{E} \int_{\mathbb{R}^d} \langle \Delta_u(\nabla(f_h))(X) ; u \rangle \left(1 - \omega_X(\|u\|)\right)\nu_\alpha(du) \right| \\
& \quad\quad + \bigg| \mathbb{E} \bigg( -\int_{\mathbb{R}^d} \langle \Delta_u(G_{0^+}(g_R))(X) ; \Delta_u(\nabla(f_h))(X) \rangle \omega_X(\|u\|) \nu_\alpha(du) \\
& \quad\quad + \int_{\mathbb{R}^d} \langle \Delta_u(\nabla(f_h))(X) ; u \rangle \omega_X(\|u\|) \nu_\alpha(du) \bigg) \bigg|.
\end{align*}
For the first term on the right-hand side of the previous inequality, cutting the integral in $u$ into two parts:
\begin{align}\label{ineq:easy_bound}
\left| \mathbb{E} \int_{\mathbb{R}^d} \langle \Delta_u(\nabla(f_h))(X) ; u \rangle \left(1 - \omega_X(\|u\|)\right)\nu_\alpha(du) \right| & \leq 2 \int_{\mathcal{B}(0,1)^c} \|u \| \left| \omega_X( \|u\| )-1 \right|\nu_\alpha(du) \nonumber \\
& + C_{\alpha,d} \int_{\mathcal{B}(0,1)} \|u \|^2 \left| \omega_X( \|u\| )-1 \right| \nu_\alpha(du),
\end{align}
where $C_{\alpha,d}$ is given by Proposition \ref{prop:Stein_Method_Stable_NDS}.~So, it remains to deal with
\begin{align}\label{eq:error_term_R12}
E_{R,1,2} & = \bigg| \mathbb{E} \bigg( -\int_{\mathbb{R}^d} \langle \Delta_u(G_{0^+}(g_R))(X) ; \Delta_u(\nabla(f_h))(X) \rangle \omega_X(\|u\|) \nu_\alpha(du) \nonumber \\
& \quad\quad + \int_{\mathbb{R}^d} \langle \Delta_u(\nabla(f_h))(X) ; u \rangle \omega_X(\|u\|) \nu_\alpha(du) \bigg) \bigg| , \quad R>0. 
\end{align}
Observe that, for all $R>0$, 
\begin{align*}
&\bigg| \mathbb{E} \int_{\mathbb{R}^d} \langle \Delta_u(\nabla(f_h))(X) ; g_R(X+u) - g_R(X)- u \rangle \omega_X(\|u\|) \nu_\alpha(du)\bigg| \nonumber \\
&\quad \quad \leq  2 \mathbb{E} \int_{\mathcal{B}(0,1)^c} \| F_R(X+u) - F_R(X) \| \omega_X(\|u\|) \nu_\alpha(du) \\
&\quad \quad + C_{\alpha,d} \mathbb{E} \int_{\mathcal{B}(0,1)} \| u \| \| F_R(X+u) - F_R(X) \| \omega_X(\|u\|) \nu_\alpha(du),
\end{align*}
where, for all $x \in \mathbb{R}^d$ and all $R>0$, 
\begin{align*}
F_R(x) = x \left(1 - \exp\left( - \frac{\|x\|^2}{R^2}\right)\right).
\end{align*}
Moreover, for all $x,u \in \mathbb{R}^d$, 
\begin{align*}
\underset{R \longrightarrow +\infty}{\lim} \| F_R(x+u) - F_R(x) \| = 0,
\end{align*}
and, for all $x \in \mathbb{R}^d$, all $R>0$ and all $ j , k \in \{1 , \cdots, d\}$, 
\begin{align*}
\partial_k \left(F_{R,j}(x)\right) = \left\{
    \begin{array}{ll}
         & \left(1 - \exp\left(-\frac{\|x\|^2}{R^2}\right)\right) + 2 \frac{x_j^2}{R^2} \exp\left( - \frac{\|x\|^2}{R^2}\right)  \mbox{if } j=k ,\\
         & 2 \frac{x_jx_k}{R^2} \exp\left( - \frac{\|x\|^2}{R^2}\right) \mbox{if not}.
    \end{array}
\right.
\end{align*}
Thus, for all $x,u \in \mathbb{R}^d$ and all $R>0$,
\begin{align*}
\| F_R(x+u) - F_R(x) \| \leq C_d \|u\|,
\end{align*}
for some $C_d >0$. Then, uniformly in $h \in \mathcal{C}_c^{\infty}(\mathbb{R}^d)$ with $|h|_{\operatorname{Lip}} \leq 1$, 
\begin{align*}
\bigg| \mathbb{E} \int_{\mathbb{R}^d} \langle \Delta_u(\nabla(f_h))(X) ; g_R(X+u) - g_R(X)- u \rangle \omega_X(\|u\|) \nu_\alpha(du)\bigg| \longrightarrow 0,
\end{align*}
as $R$ tends to $+\infty$. To conlcude, it remains to control the following error term: for all $R>0$, 
\begin{align*}
E_{R,1,2,2} : = \left| \mathcal{E}_{\nu_\alpha , \omega_X} ( \nabla(f_h); G_{0^+}\left(g_R\right) - g_R ) \right|.
\end{align*}
By Cauchy-Schwarz inequality, for all $R>0$, 
\begin{align*}
E_{R,1,2,2}  \leq \left( \mathcal{E}_{\nu_\alpha , \omega_X} \left(  \nabla(f_h) ;  \nabla(f_h) \right)\right)^{\frac{1}{2}} \left(\mathcal{E}_{\nu_\alpha , \omega_X} \left(  G_{0^+}\left(g_R\right) - g_R ; G_{0^+}\left(g_R\right) - g_R \right)\right)^{\frac{1}{2}}.
\end{align*}
Now, developing the square, for all $R>0$, 
\begin{align*}
\mathcal{E}_{\nu_\alpha , \omega_X} \left(  G_{0^+}\left(g_R\right) - g_R ; G_{0^+}\left(g_R\right) - g_R \right) & = \mathcal{E}_{\nu_\alpha , \omega_X} \left(  G_{0^+}\left(g_R\right) ; G_{0^+}\left(g_R\right) \right) + \mathcal{E}_{\nu_\alpha , \omega_X} \left(  g_R ; g_R \right) \\
&\quad\quad - 2 \mathcal{E}_{\nu_\alpha , \omega_X} \left(  G_{0^+}\left(g_R\right) ; g_R \right). 
\end{align*}
Moreover, for all $R>0$, 
\begin{align*}
\mathcal{E}_{\nu_\alpha , \omega_X} \left(  G_{0^+}\left(g_R\right) ; g_R \right) & = \langle g_R ; g_R \rangle_{L^2(\mu_X)} , \\
 \mathcal{E}_{\nu_\alpha , \omega_X} \left(  G_{0^+}\left(g_R\right) ; G_{0^+}\left(g_R\right) \right) & = \langle g_R ; G_{0^+}\left(g_R\right) \rangle_{L^2(\mu_X)} \\
& \leq \| g_R\|_{L^2(\mu_X)} \| G_{0^+}\left(g_R\right) \|_{L^2(\mu_X)} , \\
& \leq \| g_R\|^2_{L^2(\mu_X)}. 
\end{align*}
Thus, for all $R>0$, 
\begin{align*}
\mathcal{E}_{\nu_\alpha , \omega_X} \left(  G_{0^+}\left(g_R\right) -g_R ; G_{0^+}\left(g_R\right) -g_R \right)  \leq \mathcal{E}_{\nu_\alpha , \omega_X} \left(  g_R ; g_R \right) - \langle g_R ; g_R \rangle_{L^2(\mu_X)},
\end{align*}
This concludes the proof of the theorem. $\square$

\begin{rem}\label{rem:discussion_condition_constant_weight_function}
(i) The condition
\begin{align*}
\underset{R \longrightarrow +\infty}{\limsup} \left(\mathcal{E}_{\nu_\alpha, \omega_X}(g_R , g_R) - \langle g_R ; g_R \rangle_{L^2(\mu_X)}\right) \leq \delta,
\end{align*}
is natural since Proposition \ref{prop:Weyl_sequence_general_case} ensures that 
\begin{align*}
\underset{R \longrightarrow +\infty}{\lim} \left(\mathcal{E}^{\nu_\alpha}(g_R , g_R) - \langle g_R ; g_R \rangle_{L^2(\mu_\alpha)} \right) = 0,
\end{align*}
where $\mathcal{E}^{\nu_\alpha}$ is given by \eqref{eq:form_ND_symmetric_case}.\\
(ii) Assume that $\mu_X$ verifies a weighted Poincar\'e-type inequality with constant weight function equal to $C$ and with L\'evy measure $\nu_\alpha$. Then, under the other assumptions of Theorem \ref{thm:stability_estimate_NDS_stable}, the stability estimate boils down to  
\begin{align}\label{eq:stability_estimate_rotation_invariant_constant_weight}
W_1\left(\mu_X , \mu_\alpha\right) & \leq \sqrt{C} 
\left(C_{\alpha,d}^2 \int_{\|u\| \leq 1} \|u\|^2 \nu_\alpha(du) + 4 \int_{\|u\| \geq 1}\nu_\alpha(du)\right)^{\frac{1}{2}}\sqrt{\delta} \nonumber \\
& \quad\quad + \bigg(2 \int_{\mathcal{B}(0,1)^c} \|u \|\nu_\alpha(du) + C_{\alpha,d} \int_{\mathcal{B}(0,1)} \|u \|^2\nu_\alpha(du)\bigg) \left| C - 1 \right|,
\end{align}
which should be compared to \cite[Theorem 5.15]{AH20_3}. \\
(iii) Let us assume that $\delta = 0$. Then, from the proof of Theorem \ref{thm:stability_estimate_NDS_stable}, it is clear that the following Stein-type identity holds true: for all $f \in \mathcal{S}(\mathbb{R}^d)$, 
\begin{align}\label{eq:Stein_type_identity_mu_X}
\langle g ; \nabla(f) \rangle_{L^2(\mu_X)} = \int_{\mathbb{R}^d} \int_{\mathbb{R}^d} \langle u ; \nabla(f)(x+u) - \nabla(f)(x) \rangle \omega\left( \|u\| \right) \nu_\alpha(du) \mu_X(dx), 
\end{align}
where $g(x) = x$, for all $x \in \mathbb{R}^d$. Then, using Fourier inversion formula, the Fourier transform of $\mu_X$ verifies the following partial differential equation: for all $\xi \in \mathbb{R}^d$, 
\begin{align}\label{eq:edp_fourier_transform_muX}
\langle \xi ; \nabla(\widehat{\mu_X})(\xi) \rangle = \widehat{\mu_X}\left(\xi\right) \left(\int_{\mathbb{R}^d} \left(e^{i \langle u ; \xi \rangle}-1\right) \langle i \xi ; u \rangle \omega_X(\|u\|) \nu_\alpha(du)\right).
\end{align}
Moreover, $ \widehat{\mu_X}(0) = 1$. Then, for all $\xi \in \mathbb{R}^d$, 
\begin{align*}
\widehat{\mu_X}(\xi) = \exp \left( \int_{\mathbb{R}^d} \left(e^{i \langle u ; \xi \rangle} -1 - i \langle u ; \xi \rangle\right) \omega_X(\|u\|) \nu_\alpha(du)\right). 
\end{align*}
(iv) At this point, let us briefly discuss the different semigroups generated by closed extension(s) of the bilinear form $\mathcal{E}_{\nu_\alpha, \mu_\alpha}$ given, for all $f_1 , f_2 \in \mathcal{C}_b^1(\mathbb{R}^d)$, by
\begin{align}\label{eq:form_nua_mua}
\mathcal{E}_{\nu_\alpha, \mu_\alpha}(f_1, f_2) = \int_{\mathbb{R}^{2d}} (f_1(x+u) - f_1(x))(f_2(x+u) - f_2(x))\nu_\alpha(du) \mu_\alpha(dx).
\end{align}
Recall that $(P^{\nu_\alpha}_t)_{t \geq 0}$ denotes the $\alpha$-stable Ornstein-Uhlenbeck semigroup given, for all $f \in \mathcal{S}(\mathbb{R}^d)$, all $x \in \mathbb{R}^d$ and all $t \geq 0$, by
\begin{align*}
P^{\nu_\alpha}_t(f)(x) = \int_{\mathbb{R}^d} f\left(xe^{-t} + \left(1-e^{-\alpha t}\right)^{\frac{1}{\alpha}}y \right) \mu_\alpha(dy). 
\end{align*}
Since the measure $\mu_\alpha$ is invariant for the semigroup $\left(P_t^{\nu_\alpha}\right)_{t \geq 0}$, by standard arguments, $\left(P_t^{\nu_\alpha}\right)_{t \geq 0}$ admits a continuous extension from $L^2(\mu_\alpha)$ to itself. Then, let us denote by $\left((P_t^{\nu_\alpha})^*\right)_{t \geq 0}$ the dual semigroup of $\left(P_t^{\nu_\alpha}\right)_{t \geq 0}$ which is, then, continuous from $L^2(\mu_\alpha)$ to itself. As proved in \cite[Theorem 5.4]{AH20_4}, for all $s,t \geq 0$, the operators $P_t^{\nu_\alpha}$ and $(P_s^{\nu_\alpha})^*$ commute so that the continuous family of linear contractions $\left(\mathcal{P}_t\right)_{t \geq 0}$ defined, for all $t \geq 0$, by 
\begin{align}\label{eq:carre_mehler_sg}
\mathcal{P}_t = P_{\frac{t}{\alpha}}^{\nu_\alpha} \circ (P_{\frac{t}{\alpha}}^{\nu_\alpha})^* = (P_{\frac{t}{\alpha}}^{\nu_\alpha})^* \circ P_{\frac{t}{\alpha}}^{\nu_\alpha},
\end{align}
is a $\mathcal{C}_0$-semigroup on $L^2(\mu_\alpha)$.~This semigroup of operators is called the ``carr\'e de Mehler" semigroup. Next, since $\mu_\alpha$ is infinitely divisible on $\mathbb{R}^d$ with L\'evy measure $\nu_\alpha$, $\mu_\alpha \ast \nu_\alpha << \mu_\alpha$ (see, e.g., \cite[Lemma $4.1.$]{Chen_85}). Now, the form $\left(\mathcal{E}_{\nu_\alpha, \mu_\alpha}, \mathcal{C}_b^1(\mathbb{R}^d)\right)$ is closable and let us denote by $\left( \overline{\mathcal{E}_{\nu_\alpha, \mu_\alpha}},  D \left(\overline{\mathcal{E}_{\nu_\alpha, \mu_\alpha}} \right)\right)$ its closure on $L^2(\mu_\alpha)$.~The $\mathcal{C}_0$-semigroup of symmetric contractions on $L^2(\mu_\alpha)$ associated with $\left( \overline{\mathcal{E}_{\nu_\alpha, \mu_\alpha}},  D \left(\overline{\mathcal{E}_{\nu_\alpha, \mu_\alpha}} \right)\right)$ is denoted by $(P_t)_{t\geq 0}$. Finally, 
\begin{align}\label{eq:closure_compact_support_rep}
D \left(\overline{\mathcal{E}_{\nu_\alpha, \mu_\alpha}} \right) = \overline{\mathcal{C}_c^{\infty}(\mathbb{R}^d)}^{\| \cdot \|_{H+E}}, 
\end{align}
where $\| \cdot \|_{H+E}$ is defined by 
\begin{align*}
\| f \|^2_{H+E} = \| f \|^2_{L^2(\mu_\alpha)} + \mathcal{E}_{\nu_\alpha, \mu_\alpha}(f ,f).  
\end{align*}
Thanks to Theorem \ref{thm:thesame1} of the appendix section, both semigroups $(\mathcal{P}_t)_{t \geq 0}$ and $(P_t)_{t \geq 0}$ coincide under the assumption that 
\begin{align}\label{eq:cond_uniform_boundedness_remark}
\left\| \dfrac{\nabla(p_\alpha)}{p_\alpha} \right\|_{\infty} < +\infty,
\end{align}
where $p_\alpha$ is the Lebesgue density of the probability measure $\mu_\alpha$. The previous condition is clearly satisfied when $\mu_\alpha = \mu_\alpha^{\operatorname{rot}}$ since in this situation the following pointwise bounds hold true: for all $x \in \mathbb{R}^d$, 
\begin{align}\label{eq:two_sided_bounds}
p_\alpha(x) \asymp \frac{1}{\left(1+ \|x\|\right)^{\alpha + d}}, \quad \left\| \nabla(p_\alpha)(x) \right\| \leq \dfrac{M_{\alpha,d}}{(1+\|x\|)^{\alpha+d+1}},
\end{align}
where $\asymp$ means that $p_\alpha$ is bounded from below and from above by the function $x \mapsto 1/(1+\|x\|)^{\alpha+d}$ up to some constants depending on $\alpha$ and on $d$ and where $M_{\alpha,d}>0$. Finally, condition \eqref{eq:cond_uniform_boundedness_remark} is also verified for the $\alpha$-stable probability measure $\mu_{\alpha,d}$ defined by \eqref{eq:characteristic_function_ind}. 
\end{rem}
\noindent

Now, let us discuss applications of the previous stability theorems to generalized central limit theorems with $\alpha$-stable limiting laws for $\alpha \in (1,2)$. As in Section \ref{sec:stability_estimate_finite_second_moment}, let us start with the canonical example presented in \eqref{eq:def_cf_k}.~So, let $(Z_k)_{k \geq 1}$ be a sequence of independent random vectors of $\mathbb{R}^d$ with characteristic function defined, for all $k \geq 1$ and all $\xi \in \mathbb{R}^d$, by 
\begin{align}\label{eq:cf_canonical_case_NDS_stable}
\varphi_k(\xi) = \dfrac{\widehat{\mu_\alpha}\left((k+1) \xi\right)}{\widehat{\mu_\alpha}\left(k \xi\right)},
\end{align}
where $\widehat{\mu_\alpha}$ is the Fourier transform of $\mu_\alpha$ given by \eqref{eq:fourier_symmetric_stable}.~Thanks to stability, for all $k \geq 1$ and all $\xi \in \mathbb{R}^d$, 
\begin{align*}
\varphi_k(\xi) = \exp\left(\left((k+1)^\alpha - k^\alpha \right) \int_{\mathbb{R}^d} \left(e^{i \langle u ; \xi \rangle} - 1- i \langle u ; \xi \rangle \right) \nu_\alpha(du)\right).
\end{align*}
Then, the characteristic function of the random vector $S_n$ is given, for all $n \geq 1$ and all $\xi \in \mathbb{R}^d$, by
\begin{align*}
\varphi_{S_n}(\xi) = \exp\left(\left(\left(1+\frac{1}{n}\right)^\alpha - \frac{1}{n^\alpha} \right) \int_{\mathbb{R}^d} \left(e^{i \langle u ; \xi \rangle} - 1- i \langle u ; \xi \rangle \right) \nu_\alpha(du)\right),
\end{align*}
from which one deduces the following equality in law: $S_n =_{\mathcal{L}} \left(\left(1+\frac{1}{n}\right)^\alpha - \frac{1}{n^\alpha} \right)^{\frac{1}{\alpha}} X_\alpha$, for all $n \geq 1$, where $X_\alpha \sim \mu_\alpha$. Then, the law of $S_n$ verifies a weighted Poincar\'e-type inequality with weight function $\omega_n = \left(\left(1+\frac{1}{n}\right)^\alpha - \frac{1}{n^\alpha} \right)$ and with L\'evy measure $\nu_\alpha$, for all $n \geq 1$.~In order to apply Theorem \ref{thm:stability_estimate_NDS_stable},~it remains to prove, for all $n \geq 1$, that
\begin{align}\label{eq:weyl_condition}
\underset{R \longrightarrow +\infty}{\lim} \left(\mathcal{E}_{\nu_\alpha , \omega_n}(g_R , g_R) - \langle g_R ; g_R \rangle_{L^2(\tilde{\mu}_n)} \right) = 0,
\end{align}
where $S_n \sim \tilde{\mu}_n$ and where $\mathcal{E}_{\nu_\alpha , \omega_n}$ is defined, for all $f_1, f_2 \in \mathcal{C}_b^1(\mathbb{R}^d , \mathbb{R}^d)$, by 
\begin{align*}
\mathcal{E}_{\nu_\alpha , \omega_n}(f_1 , f_2) & =\left( \left(1+\frac{1}{n}\right)^{\alpha} - \frac{1}{n^{\alpha}}\right) \\
& \quad \times \int_{\mathbb{R}^d} \int_{\mathbb{R}^d} \langle f_1(x+u) - f_1(x) ; f_2(x+u) - f_2(x) \rangle \nu_\alpha(du) \tilde{\mu}_n(dx). 
\end{align*}
Moreover, it is clear that $\tilde{\mu}_n = \mu_\alpha \circ T_{a_n}^{-1}$ with,
\begin{align*}
a_n = \left(\left(1+\frac{1}{n}\right)^\alpha - \frac{1}{n^{\alpha}}\right)^{\frac{1}{\alpha}}, \quad n \geq 1. 
\end{align*}
Thus, by scale invariance, for all $n \geq 1$ and all $R>0$, 
\begin{align*}
\mathcal{E}_{\nu_\alpha , \omega_n}(g_R , g_R) = \mathcal{E}^{\nu_\alpha}(g_R^n ; g_R^n), 
\end{align*}
where $g_R^n(x) = g_R(a_n x)$, for all $x \in \mathbb{R}^d$. Similarly, $\|g_R^n\|_{L^2(\mu_\alpha)} = \|g_R\|_{L^2(\tilde{\mu}_n)}$, for all $R>0$ and all $n \geq 1$.~Proposition \ref{prop:Weyl_sequence_general_case} ensures that \eqref{eq:weyl_condition} holds true. A straightforward application of Theorem \ref{thm:stability_estimate_NDS_stable} gives, for all $n \geq 1$,
\begin{align}\label{ineq:quantitative_canonical_example_NDS_stable}
W_1(\tilde{\mu}_n , \mu_\alpha) \leq \bigg(2 \int_{\mathcal{B}(0,1)^c} \|u \| \nu_\alpha(du)+ C_{\alpha,d} \int_{\mathcal{B}(0,1)} \|u \|^2 \nu_\alpha(du) \bigg) \left| \left(1+\frac{1}{n}\right)^\alpha - \frac{1}{n^\alpha} -1 \right|.
\end{align}
\noindent
Next, let us provide the proof of Theorem \ref{thm:all_dimensions_NDS}.\\
\\
\textit{Proof of Theorem \ref{thm:all_dimensions_NDS}}. First, let us prove the spectral condition. Let $\mathcal{E}_{\nu_\alpha, \omega_n}$ be the bilinear form defined, for all $f_1,f_2 \in \mathcal{C}^1_b\left(\mathbb{R}^d, \mathbb{R}^d\right)$, by 
\begin{align}\label{def:bilinear_form}
\mathcal{E}_{\nu_\alpha, \omega_n}\left(f_1,f_2\right) : = \int_{\mathbb{R}^d} \int_{\mathbb{R}^d} \langle \Delta_u(f_1)(x) ;\Delta_u(f_2)(x) \rangle \omega_n(\|u\|) \nu_{\alpha}(du) \mu_n^\alpha(dx),
\end{align}
with, for all $n \geq 1$ and all $u \in \mathbb{R}^d \setminus \{0\}$, 
\begin{align}\label{eq:definition_omegan}
\omega_n\left(\| u \|\right) = n k \left(n^{\frac{1}{\alpha}} \|u\|\right)\|u\|^\alpha.
\end{align}
For all $R>0$, let $g_R$ be the function defined by \eqref{eq:eigen_approximate_eigenvector}. Then, for all $R>0$, all $j \in \{1, \dots, d\}$ and all $n \geq 1$,
\begin{align}\label{eq:fourier_rep_norm_2_alldim}
\| g_{R,j} \|^2_{L^2(\mu_n^\alpha)} & = \int_{\mathbb{R}^d} g_{R,j}(x)^2 \mu_n^\alpha(dx) = \frac{1}{(2\pi)^d} \int_{\mathbb{R}^d} \mathcal{F}(g^2_{R,j})(\xi) \widehat{\mu_n^{\alpha}}(\xi) d\xi , \nonumber \\
& = \frac{R^{d+2}}{(2\pi)^d} \int_{\mathbb{R}^d} \mathcal{F}(g^2_{1,j})(R\xi) \widehat{\mu_n^{\alpha}}(\xi) d\xi = \frac{R^{2}}{(2\pi)^d} \int_{\mathbb{R}^d} \mathcal{F}(g^2_{1,j})(\xi) \widehat{\mu_n^{\alpha}}\left(\frac{\xi}{R}\right) d\xi , \nonumber \\
& = \frac{i R^2}{\left(2\pi\right)^d} \int_{\mathbb{R}^d} \partial_{\xi_j}\left(\mathcal{F}(x_j \exp\left(- 2\|x\|^2\right))(\xi)\right) \widehat{\mu_n^{\alpha}}\left(\frac{\xi}{R}\right) d\xi , \nonumber \\
& = - \frac{i R^2}{\left(2\pi\right)^d} \int_{\mathbb{R}^d} \mathcal{F}(x_j \exp\left(- 2\|x\|^2\right))(\xi) \partial_{\xi_j} \left(\widehat{\mu_n^{\alpha}}\left(\frac{\xi}{R}\right)\right) d\xi. 
\end{align}
But, for all $\xi \in \mathbb{R}^d$, all $R>0$, all $j \in \{1 , \dots, d\}$ and all $n \geq 1$, 
$$ \partial_{\xi_j} \left(\widehat{\mu_n^{\alpha}}\left(\frac{\xi}{R}\right)\right) = \left(i\int_{(0, +\infty) \times \mathbb{S}^{d-1}} \theta_j \left(e^{i \langle r\theta ; \xi \rangle}-1\right) n  k\left(n^{\frac{1}{\alpha}} R r\right)dr \sigma(d\theta) \right) \widehat{\mu_n^\alpha}\left(\frac{\xi}{R}\right).$$
Thus, for all $R>0$, all $j \in \{1 , \dots, d\}$ and all $n \geq 1$,
\begin{align*} \| g_{R,j} \|^2_{L^2(\mu_n^\alpha)} & =  \frac{R^2}{\left(2\pi\right)^d} \int_{\mathbb{R}^d} \mathcal{F}(x_j \exp\left(- 2\|x\|^2\right))(\xi) \left(\int_{(0, +\infty) \times \mathbb{S}^{d-1}} \theta_j \left(e^{i \langle r\theta ; \xi \rangle}-1\right) n  k\left(n^{\frac{1}{\alpha}} R r\right)dr \sigma(d\theta) \right) \\
& \quad\quad \times \widehat{\mu_n^\alpha}\left(\frac{\xi}{R}\right) d\xi.
\end{align*}
But, for all $u \in \mathbb{R}^d \setminus \{0\}$ and all $n \geq 1$, 
\begin{align*}
\underset{R \longrightarrow +\infty}{\lim} R^\alpha n k \left(n^{\frac{1}{\alpha}} R \| u \|\right) = \frac{1}{\|u\|^\alpha}. 
\end{align*}
Moreover, for all $u \in \mathbb{R}^d \setminus \{0\}$, all $n \geq 1$ and all $R>0$, 
\begin{align*}
R^\alpha n k \left(n^{\frac{1}{\alpha}} R \| u \|\right) \leq \underset{z \in (0,+\infty)}{\sup} \left(z^\alpha k(z)\right) \frac{1}{\|u\|^\alpha}. 
\end{align*}
Then, for all $n \geq 1$ and all $j \in \{1, \dots, d\}$, 
\begin{align*}
\frac{\| g_{R,j} \|^2_{L^2(\mu_n^\alpha)}}{R^{2-\alpha}} & = \frac{1}{(2\pi)^d} \int_{\mathbb{R}^d} \mathcal{F}(x_j \exp\left(- 2\|x\|^2\right))(\xi) \\
 & \quad\quad \times \left(\int_{(0,+\infty) \times \mathbb{S}^{d-1}} \theta_j \left(e^{i \langle r \theta ; \xi \rangle}-1\right) n R^\alpha k\left(n^{\frac{1}{\alpha}} R r\right)dr\sigma(d\theta) \right)  \widehat{\mu_n^\alpha}\left(\frac{\xi}{R}\right) d\xi \\
&\quad\quad \longrightarrow \frac{1}{(2\pi)^d} \int_{\mathbb{R}^d} \mathcal{F}(x_j \exp\left(- 2\|x\|^2\right))(\xi) \left(\int_{\mathbb{R}^d} u_j \left(e^{i \langle u ; \xi \rangle}-1\right)\nu_\alpha(du)\right)d\xi,
\end{align*}
as $R$ tends to $+\infty$. Let us deal now with the term given, for all $R>0$ and all $n \geq 1$, by
\begin{align*}
\mathcal{E}_{\nu_\alpha, \omega_n}\left(g_R,g_R\right) = \int_{\mathbb{R}^d} \int_{\mathbb{R}^d} \langle \Delta_u(g_R)(x) ; \Delta_u(g_R)(x) \rangle \omega_n(\| u \|) \nu_\alpha(du) \mu_n^\alpha(dx). 
\end{align*} 
Now, the following Fourier representation formula holds true for the ``carr\'e du champs" operator applied to $g_{R,j}$: for all $R>0$, all $j \in \{1, \dots, d\}$ and all $x \in \mathbb{R}^d$, 
\begin{align*}
\int_{\mathbb{R}^d} (\Delta_u(g_{R,j})(x))^2 \omega_n(\|u\|) \nu_{\alpha}(du) & = \frac{1}{\left(2\pi\right)^{2d}} \int_{\mathbb{R}^d} \int_{\mathbb{R}^d} \mathcal{F}(g_{R,j})(\xi_1) \mathcal{F}(g_{R,j})(\xi_2) \\
&\quad\quad \times m_n(\xi_1, \xi_2) e^{i \langle x ; \xi_1+\xi_2 \rangle} d\xi_1d\xi_2 , \\
& =  \frac{R^{2d+2}}{\left(2\pi\right)^{2d}} \int_{\mathbb{R}^d} \int_{\mathbb{R}^d}\mathcal{F}(g_{1,j})(R\xi_1) \mathcal{F}(g_{1,j})(R\xi_2)\\
&\quad\quad \times m_n(\xi_1, \xi_2) e^{i \langle x ; \xi_1+\xi_2 \rangle} d\xi_1d\xi_2 , \\
& = \frac{R^{2}}{\left(2\pi\right)^{2d}} \int_{\mathbb{R}^d} \int_{\mathbb{R}^d}\mathcal{F}(g_{1,j})(\xi_1) \mathcal{F}(g_{1,j})(\xi_2)\\
&\quad\quad \times m_n\left(\frac{\xi_1}{R}, \frac{\xi_2}{R}\right) e^{i \langle x ; \frac{\xi_1+\xi_2}{R} \rangle} d\xi_1d\xi_2,
\end{align*}
where, for all $\xi_1, \xi_2 \in \mathbb{R}^d$ and all $n \geq 1$, 
\begin{align*}
m_n(\xi_1, \xi_2) = \int_{(0,+\infty) \times \mathbb{S}^{d-1}} \left(e^{i \langle r \theta ; \xi_1 \rangle }-1\right) \left(e^{i \langle r \theta ; \xi_2\rangle} - 1\right) n k \left(n^{\frac{1}{\alpha}} r\right) \frac{dr\sigma(d\theta)}{r}. 
\end{align*}
Thus, 
\begin{align*}
\mathcal{E}_{\nu_\alpha, \omega_n}\left(g_R,g_R\right) =  \frac{R^{2}}{\left(2\pi\right)^{2d}} \sum_{j = 1}^d \int_{\mathbb{R}^d} \int_{\mathbb{R}^d}\mathcal{F}(g_{1,j})(\xi_1) \mathcal{F}(g_{1,j})(\xi_2)m_n\left(\frac{\xi_1}{R}, \frac{\xi_2}{R}\right) \widehat{\mu_n^\alpha}\left(\frac{\xi_1+\xi_2}{R}\right) d\xi_1d\xi_2. 
\end{align*}
Now, by integration by parts, for all $j \in \{1 , \dots, d \}$, 
\begin{align*}
\int_{\mathbb{R}^d} \int_{\mathbb{R}^d}\mathcal{F}(g_{1,j})(\xi_1) \mathcal{F}(g_{1,j})(\xi_2)m_n\left(\frac{\xi_1}{R}, \frac{\xi_2}{R}\right) & \widehat{\mu_n^\alpha}\left(\frac{\xi_1+\xi_2}{R}\right) d\xi_1d\xi_2 \\
&= i\int_{\mathbb{R}^d} \int_{\mathbb{R}^d}\partial_{\xi_{1,j}}\left(\mathcal{F}(\exp\left(- \|x\|^2\right))(\xi_1)\right) \\
&\quad\times \mathcal{F}(g_{1,j})(\xi_2)  m_n\left(\frac{\xi_1}{R}, \frac{\xi_2}{R}\right) \widehat{\mu_n^\alpha}\left(\frac{\xi_1+\xi_2}{R}\right) d\xi_1d\xi_2, \\
& =-i \int_{\mathbb{R}^d} \int_{\mathbb{R}^d} \mathcal{F}(\exp\left(- \|x\|^2\right))(\xi_1) \mathcal{F}(g_{1,j})(\xi_2) \\
& \quad \times \dfrac{\partial}{\partial \xi_{1,j}}\left[m_n\left(\frac{\xi_1}{R}, \frac{\xi_2}{R}\right) \widehat{\mu_n^\alpha}\left(\frac{\xi_1+\xi_2}{R}\right)\right] d\xi_1d\xi_2.
\end{align*}
But, for all $R>0$ and all $n \geq 1$, 
\begin{align*}
\dfrac{\partial}{\partial \xi_{1,j}}\left[m_n\left(\frac{\xi_1}{R}, \frac{\xi_2}{R}\right) \widehat{\mu_n^\alpha}\left(\frac{\xi_1+\xi_2}{R}\right)\right] = I_{R,n,j}(\xi_1, \xi_2) + II_{R,n,j}(\xi_1, \xi_2),
\end{align*}
where, 
\begin{align*}
I_{R,n,j}(\xi_1, \xi_2) &  = i \left(\int_{(0,+\infty) \times \mathbb{S}^{d-1}}\theta_j e^{i \langle r\theta ;\xi_1 \rangle} \left(e^{i \langle r\theta ;\xi_2 \rangle}-1\right) n k\left(n^{\frac{1}{\alpha}} R r \right) dr \sigma(d\theta)\right)\widehat{\mu_n^\alpha}\left(\frac{\xi_1+\xi_2}{R}\right), \\
II_{R,n,j}(\xi_1, \xi_2) & = m_{n} \left(\frac{\xi_1}{R} , \frac{\xi_2}{R}\right)  \left(i\int_{(0,+\infty) \times \mathbb{S}^{d-1}} \theta_j \left(e^{i \langle r\theta ; \xi_1+ \xi_2 \rangle}-1\right) n  k\left(n^{\frac{1}{\alpha}} R r\right)dr \sigma(d\theta) \right) \\
&\quad\quad \times \widehat{\mu_n^\alpha}\left(\frac{\xi_1+ \xi_2}{R}\right).
\end{align*}
Now, for all $j \in \{1,\dots, d\}$ fixed, the difference between the two quadratic terms can be written as:
\begin{align*}
\frac{R^{2-\alpha}}{\left(2\pi\right)^{2d}} & \int_{\mathbb{R}^d} \int_{\mathbb{R}^d} \mathcal{F}(\exp\left(- \|x\|^2\right))(\xi_1) \mathcal{F}(g_{1,j})(\xi_2)\widehat{\mu_n^\alpha}\left(\frac{\xi_1+\xi_2}{R}\right) \\
&\quad\quad \times \left(\int_{(0,+\infty)\times\mathbb{S}^{d-1}} \theta_j e^{i \langle r\theta ;\xi_1 \rangle} \left(e^{i \langle r\theta ;\xi_2 \rangle}-1\right) n R^\alpha k\left(n^{\frac{1}{\alpha}} R r \right)dr\sigma(d\theta)\right)d\xi_1d\xi_2 \\
&\quad\quad + \frac{R^{2-2\alpha}}{(2\pi)^{2d}}\int_{\mathbb{R}^d} \int_{\mathbb{R}^d} \mathcal{F}(\exp\left(- \|x\|^2\right))(\xi_1) \mathcal{F}(g_{1,j})(\xi_2)\widehat{\mu_n^\alpha}\left(\frac{\xi_1+\xi_2}{R}\right) \\
&\quad\quad \times R^\alpha m_{n} \left(\frac{\xi_1}{R} , \frac{\xi_2}{R}\right) \left(\int_{(0,+\infty) \times \mathbb{S}^{d-1}} \theta_j \left(e^{i \langle r\theta ; \xi_1+ \xi_2 \rangle}-1\right) n R^\alpha k\left(n^{\frac{1}{\alpha}} R r\right)dr \sigma(d\theta) \right)d\xi_1d\xi_2 \\
&\quad\quad - \frac{R^{2-\alpha}}{(2\pi)^d} \int_{\mathbb{R}^d} \mathcal{F}\left(x_j \exp\left(-2\|x\|^2\right)\right)(\xi)  \\
& \quad\quad \times \left(\int_{(0,+\infty) \times \mathbb{S}^{d-1}}\theta_j \left(e^{i \langle r \theta ; \xi \rangle}-1\right) n R^\alpha k\left(n^{\frac{1}{\alpha}} R r\right)dr\sigma(d\theta) \right) \widehat{\mu_n^\alpha}\left(\frac{\xi}{R}\right) d\xi.
\end{align*}
Moreover, 
\begin{align*}
& R^\alpha m_{n} \left(\frac{\xi_1}{R} , \frac{\xi_2}{R}\right) \longrightarrow \int_{\mathbb{R}^d} \left(e^{i \langle u ; \xi_1 \rangle}-1\right)\left(e^{i \langle u ;\xi_2 \rangle}-1\right) \nu_\alpha(du), \\
& \int_{(0,+\infty) \times \mathbb{S}^{d-1}}\theta_j \left(e^{i \langle r\theta ; \xi_1+ \xi_2 \rangle}-1\right) n R^\alpha k\left(n^{\frac{1}{\alpha}} R r\right)dr \sigma(d\theta) \longrightarrow \int_{\mathbb{R}^d} u_j \left(e^{i \langle u ; \xi_1+ \xi_2 \rangle}-1\right) \nu_\alpha(du),
\end{align*}
as $R$ tends to $+\infty$. Hence, the middle term disappears in the limit and it remains to study the asymptotic of the term:
\begin{align*}
\frac{R^{2-\alpha}}{\left(2\pi\right)^{2d}} & \int_{\mathbb{R}^d} \int_{\mathbb{R}^d} \mathcal{F}(\exp\left(- \|x\|^2\right))(\xi_1) \mathcal{F}(g_{1,j})(\xi_2)\widehat{\mu_n^\alpha}\left(\frac{\xi_1+\xi_2}{R}\right) \\
& \times \bigg(\int_{(0,+\infty)\times \mathbb{S}^{d-1}}\bigg[\theta_j e^{i \langle r\theta ;\xi_1 \rangle} \left(e^{i \langle r\theta ;\xi_2 \rangle}-1\right) \\
&\quad\quad - \theta_j \left(e^{i \langle r\theta ;\xi_2 +\xi_1 \rangle}-1\right) \bigg]n R^\alpha k\left(n^{\frac{1}{\alpha}} R r \right)dr\sigma(d\theta)\bigg)d\xi_1d\xi_2 \\
& = \frac{R^{2-\alpha}}{\left(2\pi\right)^{2d}}  \int_{\mathbb{R}^d} \int_{\mathbb{R}^d} \mathcal{F}(\exp\left(- \|x\|^2\right))(\xi_1) \mathcal{F}(g_{1,j})(\xi_2)\widehat{\mu_n^\alpha}\left(\frac{\xi_1+\xi_2}{R}\right) \\
&\quad\quad \times \left(\int_{\mathbb{R}^d}\theta_j  \left(1-e^{i \langle r\theta ;\xi_1 \rangle}\right) n R^\alpha k\left(n^{\frac{1}{\alpha}} R r \right)dr \sigma(d\theta)\right)d\xi_1d\xi_2. 
\end{align*}
Observe that, for all $j \in \{1 , \dots, d\}$, $g_{1,j}(0) = 0$. Thus, the remaining term can be written as: 
\begin{align*}
\frac{R^{2-\alpha}}{\left(2\pi\right)^{2d}} & \int_{\mathbb{R}^d} \int_{\mathbb{R}^d} \mathcal{F}(\exp\left(- \|x\|^2\right))(\xi_1) \mathcal{F}(g_{1,j})(\xi_2)\left(\widehat{\mu_n^\alpha}\left(\frac{\xi_1+\xi_2}{R}\right) -1\right)\\
&\quad\quad \times \left(\int_{(0,+\infty) \times \mathbb{S}^{d-1}} \theta_j  \left(1-e^{i \langle r \theta ;\xi_1 \rangle}\right) n R^\alpha k\left(n^{\frac{1}{\alpha}} R r \right)dr\sigma(d\theta)\right)d\xi_1d\xi_2 , \\
 = \frac{R^{2-\alpha}}{\left(2\pi\right)^{2d}} & \int_{\mathbb{R}^d} \int_{\mathbb{R}^d} \mathcal{F}(\exp\left(- \|x\|^2\right))(\xi_1) \mathcal{F}(g_{1,j})(\xi_2) \psi_n^\alpha \left(\frac{\xi_1+\xi_2}{R}\right) \int_0^1 \exp\left(t \psi_n^\alpha \left(\frac{\xi_1+\xi_2}{R}\right)\right) dt\\
&\quad\quad \times \left(\int_{(0,+\infty) \times \mathbb{S}^{d-1}} \theta_j  \left(1-e^{i \langle r\theta ;\xi_1 \rangle}\right) n R^\alpha k\left(n^{\frac{1}{\alpha}} R r \right)dr\sigma(d\theta)\right)d\xi_1d\xi_2,
\end{align*}
with, 
\begin{align*}
\psi_n^\alpha \left(\frac{\xi_1+\xi_2}{R}\right) = \int_{(0,+\infty)\times \mathbb{S}^{d-1}} \left(e^{i \langle r\theta ; \xi_1 + \xi_2 \rangle}-1- i \langle r\theta ; \xi_1+ \xi_2 \rangle\right) n k \left(n^{\frac{1}{\alpha}} r R\right) \frac{dr \sigma(d\theta)}{r}.
\end{align*}
But, 
\begin{align*}
R^\alpha \psi_n^\alpha \left(\frac{\xi_1+\xi_2}{R}\right) \longrightarrow \int_{\mathbb{R}^d} \left(e^{i \langle u ; \xi_1 + \xi_2 \rangle}-1- i \langle u ; \xi_1+ \xi_2 \rangle\right) \nu_\alpha(du) , \quad R\longrightarrow +\infty. 
\end{align*}
This concludes the proof of the spectral condition.~The end of the proof follows from a straightforward application of Theorem \ref{thm:stability_estimate_NDS_stable} and from the explicit expression of  the weight $\omega_n$ together with the asymptotic expansion \eqref{eq:second_order_expansion_NDS} $\square$.

\begin{rem}\label{rem:pareto-type_examples}
In this remark, let us consider two examples of Pareto-type distribution for which one can apply Theorem \ref{thm:all_dimensions_NDS}. Let $f_1$ be the one-sided Pareto density defined, for all $x \geq 0$, by
\begin{align}\label{eq:pareto_one_sided_density}
f_1(x) = \dfrac{\alpha}{\left(1+x\right)^{\alpha+1}}, \quad \alpha \in (0,2).
\end{align}
This probability density function is \textit{subexponential} (see \cite{Asmussen_Foss_Korshunov_03,Watanabe_Yamamuro_10,Watanabe_22}) and the associated probability measure on $\mathbb{R}_+$ is self-decomposable (see \cite[Appendix B, Section 3]{Steutel_Harn_04}).~Then, \cite[Theorem 1.3]{Watanabe_Yamamuro_10} ensures that
\begin{align}\label{eq:equivalent_kf_Pareto_one_sided}
k_1(x) \sim xf_1(x) \sim \frac{\alpha}{x^\alpha}, \quad x\rightarrow +\infty,  
\end{align}
where $k_1$ is the $k$-function of the L\'evy measure of the one-sided Pareto distribution on $\mathbb{R}_+$.~Moreover, thanks to \cite{Th_77}, the one-sided Pareto distribution with density $f_1$ belongs to the set of generalized gamma convolutions (GGC) which is a subset of the set of self-decomposable distributions on $\mathbb{R}_+$. As such and thanks to \cite{Th_77}, $k_1$ admits the following integral representation: for all $u>0$, 
\begin{align}\label{eq:LTrep_kfunction_onesided}
k_1(u) = \int_0^{+\infty} \exp( - u y) u_1(y)dy,
\end{align}  
where $u_1$ is a probability density function on $(0,+\infty)$ such that $\int_{(0,1]} |\log(y)| u_1(y)dy< +\infty$ and $\int_{(1, +\infty)} u_1(y)dy/y<+\infty$. Then, at once, $k_1(0) = 1$ and $k_1$ is smooth on $(0,+\infty)$.~Now, let $Y_1$ and $Y_2$ be two independent non-negative random variables with identical law $\rho_1$ which density is given by $f_1$ with $\alpha \in (1,2)$. Let $Y$ be the centered real-valued random variable defined by $Y = Y_1 - Y_2$. Then, for all $\xi \in \mathbb{R}$, 
\begin{align}\label{eq:characteristic_function_Y}
\mathbb{E} [e^{i\xi Y}]  & = \exp \left( \int_{\mathbb{R}} \left(e^{i u \xi} - 1 - iu\xi \right) \frac{k_1(|u|)}{|u|}du\right).
\end{align}
From the previous representation, it follows that $Y$ is self-decomposable on $\mathbb{R}$ with $k$-function given by $u \mapsto k_1(|u|)$. So, to apply Theorem \ref{thm:all_dimensions_NDS}, it remains to prove the asymptotic expansion \eqref{eq:second_order_expansion_NDS} for the symmetrization of $k_1$ at infinity.~But this follows from the asymptotic expansion of the function $u_1$ at $0^+$ (see \cite{Th_77}) together with a Tauberian-type argument. Theorem \ref{thm:all_dimensions_NDS} and the previous lines of reasoning can be extended as well to the symmetric (or double) Pareto distribution with density function $f_2$ given, for all $x \in \mathbb{R}$, by 
\begin{align}
f_2(x) = \frac{\alpha}{2} \dfrac{1}{\left(1+|x|\right)^{\alpha+1}}, 
\end{align}
which is self-decomposable and EGGC (see, e.g., \cite[Chapter VI, Section 12, Example 12.20]{Steutel_Harn_04} and \cite[Chapter $7$, Section $7.3$ page $118$]{Bon_92}). In particular, a random variable $X_\alpha$, distributed with respect to the double Pareto distribution with parameter $\alpha \in (1,2)$, is a variance-mixture of the standard normal distribution and the following representation in law holds true:
\begin{align}\label{eq:rep_law_double_pareto}
X_\alpha =_d \sqrt{2Y_\alpha} Z,
\end{align}
where $=_d$ stands for equality in law, where $Z$ is a standard normal random variable and where $Y_\alpha$ is independent of $Z$ and distributed according to the absolutely continuous law given by the density $g$ in \cite[Chapter VI, Section $12$, Example 12.20]{Steutel_Harn_04} with $r = \alpha+1$. Moreover, the law of $Y_\alpha$ belongs to the set of GGC and the symmetric extended Th\"orin measure of the double Pareto distribution restricted to $\mathbb{R}_+^*$ is the image measure of the Th\"orin measure of the law of $Y_\alpha$ by the square root mapping. Finally, adapting the techniques of \cite{Th_77} to the law of $Y_\alpha$, one obtains an asymptotic expansion at $0^+$ of the Lebesgue density of the Th\"orin measure of $Y_\alpha$ from which one can deduce the asymptotic expansion at $+\infty$ of the $k$-function of the double Pareto distribution. 
\end{rem}

Finally, let us consider another example of initial law which $k$-function does not verify all the assumptions of Theorem \ref{thm:all_dimensions_NDS} and which is a simple instance of \textit{layered stable distributions} (see \cite{HK_07,RS_10}). For this purpose, let $\alpha \in (1,2)$ and let $\beta \in (\alpha, 2)$. Let $k_{\alpha,\beta}^L$ be the function defined, for all $r \in (0,+\infty)$, by
\begin{align}\label{eq:kfunction_L}
k_{\alpha,\beta}^L(r) = \frac{1}{r^{\beta}}\bbone_{(0,1]}(r) + \frac{1}{r^\alpha}\bbone_{(1,+\infty)}(r). 
\end{align}
Then, let $\nu_{\alpha,\beta}^L$ be the L\'evy measure defined through the polar decomposition
\begin{align*}
\nu_{\alpha,\beta}^L(du) =\bbone_{(0,+\infty)}(r)\bbone_{\mathbb{S}^{d-1}}(y) \dfrac{k_{\alpha,\beta}^L(r) }{r}dr \sigma(dy),
\end{align*}
where $\sigma$ is a positive finite symmetric measure on $\mathcal{B}\left(\mathbb{S}^{d-1}\right)$ such that 
\begin{align*}
\underset{e \in \mathbb{S}^{d-1}}{\inf} \int_{\mathbb{S}^{d-1}} \left| \langle e ; y \rangle \right|^\alpha \sigma(dy)>0. 
\end{align*}
First, the following limit theorem holds true.  

\begin{thm}\label{thm:limit_theorem_layered_stable_distributions}
Let $d \geq 1$ be an integer, let $\alpha \in (1,2)$ and let $\mu_\alpha$ be a non-degenerate symmetric $\alpha$-stable probability measure on $\mathbb{R}^d$ with associated L\'evy measure $\nu_\alpha$ which spherical component is denoted by $\sigma$. Let $\beta \in (\alpha,2)$ and let $\mu^L_{\alpha,\beta}$ be the probability measure on $\mathbb{R}^d$ which Fourier transform is given, for all $\xi \in \mathbb{R}^d$, by
\begin{align}\label{eq:TF_Layered_stable}
\widehat{\mu^L_{\alpha,\beta}}(\xi) = \exp \left( \int_{\mathbb{R}^d} \left(e^{i \langle \xi ; u \rangle} - 1 - i \langle \xi ;u \rangle\right) \nu^L_{\alpha,\beta}(du)\right),
\end{align}
where $\nu^L_{\alpha,\beta}$ is defined by \eqref{eq:Levymeasure_L} with spherical component $\sigma$.~Let $(Z_k)_{k \geq 1}$ be a sequence of i.i.d. random vectors of $\mathbb{R}^d$ such that $Z_1 \sim \mu^L_{\alpha,\beta}$. Let $(S_n)_{n \geq 1}$ be the sequence of random vectors defined, for all $n \geq 1$, by
\begin{align*}
S_n  = \frac{1}{n^{\frac{1}{\alpha}}} \sum_{k = 1}^n Z_k. 
\end{align*}
Then, 
\begin{align}\label{eq:limit_theorem}
S_n \overset{\mathcal{L}}{\longrightarrow} X_\alpha \sim \mu_\alpha,
\end{align}
as $n$ tends to $+\infty$.
\end{thm}

\begin{proof}
The proof is a simple application of L\'evy's continuity theorem. First, let us compute the characteristic function of $S_n$, for all $n \geq 1$. By standard computations, for all $n \geq 1$ and all $\xi \in \mathbb{R}^d$, 
\begin{align}\label{eq:characteristic_function_Sn_Ls}
\varphi_{S_n}(\xi)
& = \exp\left( \int_{\mathbb{R}^d} \left(e^{i \langle u ; \xi \rangle} - 1 - i\langle u ; \xi \rangle\right) \nu_{\alpha,\beta}^{L,n}(du) \right), 
\end{align}
with,
\begin{align}\label{eq:Levymeasure_Sn_Ls}
\nu_{\alpha,\beta}^{L,n}(du) =\bbone_{(0,+\infty)}(r)\bbone_{\mathbb{S}^{d-1}}(y) \frac{nk^L_{\alpha,\beta}\left(n^{\frac{1}{\alpha}}r\right)}{r} dr \sigma(dy). 
\end{align} 
Now, observe that, for all $r \in (0,+\infty)$, since $\beta \in (\alpha, 2)$, 
\begin{align}\label{eq:pointwise_convergence_Ls}
nk^L_{\alpha,\beta}\left(n^{\frac{1}{\alpha}}r\right) \longrightarrow \frac{1}{r^\alpha}\bbone_{(0,+\infty)}(r),  
\end{align}
as $n$ tends to $+\infty$. Moreover, for all $r \in (0,+\infty)$ and all $n \geq 1$, 
\begin{align}\label{ineq:uniform_domination_Ls}
nk^L_{\alpha,\beta}\left(n^{\frac{1}{\alpha}}r\right) & = \frac{1}{n^{\frac{\beta}{\alpha}-1}} \frac{1}{r^\beta}\bbone_{(0, n^{-1/\alpha}]}(r) + \frac{1}{r^\alpha}\bbone_{(n^{-1/\alpha}, +\infty)}(r) , \nonumber \\
& \leq \frac{1}{r^{\beta}}\bbone_{(0,1]}(r) + \frac{1}{r^{\alpha}}\bbone_{(0,+\infty)}(r). 
\end{align}
Thus, a direct application of the Lebesgue dominated convergence theorem ensures that, for all $\xi \in \mathbb{R}^d$, 
\begin{align*}
\varphi_{S_n}(\xi) \longrightarrow \widehat{\mu_\alpha}(\xi),
\end{align*}
as $n$ tends to $+\infty$. This concludes the proof of the theorem.
\end{proof}
\noindent 
Now, let us prove Theorem \ref{thm:quantitative_approximation_Ls} which is a quantitative version of the previous limit theorem as an application of the stability result Theorem \ref{thm:stability_estimate_NDS_stable}.\\
\\
\textit{Proof of Theorem \ref{thm:quantitative_approximation_Ls}}. In order to obtain \eqref{ineq:rate_W1_LayeredSimple}, let us apply Theorem \ref{thm:stability_estimate_NDS_stable} to the probability measure $\mu_n^{\alpha,\beta}$ and $\mu_\alpha$, for all $n \geq 1$.~First, thanks to \eqref{eq:characteristic_function_Sn_Ls} and to \eqref{eq:Levymeasure_Sn_Ls}, $\mu_n^{\alpha,\beta}$ is a symmetric non-degenerate ID probability measure on $\mathbb{R}^d$ with L\'evy measure $\nu_{\alpha,\beta}^{L,n}$ such that, for all $\delta \in [1, \alpha)$, 
\begin{align}\label{eq:moment_condition}
\int_{\mathbb{R}^d} \|x\|^\delta \mu^{\alpha,\beta}_n(dx) < +\infty.
\end{align}
Thus, $\mu_n^{\alpha,\beta} \ast \nu_{\alpha,\beta}^{L,n} << \mu_n^{\alpha,\beta}$ and, for all $f \in \mathcal{C}_b^1(\mathbb{R}^d, \mathbb{R}^d)$ such that $\mu_n^{\alpha,\beta}(f) = 0$, 
\begin{align*}
\int_{\mathbb{R}^d} \| f(x)\|^2 \mu_n^{\alpha,\beta}(dx) \leq \int_{\mathbb{R}^d} \int_{\mathbb{R}^d} \| f(x+u)- f(x) \|^2  \nu_{\alpha,\beta}^{L,n}(du)\mu_n^{\alpha,\beta}(dx). 
\end{align*}
Now, for all $n \geq 1$, 
\begin{align}\label{eq:weight_function}
\nu_{\alpha,\beta}^{L,n}(du) = \omega^{\alpha,\beta}_n(\|u\|) \nu_\alpha(du), 
\end{align}
where $\omega_n^{\alpha,\beta}$ is defined, for all $r \in (0,+\infty)$, by
\begin{align*}
\omega_n^{\alpha,\beta}(r) = r^\alpha n k_{\alpha,\beta}^L\left(n^{\frac{1}{\alpha}}r\right) =  \frac{1}{n^{\frac{\beta}{\alpha}-1}} \frac{1}{r^{\beta-\alpha}}\bbone_{(0, n^{-1/\alpha}]}(r) +\bbone_{(n^{-1/\alpha},+\infty)}(r). 
\end{align*}
Note that, since $\beta \in (\alpha,2)$, for all $r \in (0,+\infty)$, 
\begin{align}\label{eq:pointwise_convergence_weight_Ls}
\omega_n^{\alpha,\beta}(r) \longrightarrow 1, 
\end{align}
as $n$ tends to $+\infty$.~The first three conditions of Theorem \ref{thm:stability_estimate_NDS_stable} are fulfilled and so it remains to prove the spectral condition: for all $n \geq 1$, 
\begin{align}\label{eq:spectral_condition_Ls}
\underset{R \longrightarrow +\infty}{\lim}\left(\mathcal{E}_{\alpha,\beta}^{L,n} (g_R,g_R) - \|g_R\|^2_{L^2(\mu_n^{\alpha,\beta})}\right) = 0,
\end{align}
where $\mathcal{E}_{\alpha,\beta}^{L,n}$ is defined, for all $h_1,h_2 \in \mathcal{C}_b^1(\mathbb{R}^d, \mathbb{R}^d)$, by
\begin{align}\label{eq:quadratic_form_Ls}
\mathcal{E}_{\alpha,\beta}^{L,n} (h_1,h_2) = \int_{\mathbb{R}^d} \int_{\mathbb{R}^d} \langle \Delta_u(h_1)(x) ; \Delta_u(h_2)(x)\rangle \nu_{\alpha,\beta}^{L,n}(du) \mu^{\alpha,\beta}_n(dx). 
\end{align}
First, by computations similar to the ones of the proof of Theorem \ref{thm:all_dimensions_NDS}, for all $R>0$ and all $j \in \{1, \dots, d\}$, 
\begin{align*}
\|g_{R,j}\|^2_{L^2(\mu_n^{\alpha,\beta})} = - \frac{i R^2}{\left(2\pi\right)^d} \int_{\mathbb{R}^d} \mathcal{F}(x_j \exp\left(- 2\|x\|^2\right))(\xi) \partial_{\xi_j} \left(\widehat{\mu_n^{\alpha,\beta}}\left(\frac{\xi}{R}\right)\right) d\xi. 
\end{align*}
Now, for all $n \geq 1$ and all $\xi \in \mathbb{R}^d$, 
\begin{align*}
\partial_{\xi_j} \left(\widehat{\mu_n^{\alpha,\beta}}\left(\xi\right) \right) = \left(i \int_{\mathbb{R}^d} u_j \left(e^{i \langle u ;\xi \rangle}-1\right) \nu_{\alpha,\beta}^{L,n}(du) \right) \widehat{\mu_n^{\alpha,\beta}}(\xi). 
\end{align*}
Thus, for all $n \geq 1$ and all $R>0$, 
\begin{align*}
\|g_{R,j}\|^2_{L^2(\mu_n^{\alpha,\beta})} & = \frac{R^{2-\alpha}}{\left(2\pi\right)^d} \int_{\mathbb{R}^d} \mathcal{F}(x_j \exp\left(- 2\|x\|^2\right))(\xi) \widehat{\mu_n^{\alpha,\beta}}\left(\frac{\xi}{R}\right) \\
&\quad\quad \times \left(\int_{(0,+\infty) \times \mathbb{S}^{d-1}} ry_j \left(e^{i \langle ry ;\xi \rangle}-1\right) nR^\alpha k_{\alpha,\beta}^L(n^{\frac{1}{\alpha}}Rr) \frac{dr}{r} \sigma(dy)\right)d\xi. 
\end{align*}
Now, for all $r>0$, all $n \geq 1$ and all $R>0$, 
\begin{align*}
nR^\alpha k^L_{\alpha, \beta}\left(n^{\frac{1}{\alpha}} R r\right) = \dfrac{1}{n^{\frac{\beta}{\alpha}-1} R^{\beta-\alpha} r^{\beta}}\bbone_{(0,n^{-1/\alpha}R^{-1}]}(r) + \frac{1}{r^\alpha}\bbone_{(n^{-1/\alpha}R^{-1},+\infty)}(r). 
\end{align*}
Thus, for all $r>0$ and all $n \geq 1$, 
\begin{align*}
nR^\alpha k^L_{\alpha, \beta}\left(n^{\frac{1}{\alpha}} R r\right) \longrightarrow \frac{1}{r^\alpha},
\end{align*}
as $R$ tends to $+\infty$. Moreover, for all $n \geq 1$, all $R \geq1$ and all $r>0$, 
\begin{align*}
nR^\alpha k^L_{\alpha, \beta}\left(n^{\frac{1}{\alpha}} R r\right) \leq \frac{1}{r^\beta}\bbone_{(0,1]}(r) + \frac{1}{r^\alpha}\bbone_{(0,+\infty)}(r).  
\end{align*}
Thus, a straightforward application of the Lebesgue dominated convergence theorem ensures that, 
\begin{align}\label{eq:limit_squared_norm_Ls}
\dfrac{\|g_{R,j}\|^2_{L^2(\mu_n^{\alpha,\beta})}}{R^{2-\alpha}} \longrightarrow \frac{1}{(2\pi)^d} \int_{\mathbb{R}^d} \mathcal{F}(x_j \exp\left(- 2\|x\|^2\right))(\xi)\left(\int_{\mathbb{R}^d} u_j \left(e^{i \langle u ;\xi \rangle}-1\right) \nu_\alpha(du)\right)d\xi. 
\end{align} 
Next, let us deal with the term regarding the quadratic form. First, for all $x \in \mathbb{R}^d$, all $R>0$, all $n \geq 1$ and all $j \in \{1, \dots, d\}$,
\begin{align}\label{eq:squared_field_Ls}
\int_{\mathbb{R}^d} \left|\Delta_u(g_{R,j})(x)\right|^2 \nu_{\alpha,\beta}^{L,n}(du) & = \frac{1}{(2\pi)^{2d}} \int_{\mathbb{R}^{2d}} \mathcal{F}\left(g_{R,j}\right)(\xi_1) \mathcal{F}\left(g_{R,j}\right)(\xi_2) e^{i \langle x ; \xi_1+\xi_2 \rangle} \nonumber \\
&\quad \times m_{\alpha,\beta}^{L,n}\left(\xi_1, \xi_2\right) d\xi_1 d\xi_2, 
\end{align}
with, for all $(\xi_1 , \xi_2) \in \mathbb{R}^{2d}$, 
\begin{align}\label{eq:symbol_squared_field_Ls}
m_{\alpha,\beta}^{L,n}\left(\xi_1, \xi_2\right) = \int_{\mathbb{R}^d} \left(e^{i \langle u ;\xi_1 \rangle}-1\right)\left(e^{i \langle u ; \xi_2 \rangle}-1\right) \nu_{\alpha,\beta}^{L,n}(du).
\end{align}
Then, by standard computations, for all $R>0$, all $j \in \{1, \dots, d\}$ and all $x \in \mathbb{R}^d$, 
\begin{align*}
\int_{\mathbb{R}^d} \left|\Delta_u(g_{R,j})(x)\right|^2 \nu_{\alpha,\beta}^{L,n}(du) & =  \frac{R^{2}}{\left(2\pi\right)^{2d}} \int_{\mathbb{R}^d} \int_{\mathbb{R}^d}\mathcal{F}(g_{1,j})(\xi_1) \mathcal{F}(g_{1,j})(\xi_2)\\
&\quad\quad \times m^{L,n}_{\alpha,\beta}\left(\frac{\xi_1}{R}, \frac{\xi_2}{R}\right) e^{i \langle x ; \frac{\xi_1+\xi_2}{R} \rangle} d\xi_1d\xi_2.
\end{align*}
Thus, as in the proof of Theorem \ref{thm:all_dimensions_NDS}, 
\begin{align*}
\mathcal{E}_{\alpha,\beta}^{L,n}\left(g_R,g_R\right) & =  \frac{R^{2}}{\left(2\pi\right)^{2d}} \sum_{j = 1}^d \int_{\mathbb{R}^d} \int_{\mathbb{R}^d}\mathcal{F}(g_{1,j})(\xi_1) \mathcal{F}(g_{1,j})(\xi_2)m_{\alpha,\beta}^{L,n}\left(\frac{\xi_1}{R}, \frac{\xi_2}{R}\right) \\
&\quad \times \widehat{\mu_n^{\alpha, \beta}}\left(\frac{\xi_1+\xi_2}{R}\right) d\xi_1d\xi_2. 
\end{align*}
Now, by integration by parts, for all $j \in \{1 , \dots, d \}$, 
\begin{align*}
\int_{\mathbb{R}^d} \int_{\mathbb{R}^d}\mathcal{F}(g_{1,j})(\xi_1) \mathcal{F}(g_{1,j})(\xi_2)& m_{\alpha,\beta}^{L,n}\left(\frac{\xi_1}{R}, \frac{\xi_2}{R}\right) \widehat{\mu_n^{\alpha, \beta}}\left(\frac{\xi_1+\xi_2}{R}\right) d\xi_1d\xi_2 \\
& = i\int_{\mathbb{R}^d} \int_{\mathbb{R}^d}\partial_{\xi_{1,j}}\left(\mathcal{F}(\exp\left(- \|x\|^2\right))(\xi_1)\right) \\
&\quad\quad \times \mathcal{F}(g_{1,j})(\xi_2)  m_{\alpha,\beta}^{L,n}\left(\frac{\xi_1}{R}, \frac{\xi_2}{R}\right) \widehat{\mu_n^{\alpha, \beta}}\left(\frac{\xi_1+\xi_2}{R}\right) d\xi_1d\xi_2, \\
& =-i \int_{\mathbb{R}^d} \int_{\mathbb{R}^d} \mathcal{F}(\exp\left(- \|x\|^2\right))(\xi_1) \mathcal{F}(g_{1,j})(\xi_2) \\
& \quad\quad \times \dfrac{\partial}{\partial \xi_{1,j}}\left[m_{\alpha,\beta}^{L,n}\left(\frac{\xi_1}{R}, \frac{\xi_2}{R}\right) \widehat{\mu_n^{\alpha, \beta}}\left(\frac{\xi_1+\xi_2}{R}\right)\right] d\xi_1d\xi_2.
\end{align*}
But, for all $R>0$ and all $n \geq 1$, 
\begin{align*}
\dfrac{\partial}{\partial \xi_{1,j}}\left[m_{\alpha,\beta}^{L,n}\left(\frac{\xi_1}{R}, \frac{\xi_2}{R}\right) \widehat{\mu_n^{\alpha, \beta}}\left(\frac{\xi_1+\xi_2}{R}\right)\right] = I^L_{R,n,j}(\xi_1, \xi_2) + II^L_{R,n,j}(\xi_1, \xi_2),
\end{align*}
where,
\begin{align*}
I^L_{R,n,j}(\xi_1, \xi_2) &  = i \left(\int_{(0,+\infty)\times \mathbb{S}^{d-1}}ry_j e^{i \langle ry ;\xi_1 \rangle} \left(e^{i \langle ry ;\xi_2 \rangle}-1\right) n k_{\alpha,\beta}^{L}\left(n^{\frac{1}{\alpha}} R r \right) \frac{dr}{r} \sigma(dy)\right) \\
&\quad\quad \times \widehat{\mu_n^{\alpha, \beta}}\left(\frac{\xi_1+\xi_2}{R}\right), \\
II^L_{R,n,j}(\xi_1, \xi_2) & = m_{\alpha,\beta}^{L,n} \left(\frac{\xi_1}{R} , \frac{\xi_2}{R}\right)  \left(i\int_{(0,+\infty) \times \mathbb{S}^{d-1}} ry_j \left(e^{i \langle ry ; \xi_1+ \xi_2 \rangle}-1\right) n  k_{\alpha,\beta}^L\left(n^{\frac{1}{\alpha}} R r\right)\frac{dr}{r} \sigma(dy)\right) \\
&\quad\quad \times \widehat{\mu_n^{\alpha, \beta}}\left(\frac{\xi_1+ \xi_2}{R}\right).
\end{align*}
Now, for all $j \in \{1,\dots, d\}$ fixed, the difference between the two quadratic terms can be written as:
\begin{align*}
\frac{R^{2-\alpha}}{\left(2\pi\right)^{2d}} & \int_{\mathbb{R}^d} \int_{\mathbb{R}^d} \mathcal{F}(\exp\left(- \|x\|^2\right))(\xi_1) \mathcal{F}(g_{1,j})(\xi_2)\widehat{\mu_n^{\alpha, \beta}}\left(\frac{\xi_1+\xi_2}{R}\right) \\
&\quad\quad \times \left(\int_{(0, +\infty) \times \mathbb{S}^{d-1}}ry_j e^{i \langle ry ;\xi_1 \rangle} \left(e^{i \langle ry ;\xi_2 \rangle}-1\right) n R^\alpha k_{\alpha,\beta}^L\left(n^{\frac{1}{\alpha}} R r \right) \frac{dr}{r} \sigma(dy)\right)d\xi_1d\xi_2 \\
&\quad\quad + \frac{R^{2-2\alpha}}{(2\pi)^{2d}}\int_{\mathbb{R}^d} \int_{\mathbb{R}^d} \mathcal{F}(\exp\left(- \|x\|^2\right))(\xi_1) \mathcal{F}(g_{1,j})(\xi_2)\widehat{\mu_n^{\alpha, \beta}}\left(\frac{\xi_1+\xi_2}{R}\right) \\
&\quad\quad \times R^\alpha m_{\alpha,\beta}^{L,n} \left(\frac{\xi_1}{R} , \frac{\xi_2}{R}\right) \left(\int_{(0,+\infty)\times\mathbb{S}^{d-1}} ry_j \left(e^{i \langle ry ; \xi_1+ \xi_2 \rangle}-1\right) n R^\alpha k_{\alpha,\beta}^L\left(n^{\frac{1}{\alpha}} R r\right)\frac{dr}{r} \right)d\xi_1d\xi_2 \\
&\quad\quad - \frac{R^{2-\alpha}}{(2\pi)^d} \int_{\mathbb{R}^d} \mathcal{F}\left(x_j \exp\left(-2\|x\|^2\right)\right)(\xi)\widehat{\mu_n^{\alpha, \beta}}\left(\frac{\xi}{R}\right)  \\
& \quad\quad \times \left(\int_{(0,+\infty)\times\mathbb{S}^{d-1}} ry_j \left(e^{i \langle ry ; \xi \rangle}-1\right) n R^\alpha k^L_{\alpha,\beta}\left(n^{\frac{1}{\alpha}} R r\right)\frac{dr}{r} \sigma(dy) \right) d\xi.
\end{align*}
Moreover, 
\begin{align*}
& R^\alpha m_{\alpha,\beta}^{L,n} \left(\frac{\xi_1}{R} , \frac{\xi_2}{R}\right) \longrightarrow \int_{\mathbb{R}^d} \left(e^{i \langle u ; \xi_1 \rangle}-1\right)\left(e^{i \langle u ;\xi_2 \rangle}-1\right) \nu_\alpha(du), \\
& \int_{(0,+\infty)\times\mathbb{S}^{d-1}} ry_j \left(e^{i \langle ry ; \xi_1+ \xi_2 \rangle}-1\right) n R^\alpha k^L_{\alpha, \beta}\left(n^{\frac{1}{\alpha}} R r\right)\frac{dr}{r} \sigma(dy) \longrightarrow \int_{\mathbb{R}^d} u_j \left(e^{i \langle u ; \xi_1+ \xi_2 \rangle}-1\right) \nu_\alpha(du),
\end{align*}
as $R$ tends to $+\infty$. Hence, the middle term disappears in the limit and it remains to study the asymptotic of the term:
\begin{align*}
\frac{R^{2-\alpha}}{\left(2\pi\right)^{2d}} & \int_{\mathbb{R}^d} \int_{\mathbb{R}^d} \mathcal{F}(\exp\left(- \|x\|^2\right))(\xi_1) \mathcal{F}(g_{1,j})(\xi_2)\widehat{\mu_n^{\alpha, \beta}}\left(\frac{\xi_1+\xi_2}{R}\right)d\xi_1d\xi_2 \\
&\times \bigg(\int_{(0, +\infty) \times \mathbb{S}^{d-1}}\left[ry_j e^{i \langle ry ;\xi_1 \rangle} \left(e^{i \langle ry ;\xi_2 \rangle}-1\right) - ry_j \left(e^{i \langle ry ;\xi_2 +\xi_1 \rangle}-1\right) \right] \\
&\times n R^\alpha k_{\alpha,\beta}^L\left(n^{\frac{1}{\alpha}} R r \right) \frac{dr}{r} \sigma(dy)\bigg) \\
& = \frac{R^{2-\alpha}}{\left(2\pi\right)^{2d}}  \int_{\mathbb{R}^d} \int_{\mathbb{R}^d} \mathcal{F}(\exp\left(- \|x\|^2\right))(\xi_1) \mathcal{F}(g_{1,j})(\xi_2)\widehat{\mu_n^{\alpha, \beta}}\left(\frac{\xi_1+\xi_2}{R}\right) \\
&\quad\quad \times \left(\int_{(0,+\infty)\times\mathbb{S}^{d-1}} ry_j  \left(1-e^{i \langle ry ;\xi_1 \rangle}\right) n R^\alpha k^L_{\alpha, \beta}\left(n^{\frac{1}{\alpha}} Rr \right) \frac{dr}{r} \sigma(dy)\right)d\xi_1d\xi_2. 
\end{align*}
Observe that, for all $j \in \{1 , \dots, d\}$, $g_{1,j}(0) = 0$. Thus, the remaining term can be written as: 
\begin{align*}
\frac{R^{2-\alpha}}{\left(2\pi\right)^{2d}} & \int_{\mathbb{R}^d} \int_{\mathbb{R}^d} \mathcal{F}(\exp\left(- \|x\|^2\right))(\xi_1) \mathcal{F}(g_{1,j})(\xi_2)\left(\widehat{\mu_n^{\alpha, \beta}}\left(\frac{\xi_1+\xi_2}{R}\right) -1\right)\\
&\quad\quad \times \left(\int_{(0,+\infty)\times\mathbb{S}^{d-1}} ry_j  \left(1-e^{i \langle ry ;\xi_1 \rangle}\right) n R^\alpha k_{\alpha,\beta}^L\left(n^{\frac{1}{\alpha}} R r \right) \frac{dr}{r} \sigma(dy)\right)d\xi_1d\xi_2 , \\
 = \frac{R^{2-\alpha}}{\left(2\pi\right)^{2d}} & \int_{\mathbb{R}^d} \int_{\mathbb{R}^d} \mathcal{F}(\exp\left(- \|x\|^2\right))(\xi_1) \mathcal{F}(g_{1,j})(\xi_2) \psi_n^{\alpha, \beta} \left(\frac{\xi_1+\xi_2}{R}\right) \int_0^1 \exp\left(t \psi_n^{\alpha, \beta} \left(\frac{\xi_1+\xi_2}{R}\right)\right) dt\\
&\quad\quad \times \left(\int_{(0,+\infty)\times\mathbb{S}^{d-1}} ry_j  \left(1-e^{i \langle ry ;\xi_1 \rangle}\right) n R^\alpha k_{\alpha,\beta}^L\left(n^{\frac{1}{\alpha}} R r \right) \frac{dr}{r}\sigma(dy)\right)d\xi_1d\xi_2,
\end{align*}
with, 
\begin{align*}
\psi_n^{\alpha, \beta} \left(\frac{\xi_1+\xi_2}{R}\right) = \int_{(0,+\infty)\times\mathbb{S}^{d-1}} \left(e^{i \langle ry ; \xi_1 + \xi_2 \rangle}-1- i \langle ry ; \xi_1+ \xi_2 \rangle\right) n k_{\alpha,\beta}^L \left(n^{\frac{1}{\alpha}} r R\right) \frac{dr}{r}\sigma(dy).
\end{align*}
But, 
\begin{align*}
R^\alpha \psi_n^{\alpha, \beta} \left(\frac{\xi_1+\xi_2}{R}\right) \longrightarrow \int_{\mathbb{R}^d} \left(e^{i \langle u ; \xi_1 + \xi_2 \rangle}-1- i \langle u ; \xi_1+ \xi_2 \rangle\right) \nu_\alpha(du) , \quad R\longrightarrow +\infty. 
\end{align*}
This concludes the proof of the spectral condition. So, thanks to Theorem \ref{thm:stability_estimate_NDS_stable}, for all $n \geq 1$, 
\begin{align*}
W_1\left(\mu_n^{\alpha,\beta} , \mu_\alpha \right) & \leq 2 \int_{\|u\|\geq 1} \|u \| \left| \omega_n^{\alpha,\beta}( \|u\| )-1 \right|\nu_\alpha(du) \nonumber \\
& \quad\quad + C_{\alpha,d} \int_{\|u\|\leq 1} \|u \|^2 \left| \omega_n^{\alpha,\beta}( \|u\| )-1 \right| \nu_\alpha(du).
\end{align*}
Now, recall that, for all $r>0$ and all $n \geq 2$, 
\begin{align*}
\omega_n^{\alpha,\beta}(r) =  \frac{1}{n^{\frac{\beta}{\alpha}-1}} \frac{1}{r^{\beta-\alpha}}\bbone_{(0, n^{-1/\alpha}]}(r) +\bbone_{(n^{-1/\alpha},+\infty)}(r).
\end{align*}
Thus, for all $r>1$ and all $n \geq 2$, $\omega_n^{\alpha,\beta}(r) = 1$, so that 
\begin{align*}
\int_{\|u\|\geq 1} \|u \| \left| \omega_n^{\alpha,\beta}( \|u\| )-1 \right|\nu_\alpha(du) = 0. 
\end{align*}
Moreover, by spherical coordinates, for all $n \geq 2$,
\begin{align*}
\int_{\|u\|\leq 1} \|u \|^2 \left| \omega_n^{\alpha,\beta}( \|u\| )-1 \right| \nu_\alpha(du) = \sigma\left(\mathbb{S}^{d-1}\right) \int_0^1 \left| \omega_n^{\alpha,\beta}( r )-1 \right| \dfrac{dr}{r^{\alpha-1}}. 
\end{align*}
Now, for all $r \in (0,1)$, 
\begin{align*}
\left| \omega_n^{\alpha,\beta}( r )-1 \right| \leq \frac{1}{n^{\frac{\beta}{\alpha}-1}} \frac{1}{r^{\beta-\alpha}}\bbone_{(0, n^{-1/\alpha}]}(r) + \left|\bbone_{(n^{-1/\alpha},+\infty)}(r) -1 \right|. 
\end{align*}
Thus, for all $n \geq 2$, 
\begin{align*}
\int_0^1 \left| \omega_n^{\alpha,\beta}( r )-1 \right| \dfrac{dr}{r^{\alpha-1}} & = \int_0^{n^{-\frac{1}{\alpha}}} \left| \omega_n^{\alpha,\beta}( r )-1 \right| \dfrac{dr}{r^{\alpha-1}}, \\
& \leq \int_0^{n^{- \frac{1}{\alpha}}} \frac{1}{n^{\frac{\beta}{\alpha}-1}} \frac{1}{r^{\beta-\alpha}} \frac{dr}{r^{\alpha-1}}
+ \int_0^{n^{-\frac{1}{\alpha}}} \frac{dr}{r^{\alpha-1}},\\
& \leq \left(\int_0^1\frac{dr}{r^{\beta-1}}\right) \frac{1}{n^{\frac{2}{\alpha}-1}} + \frac{1}{2-\alpha} \frac{1}{n^{\frac{2}{\alpha}-1}}. 
\end{align*}
This concludes the proof of the theorem. $\square$

\begin{rem}\label{rem:generalization_layered_stable_distributions}
Based on the proofs of Theorem \ref{thm:limit_theorem_layered_stable_distributions} and of Theorem \ref{thm:quantitative_approximation_Ls}, it is clearly possible to obtain similar quantitative estimates in $1$-Wasserstein distance for initial laws which $k$-function, denoted by $k_L$, verifies the following conditions: $k_L$ is a Borel measurable function from $(0,+\infty)$ to $\mathbb{R}_+$ such that 
\begin{align}\label{eq:asymptotics}
\underset{r \rightarrow +\infty}{\lim} r^\alpha k_L(r) = 1, \quad \underset{r \rightarrow 0^+}{\lim} r^\beta k_L(r) = c_{\alpha,\beta},
\end{align} 
where $\alpha \in (1,2)$, $\beta \in (\alpha,2)$ and $c_{\alpha,\beta}>0$.~Moreover, there exists a Borel measurable function $\phi$ from $(0,+\infty)$ to $\mathbb{R}_+$ such that, for all $r>0$ and all $n \geq 1$, 
\begin{align}\label{eq:hp_domination_general}
n k_L(n^{\frac{1}{\alpha}}r) \leq \phi(r) , \quad \int_{0}^{+\infty} \min(1,r^2) \dfrac{\phi(r)}{r}dr <+\infty. 
\end{align}
Finally, in order to obtain an explicit rate of convergence, as in Theorem \ref{thm:all_dimensions_NDS}, one should specify the second order asymptotic at $+\infty$ of the function $k_L$.
\end{rem}

\section{Spectral Analysis of Non-Local Operators associated with Cauchy Probability Measures on $\mathbb{R}^d$}\label{sec:spectral_cauchy}
\subsection{$L^2$-spectral properties}\label{sec:L2_Spec_Prop}
\noindent
In the sequel, $\mu_1$ denotes a non-degenerate symmetric $1$-stable probability measure on $\mathbb{R}^d$ with Fourier transform given by \eqref{eq:fourier_symmetric_stable} with $\alpha=1$.~$\widehat{\mu_1}$ is also given, for all $\xi \in \mathbb{R}^d$, by 
\begin{align}\label{eq:FT_Symmetric_Cauchy}
\widehat{\mu_1}\left(\xi\right) = \exp \left(\int_{\mathbb{R}^d} \left(e^{i \langle u ; \xi \rangle} - 1 - i \langle u ; \xi \rangle\bbone_{\|u\|\leq 1}\right) \nu_1(du)\right),
\end{align}
where $\nu_1$ is a symmetric L\'evy measure on $\mathcal{B}(\mathbb{R}^d)$ satisfying \eqref{eq:def_scale_invariance} with $\alpha=1$.~Next, $(P^{\nu_1}_t)_{t \geq 0}$ denotes the $1$-stable Ornstein-Uhlenbeck semigroup with invariant measure $\mu_1$ and given by the Mehler integral representation formula of equation \eqref{eq:1_stable_symmetric_sg}. Thanks to \cite[Proposition $5.5$]{AH20_4}, its generator is given, for all $f \in \mathcal{S}(\mathbb{R}^d)$ and all $x \in \mathbb{R}^d$, by
\begin{align}\label{eq:1_stable_symmetric_gene}
\mathcal{L}^{\nu_1}(f)(x) = - \langle x ; \nabla(f)(x) \rangle + \mathcal{A}_1(f)(x),
\end{align}
 with $\mathcal{A}_1$ defined by \eqref{eq:1_stable_symmetric_diff}.~$\mu_1$ verifies the following Poincar\'e-type inequality: for all $f \in \mathcal{S}(\mathbb{R}^d)$ such that $\int_{\mathbb{R}^d} f(x) \mu_1(dx) = 0$, 
\begin{align}\label{Poinc_type_Cauchy}
\int_{\mathbb{R}^d} f(x)^2 \mu_1(dx) \leq \int_{\mathbb{R}^{2d}} \left| f(x+u)-f(x) \right|^2\nu_1(du) \mu_1(dx). 
\end{align}
Several proofs of Inequality \eqref{Poinc_type_Cauchy} have appeared in the literature and let us refer the reader to the classical works \cite{Chen_85,Chen_Lou_87,HPA_95,Houdre_al_98,Wu_PTRF00,Rockner_Wang_03,Last_Penrose_11}.~Below, a semigroup proof of this functional inequality is presented reminiscent of the well-known proof of the Gaussian Poincar\'e inequality along the classical Ornstein-Uhlenbeck semigroup.

\begin{prop}\label{prop:Poinc_type_Cauchy}
Let $d \geq 1$ be an integer, let $\nu_1$ be a non-degenerate symmetric L\'evy measure verifying \eqref{eq:def_scale_invariance} with $\alpha = 1$ and let $\mu_1$ be the corresponding non-degenerate symmetric $1$-stable 
probability measure on $\mathbb{R}^d$. Then,  for all $f \in \mathcal{S}(\mathbb{R}^d)$ such that $\int_{\mathbb{R}^d} f(x) \mu_1(dx) = 0$,
\begin{align*}
\int_{\mathbb{R}^d} f(x)^2 \mu_1(dx) \leq \int_{\mathbb{R}^{2d}} |f(x+u)-f(x)|^2 \nu_1(du) \mu_1(dx).
\end{align*}
\end{prop}

\begin{proof}
Let $f \in \mathcal{S}(\mathbb{R}^d)$ be such that $\int_{\mathbb{R}^d} f(x)\mu_1(dx) = 0$.~Thanks to Fourier inversion formula, for all $t \in [0,+\infty)$ and all $x \in \mathbb{R}^d$,
\begin{align*}
P^{\nu_1}_t(f)(x) = \frac{1}{(2\pi)^d} \int_{\mathbb{R}^d} \mathcal{F}(f)(\xi) e^{ i \langle x ; \xi e^{-t} \rangle} \dfrac{\widehat{\mu_1}(\xi)}{\widehat{\mu_1}(e^{-t}\xi)} d\xi.
\end{align*}
Moreover, by Fourier argument, for all $t \geq 0$ and all $x \in \mathbb{R}^d$,
\begin{align*}
\dfrac{d}{dt} \left(P^{\nu_1}_t(f)(x)\right)  = \mathcal{L}^{\nu_1}\left(P^{\nu_1}_t(f)\right)(x). 
\end{align*}
Finally, thanks to the \cite[proof of Proposition $5.5$]{AH20_4}, for all $t\in [0,+\infty)$ and all $x \in \mathbb{R}^d$,
\begin{align*}
\mathcal{L}^{\nu_1}  \left(P^{\nu_1}_t(f)\right)(x) =P^{\nu_1}_t(\mathcal{L}^{\nu_1}(f))(x).
\end{align*}
Thus,  for all $x \in \mathbb{R}^d$ and all $t \geq 0$, 
\begin{align*}
\langle x ; \nabla\left(P^{\nu_1}_t(f)\right)(x) \rangle = \mathcal{A}_1(P^{\nu_1}_t(f))(x)-P^{\nu_1}_t(\mathcal{L}_1  (f))(x).
\end{align*}
Then, differentiating under the integral sign, 
\begin{align*}
\dfrac{d}{dt} \left( \int_{\mathbb{R}^d} (P^{\nu_1}_t(f)(x))^2 \mu_1(dx) \right) & = 2 \int_{\mathbb{R}^d} P^{\nu_1}_t(f)(x) \dfrac{d}{dt} \left(P^{\nu_1}_t(f)(x)\right) \mu_1(dx) , \\
& = - 2 \int_{\mathbb{R}^d} P^{\nu_1}_t(f)(x) \langle x ; \nabla\left(P^{\nu_1}_t(f)\right)(x) \rangle \mu_1(dx) \\
& \quad \quad + 2 \int_{\mathbb{R}^d} P^{\nu_1}_t(f)(x) \mathcal{A}_1(P^{\nu_1}_t(f))(x) \mu_1(dx).
\end{align*}
Now,  by the chain rule and integration by parts, 
\begin{align*}
2 \int_{\mathbb{R}^d} P^{\nu_1}_t(f)(x) \langle x ; \nabla\left(P^{\nu_1}_t(f)\right)(x) \rangle \mu_1(dx) & = \int_{\mathbb{R}^d}  \langle x ; \nabla\left((P^{\nu_1}_t(f))^2\right)(x) \rangle \mu_1(dx) ,  \\
& = \int_{\mathbb{R}^d}  \mathcal{A}_1 \left((P^{\nu_1}_t(f))^2\right)(x) \mu_1(dx).
\end{align*}
Note that $(P^{\nu_1}_t(f))^2$ belongs to the set of infinitely continuously differentiable functions on $\mathbb{R}^d$ with bounded derivatives up to the second order and such that $\|\langle x; \nabla (f) \rangle\|_{\infty}<+\infty$.~This space is a core for the $L^p(\mu_1)$-generator of $(P_t^{\nu_1})_{t \geq 0}$, with $p \in (1,+\infty)$, and on it, this operator is equal to $\mathcal{L}^{\nu_1}$.~Moreover, it is stable under pointwise multiplication. Thus, 
\begin{align*}
- \dfrac{d}{dt} \left( \int_{\mathbb{R}^d} (P^{\nu_1}_t(f)(x))^2 \mu_1(dx) \right) = \int_{\mathbb{R}^d} \left(\mathcal{A}_1 \left((P^{\nu_1}_t(f))^2\right)(x) - 2P^{\nu_1}_t(f)(x) \mathcal{A}_1(P^{\nu_1}_t(f))(x) \right) \mu_1(dx).
\end{align*}
Now, for all $x \in \mathbb{R}^d$ and all $t \in [0,+\infty)$, 
\begin{align*}
\mathcal{A}_1 \left((P^{\nu_1}_t(f))^2\right)(x) - 2P^{\nu_1}_t(f)(x) \mathcal{A}_1(P^{\nu_1}_t(f))(x) & =  \int_{\mathbb{R}^d \times \mathbb{R}^d} \mathcal{F}(P^{\nu_1}_t(f))(\xi) \mathcal{F}(P^{\nu_1}_t(f))(\zeta) e^{i \langle x ; \xi +\zeta \rangle} \\
&\quad\quad \times \psi_1(\xi,\zeta)\frac{d \xi d\zeta}{(2\pi)^{2d}} , 
\end{align*}
where $\psi_1$ is defined, for all $\xi, \zeta \in \mathbb{R}^d$, by
\begin{align*}
\psi_1(\xi,\zeta) = \int_{\mathbb{R}^d} \left(e^{i \langle u ;\xi \rangle}-1\right) \left(e^{i \langle u ;\zeta \rangle}-1\right) \nu_1(du). 
\end{align*}
Thus,  for all $x \in \mathbb{R}^d$ and all $t \in [0,+\infty)$, 
\begin{align*}
\mathcal{A}_1 \left((P^{\nu_1}_t(f))^2\right)(x) - 2P^{\nu_1}_t(f)(x) \mathcal{A}_1(P^{\nu_1}_t(f))(x) & = \int_{\mathbb{R}^d} |P^{\nu_1}_t(f)(x+u) - P^{\nu_1}_t(f)(x)|^2 \nu_1(du),
\end{align*}
so that, 
\begin{align*}
- \dfrac{d}{dt} \left( \int_{\mathbb{R}^d} (P^{\nu_1}_t(f)(x))^2 \mu_1(dx) \right) = \int_{\mathbb{R}^d} \Gamma_1\left(P^{\nu_1}_t(f),P^{\nu_1}_t(f)\right)(x) \mu_1(dx),
\end{align*}
where, for all $f_1,f_2$ suitable and all $x \in \mathbb{R}^d$, 
\begin{align*}
\Gamma_1\left(f_1,f_2\right)(x) =  \int_{\mathbb{R}^d} (f_1(x+u)-f_1(x))(f_2(x+u)-f_2(x)) \nu_1(du).
\end{align*} 
Thanks to Jensen's inequality,  a change of variables in the radial coordinate and using the fact that, under $\mu_1 \otimes \mu_1$, $e^{-t}x + (1-e^{-t})y$ is distributed according to $\mu_1$, 
\begin{align*}
\int_{\mathbb{R}^d} \Gamma_1\left(P^{\nu_1}_t(f),P^{\nu_1}_t(f)\right)(x) \mu_1(dx) \leq e^{-t} \int_{\mathbb{R}^d} \Gamma_1\left(f,f\right)(x) \mu_1(dx).
\end{align*}
Then, for all $t \geq 0$,
\begin{align*}
- \dfrac{d}{dt} \left( \int_{\mathbb{R}^d} (P^{\nu_1}_t(f)(x))^2 \mu_1(dx) \right) \leq e^{-t} \int_{\mathbb{R}^d} \Gamma_1\left(f,f\right)(x) \mu_1(dx).
\end{align*}
Integrating with respect to $t$ and using the fact that $\mathbb{E} f(X_1) = 0$, with $X_1 \sim \mu_1$, concludes the proof of the Poincar\'e-type inequality for $\mu_1$. 
\end{proof}
\noindent
Next, let us focus our analysis on the rotationally invariant case for which the Lebesgue density is  known explicitly. Indeed, for all $x \in \mathbb{R}^d$ and all integer $d \geq 1$, 
\begin{align}\label{eq:def_rot_inv}
p_1^{\operatorname{rot}}(x) : = \dfrac{c_d}{\left(1+ \|x\|^2\right)^{\frac{d+1}{2}}}, \quad c_d := \dfrac{\Gamma\left(\frac{d+1}{2}\right)}{\pi^{\frac{d+1}{2}}}.
\end{align}
Let $\nu_1^{\operatorname{rot}}$ be the L\'evy measure on $\mathbb{R}^d$ associated with the $1$-stable probability measure $\mu_1^{\operatorname{rot}}$ whose Lebesgue density is given by \eqref{eq:def_rot_inv}. In the sequel, $\mu_1^{\operatorname{rot}}$ is called the \textit{standard Cauchy probability measure} on $\mathbb{R}^d$. Its Fourier transform is given, for all $\xi \in \mathbb{R}^d$, by
\begin{align}\label{eq:Fourier_transform_standard_Cauchy}
\widehat{\mu_1^{\operatorname{rot}}}(\xi) = \exp \left( - \|\xi\|\right). 
\end{align}
Let $(P_t^{\nu_1^{\operatorname{rot}}})_{t \geq 0}$ be the semigroup of operators defined by \eqref{eq:1_stable_symmetric_sg} with $\mu_1 = \mu_1^{\operatorname{rot}}$.~Thanks to the integral representation \eqref{eq:1_stable_symmetric_sg}, this semigroup of operators can be extended in a compatible way to semigroups of linear contractions on $L^p(\mu_1^{\operatorname{rot}})$, for all $p \in [1,+\infty)$, since $\mu_1^{\operatorname{rot}}$ is an invariant measure for the semigroup.~As previously mentioned, thanks to \cite[Proposition $5.5$]{AH20_4}, for all $p \in (1, +\infty)$, $\mathcal{S}(\mathbb{R}^d)$ is a subset of the $L^p$-domain of the generator of $(P_t^{\nu_1^{\operatorname{rot}}})_{t \geq 0}$ and, for all $f \in \mathcal{S}(\mathbb{R}^d)$ and all $x \in \mathbb{R}^d$, 
\begin{align}\label{eq:OU_Cauchy_rot_inv}
\mathcal{L}_1^{\operatorname{rot}}(f)(x) = - \langle x ; \nabla(f)(x) \rangle + \mathcal{A}_1^{\operatorname{rot}}(f)(x), 
\end{align}
where, 
\begin{align}\label{eq:non_local_square_root_Laplacian}
\mathcal{A}_1^{\operatorname{rot}}(f)(x) & =\int_{\mathbb{R}^d} \mathcal{F}(f)(\xi) e^{i \langle x ; \xi \rangle} \left( -  \|\xi\|\right) \frac{d\xi}{(2\pi)^d},  \nonumber \\
& = \int_{\mathbb{R}^d} \left(f(x+u) - f(x) -\bbone_{\| u \|\leq 1} \langle u  ; \nabla(f)(x) \rangle\right) \nu_1^{\operatorname{rot}}(du). 
\end{align}
Denoting by $\left((P_t^{\nu_1^{\operatorname{rot}}})^*\right)_{t\geq 0}$ the dual semigroup, one can consider the ``carr\'e de Mehler" semigroup associated with the standard Cauchy probability measure: for all $f \in L^2\left(\mu_1^{\operatorname{rot}}\right)$ and all $t \geq 0$, 
\begin{align}\label{eq:carre_de_mehler_sg}
\mathcal{P}^{\operatorname{rot}}_t(f) = ((P_t^{\nu_1^{\operatorname{rot}}})^* \circ P_t^{\nu_1^{\operatorname{rot}}})(f) = (P_t^{\nu_1^{\operatorname{rot}}} \circ (P_t^{\nu_1^{\operatorname{rot}}})^*)(f). 
\end{align}
The generator of the ``carr\'e de Mehler" semigroup is denoted by $\mathcal{L}_1$ in the sequel.~Using the techniques developed in \cite[Section $3$]{AH22_5}, the next proposition provides explicit formulas for the dual semigroup $((P_t^{\nu_1^{\operatorname{rot}}})^*)_{t \geq 0}$, for $(\mathcal{L}^{\operatorname{rot}}_1)^*$, the adjoint of $\mathcal{L}_1^{\operatorname{rot}}$, and for $\mathcal{L}_{1}$ on $\mathcal{S}(\mathbb{R}^d)$.  

\begin{prop}\label{lem:dual_Cauchy_SG}
Let $d \geq 1$ be an integer and let $p_1^{\operatorname{rot}}$ be defined by \eqref{eq:def_rot_inv}. Then, for all $f \in \mathcal{S}(\mathbb{R}^d)$, all $t>0$ and all $x \in \mathbb{R}^d$, 
\begin{align}\label{eq:Cauchy_dual_SG}
(P_t^{\nu_1^{\operatorname{rot}}})^*(f)(x) = \frac{1}{p_1^{\operatorname{rot}}(x)} \int_{\mathbb{R}^d} f(y) p_1^{\operatorname{rot}}(y) p_1^{\operatorname{rot}}\left(\dfrac{x-ye^{-t}}{(1-e^{-t})}\right) \frac{dy}{(1-e^{-t})^d}.
\end{align}
Moreover, for all $f \in \mathcal{S}(\mathbb{R}^d)$ and all $x \in \mathbb{R}^d$, 
\begin{align*}
& (\mathcal{L}^{\operatorname{rot}}_1)^*(f)(x) =  d f(x) + \frac{1}{p_1^{\operatorname{rot}}(x)} \langle x ; \nabla(p_1^{\operatorname{rot}}f)(x) \rangle+ \frac{1}{p_1^{\operatorname{rot}}(x)} \mathcal{A}_1^{\operatorname{rot}}(p_1^{\operatorname{rot}}f)(x), \\
& \mathcal{L}_{1}(f)(x) : = \mathcal{L}^{\operatorname{rot}}_1(f)(x) + (\mathcal{L}^{\operatorname{rot}}_1)^*(f)(x) = \dfrac{d-\|x\|^2}{1+\|x\|^2}f(x) + \mathcal{A}^{\operatorname{rot}}_1(f)(x) + \frac{1}{p_1^{\operatorname{rot}}(x)} \mathcal{A}_1^{\operatorname{rot}}(p_1^{\operatorname{rot}}f)(x).
\end{align*}
\end{prop}

\begin{proof}
To prove formula \eqref{eq:Cauchy_dual_SG}, one can proceed as in the proof of \cite[Lemma $3.1$]{AH22_5} by a duality argument. Now, let $\left(T^1_t\right)_{t\geq 0}$ be the continuous family of operators defined, for all $f \in \mathcal{S}(\mathbb{R}^d)$, all $t>0$ and all $x \in \mathbb{R}^d$, by
\begin{align*}
T_t^1(f)(x) =  \int_{\mathbb{R}^d} f(y) p_1^{\operatorname{rot}}\left(\dfrac{x-ye^{-t}}{(1-e^{-t})}\right) \frac{dy}{(1-e^{-t})^d}, \quad T_0(f)(x) = f(x), 
\end{align*}
and let $M_1$ be the multiplication operator by the function $p_1^{\operatorname{rot}}$. It is clear that, for all $f \in \mathcal{S}(\mathbb{R}^d)$, all $t>0$ and all $x \in \mathbb{R}^d$, 
\begin{align*}
(P_t^{\nu_1^{\operatorname{rot}}})^*(f)(x) = \left(M_1^{-1} \circ T^1_t \circ M_1\right)(f)(x). 
\end{align*}
$\left(T^1_t\right)_{t\geq 0}$ is a semigroup of operators which admits the following representation thanks to Fourier inversion formula: for all $f \in \mathcal{S}(\mathbb{R}^d)$, all $x \in \mathbb{R}^d$ and all $t>0$, 
\begin{align*}
T_t^1(f)(x) = \frac{e^{td}}{(2\pi)^d} \int_{\mathbb{R}^d} \mathcal{F}(f)(\xi) e^{i \langle x ; \xi \rangle e^t} \dfrac{\varphi_1(e^{t} \xi)}{ \varphi_1(\xi)} d\xi  , \quad \varphi_1(\xi) =\int_{\mathbb{R}^d} e^{i \langle x ; \xi \rangle} \mu^{\operatorname{rot}}_1(dx).
\end{align*}
Based on this Fourier representation formula, it is readily seen that the generator of $(T_t^1)_{t \geq 0}$ is given, for all $f \in \mathcal{S}(\mathbb{R}^d)$ and all $x \in \mathbb{R}^d$, by
\begin{align}\label{eq:generator_T_t}
A_1(f)(x) = d f(x) + \langle x ; \nabla(f)(x) \rangle + \mathcal{A}_1^{\operatorname{rot}}(f)(x).  
\end{align}
Then, the formulas for $(\mathcal{L}^{\operatorname{rot}}_1)^*(f)$ and for $\mathcal{L}_{1}(f)$ follow easily. 
\end{proof}
\noindent
To pursue our preliminary analysis, let us compute the action of the operator $\mathcal{A}_1^{\operatorname{rot}}$ on the Lebesgue density of the standard Cauchy probability measure. 

\begin{lem}\label{lem:action_square_root_laplacian}
Let $d \geq 1$ be an integer, let $p_1^{\operatorname{rot}}$ be defined by \eqref{eq:def_rot_inv} and let $\mathcal{A}_{1}^{\operatorname{rot}}$ be given by \eqref{eq:non_local_square_root_Laplacian}. Then, for all $x \in \mathbb{R}^d$, 
\begin{align}\label{eq:action_square_root_laplacian}
\mathcal{A}_1^{\operatorname{rot}}(p_1^{\operatorname{rot}})(x) =  p_1^{\operatorname{rot}}(x) \dfrac{\|x\|^2 - d}{1+\|x\|^2}.
\end{align}
Moreover, for all $x \in \mathbb{R}^d$
\begin{align}\label{eq:harmonic}
A_1(p_1^{\operatorname{rot}})(x) = 0.
\end{align}
\end{lem}

\begin{proof}
Equation \eqref{eq:action_square_root_laplacian} can be seen as a direct consequence of $\mathcal{L}_1(1)= 0$ which follows from the fact that the semigroups $(P_t^{\nu^{\operatorname{rot}}_1})_{t \geq 0}$ and $((P_t^{\nu^{\operatorname{rot}}_1})^*)_{t \geq 0}$ are mass conservative in the sense that $P_t^{\nu^{\operatorname{rot}}_1}(1) = 1$ and $(P_t^{\nu^{\operatorname{rot}}_1})^*(1)=1$, for all $t \geq 0$. Equation \eqref{eq:harmonic} follows by direct computations. 
\end{proof}
\noindent
Now, let $\Gamma_1^{\operatorname{rot}}$ be the carr\'e du champs operator associated with $\mathcal{A}_1^{\operatorname{rot}}$ defined, for all $f,g \in \mathcal{S}(\mathbb{R}^d)$ and all $x \in \mathbb{R}^d$, by
\begin{align}\label{eq:definition_squared_field_operator}
\Gamma_1^{\operatorname{rot}}(f,g)(x) & : = \mathcal{A}_1^{\operatorname{rot}}(fg)(x) - g(x)\mathcal{A}_1^{\operatorname{rot}}(f)(x) -f(x)\mathcal{A}_1^{\operatorname{rot}}(g)(x) , \nonumber \\
& = \int_{\mathbb{R}^d} \left(f(x+u)-f(x)\right)\left(g(x+u)-g(x)\right) \nu_1^{\operatorname{rot}}(du).
\end{align}
Then, for all $x \in \mathbb{R}^d$ and all $f \in \mathcal{S}(\mathbb{R}^d)$, 
\begin{align*}
\frac{1}{p^{\operatorname{rot}}_1(x)}\mathcal{A}_1^{\operatorname{rot}}(p_1^{\operatorname{rot}}f)(x) = \frac{1}{p_1^{\operatorname{rot}}(x)} \Gamma_1^{\operatorname{rot}}(p_1^{\operatorname{rot}},f)(x) + \mathcal{A}_1^{\operatorname{rot}}(f)(x) + \frac{f(x)}{p_1^{\operatorname{rot}}(x)} \mathcal{A}_1^{\operatorname{rot}}(p_1^{\operatorname{rot}})(x),
\end{align*}
which implies that, 
\begin{align}\label{eq:smart_decomposition}
\mathcal{L}_1(f)(x) = 2\mathcal{A}_1^{\operatorname{rot}}(f)(x) + \frac{1}{p_1^{\operatorname{rot}}(x)} \Gamma_1^{\operatorname{rot}}(p_1^{\operatorname{rot}},f)(x). 
\end{align}
Finally, for all $f_1 \in \mathcal{D}(\mathcal{L}_1)$ and all $f_2 \in \mathcal{D}(\mathcal{E}_{\nu_1^{\operatorname{rot}}, \mu_1^{\operatorname{rot}}})$, 
\begin{align}\label{eq:ipp}
\mathcal{E}_{\nu_1^{\operatorname{rot}}, \mu_1^{\operatorname{rot}}} (f_1,f_2) = \langle (- \mathcal{L}_1)(f_1) ; f_2 \rangle_{L^2(\mu_1^{\operatorname{rot}})}, 
\end{align}
where $\mathcal{D}(\mathcal{L}_1)$ is the $L^2(\mu_1^{\operatorname{rot}})$-domain of the generator $\mathcal{L}_1$, where $\mathcal{E}_{\nu_1^{\operatorname{rot}}, \mu_1^{\operatorname{rot}}}$ is given, for all $f_1,f_2 \in \mathcal{C}^1_b(\mathbb{R}^d)$, by 
\begin{align}\label{eq:form_standard_cauchy_pm}
\mathcal{E}_{\nu_1^{\operatorname{rot}}, \mu_1^{\operatorname{rot}}}(f_1 , f_2) = \int_{\mathbb{R}^d} \int_{\mathbb{R}^d} (f_1(x+u)-f_1(x))(f_2(x+u)-f_2(x))\nu_1^{\operatorname{rot}}(du) \mu_1^{\operatorname{rot}}(dx),
\end{align}
and where $\mathcal{D}(\mathcal{E}_{\nu_1^{\operatorname{rot}}, \mu_1^{\operatorname{rot}}})$ is the $L^2(\mu_1^{\operatorname{rot}})$-domain of the $L^2(\mu_1^{\operatorname{rot}})$-extension of $(\mathcal{E}_{\nu_1^{\operatorname{rot}}, \mu_1^{\operatorname{rot}}}, \mathcal{C}^1_b(\mathbb{R}^d))$ compatible with the self-adjoint operator $\mathcal{L}_1$. In the remainder of this section, let us investigate $L^2(\mu_1^{\operatorname{rot}})$-spectral properties of the operators $\mathcal{L}_1^{\operatorname{rot}}$ and $\mathcal{L}_1$.~Thanks to the integral representation \eqref{eq:1_stable_symmetric_sg}, for all $x \in \mathbb{R}^d$ and all $t \geq 0$, 
\begin{align}\label{eq:pointwise_approximate_eigenvector_sg}
P_t^{\nu_1^{\operatorname{rot}}}(g_{R})(x) \longrightarrow e^{-t} g(x) = e^{-t} x ,  
\end{align}
as $R$ tends to $+\infty$. In particular, for all $x \in \mathbb{R}^d$, 
\begin{align}\label{eq:pointwise_approximate_eigenvector_gen}
\mathcal{L}_1^{\operatorname{rot}}(g_R)(x) \longrightarrow -x,
\end{align}
as $R$ tends to $+\infty$. Based on these observations, it is natural to wonder if the vector-valued function $g_R$ is an approximate eigenvector of the operator $\mathcal{L}_1^{\operatorname{rot}}$ associated with the approximate eigenvalue $-1$. More precisely, do we have, for all $j \in \{1, \dots, d\}$, 
\begin{align}\label{eq:question_approximate_eigenvector}
\underset{R \longrightarrow +\infty}{\lim} \dfrac{\|\mathcal{L}_1^{\operatorname{rot}}(g_{R,j}) + g_{R,j}\|_{L^2(\mu_1^{\operatorname{rot}})}}{\|g_{R,j}\|_{L^2(\mu_1^{\operatorname{rot}})}} =0 \quad ?
\end{align}
The next results of this section show that the answer to this question is no. First, let us prove one technical lemma. 

\begin{lem}\label{lem:limit_first_term}
Let $d \geq 1$ be an integer, let $\mu_1^{\operatorname{rot}}$ be the standard Cauchy probability measure on $\mathbb{R}^d$ and, for $R>0$, let $g_{R}$ be defined by \eqref{eq:eigen_approximate_eigenvector}. Then, for all $j \in \{1, \cdots,d\}$
\begin{align*}
\underset{R \longrightarrow +\infty}{\lim} \langle g_{R,j} \left(1 - 4 \frac{\|x\|^2}{R^2} \right) ; g_{R,j} \rangle_{L^2(\mu_1^{\operatorname{rot}})} & = c_d \left(\int_{\mathbb{S}^{d-1}} |\langle y ; e_1 \rangle|^2 \sigma_L(dy) \right) \\
&\quad\quad \times \int_{0}^{+\infty} \left(\dfrac{1}{(1+\frac{1}{r^2})^{\frac{d+1}{2}}}-1\right)dr,
\end{align*}
where $\sigma_L$ is the spherical part of the $d$-dimensional Lebesgue measure.
\end{lem}

\begin{proof}
Changing variables, for all $R>0$ and all $j \in \{1, \dots, d\}$, 
\begin{align*}
\langle g_{R,j} \left(1 - 4 \frac{\|x\|^2}{R^2}\right) ; g_{R,j}  \rangle_{L^2(\mu_1^{\operatorname{rot}})} & = \int_{\mathbb{R}^d} p_1^{\operatorname{rot}}(x) x_j^2 \exp\left( - 2 \frac{\|x\|^2}{R^2}\right) \left(1 - 4 \frac{\|x\|^2}{R^2}\right) dx , \\
& = \int_{\mathbb{R}^d} p_1^{\operatorname{rot}}(Rx) R^{d+2} x_j^2 \exp\left( - 2 \|x\|^2\right) \left(1 - 4 \|x\|^2\right)dx , \\
& = c_d \int_{\mathbb{R}^d} \dfrac{R^{d+2}}{\left(1+ R^2 \|x\|^2\right)^{\frac{d+1}{2}}} x_j^2 \exp\left( - 2 \|x\|^2\right) \left(1 - 4 \|x\|^2\right)dx.
\end{align*}
Recall that, 
\begin{align}
\int_0^{+\infty} e^{-2r^2}(1- 4r^2) dr = 0,
\end{align}
which can be seen as a consequence of the orthogonality relation for the Hermite polynomials. Then, using spherical coordinates, 
{\allowdisplaybreaks
\begin{align*}
\langle g_{R,j} \left(1 - 4 \frac{\|x\|^2}{R^2}\right) ; g_{R,j}  \rangle_{L^2(\mu_1^{\operatorname{rot}})} & = c_d \int_{\mathbb{R}^d} \dfrac{R^{d+2}}{\left(1+ R^2 \|x\|^2\right)^{\frac{d+1}{2}}} x_j^2 \exp\left( - 2 \|x\|^2\right) \left(1 - 4 \|x\|^2\right)dx , \\
& = c_d R \int_{\mathbb{S}^{d-1}} |\langle  y ; e_j\rangle|^2 \sigma_L(dy) \int_{0}^{+\infty}  \dfrac{R^{d+1} r^{d+1}}{\left(1+ R^2 r^2 \right)^{\frac{d+1}{2}}} \exp\left( - 2 r^2\right) \\
&\quad\quad \times \left(1 - 4 r^2\right)dr , \\
& = c_d R \int_{\mathbb{S}^{d-1}} |\langle  y ; e_j\rangle|^2 \sigma_L(dy) \int_{0}^{+\infty}  \left(\dfrac{R^{d+1} r^{d+1}}{\left(1+ R^2 r^2 \right)^{\frac{d+1}{2}}} -1 \right) \\
&\quad\quad \times \exp\left( - 2 r^2\right) \left(1 - 4 r^2\right)dr , \\
& = c_d R \int_{\mathbb{S}^{d-1}} |\langle  y ; e_j\rangle|^2 \sigma_L(dy) \int_{0}^{+\infty}  \left(\dfrac{1}{\left(1+ \frac{1}{R^2 r^2} \right)^{\frac{d+1}{2}}} -1 \right) \\
&\quad\quad \times \exp\left( - 2 r^2\right) \left(1 - 4 r^2\right)dr , \\
& = c_d  \int_{\mathbb{S}^{d-1}} |\langle  y ; e_j\rangle|^2 \sigma_L(dy) \int_{0}^{+\infty}  \left(\dfrac{1}{\left(1+ \frac{1}{r^2} \right)^{\frac{d+1}{2}}} -1 \right) \\
&\quad\quad \times \exp\left( - 2 \frac{r^2}{R^2}\right) \left(1 - 4 \frac{r^2}{R^2}\right)dr , \\
& \longrightarrow c_d  \int_{\mathbb{S}^{d-1}} |\langle  y ; e_j\rangle|^2 \sigma_L(dy) \int_{0}^{+\infty}  \left(\dfrac{1}{\left(1+ \frac{1}{r^2} \right)^{\frac{d+1}{2}}} -1 \right) dr,
\end{align*}
}
\noindent as $R$ tends to $+\infty$ (thanks to the Lebesgue dominated convergence theorem). This concludes the proof of the lemma.  
\end{proof}
\noindent
The next result answers by the negative question \eqref{eq:question_approximate_eigenvector}. 

\begin{prop}\label{prop:aps_ou_stable}
Let $d \geq 1$ be an integer, let $\mathcal{L}_1^{\operatorname{rot}}$ be the operator defined by \eqref{eq:OU_Cauchy_rot_inv} and, for $R>0$, let $g_{R}$ be defined by \eqref{eq:eigen_approximate_eigenvector}. Then, for all $j \in \{1, \dots, d\}$,
\begin{align}\label{eq:minus_one_outside}
\underset{R\longrightarrow +\infty}{\lim} \dfrac{\| \mathcal{L}^{\operatorname{rot}}_1(g_{R,j}) + g_{R,j} \|^2_{L^2(\mu_1^{\operatorname{rot}})}}{\|g_{R,j}\|_{L^2(\mu_1^{\operatorname{rot}})}^2} = \frac{3}{4}. 
\end{align}
\end{prop}

\begin{proof}
First, for all $R>0$ and all $j \in \{1, \dots, d\}$, 
\begin{align*}
\|g_{R,j}\|^2_{L^2(\mu_1^{\operatorname{rot}})} & = \int_{\mathbb{R}^d} x_j^2 \exp \left(-2\frac{\|x\|^2}{R^2}\right) \dfrac{c_d dx}{\left(1+\|x\|^2\right)^{\frac{d+1}{2}}}, \\
& = c_d \int_{\mathbb{S}^{d-1}} \left| \langle y ; e_j \rangle \right|^2 \sigma_L(dy) \int_0^{+\infty} \exp \left(-2\frac{r^2}{R^2}\right) \dfrac{r^{d+1} dr}{\left(1+r^2\right)^{\frac{d+1}{2}}}.
\end{align*}
Thus, 
\begin{align*}
\frac{1}{R} \|g_{R,j}\|^2_{L^2(\mu_1^{\operatorname{rot}})} & = c_d \int_{\mathbb{S}^{d-1}} \left| \langle y ; e_j \rangle \right|^2 \sigma_L(dy) \int_0^{+\infty} \exp \left(-2r^2\right) \dfrac{R^{d+1} r^{d+1} dr}{\left(1+R^2 r^2\right)^{\frac{d+1}{2}}}, \\
& \longrightarrow c_d \int_{\mathbb{S}^{d-1}} \left| \langle y ; e_j \rangle \right|^2 \sigma_L(dy) \int_0^{+\infty} \exp \left(-2r^2\right)dr.
\end{align*}
Moreover, for all $R>0$ and all $ j \in \{1, \dots, d\}$,
\begin{align*}
\| \mathcal{L}^{\operatorname{rot}}_1(g_{R,j}) + g_{R,j} \|^2_{L^2(\mu_1^{\operatorname{rot}})} = \| \mathcal{L}^{\operatorname{rot}}_1(g_{R,j}) \|^2_{L^2(\mu_1^{\operatorname{rot}})}+\|g_{R,j} \|^2_{L^2(\mu_1^{\operatorname{rot}})}+2\langle \mathcal{L}^{\operatorname{rot}}_1(g_{R,j}) ; g_{R,j} \rangle_{L^2(\mu_1^{\operatorname{rot}})}.
\end{align*}
Now, since $g_{R,j} \in \mathcal{S}(\mathbb{R}^d)$, for all $R>0$, 
\begin{align*}
\mathcal{E}_{\nu_1^{\operatorname{rot}}, \mu_1^{\operatorname{rot}}}(g_{R,j},g_{R,j}) = \langle (- \mathcal{L}_1)(g_{R,j}) ; g_{R,j} \rangle_{L^2(\mu_1^{\operatorname{rot}})} = 2 \langle (- \mathcal{L}_1^{\operatorname{rot}}) (g_{R,j}) ; g_{R,j} \rangle_{L^2(\mu_1^{\operatorname{rot}})}.
\end{align*}
Moreover, by Fourier methods, 
\begin{align*}
\dfrac{\|g_{R,j} \|^2_{L^2(\mu_1^{\operatorname{rot}})}+2\langle \mathcal{L}^{\operatorname{rot}}_1(g_{R,j}) ; g_{R,j} \rangle_{L^2(\mu_1^{\operatorname{rot}})}}{\|g_{R,j}\|_{L^2(\mu_1^{\operatorname{rot}})}^2} = 1 - \dfrac{\mathcal{E}_{\nu_1^{\operatorname{rot}}, \mu_1^{\operatorname{rot}}}(g_{R,j} ; g_{R,j})}{\|g_{R,j}\|_{L^2(\mu_1^{\operatorname{rot}})}^2} \longrightarrow 0,
\end{align*}
as $R$ tends to $+\infty$. So, it remains to compute the limit of the first term\\ $ \| \mathcal{L}^{\operatorname{rot}}_1(g_{R,j}) \|^2_{L^2(\mu_1^{\operatorname{rot}})} / \|g_{R,j}\|^2_{L^2(\mu_1^{\operatorname{rot}})}$ as $R$ tends to $+\infty$. Developing the square, for all $R>0$,
\begin{align}\label{eq:decomposition_L2_norm_OU_gen}
\| \mathcal{L}^{\operatorname{rot}}_1(g_{R,j}) \|^2_{L^2(\mu_1^{\operatorname{rot}})} & = \|x \cdot \nabla(g_{R,j})\|^2_{L^2(\mu_1^{\operatorname{rot}})} + \|\mathcal{A}_1^{\operatorname{rot}} (g_{R,j}) \|^2_{L^2(\mu_1^{\operatorname{rot}})} \nonumber \\
&\quad\quad - 2\langle x \cdot \nabla(g_{R,j}) ; \mathcal{A}_1^{\operatorname{rot}} (g_{R,j}) \rangle_{L^2(\mu_1^{\operatorname{rot}})}.
\end{align}
Now, based on standard computations, for all $R>0$, all $x \in \mathbb{R}^d$ and all $j \in \{1, \dots, d\}$, 
\begin{align}\label{eq:drift_term}
 \langle x ; \nabla(g_{R,j})(x) \rangle &  = g_{R,j}(x) - 2 \sum_{k = 1}^d \dfrac{x_jx_k^2}{R^2} \exp\left( - \frac{\|x\|^2}{R^2}\right) , \nonumber \\
 & = g_{R,j}(x) \left(1 - 2 \frac{\|x\|^2}{R^2}\right).
\end{align}
For the first term in \eqref{eq:decomposition_L2_norm_OU_gen}, thanks to \eqref{eq:drift_term} and using spherical coordinates, 
\begin{align*}
\|x \cdot \nabla(g_{R,j})\|^2_{L^2(\mu_1^{\operatorname{rot}})} & = \int_{\mathbb{R}^d} g_{R,j}(x)^2 \left(1 - 2 \frac{\|x\|^2}{R^2}\right)^2 p_1^{\operatorname{rot}}(x)dx , \\
& = c_d \int_{\mathbb{S}^{d-1}} \left| \langle y ; e_j \rangle \right|^2 \sigma_L(dy) \int_0^{+\infty} \left(1 - 2 \frac{r^2}{R^2}\right)^2 \exp(-2\frac{r^2}{R^2}) \dfrac{r^{d+1}dr}{\left(1+r^2\right)^{\frac{d+1}{2}}}.
\end{align*}
Changing variables, 
\begin{align*}
\|x \cdot \nabla(g_{R,j})\|^2_{L^2(\mu_1^{\operatorname{rot}})} & = R c_d \int_{\mathbb{S}^{d-1}} \left| \langle y ; e_j \rangle \right|^2 \sigma_L(dy) \int_0^{+\infty} \left(1 - 2 r^2\right)^2 \exp(-2r^2) \dfrac{R^{d+1} r^{d+1}dr}{\left(1+R^2 r^2\right)^{\frac{d+1}{2}}}.
\end{align*}
Then, 
\begin{align*}
\frac{1}{R} \|x \cdot \nabla(g_{R,j})\|^2_{L^2(\mu_1^{\operatorname{rot}})} \longrightarrow c_d \int_{\mathbb{S}^{d-1}} \left| \langle y ; e_j \rangle \right|^2 \sigma_L(dy) \int_0^{+\infty} \left(1 - 2 r^2\right)^2 \exp(-2r^2) dr \ne 0,
\end{align*}
as $R$ tends to $+\infty$. Regarding the second term, for all $j \in \{1, \dots, d\}$,
\begin{align*}
 \|\mathcal{A}_1^{\operatorname{rot}} (g_{R,j}) \|^2_{L^2(\mu_1^{\operatorname{rot}})} & = \int_{\mathbb{R}^d} \left|  \mathcal{A}_1^{\operatorname{rot}} (g_{R,j})(x)\right|^2 \frac{c_d dx}{\left(1+\|x\|^2\right)^{\frac{d+1}{2}}} , \\
 & = \int_{\mathbb{R}^d} \left|  \mathcal{A}_1^{\operatorname{rot}} (g_{R,j})(Rx)\right|^2 \frac{R^d c_d dx}{\left(1+R^2\| x\| ^2\right)^{\frac{d+1}{2}}} , \\
 & = \int_{\mathbb{R}^d} \left|  \mathcal{A}_1^{\operatorname{rot}} (g_{1,j})(x)\right|^2 \frac{R^d c_d dx}{ \left(1+R^2 \| x\|^2\right)^{\frac{d+1}{2}}} , \\
 & = \int_{\mathbb{R}^d} \left|  \mathcal{A}_1^{\operatorname{rot}} (g_{1,j})\left(\frac{x}{R}\right)\right|^2 \frac{c_d dx}{ \left(1+  \| x\|^2\right)^{\frac{d+1}{2}}}, \\
 & \longrightarrow \int_{\mathbb{R}^d} \left|  \mathcal{A}_1^{\operatorname{rot}} (g_{1,j})\left(0\right)\right|^2 \frac{c_d dx}{ \left(1+  \| x\|^2\right)^{\frac{d+1}{2}}} = 0, 
\end{align*}
as $R$ tends to $+\infty$ (thanks to the Lebesgue dominated convergence theorem). Finally, for the last term, 
\begin{align*}
\langle x \cdot \nabla(g_{R,j}) ; \mathcal{A}_1^{\operatorname{rot}} (g_{R,j}) \rangle_{L^2(\mu_1^{\operatorname{rot}})} & = \int_{\mathbb{R}^d} x_j \exp\left(- \frac{\|x\|^2}{R^2}\right) \left(1 -2 \frac{\|x\|^2}{R^2}\right) \\
&\quad\quad \times  \mathcal{A}_1^{\operatorname{rot}} (g_{R,j})(x) \frac{c_d dx}{\left(1+\|x\|^2\right)^{\frac{d+1}{2}}}, \\
&= \int_{\mathbb{R}^d} x_j \exp\left(-\| x\|^2\right) \left(1 -2 \|x\|^2\right) \\
&\quad\quad \times  \mathcal{A}_1^{\operatorname{rot}} (g_{1,j})(x) \frac{R^{d+1} c_d dx}{\left(1+R^2 \|x\|^2\right)^{\frac{d+1}{2}}} , \\
& \longrightarrow \int_{\mathbb{R}^d} x_j \exp\left(-\|x\|^2\right) \left(1 -2 \|x\|^2\right) \\
&\quad\quad \times  \mathcal{A}_1^{\operatorname{rot}} (g_{1,j})(x) \frac{c_d dx}{ \|x\|^{d+1}} ,
\end{align*}
as $R$ tends to $+\infty$. Note in particular that the previous integral is well-defined since $ \mathcal{A}_1^{\operatorname{rot}} (g_{1,j})(0) = 0$ and $ \mathcal{A}_1^{\operatorname{rot}} (g_{1,j})$ is smooth on $\mathbb{R}^d$. Thus, for all $j \in \{1, \dots, d\}$, 
\begin{align*}
\underset{R \rightarrow +\infty}{\lim} \dfrac{\| \mathcal{L}^{\operatorname{rot}}_1(g_{R,j}) \|^2_{L^2(\mu_1^{\operatorname{rot}})}} { \|g_{R,j}\|^2_{L^2(\mu_1^{\operatorname{rot}})}} = \underset{R \rightarrow +\infty}{\lim} \dfrac{\|x \cdot \nabla(g_{R,j})\|^2_{L^2(\mu_1^{\operatorname{rot}})}} { \|g_{R,j}\|^2_{L^2(\mu_1^{\operatorname{rot}})}} = \dfrac{\int_0^{+\infty} \left(1 - 2 r^2\right)^2 \exp(-2r^2) dr}{\int_0^{+\infty} \exp \left(-2r^2\right)dr} = \frac{3}{4}.
\end{align*}
This concludes the proof of the proposition. 
\end{proof}
\noindent
To conclude this part regarding the operator $\mathcal{L}_1^{\operatorname{rot}}$, let us compute in dimension $1$,  
\begin{align}\label{eq:first_difference_one_dimension}
\underset{R \longrightarrow + \infty}{\lim} \left(\mathcal{E}^{\nu_1^{\operatorname{rot}}} (g_{R},g_{R}) - \| g_{R} \|^2_{L^2(\mu_1^{\operatorname{rot}})}\right),
\end{align}
where, for all $f_1,f_2 \in \mathcal{C}^1_b(\mathbb{R})$, 
\begin{align}\label{eq:1_stable_rot_invariant_vectorial_form}
\mathcal{E}^{\nu_1^{\operatorname{rot}}} (f_1,f_2) := \int_{\mathbb{R}^{2}}  (f_1(x+u) - f_1(x))(f_2(x+u) - f_2(x))\nu_1^{\operatorname{rot}}(du) \mu_1^{\operatorname{rot}}(dx).
\end{align}
By standard computations, for all $R >0$,
\begin{align*}
\mathcal{E}^{\nu_1^{\operatorname{rot}}} (g_{R},g_{R}) - \langle g_{R} ; g_{R} \rangle_{L^2(\mu_1^{\operatorname{rot}})}
& = \langle g_{R} \left(1 - 4 \frac{|x|^2}{R^2}\right) ; g_{R}  \rangle_{L^2(\mu_1^{\operatorname{rot}})} + 2 \langle (-\mathcal{A}_{1}^{\operatorname{rot}})(g_{R}) ; g_{R} \rangle_{L^2(\mu_1^{\operatorname{rot}})} .
\end{align*}
Then, the following result holds.

\begin{prop}\label{prop:first_difference_one_dimension}
Let $\mu_1^{\operatorname{rot}}$ be the standard one-dimensional Cauchy probability measure on $\mathbb{R}$ with L\'evy measure $\nu_1^{\operatorname{rot}}$. Let $\mathcal{E}^{\nu_1^{\operatorname{rot}}}$ be given by \eqref{eq:1_stable_rot_invariant_vectorial_form} and, for $R>0$, let $g_R$ be defined by \eqref{eq:eigen_approximate_eigenvector}.~Then, 
\begin{align}
\underset{R \longrightarrow + \infty}{\lim} \left(\mathcal{E}^{\nu_1^{\operatorname{rot}}} (g_{R},g_{R}) - \|g_R\|^2_{L^2(\mu_1^{\operatorname{rot}})}\right) = \frac{2}{\pi}. 
\end{align}
\end{prop}

\begin{proof}
First, thanks to Lemma \ref{lem:limit_first_term} and standard computations,
\begin{align*}
\langle g_{R} \left(1 - 4 \frac{| x |^2}{R^2}\right) ; g_{R}  \rangle_{L^2(\mu_1^{\operatorname{rot}})} \longrightarrow -\pi c_1 ,
\end{align*}
as $R$ tends to $+\infty$.~Regarding the non-local term, for all $R>0$, 
\begin{align*}
2  \langle (-\mathcal{A}_{1}^{\operatorname{rot}})(g_{R}) ; g_{R} \rangle_{L^2(\mu_1^{\operatorname{rot}})} & = 2 \int_{\mathbb{R}} g_{R}(x) (-\mathcal{A}_{1}^{\operatorname{rot}})(g_{R})(x) p_1^{\operatorname{rot}}(x)dx , \\
& = 2 \int_{\mathbb{R}} g_{R}(Rx) (-\mathcal{A}_{1}^{\operatorname{rot}})(g_{R})(Rx) p_1^{\operatorname{rot}}(Rx) R dx , \\
& = 2 \int_{\mathbb{R}} R x \exp\left( -x^2\right) (-\mathcal{A}_{1}^{\operatorname{rot}})(g_{1})(x) \dfrac{c_1 R dx}{\left(1+ R^2 x^2\right)}, \\
& \longrightarrow 2 \int_{\mathbb{R}} x \exp\left( - x^2\right) (-\mathcal{A}_{1}^{\operatorname{rot}})(g_{1})(x) \dfrac{c_1 dx}{x^2},
\end{align*}
as $R$ tends to $+ \infty$. Next, let us compute precisely the limit appearing in the previous equation. 
But, first, let us compute the action of the operator $\mathcal{A}_1^{\operatorname{rot}}$ on the function $g_{1}$. Since $g_{1}$ belongs to the Schwartz space, by Fourier inversion formula, for all $x \in \mathbb{R}$
\begin{align*}
(-\mathcal{A}_{1}^{\operatorname{rot}})(g_{1})(x) = \frac{1}{2\pi} \int_{\mathbb{R}} \mathcal{F}(g_{1})(\xi) |\xi| e^{i x \xi} d\xi.
\end{align*}
Now, the Fourier transform of the function $g_{1}$ is given, for all $\xi \in \mathbb{R}$, by 
\begin{align*}
\mathcal{F}(g_{1})(\xi) =  - \frac{(i\xi)}{2} \sqrt{\pi} e^{- \frac{\xi^2}{4}}.
\end{align*}
Thus, for all $x \in \mathbb{R}$, 
\begin{align*}
(-\mathcal{A}_{1}^{\operatorname{rot}})(g_{1})(x) & = -\frac{i}{4\sqrt{\pi}} \int_{\mathbb{R}} \xi e^{- \frac{\xi^2}{4}} |\xi| e^{i x \xi} d\xi , \\
& =  -\frac{i}{4\sqrt{\pi}} \left( \int_{0}^{+\infty} \xi e^{- \frac{\xi^2}{4}} |\xi| e^{i x \xi} d\xi +\int_{-\infty}^0 \xi e^{- \frac{\xi^2}{4}} |\xi| e^{i x \xi} d\xi    \right) , \\
& = -\frac{i}{4\sqrt{\pi}} \left( \int_{0}^{+\infty} \xi e^{- \frac{\xi^2}{4}} |\xi| e^{i x \xi} d\xi  - \int_{0}^{+\infty} \xi e^{- \frac{\xi^2}{4}} |\xi| e^{-i x \xi} d\xi    \right) , \\
& =  \frac{1}{2\sqrt{\pi}} \int_{0}^{+\infty} \xi^2 e^{- \frac{\xi^2}{4}} \sin\left(x\xi\right) d\xi.
\end{align*}
Now, performing two integrations by parts, for all $x \in \mathbb{R}$ with $x \ne 0$, 
\begin{align*}
\int_{0}^{+\infty} \xi^2 e^{- \frac{\xi^2}{4}} \sin\left(x\xi\right) d\xi & = - 2 \int_0^{+\infty} \dfrac{d}{d\xi}\left(e^{- \frac{\xi^2}{4}}\right) \xi \sin\left(x \xi\right) d\xi , \\
& = 2 \int_0^{+\infty} e^{ - \frac{\xi^2}{4}} \dfrac{d}{d\xi}\left(\xi \sin\left(x \xi\right)\right) d\xi , \\
& = 2 \int_0^{+\infty} e^{ - \frac{\xi^2}{4}} \sin\left(x \xi\right) d\xi + 2 x \int_0^{+\infty} e^{ - \frac{\xi^2}{4}} \xi \cos\left(x \xi\right) d\xi , \\
& = 2 \int_0^{+\infty} e^{ - \frac{\xi^2}{4}} \sin\left(x \xi\right) d\xi -4x  \int_0^{+\infty} \dfrac{d}{d\xi} \left( e^{ - \frac{\xi^2}{4}} \right)\cos\left(x \xi\right) d\xi , \\
& = 2 \int_0^{+\infty} e^{ - \frac{\xi^2}{4}} \sin\left(x \xi\right) d\xi  + 4x - 4x^2 \int_0^{+\infty} e^{ - \frac{\xi^2}{4}} \sin \left(x\xi\right) d\xi , \\
& = 4F(x) +4x - 8x^2F(x) , 
\end{align*}
where $F$ is the Dawson's integral function (see \cite[Chapter $7$ page $160$]{Olver_10} for a definition).~Then, for all $x \in \mathbb{R}$,
\begin{align}\label{eq:representation_square_root_g1_dawson}
(-\mathcal{A}_{1}^{\operatorname{rot}})(g_{1})(x) =  \frac{1}{2\sqrt{\pi}} \left(4F(x) +4x - 8x^2F(x)\right). 
\end{align}
Then, for the limiting term, thanks to Lemma \ref{lem:technical_dawson}, 
\begin{align*}
2 \int_{\mathbb{R}} x \exp\left( - x^2\right) (-\mathcal{A}_{1}^{\operatorname{rot}})(g_{1})(x) \dfrac{c_1 dx}{ |x|^{2}} & = \frac{c_1}{\sqrt{\pi}} \int_{\mathbb{R}} x \exp\left( - x^2\right) \left(4F(x) +4x - 8x^2F(x)\right) \frac{dx}{x^2} , \\
& = (2+\pi)c_1.
\end{align*}
This concludes the proof of the proposition. 
\end{proof}
\noindent
Next, let us investigate the spectral properties of the ``carr\'e de Mehler" generator $\mathcal{L}_1$.~The next lemma analyzes the asymptotic pointwise behavior of $\mathcal{L}_1(g_{R,j})$ as $R$ tends to $+\infty$, for all $j \in \{1, \dots, d\}$.

\begin{lem}\label{lem:pointwise_eigenfunction}
Let $d \geq 1$ be an integer, let $\mathcal{L}_1$ be the operator given by \eqref{eq:smart_decomposition} and, for $R>0$, let $g_{R}$ be the function defined by \eqref{eq:eigen_approximate_eigenvector}. 
Then, for all $j \in \{1, \dots, d\}$ and all $x \in \mathbb{R}^d$, 
\begin{align}\label{eq:weak_spectral}
(\mathcal{L}_1)(g_{R,j})(x) \longrightarrow -x_j,
\end{align}
as $R$ tends to $+\infty$. 
\end{lem}

\begin{proof}
First, for all $j \in \{1, \dots, d\}$ and all $x \in \mathbb{R}^d$, 
\begin{align*}
\mathcal{A}_1^{\operatorname{rot}}(g_{R,j})(x) \longrightarrow 0,
\end{align*}
as $R$ tends to $+\infty$. So, let us prove that, for all $j \in \{1, \dots, d\}$ and all $x \in \mathbb{R}^d$, 
\begin{align*}
\frac{1}{p_1^{\operatorname{rot}}(x)} \Gamma_1^{\operatorname{rot}}(p_1^{\operatorname{rot}},g_{R,j})(x) \longrightarrow -x_j,
\end{align*}
as $R$ tends to $+\infty$. By Fourier inversion formula, for all $j \in \{1, \dots, d\}$ and all $x \in \mathbb{R}^d$, 
\begin{align*}
\Gamma_1^{\operatorname{rot}}(p_1^{\operatorname{rot}},g_{R,j})(x) & = \frac{1}{(2\pi)^{2d}} \int_{\mathbb{R}^{2d}} R^{d+1} \mathcal{F}(g_{1,j})(R \xi_1) \mathcal{F}(p^{\operatorname{rot}}_1)(\xi_2) e^{i \langle x ; \xi_1 + \xi_2 \rangle} \\
& \quad\quad \times \left(- \|\xi_1+ \xi_2\| + \|\xi_1\| +\|\xi_2\| \right) d\xi_1 d\xi_2 , \\
& = \frac{1}{(2\pi)^{2d}} \int_{\mathbb{R}^{2d}} \mathcal{F}(g_{1,j})(\xi_1) \mathcal{F}(p^{\operatorname{rot}}_1)(\xi_2) e^{i \langle x ; \frac{\xi_1}{R} + \xi_2 \rangle}\\
&\quad\quad \times R \left(- \| \frac{\xi_1}{R}+ \xi_2\| + \frac{\|\xi_1\|}{R} +\|\xi_2\| \right) d\xi_1 d\xi_2.
\end{align*}
Now, for all $\xi_1, \xi_2 \in \mathbb{R}^d \setminus \{0\}$, 
\begin{align*}
R \left(- \| \frac{\xi_1}{R}+ \xi_2\| + \frac{\|\xi_1\|}{R} +\|\xi_2\| \right) \longrightarrow \|\xi_1\| - \langle \frac{\xi_2}{\|\xi_2\|} ; \xi_1\rangle,
\end{align*}
as $R$ tends to $+\infty$. Thus, for all $j \in \{1, \dots, d\}$ and all $x \in \mathbb{R}^d$,
\begin{align*}
\Gamma_1^{\operatorname{rot}}(p_1^{\operatorname{rot}},g_{R,j})(x) \longrightarrow  \frac{1}{(2\pi)^{2d}} \int_{\mathbb{R}^{2d}} \mathcal{F}(g_{1,j})(\xi_1) \mathcal{F}(p^{\operatorname{rot}}_1)(\xi_2) e^{i \langle x ; \xi_2 \rangle} \left(\|\xi_1\| - \langle \frac{\xi_2}{\|\xi_2\|} ; \xi_1\rangle \right) d\xi_1 d\xi_2. 
\end{align*}
Now, by antisymmetry, 
\begin{align*}
\int_{\mathbb{R}^d} \mathcal{F}(g_{1,j})(\xi_1) \|\xi_1\| d\xi_1 = 0. 
\end{align*}
Finally, by integration by parts, for all $x \in \mathbb{R}^d$ and all $j \in \{1 , \dots, d \}$, 
\begin{align*}
\frac{1}{(2\pi)^{2d}}&  \int_{\mathbb{R}^{2d}} \mathcal{F}(g_{1,j})(\xi_1) \mathcal{F}(p^{\operatorname{rot}}_1)(\xi_2) e^{i \langle x ; \xi_2 \rangle} \langle \frac{i \xi_2}{\|\xi_2\|} ; i \xi_1\rangle d\xi_1 d\xi_2 , \\
& = - \langle \nabla(g_{1,j})(0) ; x \rangle p_1^{\operatorname{rot}}(x) = -x_j p_1^{\operatorname{rot}}(x),
\end{align*}
where we have used the fact that, for all $\xi \in \mathbb{R}^d$ with $\xi \ne 0$, 
\begin{align*}
\nabla_{\xi} \left(\mathcal{F}(p_{1}^{\operatorname{rot}})\right)(\xi) = - \frac{\xi}{\| \xi \|} \mathcal{F}(p_{1}^{\operatorname{rot}})(\xi). 
\end{align*}
This concludes the proof of the lemma.  
\end{proof}
\noindent
As previously, let us investigate the $L^2$-spectral properties of the operator $\mathcal{L}_1$ and, in particular, let us locate in the spectrum the negative real $-1$.
For this purpose, let us compute the following limit (if it exists): for all $j \in \{1, \dots, d\}$,
\begin{align*}
\tilde{L}_j := \underset{R \longrightarrow +\infty}{\lim}\dfrac{\| \mathcal{L}_1 \left(g_{R,j}\right) + g_{R,j} \|^2_{L^2(\mu_1^{\operatorname{rot}})}}{\| g_{R,j} \|^2_{L^2(\mu_1^{\operatorname{rot}})}}. 
\end{align*}

\begin{prop}\label{prop:approximate_point_spectrum}
Let $d \geq 1$ be an integer, let $\mathcal{L}_1$ be the operator given by \eqref{eq:smart_decomposition} and, for $R>0$, let $g_{R}$ be the function defined by \eqref{eq:eigen_approximate_eigenvector}. 
Then, for all $j \in \{1, \dots, d\}$,
\begin{align*}
 \underset{R \longrightarrow +\infty}{\lim}\dfrac{\| \mathcal{L}_1 \left(g_{R,j}\right) + g_{R,j} \|^2_{L^2(\mu_1^{\operatorname{rot}})}}{\| g_{R,j} \|^2_{L^2(\mu_1^{\operatorname{rot}})}} = 0.
\end{align*}
\end{prop}

\begin{proof}
By straightforward computations and integration by parts, 
\begin{align*}
\| \mathcal{L}_1 \left(g_{R,j}\right) + g_{R,j} \|^2_{L^2(\mu_1^{\operatorname{rot}})} & = \| \mathcal{L}_1(g_{R,j}) \|^2_{L^2(\mu_1^{\operatorname{rot}})} + \| g_{R,j} \|^2_{L^2(\mu_1^{\operatorname{rot}})} + 2 \langle \mathcal{L}_1(g_{R,j}) ; g_{R,j}  \rangle_{L^2(\mu_1^{\operatorname{rot}})} , \\
& =  \| \mathcal{L}_1(g_{R,j}) \|^2_{L^2(\mu_1^{\operatorname{rot}})} - \mathcal{E}_{\nu_1^{\operatorname{rot}}, \mu_1^{\operatorname{rot}}}(g_{R,j} , g_{R,j}) \\
&\quad + \| g_{R,j} \|^2_{L^2(\mu_1^{\operatorname{rot}})}  - \mathcal{E}_{\nu_1^{\operatorname{rot}}, \mu_1^{\operatorname{rot}}}(g_{R,j} , g_{R,j}).
\end{align*}
But, for all $j \in \{ 1, \dots, d\}$, 
\begin{align}\label{eq:limit_ratio_1}
 \underset{R \longrightarrow +\infty}{\lim} \dfrac{\mathcal{E}_{\nu_1^{\operatorname{rot}}, \mu_1^{\operatorname{rot}}}(g_{R,j} , g_{R,j})}{\|g_{R,j}\|^2_{L^2(\mu_1^{\operatorname{rot}})}} = 1. 
\end{align}
So, let us compute the following limit: for all $j \in \{1, \dots, d\}$, 
\begin{align*}
\underset{R \longrightarrow +\infty}{\lim} \dfrac{\| \mathcal{L}_1(g_{R,j}) \|_{L^{2}(\mu_1^{\operatorname{rot}})}^2}{\| g_{R,j} \|^2_{L^2(\mu_1^{\operatorname{rot}})}}. 
\end{align*}
Thanks to \eqref{eq:smart_decomposition} and by standard computations, 
\begin{align*}
\| \mathcal{L}_1(g_{R,j}) \|^2_{L^{2}(\mu_1^{\operatorname{rot}})} & = 4 \| \mathcal{A}^{\operatorname{rot}}_1(g_{R,j}) \|^2_{L^2(\mu_1^{\operatorname{rot}})} + 4 \langle \mathcal{A}_1^{\operatorname{rot}}(g_{R,j})  ; \frac{1}{p_1^{\operatorname{rot}}}\Gamma^{\operatorname{rot}}_1\left( p_1^{\operatorname{rot}}; g_{R,j} \right) \rangle_{L^2(\mu_1^{\operatorname{rot}})} \\
&\quad\quad + \left\| \frac{1}{p_1^{\operatorname{rot}}}\Gamma^{\operatorname{rot}}_1\left( p_1^{\operatorname{rot}}; g_{R,j} \right)  \right\|^2_{L^2(\mu_1^{\operatorname{rot}})}. 
\end{align*}
Now, by the Lebesgue dominated convergence theorem, for all $j \in \{1, \dots, d\}$, 
\begin{align*}
\underset{R \longrightarrow +\infty}{\lim}  \| \mathcal{A}^{\operatorname{rot}}_1(g_{R,j}) \|^2_{L^2(\mu_1^{\operatorname{rot}})} = 0. 
\end{align*}
Moreover, by Parseval-Plancherel formula, 
\begin{align*}
\left\langle \mathcal{A}_1^{\operatorname{rot}}(g_{R,j})  ; \frac{1}{p_1^{\operatorname{rot}}}\Gamma^{\operatorname{rot}}_1\left( p_1^{\operatorname{rot}}; g_{R,j} \right) \right\rangle_{L^2(\mu_1^{\operatorname{rot}})} & = \left\langle \mathcal{A}_1^{\operatorname{rot}}(g_{R,j})  ; \Gamma^{\operatorname{rot}}_1\left( p_1^{\operatorname{rot}}; g_{R,j} \right) \right\rangle_{L^2(\mathbb{R}^d,dx)} , \\
&= \frac{1}{(2\pi)^d} \int_{\mathbb{R}^d} \mathcal{F}\left( \mathcal{A}_1^{\operatorname{rot}}(g_{R,j}) \right)(\xi) \\
& \quad \times \overline{ \mathcal{F}\left( \Gamma^{\operatorname{rot}}_1\left( p_1^{\operatorname{rot}}; g_{R,j} \right) \right)(\xi)} d\xi. 
\end{align*}
Let us compute the Fourier transform of the function $\Gamma^{\operatorname{rot}}_1\left( p_1^{\operatorname{rot}}; g_{R,j} \right)$.  First, for all $x \in \mathbb{R}^d$, all $R>0$ and all $j \in \{1, \dots, d\}$, 
\begin{align*}
\Gamma^{\operatorname{rot}}_1\left( p_1^{\operatorname{rot}}; g_{R,j} \right)(x) & = \frac{1}{(2\pi)^{2d}} \int_{\mathbb{R}^{2d}} \mathcal{F}(g_{R,j})(\xi_1) \mathcal{F}(p^{\operatorname{rot}}_1)(\xi_2)e^{i \langle x ; \xi_1 + \xi_2\rangle} \\
& \quad \times \left(- \|\xi_1+ \xi_2\| + \|\xi_1\| + \|\xi_2\|\right)d\xi_1 d\xi_2.  
\end{align*}
Then, by Fourier inversion formula, for all $\xi_1 \in \mathbb{R}^d$, all $R>0$ and all $j \in \{1, \dots, d\}$, 
\begin{align*}
\mathcal{F}\left(\Gamma^{\operatorname{rot}}_1\left( p_1^{\operatorname{rot}}; g_{R,j} \right)\right)(\xi_1) & = \frac{1}{(2\pi)^d} \int_{\mathbb{R}^d} \mathcal{F}(g_{R,j})(\xi_1-\xi_2) \mathcal{F}(p^{\operatorname{rot}}_1)(\xi_2) \left(- \|\xi_1\| + \|\xi_1-\xi_2\| + \|\xi_2\|\right) d\xi_2.
\end{align*}
Then, 
\begin{align}\label{eq:comp_crossed_term}
\left\langle \mathcal{A}_1^{\operatorname{rot}}(g_{R,j})  ; \frac{1}{p_1^{\operatorname{rot}}}\Gamma^{\operatorname{rot}}_1\left( p_1^{\operatorname{rot}}; g_{R,j} \right) \right\rangle_{L^2(\mu_1^{\operatorname{rot}})} &  = \frac{1}{(2\pi)^d} \int_{\mathbb{R}^d} \mathcal{F}(g_{R,j})(\omega) \|\omega\| \mathcal{F} \left(\Gamma_1^{\operatorname{rot}}(p_1^{\operatorname{rot}} ; g_{R,j})\right)(\omega) d\omega ,  \nonumber \\
& = \frac{R^{d+1}}{(2\pi)^d} \int_{\mathbb{R}^d} \mathcal{F}(g_{1,j})(R\omega) \|\omega\| \mathcal{F}\left(\Gamma_1^{\operatorname{rot}}(p_1^{\operatorname{rot}} ; g_{R,j})\right)(\omega) d\omega ,  \nonumber \\
& = \frac{1}{(2\pi)^d} \int_{\mathbb{R}^d} \mathcal{F}(g_{1,j})(\omega) \|\omega\| \mathcal{F}\left(\Gamma_1^{\operatorname{rot}}(p_1^{\operatorname{rot}} ; g_{R,j})\right)\left(\frac{\omega}{R}\right) d\omega.
\end{align}
Now, for all $\omega \in \mathbb{R}^d$, all $R>0$ and all $j \in \{1, \dots, d\}$, 
\begin{align*}
\mathcal{F}\left(\Gamma_1^{\operatorname{rot}}(p_1^{\operatorname{rot}} ; g_{R,j})\right)\left(\frac{\omega}{R}\right) & = \frac{1}{(2\pi)^d} \int_{\mathbb{R}^d} \mathcal{F}(g_{R,j})\left(\frac{\omega}{R}-\xi_2\right) \mathcal{F}(p^{\operatorname{rot}}_1)(\xi_2) \\
& \quad \times \left(- \left\|\frac{\omega}{R}\right\| + \left\|\frac{\omega}{R}-\xi_2\right\| + \|\xi_2\|\right) d\xi_2 , \\
& = \frac{R^{d+1}}{(2\pi)^d} \int_{\mathbb{R}^d} \mathcal{F}(g_{1,j})\left(\omega-R\xi_2\right) \mathcal{F}(p^{\operatorname{rot}}_1)(\xi_2) \\
&\quad \times \left(- \left\|\frac{\omega}{R}\right\| + \left\|\frac{\omega}{R}-\xi_2\right\| + \|\xi_2\|\right) d\xi_2 , \\
& = \frac{1}{(2\pi)^d} \int_{\mathbb{R}^d} \mathcal{F}(g_{1,j})\left(\omega-\xi_2\right) \mathcal{F}(p^{\operatorname{rot}}_1)\left(\frac{\xi_2}{R}\right) \\
&\quad \times \left(- \left\|\omega\right\| + \left\|\omega-\xi_2\right\| + \|\xi_2\|\right) d\xi_2 , \\
& \longrightarrow \frac{1}{(2\pi)^d} \int_{\mathbb{R}^d} \mathcal{F}(g_{1,j})\left(\omega-\xi_2\right) \left(- \left\|\omega\right\| + \left\|\omega-\xi_2\right\| + \|\xi_2\|\right) d\xi_2 , 
\end{align*}
as $R$ tends to $+\infty$.  Then, passing to the limit in \eqref{eq:comp_crossed_term}, 
\begin{align}\label{eq:limit_crossed_term}
\left\langle \mathcal{A}_1^{\operatorname{rot}}(g_{R,j})  ; \frac{1}{p_1^{\operatorname{rot}}}\Gamma^{\operatorname{rot}}_1\left( p_1^{\operatorname{rot}}; g_{R,j} \right) \right\rangle_{L^2(\mu_1^{\operatorname{rot}})} &  \longrightarrow \frac{1}{(2\pi)^{2d}} \int_{\mathbb{R}^{2d}} \mathcal{F}(g_{1,j})(\omega) \|\omega\| \mathcal{F}(g_{1,j})(\omega - \xi_2) \nonumber \\
&\quad\quad \times \left( - \|\omega\| + \|\omega - \xi_2\| + \|\xi_2\|\right) d\xi_2d\omega. 
\end{align}
Let us compute this limit.  First, 
\begin{align*}
\int_{\mathbb{R}^{2d}} \mathcal{F}(g_{1,j})(\omega) \|\omega\|^2 \mathcal{F}(g_{1,j})(\omega - \xi_2) d\xi_2d\omega  = \left(\int_{\mathbb{R}^d} \mathcal{F}(g_{1,j})(\omega) \|\omega\|^2 d\omega \right) \left(\int_{\mathbb{R}^d} \mathcal{F}(g_{1,j})(\xi_2) d\xi_2 \right) = 0.
\end{align*}
Moreover,  
\begin{align*}
\int_{\mathbb{R}^{2d}} \mathcal{F}(g_{1,j})(\omega) \|\omega\| \mathcal{F}(g_{1,j})(\omega - \xi_2) \|\omega - \xi_2\| d\xi_2d\omega = \left( \int_{\mathbb{R}^d} \mathcal{F}(g_{1,j})(\omega) \|\omega\| d \omega \right)^2  = 0.
\end{align*}
Finally, since the cross term is normalized by $\|g_{R,j}\|^2_{L^2(\mu_1^{\operatorname{rot}})}$, the last term in \eqref{eq:limit_crossed_term} does not contribute.  Thus, 
\begin{align*}
 \frac{1}{\|g_{R,j}\|^2_{L^2(\mu_1^{\operatorname{rot}})}}\left\langle \mathcal{A}_1^{\operatorname{rot}}(g_{R,j})  ; \frac{1}{p_1^{\operatorname{rot}}}\Gamma^{\operatorname{rot}}_1\left( p_1^{\operatorname{rot}}; g_{R,j} \right) \right\rangle_{L^2(\mu_1^{\operatorname{rot}})} \longrightarrow 0,
\end{align*}
as $R$ tends to $+\infty$. Now, by scale invariance of $\nu_1^{\operatorname{rot}}$, for all $x \in \mathbb{R}^d$ and all $R>0$, 
\begin{align*}
\Gamma^{\operatorname{rot}}_1\left( p_1^{\operatorname{rot}}; g_{R,j} \right)(x) = \Gamma^{\operatorname{rot}}_1\left( p_{1,R}^{\operatorname{rot}}; g_{1,j} \right)\left(\frac{x}{R}\right),
\end{align*}
where $p_{1,R}^{\operatorname{rot}}(x) = p_{1}^{\operatorname{rot}}(Rx)$, $x \in \mathbb{R}^d$. Thus, 
\begin{align*}
\left\| \dfrac{\Gamma^{\operatorname{rot}}_1\left( p_1^{\operatorname{rot}}; g_{R,j} \right)}{p_1^{\operatorname{rot}}}  \right\|^2_{L^2(\mu_1^{\operatorname{rot}})} & = \frac{1}{c_d} \int_{\mathbb{R}^d} \left(1+ \|x\|^2\right)^{\frac{d+1}{2}} \left|\Gamma^{\operatorname{rot}}_1\left( p_1^{\operatorname{rot}}; g_{R,j} \right)(x)\right|^2dx , \\
& = \frac{1}{c_d} \int_{\mathbb{R}^d} \left(1+ \|x\|^2\right)^{\frac{d+1}{2}} \left| \Gamma^{\operatorname{rot}}_1\left( p_{1,R}^{\operatorname{rot}}; g_{1,j} \right)\left(\frac{x}{R}\right)\right|^2dx , \\
& = \frac{R^d}{c_d} \int_{\mathbb{R}^d} \left(1+ R^2\|x\|^2\right)^{\frac{d+1}{2}} \left| \Gamma^{\operatorname{rot}}_1\left( p_{1,R}^{\operatorname{rot}}; g_{1,j} \right)\left(x\right)\right|^2dx.
\end{align*}
Next, let us prove, for all $x \in \mathbb{R}^d$ with $x \ne 0$ and all $j \in \{1, \dots, d\}$, that
\begin{align}\label{eq:limit_pointwise_renormalisation_all_dimensions}
\underset{R \longrightarrow +\infty}{\lim} R^d \Gamma_1^{\operatorname{rot}}\left(p^{\operatorname{rot}}_{1,R} ; g_{1,j}\right)(x) =- \frac{c_d}{\|x\|^{d+1}} g_{1,j}(x),
\end{align}
where $c_d$ is given by \eqref{eq:def_rot_inv}. Recall that, for all $x \in \mathbb{R}^d$, all $R>0$ and all $j \in \{1, \dots, d\}$, 
\begin{align*}
\Gamma_1^{\operatorname{rot}}\left(p^{\operatorname{rot}}_{1,R} ; g_{1,j}\right)(x) & = \frac{1}{(2\pi)^{2d}} \int_{\mathbb{R}^{2d}} \mathcal{F}\left(p^{\operatorname{rot}}_{1,R}\right)(\xi_1) \mathcal{F}\left(g_{1,j}\right)(\xi_2) e^{i \langle x ; \xi_1 + \xi_2 \rangle } \\
&\quad\quad \times\left( - \|\xi_1 + \xi_2\| + \|\xi_1\| +\|\xi_2\|\right) d\xi_1 d\xi_2 , \\
& = \frac{1}{(2\pi)^{2d}} \int_{\mathbb{R}^{2d}} \frac{1}{R^{d}} \mathcal{F}\left(p_1^{\operatorname{rot}}\right)\left(\frac{\xi_1}{R}\right)\mathcal{F}\left(g_{1,j}\right)(\xi_2) e^{i \langle x ; \xi_1 + \xi_2 \rangle } \\
&\quad\quad \times\left( - \|\xi_1 + \xi_2\| + \|\xi_1\| +\|\xi_2\|\right) d\xi_1 d\xi_2. 
\end{align*}
Thus, for all $R>0$, all $x \in \mathbb{R}^d$ and all $j \in \{1, \dots, d\}$, 
\begin{align*}
R^d\Gamma_1^{\operatorname{rot}}\left(p^{\operatorname{rot}}_{1,R} ; g_{1,j}\right)(x) & = \frac{1}{(2\pi)^{2d}} \int_{\mathbb{R}^{2d}} \exp \left(- \frac{\|\xi_1\|}{R}\right)\mathcal{F}\left(g_{1,j}\right)(\xi_2) e^{i \langle x ; \xi_1 + \xi_2 \rangle } \\
& \quad\quad \times \left( - \|\xi_1 + \xi_2\| + \|\xi_1\| +\|\xi_2\|\right) d\xi_1 d\xi_2, \\
& = \frac{1}{(2\pi)^{2d}} \int_{\mathbb{R}^d} \mathcal{F}\left(g_{1,j}\right)(\xi_2) e^{i \langle x ; \xi_2 \rangle} \bigg( \int_{\mathbb{R}^d} e^{- \frac{\|\xi_1\|}{R}} e^{i \langle \xi_1 ; x \rangle } \\
&\quad\quad \times \left( - \|\xi_1+\xi_2\| + \|\xi_1\| + \|\xi_2\|\right) d\xi_1 \bigg) d\xi_2.
\end{align*}
First, by Fourier inversion formula, for all $x \in \mathbb{R}^d$ such that $x \ne 0$, all $R>0$ and all $j \in \{1, \dots, d\}$, 
\begin{align*}
\frac{1}{(2\pi)^{2d}} \int_{\mathbb{R}^{2d}} & \mathcal{F}\left(g_{1,j}\right)(\xi_2)  \|\xi_2\| e^{i \langle x ; \xi_2 \rangle} \bigg( \int_{\mathbb{R}^d} e^{- \frac{\|\xi_1\|}{R}} e^{i \langle \xi_1 ; x \rangle }d\xi_1 \bigg) d\xi_2 = \\
& \frac{1}{(2\pi)^d} \int_{\mathbb{R}^d}  \mathcal{F}\left(g_{1,j}\right)(\xi_2)  \|\xi_2\| e^{i \langle x ; \xi_2 \rangle}  d\xi_2 \dfrac{c_d R^d}{\left(1+ R^2 \|x\|^2\right)^{\frac{d+1}{2}}} , \\
& \longrightarrow 0,
\end{align*}
as $R$ tends to $+\infty$. Moreover, thanks to Lemma \ref{lem:riesz_potential_critical} of the Appendix section, for all $x \in \mathbb{R}^d$ such that $x \ne 0$,  
\begin{align*}
\underset{R \longrightarrow +\infty}{\lim} \int_{\mathbb{R}^d} e^{- \frac{\|\xi_1\|}{R}} e^{i \langle x ; \xi_1 \rangle } \|\xi_1\| \frac{d\xi_1}{(2\pi)^d} = \dfrac{-c_d}{\|x\|^{d+1}}.
\end{align*}
Then, 
\begin{align*}
\underset{R \longrightarrow +\infty}{\lim} \frac{1}{(2\pi)^{2d}} \int_{\mathbb{R}^d} \mathcal{F}\left(g_{1,j}\right)(\xi_2) e^{i \langle x ; \xi_2 \rangle} \bigg( \int_{\mathbb{R}^d} e^{- \frac{\|\xi_1\|}{R}} e^{i \langle \xi_1 ; x \rangle } \|\xi_1\| d\xi_1 \bigg) d\xi_2 = \dfrac{-c_d g_{1,j}(x)}{\|x\|^{d+1}}.
\end{align*}
Finally, using tempered distributions, for all $x \in \mathbb{R}^d$ such that $x \ne 0$ and all $\xi_2 \in \mathbb{R}^d$, 
\begin{align*}
\mathcal{F}^{-1}\left( \| \cdot+\xi_2\| \right)(x) = - \frac{c_d e^{- i \langle x ; \xi_2 \rangle }}{\|x\|^{d+1}}.
\end{align*} 
Moreover, by antisymmetry,
\begin{align*}
\int_{\mathbb{R}^d} \mathcal{F}\left(g_{1,j}\right)(\xi_2) d\xi_2 = 0.
\end{align*}
Thus, for all $x \in \mathbb{R}^d$ such that $x \ne 0$, 
\begin{align*}
R^d \Gamma_1^{\operatorname{rot}}\left(p^{\operatorname{rot}}_{1,R} ; g_{1,j}\right)(x) & \longrightarrow -\dfrac{c_d g_{1,j}(x)}{\|x\|^{d+1}},
\end{align*}
as $R$ tends to $+\infty$. Then, 
\begin{align*}
\underset{R \longrightarrow +\infty}{\lim} \frac{1}{R} \frac{R^d}{c_d} \int_{\mathbb{R}^d} \left(1+ R^2\|x\|^2\right)^{\frac{d+1}{2}} \left| \Gamma^{\operatorname{rot}}_1\left( p_{1,R}^{\operatorname{rot}}; g_{1,j} \right)\left(x\right)\right|^2dx & = \int_{\mathbb{R}^d} x_j^2 \exp\left(-2 \|x\|^2\right) \frac{c_d dx}{\|x\|^{d+1}}.
\end{align*}
This concludes the proof of the proposition. 
\end{proof}
\noindent
The previous result implies that the negative real $-1$ belongs to the approximate point spectrum of $\mathcal{L}_1$. Based on the previous analysis, it is legitimate to introduce the following functional (informally) defined on the set of probability measure on $\mathbb{R}^d$: for all $\mu \in \mathcal{M}_1(\mathbb{R}^d)$, 
\begin{align}\label{eq:def_new_functional}
U_1(\mu) = \underset{f\in \mathcal{S}(\mathbb{R}^d) \setminus \{0\},\, \int_{\mathbb{R}^d} f(x) \mu(dx) = 0}{\sup} \dfrac{\|f\|_{L^2(\mu)}}{\|\mathcal{L}_1(f)\|_{L^2(\mu)}}. 
\end{align}
First, the following technical lemma holds true.

\begin{lem}\label{lem:U_equal_1}
Let $d \geq 1$ be an integer, let $\mu_1^{\operatorname{rot}}$ be the standard Cauchy probability measure on $\mathbb{R}^d$, let $\mathcal{L}_1$ be the operator given by \eqref{eq:smart_decomposition} and let $U_1$ be the functional defined by \eqref{eq:def_new_functional}. Then, 
\begin{align}
U_1(\mu_1^{\operatorname{rot}}) = 1. 
\end{align}
\end{lem}

\begin{proof}
Thanks to Poincar\'e's inequality for the standard Cauchy probability measure, for all $f \in D(\mathcal{L}_1) \setminus \{0\}$ such that $\int_{\mathbb{R}^d} f(x) \mu_1^{\operatorname{rot}}(dx) = 0$, 
\begin{align*}
\int_{\mathbb{R}^d} f(x)^2 \mu_1^{\operatorname{rot}}(dx) \leq \mathcal{E}_{\nu_1^{\operatorname{rot}}, \mu_1^{\operatorname{rot}}}(f,f)  = \langle f ; (- \mathcal{L}_1)(f) \rangle_{L^2(\mu_1^{\operatorname{rot}})}.
\end{align*}
Now, by Cauchy-Schwarz inequality,
\begin{align*}
\|f\|^2_{L^2(\mu_1^{\operatorname{rot}})}  \leq \|f\|_{L^2(\mu_1^{\operatorname{rot}})} \|\mathcal{L}_1\left(f\right)\|_{L^2(\mu_1^{\operatorname{rot}})},
\end{align*} 
which implies that $U_1(\mu_1^{\operatorname{rot}}) \leq 1$. Moreover, since, for all $R>0$ and all $j \in \{ 1, \dots, d\}$, $g_{R,j} \in \mathcal{S}(\mathbb{R}^d)$ with $\int_{\mathbb{R}^d} g_{R,j}(x) \mu^{\operatorname{rot}}_1(dx) = 0$,
\begin{align*}
U_1(\mu_1^{\operatorname{rot}}) \geq  \dfrac{\| g_{R,j} \|_{L^2(\mu_1^{\operatorname{rot}})}}{\|\mathcal{L}_1(g_{R,j})\|_{L^2(\mu_1^{\operatorname{rot}})}}. 
\end{align*}
But, in the proof of Lemma \ref{prop:approximate_point_spectrum}, it is proved that 
\begin{align*}
\underset{R \longrightarrow +\infty}{\lim} \dfrac{\| g_{R,j} \|_{L^2(\mu_1^{\operatorname{rot}})}}{\|\mathcal{L}_1(g_{R,j})\|_{L^2(\mu_1^{\operatorname{rot}})}} = 1. 
\end{align*}
Thus, $U_1(\mu_1^{\operatorname{rot}}) \geq 1$. This concludes the proof of the lemma. 
\end{proof}
\noindent
To end this subsection, let us compute the following limit (in dimension one):
\begin{align*}
\underset{R \longrightarrow +\infty}{\lim} \left(\|g_{R}\|^2_{L^2(\mu_1^{\operatorname{rot}})} - \| \mathcal{L}_1(g_{R})\|^2_{L^2(\mu_1^{\operatorname{rot}})}\right). 
\end{align*}

\begin{prop}\label{prop:diff_limit_carre_mehler_1D}
Let $\mu_1^{\operatorname{rot}}$ be the standard Cauchy probability measure on $\mathbb{R}$, let $\mathcal{L}_1$ be the operator given by \eqref{eq:smart_decomposition} and, for $R>0$, let $g_R$ be the function defined by \eqref{eq:eigen_approximate_eigenvector}.~Then, 
\begin{align*}
\underset{R \longrightarrow +\infty}{\lim}  \left(\|g_{R}\|^2_{L^2(\mu_1^{\operatorname{rot}})} - \| \mathcal{L}_1(g_{R})\|^2_{L^2(\mu_1^{\operatorname{rot}})}\right) = - \frac{4}{\pi}.
\end{align*}
\end{prop}

\begin{proof}
Thanks to the proof of Lemma \ref{prop:approximate_point_spectrum}, for all $R>0$, 
\begin{align*}
\| \mathcal{L}_1(g_{R}) \|^2_{L^2(\mu_1^{\operatorname{rot}})} = & = 4 \| \mathcal{A}^{\operatorname{rot}}_1(g_{R}) \|^2_{L^2(\mu_1^{\operatorname{rot}})} + 4 \langle \mathcal{A}_1^{\operatorname{rot}}(g_{R})  ; \frac{1}{p_1^{\operatorname{rot}}}\Gamma^{\operatorname{rot}}_1\left( p_1^{\operatorname{rot}}; g_{R} \right) \rangle_{L^2(\mu_1^{\operatorname{rot}})} \\
&\quad\quad + \left\| \frac{1}{p_1^{\operatorname{rot}}}\Gamma^{\operatorname{rot}}_1\left( p_1^{\operatorname{rot}}; g_{R} \right)  \right\|^2_{L^2(\mu_1^{\operatorname{rot}})}. 
\end{align*}
Moreover, still from the proof of Lemma \ref{prop:approximate_point_spectrum},
\begin{align*}
 \underset{R\longrightarrow +\infty}{\lim} \| \mathcal{A}^{\operatorname{rot}}_1(g_{R}) \|^2_{L^2(\mu_1^{\operatorname{rot}})} &  = 0 , \\
 \underset{R\longrightarrow +\infty}{\lim} \left\langle \mathcal{A}_1^{\operatorname{rot}}(g_{R})  ; \frac{1}{p_1^{\operatorname{rot}}}\Gamma^{\operatorname{rot}}_1\left( p_1^{\operatorname{rot}}; g_{R} \right) \right\rangle_{L^2(\mu_1^{\operatorname{rot}})}  & = \frac{1}{(2\pi)^{2}} \int_{\mathbb{R}^{2}} \mathcal{F}(g_{1})(\omega) |\omega| \mathcal{F}(g_{1})(\omega - \xi_2) \nonumber \\
&\quad\quad \times \left( - |\omega| + |\omega - \xi_2| + |\xi_2|\right) d\xi_2d\omega. 
\end{align*}
Then, let us deal first with the following limit:
\begin{align*}
\underset{R \longrightarrow +\infty}{\lim} \left(\|g_{R}\|^2_{L^2(\mu_1^{\operatorname{rot}})} - \left\| \frac{1}{p_1^{\operatorname{rot}}}\Gamma^{\operatorname{rot}}_1\left( p_1^{\operatorname{rot}}; g_{R} \right)  \right\|^2_{L^2(\mu_1^{\operatorname{rot}})} \right).
\end{align*}
Thanks to \eqref{eq:definition_squared_field_operator}, 
\begin{align*}
\Gamma_1^{\operatorname{rot}}(f,g)(x) = \int_{\mathbb{R}^d} \nu_1^{\operatorname{rot}}(du) \Delta_u(f)(x) \Delta_u(g)(x) = \mathcal{A}_{1}^{\operatorname{rot}}(gf)(x) - g(x)\mathcal{A}_1^{\operatorname{rot}}(f)(x)- f(x)\mathcal{A}_1^{\operatorname{rot}}(g)(x). 
\end{align*} 
Thus, for all $x \in \mathbb{R}$ and all $R>0$, 
\begin{align*}
\left| \Gamma_1^{\operatorname{rot}}(p_1^{\operatorname{rot}},g_R)(x) \right|^2 & = \left( \mathcal{A}_{1}^{\operatorname{rot}}(g_Rp_1^{\operatorname{rot}})(x) - g_R(x)\mathcal{A}_1^{\operatorname{rot}}(p_1^{\operatorname{rot}})(x)- p_1^{\operatorname{rot}}(x)\mathcal{A}_1^{\operatorname{rot}}(g_R)(x)\right)^2 , \\
& = \mathcal{A}_{1}^{\operatorname{rot}}(g_Rp_1^{\operatorname{rot}})(x)^2 + \left(g_R(x)\mathcal{A}_1^{\operatorname{rot}}(p_1^{\operatorname{rot}})(x)\right)^2 + \left(p_1^{\operatorname{rot}}(x)\mathcal{A}_1^{\operatorname{rot}}(g_R)(x)\right)^2 \\
&\quad\quad + 2 p_1^{\operatorname{rot}}(x)\mathcal{A}_1^{\operatorname{rot}}(g_R)(x)g_R(x)\mathcal{A}_1^{\operatorname{rot}}(p_1^{\operatorname{rot}})(x) \\
&\quad\quad - 2p_1^{\operatorname{rot}}(x)\mathcal{A}_1^{\operatorname{rot}}(g_R)(x)\mathcal{A}_{1}^{\operatorname{rot}}(g_Rp_1^{\operatorname{rot}})(x) \\
&\quad\quad - 2g_R(x)\mathcal{A}_1^{\operatorname{rot}}(p_1^{\operatorname{rot}})(x)\mathcal{A}_{1}^{\operatorname{rot}}(g_Rp_1^{\operatorname{rot}})(x). 
\end{align*}
Then, for all $R>0$, 
\begin{align*}
\left\| \frac{1}{p_1^{\operatorname{rot}}}\Gamma^{\operatorname{rot}}_1\left( p_1^{\operatorname{rot}}; g_{R} \right)  \right\|^2_{L^2(\mu_1^{\operatorname{rot}})} & = \int_{\mathbb{R}} \frac{1}{p_1^{\operatorname{rot}}(x)} \left| \Gamma_1^{\operatorname{rot}}(p_1^{\operatorname{rot}},g_R)(x) \right|^2 dx , \\
& = \int_{\mathbb{R}} \frac{1}{p_1^{\operatorname{rot}}(x)} \mathcal{A}_{1}^{\operatorname{rot}}(g_Rp_1^{\operatorname{rot}})(x)^2 dx \\
& \quad\quad +  \int_{\mathbb{R}} \frac{1}{p_1^{\operatorname{rot}}(x)} \left(g_R(x)\mathcal{A}_1^{\operatorname{rot}}(p_1^{\operatorname{rot}})(x)\right)^2 dx \\
& \quad\quad + \int_{\mathbb{R}} p_1^{\operatorname{rot}}(x)\left(\mathcal{A}_1^{\operatorname{rot}}(g_R)(x)\right)^2 dx \\
&\quad\quad + 2 \int_{\mathbb{R}} \mathcal{A}_1^{\operatorname{rot}}(g_R)(x)g_R(x)\mathcal{A}_1^{\operatorname{rot}}(p_1^{\operatorname{rot}})(x) dx \\
&\quad\quad - 2 \int_{\mathbb{R}} \mathcal{A}_1^{\operatorname{rot}}(g_R)(x)\mathcal{A}_{1}^{\operatorname{rot}}(g_Rp_1^{\operatorname{rot}})(x) dx \\
&\quad\quad -2 \int_{\mathbb{R}} \dfrac{1}{p_1^{\operatorname{rot}}(x)} g_R(x)\mathcal{A}_1^{\operatorname{rot}}(p_1^{\operatorname{rot}})(x)\mathcal{A}_{1}^{\operatorname{rot}}(g_Rp_1^{\operatorname{rot}})(x) dx. 
\end{align*}
Now, thanks to Lemma \ref{lem:action_square_root_laplacian}, for all $R>0$, 
\begin{align*}
\left\| \frac{1}{p_1^{\operatorname{rot}}}\Gamma^{\operatorname{rot}}_1\left( p_1^{\operatorname{rot}}; g_{R} \right)  \right\|^2_{L^2(\mu_1^{\operatorname{rot}})} & = \int_{\mathbb{R}} \frac{1}{p_1^{\operatorname{rot}}(x)} \left| \Gamma_1^{\operatorname{rot}}(p_1^{\operatorname{rot}},g_R)(x) \right|^2 dx , \\
& = \int_{\mathbb{R}} \frac{1}{p_1^{\operatorname{rot}}(x)} \mathcal{A}_{1}^{\operatorname{rot}}(g_Rp_1^{\operatorname{rot}})(x)^2 dx +  \int_{\mathbb{R}} p_1^{\operatorname{rot}}(x) \left(g_R(x) \dfrac{|x|^2-1}{|x|^2+1} \right)^2 dx \\
& \quad\quad + \int_{\mathbb{R}} p_1^{\operatorname{rot}}(x)\left(\mathcal{A}_1^{\operatorname{rot}}(g_R)(x)\right)^2 dx \\
&\quad\quad + 2 \int_{\mathbb{R}} \mathcal{A}_1^{\operatorname{rot}}(g_R)(x)g_R(x)p_1^{\operatorname{rot}}(x) \dfrac{x^2-1}{1+x^2} dx \\
&\quad\quad - 2 \int_{\mathbb{R}} \mathcal{A}_1^{\operatorname{rot}}(g_R)(x)\mathcal{A}_{1}^{\operatorname{rot}}(g_Rp_1^{\operatorname{rot}})(x) dx \\
&\quad\quad -2 \int_{\mathbb{R}} g_R(x)\dfrac{x^2-1}{1+x^2}\mathcal{A}_{1}^{\operatorname{rot}}(g_Rp_1^{\operatorname{rot}})(x) dx. 
\end{align*}
Clearly, the term $A_R$ defined, for all $R>0$, by
\begin{align*}
A_R :=  \int_{\mathbb{R}} g_R(x)^2 \left(\dfrac{|x|^2-1}{|x|^2+1} \right)^2 p_1^{\operatorname{rot}}(x) dx , 
\end{align*}
diverges as $R$ tends to $+\infty$. Thus, by the Lebesgue dominated convergence theorem
\begin{align}\label{eq:limit_diff_AR}
\int_{\mathbb{R}} g_R(x)^2 p_1^{\operatorname{rot}}(x) dx & -\int_{\mathbb{R}} g_R(x)^2 \left(\dfrac{|x|^2-1}{|x|^2+1} \right)^2 p_1^{\operatorname{rot}}(x) dx \nonumber \\
& = 4 \int_{\mathbb{R}} g_R(x)^2  \dfrac{x^2}{(1+x^2)^2} p_1^{\operatorname{rot}}(x) dx \nonumber \\
&\longrightarrow 4 \int_{\mathbb{R}} \dfrac{x^4}{(1+x^2)^2} p_1^{\operatorname{rot}}(x) dx = \frac{3}{2},
\end{align}
as $R$ tends to $+\infty$. Next, let $B_R$ be the quantity defined, for all $R>0$, by
\begin{align*}
B_R : =  \int_{\mathbb{R}} p_1^{\operatorname{rot}}(x)\left(\mathcal{A}_1^{\operatorname{rot}}(g_R)(x)\right)^2 dx.
\end{align*}
But, by antisymmetry, for all $x \in \mathbb{R}$, 
\begin{align*}
\mathcal{A}_1^{\operatorname{rot}}(g_R)(x) & = \frac{1}{2\pi} \int_{\mathbb{R}} \mathcal{F}(g_R)(\xi) e^{i x \xi} \left( - |\xi|\right) d\xi  =  \frac{R^2}{2\pi} \int_{\mathbb{R}} \mathcal{F}(g_1)(R\xi) e^{i x \xi} \left( - |\xi|\right) d\xi , \\
& = \frac{1}{2\pi} \int_{\mathbb{R}} \mathcal{F}(g_1)(\xi) e^{i \frac{x}{R} \xi} \left( - |\xi|\right) d\xi , \\
& \longrightarrow \frac{1}{2\pi} \int_{\mathbb{R}} \mathcal{F}(g_1)(\xi) \left( - |\xi|\right) d\xi = 0 ,
\end{align*}
as $R$ tends to $+\infty$. Thus, by the Lebesgue dominated convergence theorem, 
$B_R \longrightarrow 0$, as $R$ tends to $+\infty$. Now, let us consider, for all $R>0$, 
\begin{align*}
C_R & : = \int_{\mathbb{R}} \mathcal{A}_1^{\operatorname{rot}}(g_R)(x)g_R(x)p_1^{\operatorname{rot}}(x) \dfrac{x^2-1}{1+x^2} dx , \\
& = \int_{\mathbb{R}} \mathcal{A}_1^{\operatorname{rot}}(g_1)\left(\frac{x}{R}\right)g_R(x)p_1^{\operatorname{rot}}(x) \dfrac{x^2-1}{1+x^2} dx , \\
& =  \int_{\mathbb{R}} \mathcal{A}_1^{\operatorname{rot}}(g_1)\left(y\right) y \exp\left( - y^2 \right) p_1^{\operatorname{rot}}(Ry) \dfrac{R^2y^2-1}{1+R^2y^2}  R^{2} dy.
\end{align*}
From the previous representation, it is clear that, 
\begin{align}\label{eq:limit_CR}
C_R \longrightarrow \int_{\mathbb{R}} \mathcal{A}_1^{\operatorname{rot}}(g_1)\left(y\right)\exp\left( - y^2 \right) \frac{dy}{\pi y} = \frac{2}{\pi^{\frac{3}{2}}} \int_{\mathbb{R}} \left(-F(y) -y +2 y^2F(y)\right) \exp(-y^2) \frac{dy}{y} , 
\end{align}
as $R$ tends to $+\infty$. Then, it remains to deal with: for all $R>0$, 
\begin{align*}
& D_R := \int_{\mathbb{R}} \frac{1}{p_1^{\operatorname{rot}}(x)} \mathcal{A}_{1}^{\operatorname{rot}}(g_Rp_1^{\operatorname{rot}})(x)^2 dx, \quad E_R := \int_{\mathbb{R}} \mathcal{A}_1^{\operatorname{rot}}(g_R)(x)\mathcal{A}_{1}^{\operatorname{rot}}(g_Rp_1^{\operatorname{rot}})(x) dx, \\
& F_R := \int_{\mathbb{R}} g_R(x)\dfrac{x^2-1}{1+x^2}\mathcal{A}_{1}^{\operatorname{rot}}(g_Rp_1^{\operatorname{rot}})(x) dx.
\end{align*}
Let us continue with $E_R$. Then, for all $R>0$, 
\begin{align*}
E_R = \int_{\mathbb{R}} \mathcal{A}_1^{\operatorname{rot}}(g_1)\left(\frac{x}{R}\right)\mathcal{A}_{1}^{\operatorname{rot}}(g_Rp_1^{\operatorname{rot}})(x) dx.
\end{align*}
Now, for all $x \in \mathbb{R}$, 
\begin{align*}
\mathcal{A}_1^{\operatorname{rot}}(g_R p_1^{\operatorname{rot}})(x) & = \frac{1}{(2\pi)^2} \int_{\mathbb{R}^2} \mathcal{F}(p^{\operatorname{rot}}_1)(\xi) R^2 \mathcal{F}(g_1)(R\zeta) e^{i x (\xi +\zeta)} (- \left| \xi +\zeta \right| ) d\xi d\zeta, \\
& = \frac{R}{(2\pi)^2} \int_{\mathbb{R}^2} \mathcal{F}(p^{\operatorname{rot}}_1)(\xi)  \mathcal{F}(g_1)(\zeta) e^{i x (\xi +\frac{\zeta}{R})} \left(- \left| \xi +\frac{\zeta}{R} \right| \right) d\xi d\zeta , \\
& = \frac{R}{(2\pi)^2} \int_{\mathbb{R}^2} \mathcal{F}(p^{\operatorname{rot}}_1)(\xi)  \mathcal{F}(g_1)(\zeta) e^{i x \xi} \bigg[  e^{i x\frac{\zeta}{R}} \left(- \left| \xi +\frac{\zeta}{R} \right| \right)  + |\xi| \bigg] d\xi d\zeta , \\
& \longrightarrow  - \frac{2x}{x^2+1} p_1^{\operatorname{rot}}(x),
\end{align*}
as $R$ tends to $+\infty$. Then, for all $x \in \mathbb{R}$, 
\begin{align*}
\underset{R \longrightarrow +\infty}{\lim} \mathcal{A}_1^{\operatorname{rot}}(g_1)\left(\frac{x}{R}\right)\mathcal{A}_{1}^{\operatorname{rot}}(g_Rp_1^{\operatorname{rot}})(x) = 0. 
\end{align*}
Moreover, thanks to Lemma \ref{lem:one_dimension_explicit_comp}, for all $x \in \mathbb{R}$ and all $R \geq 1$, 
\begin{align*}
\left| \mathcal{A}_1^{\operatorname{rot}}(g_1)\left(\frac{x}{R}\right)\mathcal{A}_{1}^{\operatorname{rot}}(g_Rp_1^{\operatorname{rot}})(x) \right| \leq \dfrac{C}{\left(1+ |x|\right)^2},
\end{align*}
for some positive constant $C$ which does not depend on $x$ neither on $R$. Thus, $E_R$ tends to $0$, as $R$ tends to $+\infty$. So, it remains to deal with $F_R$ and $D_R$. Let us start with $D_R$. But, the convergence of $D_R$ follows from a straightforward application of the Lebesgue dominated convergence theorem since, for all $x \in \mathbb{R}$ and all $R \geq 1$, 
\begin{align*}
\frac{1}{p_1^{\operatorname{rot}}(x)} \mathcal{A}_{1}^{\operatorname{rot}}(g_Rp_1^{\operatorname{rot}})(x)^2 \leq \dfrac{D}{(1+|x|^2)},
\end{align*}
for some positive constant $D$ which does not depend on $x$ neither on $R$. Thus, 
\begin{align}\label{eq:limit_DR}
D_R \longrightarrow  \int_{\mathbb{R}} \frac{4x^2}{(x^2+1)^2} p_1^{\operatorname{rot}}(x)dx = \frac{1}{2},
\end{align}
as $R$ tends to $+\infty$. Finally, to treat the term $F_R$, one needs to use the decomposition \eqref{eq:representation_1D_exact} of Lemma \ref{lem:one_dimension_explicit_comp}. Indeed, for all $x \in \mathbb{R} \setminus \{0\}$ and all $R>0$, 
\begin{align*}
\mathcal{A}_{1}^{\operatorname{rot}}(g_Rp_1^{\operatorname{rot}})(x) = I_R(x) + J_R(x), 
\end{align*} 
with, 
\begin{align}\label{def:IR}
I_R(x) = \frac{1}{x^2 (\pi)^{\frac{3}{2}}} \exp\left(\frac{1}{R^2}\right) \bigg[  \int_0^{+\infty} \xi G_R(\xi)\sin\left(x \xi\right)d\xi + 2\int_0^{+\infty} H_R(\xi) \sin\left(x \xi\right) d\xi \bigg],
\end{align}
\begin{align}\label{def:JR}
J_R(x) = \frac{1}{x^2 (\pi)^{\frac{3}{2}}} \exp\left(\frac{1}{R^2}\right)  \left( 2 F \left(\frac{x}{R}\right) - 2\left(\frac{x}{R}\right) + 4 \left(\frac{x}{R}\right)^2 F\left(\frac{x}{R}\right)\right).
\end{align}
Moreover, thanks to the proof of Lemma \ref{lem:one_dimension_explicit_comp}, there exists $C>0$ such that, for all $x \in \mathbb{R} $ and all $R \geq 1$, 
\begin{align*}
\left| I_R(x) \right| \leq \dfrac{C}{\left(1+ |x|\right)^3}, 
\end{align*}
and, for all $x \in \mathbb{R}$,  
\begin{align*}
\underset{R \longrightarrow +\infty}{\lim} I_R(x) = - \frac{2}{\pi} \dfrac{x}{\left(1+x^2\right)^2}. 
\end{align*}
Thus, by the Lebesgue dominated convergence theorem, 
\begin{align*}
F_{R,1} := \int_{\mathbb{R}} g_R(x)\dfrac{x^2-1}{1+x^2}I_R(x) dx \longrightarrow -\frac{2}{\pi} \int_{\mathbb{R}} \dfrac{x^2-1}{1+x^2} \dfrac{x^2}{\left(1+x^2\right)^2}dx = -\frac{1}{2} ,
\end{align*}
as $R$ tends to $+\infty$. To finish the study of convergence, let us analyse $F_{R,2}$ defined, for all $R>0$, by 
\begin{align*}
F_{R,2} : = \int_{\mathbb{R}} g_R(x)\dfrac{x^2-1}{1+x^2}J_R(x) dx.
\end{align*}
Changing variables since $R\geq 1>0$, 
\begin{align*}
F_{R,2}  & =\frac{2\exp\left(\frac{1}{R^2}\right)}{(\pi)^{\frac{3}{2}}}  \int_{\mathbb{R}} \exp\left(-y^2\right)\dfrac{R^2y^2-1}{1+R^2y^2}   \left( F \left(y\right) - y + 2 y^2 F\left(y\right)\right) \frac{dy}{y} , \\
& \longrightarrow \frac{2}{(\pi)^{\frac{3}{2}}}  \int_{\mathbb{R}} \exp\left(-y^2\right)\left( F \left(y\right) - y + 2 y^2 F\left(y\right)\right) \frac{dy}{y}.
\end{align*}
Then, 
\begin{align}\label{eq:limit_FR}
F_R \longrightarrow -\frac{1}{2} + \frac{2}{(\pi)^{\frac{3}{2}}}  \int_{\mathbb{R}} \exp\left(-y^2\right)\left( F \left(y\right) - y + 2 y^2 F\left(y\right)\right) \frac{dy}{y}. 
\end{align}
Combining \eqref{eq:limit_diff_AR}, \eqref{eq:limit_CR}, \eqref{eq:limit_DR} and \eqref{eq:limit_FR}, together with the fact that 
\begin{align*}
\int_{\mathbb{R}} F(y) \exp(-y^2) \frac{dy}{y} = \frac{\pi^{\frac{3}{2}}}{4}, 
\end{align*}
one gets,
\begin{align*}
\underset{R \longrightarrow +\infty}{\lim} \left(\|g_{R}\|^2_{L^2(\mu_1^{\operatorname{rot}})} - \left\| \frac{1}{p_1^{\operatorname{rot}}}\Gamma^{\operatorname{rot}}_1\left( p_1^{\operatorname{rot}}; g_{R} \right)  \right\|^2_{L^2(\mu_1^{\operatorname{rot}})} \right) = 2.
\end{align*}
Finally, 
\begin{align*}
 \underset{R\longrightarrow +\infty}{\lim} \left\langle \mathcal{A}_1^{\operatorname{rot}}(g_{R})  ; \frac{1}{p_1^{\operatorname{rot}}}\Gamma^{\operatorname{rot}}_1\left( p_1^{\operatorname{rot}}; g_{R} \right) \right\rangle_{L^2(\mu_1^{\operatorname{rot}})}  & = - \underset{R \longrightarrow +\infty}{\lim} C_R , \\
 & = \frac{2}{\pi^{\frac{3}{2}}} \int_{\mathbb{R}} \left(F(y) +y - 2 y^2F(y)\right) \exp(-y^2) \frac{dy}{y} , \\
 &= \frac{1}{\pi} + \frac{1}{2}. 
\end{align*}
This concludes the proof of the proposition. 
\end{proof}

\subsection{Stability estimates for non-degenerate symmetric Cauchy probability measures}\label{sec:SE_NDS_Cauchy}
\noindent
In this subsection, let us prove the stability estimate of Theorem \ref{thm:stability_result_sym_Cauchy_measure} and the quantitative approximation result of Theorem \ref{thm:quantitative_approximation_CauchyLs}, when the target measure is a non-degenerate symmetric $1$-stable probability measure on $\mathbb{R}^d$.\\
\\
\textit{Proof of Theorem \ref{thm:stability_result_sym_Cauchy_measure}}. \textit{Step 1}: First, let us prove that, for all $f\in \mathcal{S}(\mathbb{R}^d)$, 
\begin{align}\label{eq:Stein_identity}
\int_{\mathbb{R}^d} \left[ - \langle x ; \nabla(f)(x) \rangle + \frac{1}{2} \int_{\mathbb{R}^d} \langle \nabla(f)(x+u) - \nabla(f)(x-u) ; u\rangle \omega_X(\|u\|) \nu_1(du)   \right] \mu_X(dx) = 0. 
\end{align} 
By L\'evy-Khintchine theorem, for all $\xi \in \mathbb{R}^d$, 
\begin{align*}
\widehat{\mu_X}(\xi) = \exp\left( \int_{\mathbb{R}^d} \left(e^{i \langle u ; \xi \rangle}-1- i \langle u ; \xi \rangle\bbone_{\|u\|\leq 1}\right) \omega_X(\|u\|) \nu_1(du) \right). 
\end{align*}
At this point, let us perform a truncation argument in the L\'evy measure of $\mu_X$. Namely, for all $R > 1$, let $\mu_{X,R}$ be the probability measure on $\mathbb{R}^d$ which Fourier transform is given, for all $\xi \in \mathbb{R}^d$, by 
\begin{align*}
\widehat{\mu_{X,R}}(\xi) & = \exp \left( \int_{\mathbb{R}^d} \left(e^{i \langle u ; \xi \rangle}-1 - i\langle u ; \xi \rangle\bbone_{\|u\| \leq 1}\right)\bbone_{1/R\leq \|u\| \leq R} \omega_X(\|u\|) \nu_1(du)\right), \\
& =  \exp \left( \int_{\mathbb{R}^d} \left(e^{i \langle u ; \xi \rangle}-1 - i\langle u ; \xi \rangle\bbone_{\|u\| \leq 1}\right)  \omega_{X,R}(\|u\|) \nu_1(du)\right),
\end{align*}
where $\omega_{X,R} (r) =\bbone_{1/R\leq r \leq R} \omega_X(r)$, for all $r>0$. Thanks to symmetry, the L\'evy-Khintchine exponent of $\mu_{X,R}$ can be written: for all $\xi \in \mathbb{R}^d$, 
\begin{align}\label{eq:rep_2_LKExp}
\int_{\mathbb{R}^d} \left(e^{i \langle u ; \xi \rangle}-1- i \langle u ; \xi \rangle\bbone_{\|u\|\leq 1}\right) \omega_{X,R}(\|u\|) \nu_1(du) &  = \frac{1}{2} \int_{\mathbb{R}^d} \left(e^{i \langle u ; \xi \rangle} + e^{- i \langle u ; \xi \rangle} - 2\right) \nonumber \\
&\quad\quad \times \omega_{X,R}(\|u\|) \nu_1(du).
\end{align}
Now, by straightforward computations, for all $r \in (0,+\infty)$ and all $y \in \mathbb{S}^{d-1}$, 
\begin{align*}
\dfrac{\partial}{\partial r}(\widehat{\mu_{X,R}})(ry) = \frac{1}{2} \left( \int_{\mathbb{R}^d} i \langle y; u \rangle \left(e^{i r \langle u ; y\rangle } - e^{- i r \langle u ; y \rangle}\right) \omega_{X,R}(\|u\|)\nu_1(du)\right) \widehat{\mu_{X,R}}(ry). 
\end{align*}
Then, the Fourier transform of $\mu_{X,R}$ verifies the following partial differential equation: for all $\xi \in \mathbb{R}^d \setminus \{0\}$, 
\begin{align}\label{eq:PDE_CF_X}
\langle \xi ; \nabla\left(\widehat{\mu_{X,R}}\right)(\xi) \rangle = \widehat{\mu_{X,R}}(\xi) m_{X,R}(\xi),
\end{align}
where, 
\begin{align}\label{eq:multiplier}
m_{X,R}(\xi) = \frac{1}{2} \int_{\mathbb{R}^d} i \langle u ; \xi \rangle \left( e^{i \langle u ; \xi \rangle} - e^{ - i  \langle u ; \xi \rangle} \right) \omega_{X,R}(\|u\|) \nu_1(du). 
\end{align}
By Fourier duality arguments, equation \eqref{eq:Stein_identity} is verified for the truncation $\mu_{X,R}$, i.e., for all $f \in \mathcal{S}(\mathbb{R}^d)$, 
\begin{align}\label{eq:Stein_identity_truncation}
\int_{\mathbb{R}^d} \left[ - \langle x ; \nabla(f)(x) \rangle + \frac{1}{2} \int_{\mathbb{R}^d} \langle \nabla(f)(x+u) - \nabla(f)(x-u) ; u\rangle \omega_{X,R}(\|u\|) \nu_1(du)   \right] \mu_{X,R}(dx) = 0.
\end{align}
Moreover, $(\mu_{X,R})_{R>1}$ converges weakly to $\mu_X$ as $R$ tends to $+\infty$, so that, for all $f \in \mathcal{S}(\mathbb{R}^d)$, 
\begin{align}\label{eq:convergence_drift_term}
\underset{R \longrightarrow +\infty}{\lim} \int_{\mathbb{R}^d} \langle x ; \nabla(f)(x) \rangle \mu_{X,R}(dx) =  \int_{\mathbb{R}^d} \langle x ; \nabla(f)(x) \rangle \mu_{X}(dx)
\end{align} 
Finally, for all $f \in \mathcal{S}(\mathbb{R}^d)$ and all $R>1$, 
\begin{align*}
\frac{1}{2}\int_{\mathbb{R}^d} & \mu_{X,R}(dx) \left[\int_{\mathbb{R}^d} \langle \nabla(f)(x+u) - \nabla(f)(x-u) ; u\rangle \omega_{X,R}(\|u\|) \nu_1(du)\right] =\\
& \frac{1}{(2\pi)^d} \int_{\mathbb{R}^d} \mathcal{F}(f)(\xi) m_{X,R}(\xi) \widehat{\mu_{X,R}}(\xi) d\xi. 
\end{align*}
Thanks to \eqref{eq:cond_semi_cv}, for all $\xi \in \mathbb{R}^d$, 
\begin{align*}
m_{X,R}(\xi) \longrightarrow m_X(\xi) =  \frac{1}{2} \int_{\mathbb{R}^d} i \langle u ; \xi \rangle \left( e^{i \langle u ; \xi \rangle} - e^{ - i  \langle u ; \xi \rangle} \right) \omega_{X}(\|u\|) \nu_1(du),
\end{align*}
as $R$ tends to $+\infty$. Now, \eqref{eq:cond_A_X} ensures that, for all $R\geq 1$ and all $\xi \in \mathbb{R}^d$, 
\begin{align*}
\left| m_{X,R}\left(\xi\right) \right| \leq  \Phi_X(\|\xi\|) \|\xi\|.
\end{align*}
The Lebesgue dominated convergence theorem ensures that \eqref{eq:Stein_identity} is satisfied by $\mu_X$, for all $f \in \mathcal{S}(\mathbb{R}^d)$. Then, by standard approximation arguments, for all $h \in \mathcal{H}_2 \cap \mathcal{C}_c^{\infty}(\mathbb{R}^d)$, 
\begin{align}\label{eq:Stein_Identity_X_Steinsol}
\int_{\mathbb{R}^d} \left[ - \langle x ; \nabla(f_h)(x) \rangle + \frac{1}{2} \int_{\mathbb{R}^d} \langle \nabla(f_h)(x+u) - \nabla(f_h)(x-u) ; u\rangle \omega_X(\|u\|) \nu_1(du)   \right] \mu_X(dx) = 0,
\end{align}
where $f_h$ is the function given by \eqref{eq:1_stable_Stein_sol}. This concludes step $1$.\\
\textit{Step 2}: Now, by Proposition \ref{prop:stein_solution_cauchy}, for all $h \in \mathcal{H}_2 \cap \mathcal{C}_c^{\infty}(\mathbb{R}^d)$ and all $x \in \mathbb{R}^d$, 
\begin{align*}
- \langle x ; \nabla(f_h)(x) \rangle + \mathcal{A}_1(f_h)(x)= h(x) - \mathbb{E} h(X_1),
\end{align*}
where $X_1 \sim \mu_1$. Thus, integrating with respect to $\mu_X$ and using \eqref{eq:Stein_Identity_X_Steinsol}, 
\begin{align*}
\mathbb{E} h(X) - \mathbb{E} h(X_1)  & = \mathbb{E} \left[- \langle X ; \nabla(f_h)(X) \rangle + \mathcal{A}_1(f_h)(X) \right] , \\
& =  \mathbb{E} \left[ - \mathcal{A}_{1, \omega_X}(f_h)(X) + \mathcal{A}_1(f_h)(X) \right], 
 \end{align*}
where, for all $x \in \mathbb{R}^d$, 
\begin{align*}
\mathcal{A}_{1, \omega_X}(f_h)(x) =  \frac{1}{2} \int_{\mathbb{R}^d} \langle \nabla(f_h)(x+u) - \nabla(f_h)(x-u) ; u\rangle \omega_X(\|u\|) \nu_1(du). 
\end{align*}
Thus, for all $h \in \mathcal{H}_2 \cap \mathcal{C}_c^{\infty}(\mathbb{R}^d)$, 
\begin{align*}
\left| \mathbb{E} h(X) - \mathbb{E} h(X_1) \right| = \frac{1}{2} \left| \mathbb{E} \left[ \int_{\mathbb{R}^d} \langle u ; \nabla(f_h)(X+u) - \nabla(f_h)(X-u) \rangle \left(\omega_X(\|u\|) - 1\right) \nu_1(du) \right] \right|. 
\end{align*}
Let us cut the integral with respect to $\mu_X$ into two parts: $\{\|x\| > 1\}$ and $\{\|x\| \leq 1\}$. Then, 
\begin{align*}
A_1 & : =  \left| \int_{\|x\| \leq 1} \int_{\mathbb{R}^d} \langle u ; \nabla(f_h)(x+u) - \nabla(f_h)(x-u) \rangle \left(\omega_X(\|u\|)-1\right) \nu_1(du) \mu_X(dx) \right| , \\
& \leq \int_{\|x\| \leq 1} \int_{\|u\|\leq 1} \left| \langle u ; \nabla(f_h)(x+u) - \nabla(f_h)(x-u) \rangle \right| \left| \omega_X(\|u\|)-1 \right| \nu_1(du) \mu_X(dx) \\
&\quad\quad + \int_{\|x\| \leq 1} \int_{\|u\|\geq 1} \left| \langle u ; \nabla(f_h)(x+u) - \nabla(f_h)(x-u) \rangle \right| \left| \omega_X(\|u\|)-1 \right| \nu_1(du) \mu_X(dx) , \\
& \leq 2 M_2(f_h) \int_{\|x\| \leq 1} \int_{\|u\|\leq 1} \|u\|^2 \left| \omega_X(\|u\|) - 1 \right| \nu_1(du) \mu_X(dx) \\
&\quad\quad + C \int_{\|x\| \leq 1} \int_{\|u\| \geq 1} \left| \omega_X(\|u\|) - 1\right| \nu_1(du) \mu_X(dx),
\end{align*}
for some $C>0$. Thus, 
\begin{align*}
A_1 \leq \int_{\|u\|\leq 1} \|u\|^2 \left| \omega_X(\|u\|) - 1 \right| \nu_1(du) + C \int_{\|u\| \geq 1} \left| \omega_X(\|u\|) - 1\right| \nu_1(du). 
\end{align*}
Moreover, 
\begin{align*}
A_2 & := \left| \int_{\|x\| > 1} \int_{\mathbb{R}^d} \langle u ; \nabla(f_h)(x+u) - \nabla(f_h)(x-u) \rangle \left(\omega_X(\|u\|)-1\right) \nu_1(du) \mu_X(dx) \right| ,  \\
& \leq \int_{\|x\| \geq 1} \int_{\|u\|\leq 1} \left| \langle u ; \nabla(f_h)(x+u) - \nabla(f_h)(x-u) \rangle \right| \left| \omega_X(\|u\|)-1 \right| \nu_1(du) \mu_X(dx) \\
&\quad\quad + \int_{\|x\| \geq 1} \int_{\|u\|\geq 1} \left| \langle u ; \nabla(f_h)(x+u) - \nabla(f_h)(x-u) \rangle \right| \left| \omega_X(\|u\|)-1 \right| \nu_1(du) \mu_X(dx) , \\
& \leq 2M_2(f_h) \int_{\|x\| \geq 1} \int_{\|u\| \leq 1} \|u\|^2 \left| \omega_X(\|u\|) -1 \right| \nu_1(du) \mu_X(dx) \\
& \quad\quad + \int_{\|x\| \geq 1} \int_{\|u\|\geq 1} \left| \langle u ; \nabla(f_h)(x+u) - \nabla(f_h)(x-u) \rangle \right| \left| \omega_X(\|u\|)-1 \right| \nu_1(du) \mu_X(dx).
\end{align*}
But, for all $x \in \mathbb{R}^d$ such that $\|x\| \geq 1$, 
\begin{align*}
\int_{\|u\|\geq 1} \left| \langle u ; \nabla(f_h)(x+u) - \nabla(f_h)(x-u) \rangle \right| \left| \omega_X(\|u\|)-1 \right| \nu_1(du) \leq A_3(x) + A_4(x), 
\end{align*}
where,
\begin{align*}
&A_3(x) := \int_{1 \leq \|u\| \leq \|x\|} \left| \langle u ; \nabla(f_h)(x+u) - \nabla(f_h)(x-u) \rangle \right| \left| \omega_X(\|u\|)-1 \right| \nu_1(du), \\
&A_4(x) := \int_{\|u\| \geq \|x\|} \left| \langle u ; \nabla(f_h)(x+u) - \nabla(f_h)(x-u) \rangle \right| \left| \omega_X(\|u\|)-1 \right| \nu_1(du). 
\end{align*}
Now, for all $x \in \mathbb{R}^d$ such that $\|x\| \geq 1$, 
\begin{align*}
A_3(x) & \leq 2 M_1(f_h) \int_{1\leq \|u\| \leq \|x\|} \|u\| \left| \omega_X(\|u\|) - 1 \right| \nu_1(du) , \\
& \leq 2 \int_{1\leq \|u\| \leq \|x\|} \|u\| \left| \omega_X(\|u\|) - 1 \right| \nu_1(du). 
\end{align*}
For $A_4$, for all $x \in \mathbb{R}^d$ such that $\|x\| \geq 1$, 
\begin{align*}
A_4(x) \leq C_5 \left(1+ \|x\|\right) \int_{\|u\| \geq \|x\|} \left| \omega_X(\|u\|) -1 \right| \nu_1(du),
\end{align*}
for some $C_5>0$. Then,
\begin{align}\label{ineq:last_bound}
 \bigg| \mathbb{E} \bigg[ \int_{\mathbb{R}^d} \langle u ; \nabla(f_h)(X+u) & - \nabla(f_h)(X-u) \rangle \left(\omega_X(\|u\|) - 1\right) \nu_1(du) \bigg] \bigg| \leq 2 \int_{\|u\| \leq 1} \|u\|^2 \nonumber\\
 &\quad\quad \times \left| \omega_X(\| u \|) - 1 \right| \nu_1(du) \nonumber \\
&\quad\quad + C_7 \int_{\|u\| \geq 1} \left| \omega_X(\|u\|) - 1 \right| \nu_1(du) \nonumber \\
& \quad\quad + 2\int_{\|x\| \geq 1} \left(\int_{1 \leq \|u\| \leq \|x\|} \|u\| \left| \omega_X(\|u\|) - 1 \right| \nu_1(du)\right) \mu_X(dx) \nonumber \\
&\quad\quad + C_8\int_{\|x\| \geq 1} \left(1+ \|x\|\right) \left(\int_{\|u\| \geq \|x\|} \left| \omega_X(\|u\|) - 1 \right| \nu_1(du)\right) \mu_X(dx),
\end{align}
for some $C_7, C_8 >0$. The conclusion of the proof follows easily. $\square$

\begin{rem}\label{rem:comments_previous_theorem}
(i) Similar quantitative estimates on $d_{W_2}(\mu_X , \mu_1)$ can be obtained using the other representations for the non-local part of $\mathcal{L}^{\nu_1}$.\\
(ii) Let us consider the trivial situation when $\omega_X$ is identically constant on $(0,+\infty)$ equal to $U_X>0$. Then, the integral appearing in \eqref{eq:cond_semi_cv} boils down to: for all $r \in (0,+\infty)$ and all $y \in \mathbb{S}^{d-1}$, 
\begin{align*}
\int_{\mathbb{R}^d} \sin\left( r \langle u ; y \rangle\right) \langle y ; u \rangle \omega_X(\|u\|) \nu_1(du) & = U_X \underset{R \longrightarrow +\infty}{\lim} \int_{\mathcal{B}(0,R)} \sin\left( r \langle u ; y \rangle\right) \langle y ; u \rangle  \nu_1(du) , \\
& = U_X  \underset{R \longrightarrow +\infty}{\lim} \int_{\mathbb{S}^{d-1}} \langle y; x \rangle \left(\int_0^R \sin \left(r \rho \langle x ; y \rangle\right) \frac{d\rho}{\rho} \right) \sigma(dx) , \\
&= U_X \int_{\mathbb{S}^{d-1}} |\langle y;x \rangle| \sigma(dx) \int_0^{+\infty} \frac{\sin(\rho)}{\rho}d\rho < +\infty,
\end{align*}
where $\sigma$ is the spherical part of $\nu_1$. Moreover, for all $r \in (0,+\infty)$, all $y \in \mathbb{S}^{d-1}$ and all $R>1$, 
\begin{align*}
\left|  \int_{\frac{1}{R} \leq \|u\| \leq R} \sin\left( r \langle u ; y \rangle\right) \langle y ; u \rangle \omega_X(\|u\|) \nu_1(du) \right| \leq U_X \sigma(\mathbb{S}^{d-1}) \underset{R \geq 1}{\sup} \left| \int_{\frac{1}{R}}^R \frac{\sin(\rho)}{\rho}d\rho \right| < +\infty.
\end{align*}
Then, the conditions on the weight $\omega_X$ are satisfied in this situation. Finally, the stability estimate \eqref{ineq:stability_result} gives
\begin{align}\label{eq:stability_cauchy_constw}
d_{W_2}(\mu_X , \mu_1) & \leq \left(\int_{\|u\| \leq 1} \|u\|^2  \nu_1(du) + M_1 \int_{\|u\| \geq 1}  \nu_1(du) \right) |U_X-1| \nonumber \\
& \quad\quad + \sigma(\mathbb{S}^{d-1}) \left(\int_{\|x\|\geq 1} \ln(\|x\|) \mu_X(dx) \right)  \left| U_X - 1 \right| \nonumber \\
& \quad\quad + 2 M_2 \sigma(\mathbb{S}^{d-1}) \left|U_X - 1\right|,
\end{align} 
for some $M_1, M_2 >0$. \\
(iii) In Theorem \ref{thm:stability_result_sym_Cauchy_measure}, one might want to relax the condition of infinite divisibility of the probability measure $\mu_X$ as in Theorem \ref{thm:stability_estimate_NDS_stable} for $\alpha \in (1,2)$.~A natural candidate for the closable non-negative definite bilinear symmetric form would be: for all $f_1, f_2 \in \mathcal{C}_b^1(\mathbb{R}^d, \mathbb{R}^d)$, 
\begin{align}\label{eq:form_central_difference}
\mathcal{E}^*_{\nu_1, \omega_X} (f_1,f_2) = \frac{1}{4} \int_{\mathbb{R}^{2d}} \langle f_1(x+u)-f_1(x-u) ; f_2(x+u) - f_2(x-u) \rangle \omega_X(\|u\|) \nu_1(du)\mu_X(dx).
\end{align}
When $\omega_X(r) = 1$, for all $r \in (0, +\infty)$, and $\mu_X = \mu_1^{\operatorname{rot}}$,  it can be proved that 
\begin{align}\label{eq:spectral_condition}
\underset{R \longrightarrow +\infty}{\lim} \left(\mathcal{E}^*_{\nu_1^{\operatorname{rot}}} (g_R,g_R) - \|g_R\|^2_{L^2(\mu^{\operatorname{rot}}_1)}\right) = 0,
\end{align}
where $g_R$ is given by \eqref{eq:eigen_approximate_eigenvector} and where, for all $f_1, f_2 \in \mathcal{C}_b^1(\mathbb{R}^d, \mathbb{R}^d)$, 
\begin{align*}
\mathcal{E}^*_{\nu_1^{\operatorname{rot}}} (f_1,f_2) = \frac{1}{4} \int_{\mathbb{R}^{2d}} \langle f_1(x+u)-f_1(x-u) ; f_2(x+u) - f_2(x-u) \rangle \nu^{\operatorname{rot}}_1(du)\mu_1^{\operatorname{rot}}(dx).
\end{align*}
However, at present, we do not know if the probability measure $\mu_1^{\operatorname{rot}}$ verifies a Poincar\'e-type inequality like \eqref{Poinc_type_Cauchy} with the energy form $\mathcal{E}^*_{\nu_1^{\operatorname{rot}}}$ instead of $\mathcal{E}^{\nu_1^{\operatorname{rot}}}$. Using standard arguments, a weak Poincar\'e-type inequality can be proved with this new bilinear form (see, e.g., \cite{Rockner_Wang_01,Wang_14,Chen_Wang_17}).\\
(iv) Next, let us consider the classical example. Let $(Z_k)_{k \geq 1}$ be a sequence of independent random vectors of $\mathbb{R}^d$ with characteristic function $(\varphi_k)_{k \geq 1}$ defined, for all $k \geq 1$ and all $\xi \in \mathbb{R}^d$, by 
\begin{align*}
\varphi_k(\xi) = \dfrac{\widehat{\mu_1}\left((k+1) \xi\right)}{\widehat{\mu_1}(k\xi)}.
\end{align*}
Let $(S_n)_{n \geq 1}$ be the sequence of random vectors of $\mathbb{R}^d$ defined, for all $n \geq 1$, by 
\begin{align*}
S_n = \frac{1}{n} \sum_{k = 1}^n Z_k.
\end{align*}
Straightforward computations ensure that $S_n =_\mathcal{L} X_1$, where $X_1 \sim \mu_1$, for all $n \geq 1$. So, both sides of \eqref{eq:stability_cauchy_constw} are equal to $0$, since $\omega_{S_n}(r) = 1$, for all $r \in (0, +\infty)$ and all $n \geq 1$.  
\end{rem}
\noindent
To conclude this section, let us consider the more involved example of Theorem \ref{thm:quantitative_approximation_CauchyLs}. Let $\beta \in (1,2)$ and let $\mu_{1,\beta}^L$ be the probability measure on $\mathbb{R}^d$ characterized in Fourier by \eqref{eq:cf_layered_cauchy}. Recall that its L\'evy measure $\nu_{1,\beta}^L$ is defined by 
\begin{align*}
\nu_{1, \beta}^L(du) =\bbone_{(0,+\infty)}(r)\bbone_{\mathbb{S}^{d-1}}(y) \left( \frac{1}{r^{\beta}}\bbone_{(0,1]}(r) + \frac{1}{r}\bbone_{(1,+\infty)}(r)\right) \dfrac{dr}{r} \sigma(dy),
\end{align*}
where $\sigma$ is a positive finite symmetric measure on $\mathcal{B}\left(\mathbb{S}^{d-1}\right)$ such that 
\begin{align*}
\underset{e \in \mathbb{S}^{d-1}}{\inf} \int_{\mathbb{S}^{d-1}} \left| \langle e ; y \rangle \right|\sigma(dy)>0. 
\end{align*}
Let $(Z_k)_{k \geq 1}$ be a sequence of i.i.d. random vectors of $\mathbb{R}^d$ such that $Z_1 \sim \mu_{1, \beta}^L$. Let $(S_n)_{n \geq 1}$ be the sequence of random vectors of $\mathbb{R}^d$ defined, for all $n \geq 1$, by 
\begin{align}\label{def:Sn_Cauchy_case}
S_n = \frac{1}{n} \sum_{k=1}^n Z_k. 
\end{align}
By a reasoning similar to the one performed in the proof of Theorem \ref{thm:limit_theorem_layered_stable_distributions} and thanks to the symmetry of $\sigma$, it is clear that $(S_n)_{n \geq 1}$ converges in law to $\mu_1$ which Fourier transform is given, for all $\xi \in \mathbb{R}^d$, by 
\begin{align*}
\widehat{\mu_1}(\xi) = \exp \left( \int_{\mathbb{R}^d} \left(e^{i \langle u ; \xi \rangle}-1 - i \langle u ; \xi \rangle\bbone_{\|u\| \leq 1}\right) \nu_1(du)\right),
 \end{align*}
 with, 
\begin{align*}
\nu_1(du) =\bbone_{(0, +\infty)}(r)\bbone_{\mathbb{S}^{d-1}}(y) \frac{dr}{r^2} \sigma(dy).
\end{align*}
Moreover, for all $n \geq 1$, the L\'evy measure of the law of $S_n$ is given by
\begin{align*}
\nu_{1,\beta}^{L,n}(du) = \omega^{1,\beta}_n(\|u\|) \nu_1(du), 
\end{align*}
where $\omega_n^{1,\beta}$ is defined, for all $r \in (0,+\infty)$, by
\begin{align}\label{eq:weight_layered_cauchy}
\omega_n^{1,\beta}(r) =  \frac{1}{n^{\beta-1}} \frac{1}{r^{\beta-1}}\bbone_{(0, n^{-1}]}(r) +\bbone_{(n^{-1},+\infty)}(r). 
\end{align}
Note that, since $\beta \in (1,2)$, for all $r \in (0,+\infty)$, 
\begin{align*}
\omega_n^{1,\beta}(r) \longrightarrow 1, \quad n \longrightarrow +\infty.
\end{align*}
For now, let us assume that the conditions \eqref{eq:cond_semi_cv} and \eqref{eq:cond_A_X} of Theorem \ref{thm:stability_result_sym_Cauchy_measure} are verified so that one can apply it.~Noting that, for all $n \geq 2$ and all $r\geq 1$, $\omega_n^{1, \beta}(r) = 1$, and performing computations similar to the ones contained in the proof of Theorem \ref{thm:quantitative_approximation_Ls}, the bound \eqref{ineq:rate_convergence_cauchy_case} of Theorem \ref{thm:quantitative_approximation_CauchyLs} holds true. The next technical lemma ensures that the conditions \eqref{eq:cond_semi_cv} and \eqref{eq:cond_A_X} are indeed verified by the weight $\omega_n^{1,\beta}$, for all $n \geq 1$. 

\begin{lem}\label{lem:technical_layered_cauchy}
Let $\beta \in (1,2)$ and let $(\omega^{1, \beta}_n)_{n \geq 1}$ be defined by \eqref{eq:weight_layered_cauchy}. Then, for all $n \geq 1$, $\omega_n^{1, \beta}$ verifies the conditions \eqref{eq:cond_semi_cv} and \eqref{eq:cond_A_X}.
\end{lem}

\begin{proof}
Let $R \geq 1$. Now, for all $n \geq 1$, all $r \in (0,+\infty)$ and all $y \in \mathbb{S}^{d-1}$
\begin{align*}
\int_{\mathcal{B}(0,R)} \sin \left( r \langle u;y \rangle\right) \langle u ; y \rangle \omega_n^{1, \beta}(\|u\|) \nu_1(du) & = \int_{\mathbb{S}^{d-1}} \langle x ; y \rangle \sigma(dx) \int_0^R \sin(r \rho \langle x ; y \rangle) \omega_n^{1, \beta}(\rho) \frac{d \rho}{\rho} , \\
& =  I_R + II_R , 
\end{align*}
with, 
\begin{align*}
&I_R := \frac{1}{n^{\beta-1}} \int_{\mathbb{S}^{d-1}} \langle x ; y \rangle \sigma(dx) \int_0^R \sin\left(r \rho \langle y ; x\rangle\right)\bbone_{(0, n^{-1}]}(\rho)   \frac{d \rho}{\rho^{\beta}} , \\
&II_R : = \int_{\mathbb{S}^{d-1}} \langle x ; y \rangle \sigma(dx) \int_0^R \sin\left(r \rho \langle y ; x\rangle\right)\bbone_{(\frac{1}{n}, + \infty)}(\rho) \frac{d \rho}{\rho}.
\end{align*}
Now, it is clear that, for all $n \geq 1$, all $r>0$ and all $y \in \mathbb{S}^{d-1}$, 
\begin{align*}
&\underset{R \longrightarrow +\infty}{\lim} I_R = \frac{1}{n^{\beta-1}} \int_{\mathbb{S}^{d-1}} \langle x ; y \rangle \sigma(dx) \int_0^{\frac{1}{n}} \sin\left(r \rho \langle y ; x\rangle\right)\frac{d \rho}{\rho^{\beta}}  , \\
&\underset{R \longrightarrow +\infty}{\lim} II_R = \int_{\mathbb{S}^{d-1}} \langle x ; y \rangle \sigma(dx) \int_{\frac{1}{n}}^{+\infty} \sin\left(r \rho \langle y ; x\rangle\right)\frac{d \rho}{\rho}. 
\end{align*}
so that the condition \eqref{eq:cond_semi_cv} is verified. Moreover, for all $R>1$, 
\begin{align*}
\int_{1/R\leq \|u\|\leq R} \sin \left( r \langle u;y \rangle\right) \langle u ; y \rangle \omega_n^{1, \beta}(\|u\|) \nu_1(du) & = \int_{\mathbb{S}^{d-1}} \langle x ; y \rangle \sigma(dx) \int_{\frac{1}{R}}^R \sin(r \rho \langle x ; y \rangle) \omega_n^{1, \beta}(\rho) \frac{d \rho}{\rho} , \\
& =  III_R + IV_R ,
\end{align*}
with, 
\begin{align*}
&III_R := \frac{1}{n^{\beta-1}} \int_{\mathbb{S}^{d-1}} \langle x ; y \rangle \sigma(dx) \int_{\frac{1}{R}}^R \sin\left(r \rho \langle y ; x\rangle\right)\bbone_{(0, n^{-1}]}(\rho)   \frac{d \rho}{\rho^{\beta}} ,  \\
&IV_R := \int_{\mathbb{S}^{d-1}} \langle x ; y \rangle \sigma(dx) \int_{\frac{1}{R}}^R \sin\left(r \rho \langle y ; x\rangle\right)\bbone_{(\frac{1}{n}, + \infty)}(\rho) \frac{d \rho}{\rho}.
\end{align*}
Now, for all $R>1$, all $n \geq 1$, all $y \in \mathbb{S}^{d-1}$ and all $r>0$, 
\begin{align*}
\left| III_R \right| \leq \frac{1}{n^{\beta-1}} \int_{\mathbb{S}^{d-1}} \sigma(dx) \int_{\frac{1}{R}}^{R} \left| \sin\left(r \rho \langle x; y \rangle\right) \right| \frac{d\rho}{\rho^\beta} \leq \sigma\left(\mathbb{S}^{d-1}\right) r^{\beta-1} \int_{0}^{+\infty} \left| \sin(\rho) \right| \frac{d\rho}{\rho^\beta}. 
\end{align*}
Finally, 
\begin{align*}
\left| IV_R \right| \leq \sigma\left(\mathbb{S}^{d-1}\right) \sup_{0<a<b<+\infty} \left| \int_{a}^{b} \sin\left(\rho\right) \frac{d\rho}{\rho} \right|.
\end{align*}
This concludes the proof of the lemma.
\end{proof}

\section{Appendix}
\noindent
\begin{lem}\label{lem:riesz_potential_critical}
Let $d \geq 1$ be an integer. Then, for all $x \in \mathbb{R}^d$ such that $x \ne 0$, 
\begin{align*}
\underset{R \longrightarrow +\infty}{\lim} \int_{\mathbb{R}^d} e^{ - \frac{\|\xi_1\|}{R}} e^{i \langle x ; \xi_1 \rangle} \|\xi_1\| \frac{d\xi_1}{(2\pi)^d} = - \frac{c_d}{\|x\|^{d+1}}.
\end{align*}
\end{lem}

\begin{proof}
By change into spherical coordinates, thanks to \cite[Appendix $B.4$]{Grafakos_1} and to \cite[Formula $10.22.48$]{Olver_10}, for all $R>0$ and all $x \in \mathbb{R}^d$ such that $x \ne 0$, 
\begin{align*}
 \int_{\mathbb{R}^d} e^{ - \frac{\|\xi_1\|}{R}} e^{i \langle x ; \xi_1 \rangle} \|\xi_1\| \frac{d\xi_1}{(2\pi)^d} & = \frac{1}{(2\pi)^d} \int_{(0,+\infty)} e^{- \frac{r}{R}} r^d \left(\int_{\mathbb{S}^{d-1}} e^{i r \langle \theta ; x \rangle} \sigma_L(d\theta)\right) dr , \\
 & = \frac{1}{(2\pi)^{\frac{d}{2}}} \frac{1}{\|x\|^{\frac{d}{2}-1}} \int_{(0,+\infty)} e^{- \frac{r}{R}} r^{\frac{d}{2}+1} J_{\frac{d}{2}-1} \left(r \|x\|\right) dr , \\
 & = \frac{1}{(2\pi)^{\frac{d}{2}}}\left(\frac{1}{2}\right)^{\frac{d}{2}-1} R^{d+1} \dfrac{\Gamma(d+1)}{\Gamma\left(\frac{d}{2}\right)} {}_{2}F_{1}\left(\frac{d+1}{2}, \frac{d}{2}+1;\frac{d}{2};- R^2 \|x\|^2\right) , \\
\end{align*}
where $J_{d/2-1}$ is the Bessel function of the first kind of order $d/2-1$ and ${}_{2}F_{1}$ is the hypergeometric function (see \cite[Chapter $15$]{Olver_10}). Now, thanks to \cite[Formula $15.8.1$]{Olver_10} combined with the definition of ${}_{2}F_{1}$, for all $R>0$ and all $x\in \mathbb{R}^d$ with $x \ne 0$, 
\begin{align*}
{}_{2}F_{1}\left(\frac{d+1}{2}, \frac{d}{2}+1;\frac{d}{2};- R^2 \|x\|^2\right) = - \frac{1}{d} \dfrac{R^2 \|x\|^2-d}{(1+R^2 \|x\|^2)^{\frac{d+1}{2}+1}}. 
\end{align*} 
Then, for all $x \in \mathbb{R}^d$ such that $x \ne 0$, 
\begin{align*}
\underset{R \rightarrow +\infty}{\lim}  \int_{\mathbb{R}^d} e^{ - \frac{\|\xi_1\|}{R}} e^{i \langle x ; \xi_1 \rangle} \|\xi_1\| \frac{d\xi_1}{(2\pi)^d} = - \frac{2}{\pi^{\frac{d}{2}}} \dfrac{\Gamma\left(d+1\right)}{d 2^d \Gamma(\frac{d}{2})} \frac{1}{\|x\|^{d+1}}.
\end{align*}
The conclusion of the proof of the lemma easily follows. 
\end{proof}

\begin{lem}\label{lem:one_dimension_explicit_comp}
Let $p_1^{\operatorname{rot}}$ be the density of the standard Cauchy probability measure on $\mathbb{R}$ and let $g_{R}$ be defined, for all $x \in \mathbb{R}$ and all $R>0$, by
\begin{align*}
g_R(x) = x \exp\left(- \frac{x^2}{R^2}\right).
\end{align*}
Let $\mathcal{A}_1^{\operatorname{rot}}$ be the square root of the Laplace operator. Then, for all $x \in \mathbb{R}$ and all $R\geq 1$, 
\begin{align*}
\left| \mathcal{A}_1^{\operatorname{rot}}(g_R p_1^{\operatorname{rot}})(x) \right| \leq \dfrac{C}{\left(1+|x|\right)^2},  
\end{align*}
for some positive constant $C$ which does not depend on $R$. Moreover, for all $x \in \mathbb{R}$ with $x \ne 0$ and all $R \geq 1$, 
\begin{align}\label{eq:representation_1D_exact}
\mathcal{A}_1^{\operatorname{rot}} \left(g_R p_1^{\operatorname{rot}}\right)(x) & = \frac{1}{x^2 (\pi)^{\frac{3}{2}}} \exp\left(\frac{1}{R^2}\right) \bigg[  \int_0^{+\infty} \xi G_R(\xi)\sin\left(x \xi\right)d\xi \nonumber \\
&\quad\quad + 2\int_0^{+\infty} H_R(\xi) \sin\left(x \xi\right) d\xi \nonumber \\
&\quad\quad + 2 F \left(\frac{x}{R}\right) - 2\left(\frac{x}{R}\right) + 4 \left(\frac{x}{R}\right)^2 F\left(\frac{x}{R}\right) \bigg] ,  
\end{align}
where, for all $\xi \in \mathbb{R}$ and all $R>0$, 
\begin{align*}
& G_R(\xi) :=  \exp(-\xi)\int_{\frac{-R\xi}{2} + \frac{1}{R}}^{+\infty}\exp\left( - \omega^2\right) d\omega -\exp(\xi)\int_{\frac{R\xi}{2} + \frac{1}{R}}^{+\infty}\exp\left( - \omega^2\right) d\omega, \\
& H_R(\xi) := - \exp(-\xi)\int_{\frac{-R\xi}{2} + \frac{1}{R}}^{+\infty}\exp\left( - \omega^2\right) d\omega  -\exp(\xi)\int_{\frac{R\xi}{2} + \frac{1}{R}}^{+\infty}\exp\left( - \omega^2\right) d\omega,
\end{align*}
and $F$ is the Dawson's integral function. 
\end{lem}

\begin{proof}
By Fourier inversion formula, for all $x \in \mathbb{R}$ and all $R>0$, 
\begin{align*}
\mathcal{A}_1^{\operatorname{rot}}(g_R p_1^{\operatorname{rot}})(x) = \frac{1}{2\pi} \int_{\mathbb{R}} \mathcal{F} \left(g_R p_1^{\operatorname{rot}}\right)(\xi) e^{i \xi x} (- |\xi|) d\xi. 
\end{align*}
So, the first step is to compute the Fourier transform of the function $g_R p_1^{\operatorname{rot}}$. But, by standard Fourier analysis, this is (up to some constant) equal to the convolution product of the Fourier transforms of the functions $p_1^{\operatorname{rot}}$ and $g_R$. Now, for all $\xi \in \mathbb{R}$, 
\begin{align*}
\mathcal{F}(p_1^{\operatorname{rot}})(\xi) = \exp\left( - |\xi|\right) , \quad \mathcal{F}(g_R)(\xi) = - R^3\frac{(i\xi)}{2} \sqrt{\pi} e^{- \frac{R^2\xi^2}{4}}.
\end{align*}
Thus, for all $\xi \in \mathbb{R}$ and all $R>0$,
\begin{align*}
\left(\mathcal{F}(g_R) \ast \mathcal{F}(p_1^{\operatorname{rot}})\right)(\xi) & = \int_{\mathbb{R}} \exp\left( - |\xi - \zeta|\right) \mathcal{F}(g_R)(\zeta)d\zeta , \\
& = - R^3 \frac{i \sqrt{\pi}}{2} \int_{\mathbb{R}} \exp\left( - |\xi - \zeta|\right) \zeta \exp\left( - \frac{R^2\zeta^2}{4}\right)  d\zeta. 
\end{align*}
At this point, let us cut the previous integral into two pieces: $\{ \zeta \geq \xi\}$ and $\{ \zeta < \xi\}$. Then, for all $\xi \in \mathbb{R}$ and all $R>0$, 
\begin{align*}
 \int_{\mathbb{R}} \exp\left( - |\xi - \zeta|\right) \zeta \exp\left( - \frac{R^2\zeta^2}{4}\right)  d\zeta & =  \int_{\{ \zeta \geq \xi\}} \exp\left( - |\xi - \zeta|\right) \zeta \exp\left( - \frac{R^2\zeta^2}{4}\right)  d\zeta \\
&\quad\quad +  \int_{\{ \zeta < \xi\}} \exp\left( - |\xi - \zeta|\right) \zeta \exp\left( - \frac{R^2\zeta^2}{4}\right)  d\zeta , \\
& = \exp(\xi) \int_{\{ \zeta \geq \xi\}} \exp\left( - \zeta \right) \zeta \exp\left( - \frac{R^2\zeta^2}{4}\right)  d\zeta \\
&\quad\quad + \exp \left(-\xi\right)  \int_{\{ \zeta < \xi\}} \exp\left(\zeta\right) \zeta \exp\left( - \frac{R^2\zeta^2}{4}\right)  d\zeta.
\end{align*}
By standard computations, for all $R>0$ and all $\xi \in \mathbb{R}$, 
\begin{align*}
\int_{\xi}^{+\infty} \exp\left( - \zeta \right) \zeta \exp\left( - \frac{R^2\zeta^2}{4}\right)  d\zeta  & = \int_{\xi}^{+\infty} \zeta \exp\left( - \left(\dfrac{R^2 \zeta^2}{4} + 2 \frac{1}{R} \frac{R\zeta}{2} \right)\right) d\zeta , \\
& = \exp\left(\frac{1}{R^2}\right) \int_{\xi}^{+\infty} \zeta \exp\left( - \left( \frac{R\zeta}{2} + \frac{1}{R} \right)^2\right) d\zeta , \\
& = \left(\frac{2}{R} \right)^2 \exp\left(\frac{1}{R^2}\right) \int_{\frac{R\xi}{2} + \frac{1}{R}}^{+\infty}\left(\omega - \frac{1}{R} \right) \exp \left( - \omega^2\right) d\omega , \\
& = \left(\frac{2}{R} \right)^2 \exp\left(\frac{1}{R^2}\right) \int_{\frac{R\xi}{2} + \frac{1}{R}}^{+\infty} \omega \exp\left( - \omega^2\right) d\omega \\
&\quad\quad -  \frac{1}{R}\left(\frac{2}{R} \right)^2 \exp\left(\frac{1}{R^2}\right) \int_{\frac{R\xi}{2} + \frac{1}{R}}^{+\infty}\exp\left( - \omega^2\right) d\omega.
\end{align*}
Now, for all $b \geq 0$, 
\begin{align*}
\int_b^{+\infty} \exp\left(-z^2\right) dz = \frac{1}{2} \Gamma\left(\frac{1}{2} , b^2\right), \quad \int_b^{+\infty} z \exp\left(-z^2\right) dz = \frac{1}{2} \Gamma\left(1 , b^2\right) = \frac{1}{2} \exp\left(-b^2\right), 
\end{align*} 
and, for all $b \leq 0$, 
\begin{align*}
\int_b^{+\infty} \exp\left(-z^2\right) dz = \frac{\sqrt{\pi}}{2} + \frac{1}{2}\gamma\left(\frac{1}{2} , b^2\right),
\end{align*}
 where, for all $z>0$ and all $\beta>0$,
 \begin{align*}
 \Gamma\left(\beta , z\right) = \int_z^{+\infty} t^{\beta-1} e^{-t} dt , \quad \gamma(\beta,z) = \int_0^z t^{\beta-1} e^{-t} dt. 
 \end{align*} 
 Thus, for all $\xi \in \mathbb{R}$ and all $R>0$, 
\begin{align*}
\int_{\xi}^{+\infty} \exp\left( - \zeta \right) \zeta \exp\left( - \frac{R^2\zeta^2}{4}\right)  d\zeta  & = \frac{1}{2} \left(\frac{2}{R} \right)^2 \exp\left(\frac{1}{R^2}\right) \exp\left( - \left(\frac{R\xi}{2} + \frac{1}{R}\right)^2\right)  \\
&\quad\quad -  \frac{1}{R}\left(\frac{2}{R} \right)^2 \exp\left(\frac{1}{R^2}\right) \int_{\frac{R\xi}{2} + \frac{1}{R}}^{+\infty}\exp\left( - \omega^2\right) d\omega.
\end{align*}
Now, for the second integral, for all $\xi \in \mathbb{R}$ and all $R>0$, 
\begin{align*}
 \int_{-\infty}^{\xi} \exp\left(\zeta\right) \zeta \exp\left( - \frac{R^2\zeta^2}{4}\right)  d\zeta & = - \int_{- \xi}^{+\infty} \exp\left(- \zeta\right) \zeta \exp\left( - \frac{R^2\zeta^2}{4}\right)  d\zeta , \\
 & = - \bigg\{\frac{1}{2} \left(\frac{2}{R} \right)^2 \exp\left(\frac{1}{R^2}\right) \exp\left( - \left(\frac{-R\xi}{2} + \frac{1}{R}\right)^2\right)  \\
&\quad\quad -  \frac{1}{R}\left(\frac{2}{R} \right)^2 \exp\left(\frac{1}{R^2}\right) \int_{\frac{-R\xi}{2} + \frac{1}{R}}^{+\infty}\exp\left( - \omega^2\right) d\omega
\bigg\}. 
\end{align*}
So, the contribution of the first integral to this convolution product is: for all $\xi \in \mathbb{R}$ and all $R>0$, 
\begin{align*}
\exp(\xi)\int_{\xi}^{+\infty} \exp\left( - \zeta \right) \zeta \exp\left( - \frac{R^2\zeta^2}{4}\right)  d\zeta & = \frac{e^{\xi}}{2} \left(\frac{2}{R} \right)^2 \exp\left(\frac{1}{R^2}\right) \exp\left( - \left(\frac{R\xi}{2} + \frac{1}{R}\right)^2\right)  \\
&\quad\quad -  \frac{e^{\xi}}{R}\left(\frac{2}{R} \right)^2 \exp\left(\frac{1}{R^2}\right) \int_{\frac{R\xi}{2} + \frac{1}{R}}^{+\infty}\exp\left( - \omega^2\right) d\omega , \\
& = \frac{1}{2} \left(\frac{2}{R} \right)^2 \exp\left( - \frac{R^2\xi^2}{4}\right)  \\
&\quad\quad -  \frac{e^{\xi}}{R}\left(\frac{2}{R} \right)^2 \exp\left(\frac{1}{R^2}\right) \int_{\frac{R\xi}{2} + \frac{1}{R}}^{+\infty}\exp\left( - \omega^2\right) d\omega
\end{align*}
Similarly, for the second integral, 
\begin{align*} 
\exp\left( - \xi\right) \int_{-\infty}^{\xi} \exp\left(\zeta\right) \zeta \exp\left( - \frac{R^2\zeta^2}{4}\right)  d\zeta & =- \bigg\{\frac{1}{2} \left(\frac{2}{R} \right)^2 \exp\left( - \dfrac{R^2\xi^2}{4}\right)  \\
&\quad\quad -  \frac{\exp(-\xi)}{R}\left(\frac{2}{R} \right)^2 \exp\left(\frac{1}{R^2}\right) \int_{\frac{-R\xi}{2} + \frac{1}{R}}^{+\infty}\exp\left( - \omega^2\right) d\omega
\bigg\}. 
\end{align*}
Thus, summing these two contributions gives: for all $\xi \in \mathbb{R}$ and all $R>0$, 
\begin{align*}
 \int_{\mathbb{R}} \exp\left( - |\xi - \zeta|\right) \zeta \exp\left( - \frac{R^2\zeta^2}{4}\right)  d\zeta & =\frac{\exp(-\xi)}{R}\left(\frac{2}{R} \right)^2 \exp\left(\frac{1}{R^2}\right) \int_{\frac{-R\xi}{2} + \frac{1}{R}}^{+\infty}\exp\left( - \omega^2\right) d\omega \\
 &\quad\quad - \frac{\exp(\xi)}{R}\left(\frac{2}{R} \right)^2 \exp\left(\frac{1}{R^2}\right) \int_{\frac{R\xi}{2} + \frac{1}{R}}^{+\infty}\exp\left( - \omega^2\right) d\omega.
\end{align*} 
Then, for all $\xi \in \mathbb{R}$ and all $R>0$, 
\begin{align}\label{eq:convolution_gR_p1}
\left(\mathcal{F}(g_R) \ast \mathcal{F}(p_1^{\operatorname{rot}})\right)(\xi) & = - \frac{i \sqrt{\pi}}{2}\bigg[4 \exp(-\xi) \exp\left(\frac{1}{R^2}\right) \int_{\frac{-R\xi}{2} + \frac{1}{R}}^{+\infty}\exp\left( - \omega^2\right) d\omega \nonumber \\
 &\quad\quad - 4\exp(\xi) \exp\left(\frac{1}{R^2}\right) \int_{\frac{R\xi}{2} + \frac{1}{R}}^{+\infty}\exp\left( - \omega^2\right) d\omega \bigg].
\end{align}
Let us observe that, for all $\xi>0$, 
\begin{align*}
\underset{R \longrightarrow+\infty}{\lim}\left(\mathcal{F}(g_R) \ast \mathcal{F}(p_1^{\operatorname{rot}})\right)(\xi) = - 2i \sqrt{\pi}\exp(-\xi) \int_{-\infty}^{+\infty} \exp(-\omega^2) d\omega,
\end{align*}
 and, for all $\xi<0$, 
 \begin{align*}
\underset{R \longrightarrow+\infty}{\lim}\left(\mathcal{F}(g_R) \ast \mathcal{F}(p_1^{\operatorname{rot}})\right)(\xi) =  2i \sqrt{\pi}\exp(\xi) \int_{-\infty}^{+\infty} \exp(-\omega^2) d\omega.
\end{align*}
Thus, for all $\xi \in \mathbb{R}$ with $\xi \ne 0$, 
\begin{align*}
\underset{R \longrightarrow+\infty}{\lim}\left(\mathcal{F}(g_R) \ast \mathcal{F}(p_1^{\operatorname{rot}})\right)(\xi)  = -2i\pi \operatorname{sign}\left(\xi\right)\exp\left(- |\xi|\right). 
\end{align*}
Now, it is clear that the function $L$ defined, for all $x \in \mathbb{R}$ with $x \ne 0$, by
\begin{align*}
L(x):= \int_{\mathbb{R}} \operatorname{sign}\left(\xi\right)\exp\left(- |\xi|\right) |\xi| e^{i x\xi} d\xi = \dfrac{4ix}{(1+x^2)^2},
\end{align*}
decays as $1/|x|^3$ as $|x|$ tends to $+\infty$. Let us consider a similar integral for $R>£0$ fixed. Namely, for all $x \in \mathbb{R}$ with $x \ne 0$, 
\begin{align*}
L_R(x) := \int_{\mathbb{R}} |\xi| G_R(\xi) e^{ix \xi} d\xi, 
\end{align*} 
where $G_R$ is defined, for all $\xi \in \mathbb{R}$, by
\begin{align*}
G_R(\xi) = \exp(-\xi)\int_{\frac{-R\xi}{2} + \frac{1}{R}}^{+\infty}\exp\left( - \omega^2\right) d\omega -\exp(\xi)\int_{\frac{R\xi}{2} + \frac{1}{R}}^{+\infty}\exp\left( - \omega^2\right) d\omega.
\end{align*}
Then, for all $x \in \mathbb{R}$ with $x \ne 0$ and all $R>0$, 
\begin{align*}
L_R(x) & = \int_{0}^{+\infty} \xi G_R(\xi) e^{ix \xi} d\xi  -  \int_{-\infty}^{0} \xi G_R(\xi) e^{ix \xi} d\xi , \\
& = \int_{0}^{+\infty} \xi G_R(\xi) e^{ix \xi} d\xi + \int_{0}^{+\infty} \xi G_R(- \xi) e^{-ix \xi} d\xi ,  \\
& = 2i \int_0^{+\infty} \xi G_R(\xi) \sin\left(x \xi\right) d\xi,
\end{align*}
where we have used the fact that the function $G_R$ is odd on $\mathbb{R}$, for all $R>0$. Now, the strategy is to obtain a bound on the previous integral as $|x|$ tends to $+\infty$ which is uniform in $R$. In order to do so, let us perform several integration by parts. Then, for all $x \in \mathbb{R}$ with $x \ne 0$ and all $R>0$,
\begin{align*}
\int_0^{+\infty} \xi G_R(\xi) \sin\left(x \xi\right) d\xi & = - \dfrac{1}{x}\int_0^{+\infty} \xi G_R(\xi) \dfrac{d}{d\xi}\left(\cos\left(x \xi\right) \right)d\xi , \\
& =  \dfrac{1}{x}\int_0^{+\infty} \dfrac{d}{d\xi} \left(\xi G_R(\xi)\right) \cos\left(x \xi\right) d\xi , \\
& = \frac{1}{x} \int_0^{+\infty} \left(G_R(\xi) + \xi G_R'(\xi)\right) \cos\left(x \xi\right) d\xi , \\
& = \frac{1}{x^2} \int_0^{+\infty}  \left(G_R(\xi) + \xi G_R'(\xi)\right) \dfrac{d}{d\xi} \left(\sin\left(x \xi\right)\right) d\xi , \\
&= - \frac{1}{x^2} \int_0^{+\infty} \left(2G'_R(\xi) + \xi G''_R(\xi)\right) \sin\left(x \xi\right)d\xi , \\
&= \frac{1}{x^3} \int_0^{+\infty} \left(2G'_R(\xi) + \xi G''_R(\xi)\right) \frac{d}{d\xi}\left(\cos\left(x \xi\right)-1\right)d\xi , \\
&= -\frac{1}{x^3} \int_0^{+\infty}\frac{d}{d\xi}\left(2G'_R(\xi) + \xi G''_R(\xi)\right) (\cos\left(x \xi\right)-1)d\xi \\
&\quad\quad + \frac{1}{x^3} \left[(2G'_R(\xi) + \xi G''_R(\xi))(\cos\left(x \xi\right)-1)\right]_0^{+\infty} \\
& = -\frac{1}{x^3} \int_0^{+\infty}\left(3G''_R(\xi) + \xi G'''_R(\xi)\right) (\cos\left(x \xi\right)-1)d\xi.
\end{align*}  
At this point, let us compute the first three derivatives of the function $G_R$, for all $R>0$. Now, for all $\xi \geq 0$ and all $R>0$, 
\begin{align}\label{eq:derivative_1}
G'_R(\xi) &=- \exp(-\xi)\int_{\frac{-R\xi}{2} + \frac{1}{R}}^{+\infty}\exp\left( - \omega^2\right) d\omega  -\exp(\xi)\int_{\frac{R\xi}{2} + \frac{1}{R}}^{+\infty}\exp\left( - \omega^2\right) d\omega \nonumber \\
&\quad\quad +R \exp\left(- \left(\frac{R^2\xi^2}{4} + \frac{1}{R^2}\right)\right),
\end{align}
\begin{align}\label{eq:derivative_2}
G''_R(\xi) & = \exp(-\xi)\int_{\frac{-R\xi}{2} + \frac{1}{R}}^{+\infty}\exp\left( - \omega^2\right) d\omega -\exp(\xi)\int_{\frac{R\xi}{2} + \frac{1}{R}}^{+\infty}\exp\left( - \omega^2\right) d\omega \nonumber \\
&\quad\quad - R \left(\dfrac{R^2\xi}{2}\right) \exp\left(- \left(\dfrac{R^2\xi^2}{4} + \frac{1}{R^2} \right) \right)\nonumber \\
& = G_R(\xi) - R \left(\dfrac{R^2\xi}{2}\right) \exp\left(- \left(\dfrac{R^2\xi^2}{4} + \frac{1}{R^2} \right) \right).
\end{align}
\begin{align}\label{eq:derivative_3}
G'''_R(\xi) = G'_R(\xi) -  \frac{R^3}{2} \exp\left(-\left(\frac{R^2 \xi^2}{2} + \frac{1}{R^2}\right)\right) +R \left(\dfrac{R^2\xi}{2}\right)^2 \exp\left(- \left(\dfrac{R^2\xi^2}{4} + \frac{1}{R^2} \right) \right) .
\end{align}
Next, let us prove that one has an uniform decay of the order $1/|x|^2$. Based on the second integration by parts, for all $x \in \mathbb{R}$ with $x \ne 0$ and all $R>0$, 
\begin{align}\label{eq:two_ipp}
\int_0^{+\infty} \xi G_R(\xi) \sin\left(x \xi\right) d\xi = - \frac{1}{x^2} \int_0^{+\infty} \left(2G'_R(\xi) + \xi G''_R(\xi)\right) \sin\left(x \xi\right)d\xi.
\end{align}
Let us start with the integral with the term $G'_R()$. For all $x \in \mathbb{R} $ with $x \ne 0$ and all $R>0$, 
\begin{align*}
\left| \int_0^{+\infty} G_R'(\xi) \sin(x\xi) d\xi\right| \leq (1) + (2) + (3), 
\end{align*}
with, 
\begin{align*}
&(1) := \left|  \int_0^{+\infty} \exp(-\xi)\left(\int_{\frac{-R\xi}{2} + \frac{1}{R}}^{+\infty}\exp\left( - \omega^2\right) d\omega\right)  \sin(x\xi) d\xi \right| , \\
&(2) := \left|  \int_0^{+\infty} \exp(\xi)\left(\int_{\frac{R\xi}{2} + \frac{1}{R}}^{+\infty}\exp\left( - \omega^2\right) d\omega\right)  \sin(x\xi) d\xi \right|, \\
&(3) :=  \left|  \int_0^{+\infty} R \exp\left(- \left(\frac{R^2\xi^2}{4} + \frac{1}{R^2}\right)\right)  \sin(x\xi) d\xi \right|.
\end{align*}
Now, it is clear that, for all $x \in \mathbb{R}$ with $x \ne 0$ and all $R>0$, $(1)$ is bounded by a positive constant not depending on $R$ neither on $x$.~Regarding $(2)$, by standard Mill's ratio bound, for all $x \in \mathbb{R}$ with $x \ne 0$ and all $R\geq 1$,
\begin{align*}
(2) \leq \int_{0}^{+\infty} \exp\left(- \frac{R^2 \xi^2}{4}\right) d\xi \leq  \int_{0}^{+\infty} \exp\left(- \frac{ \xi^2}{4}\right) d\xi <+\infty. 
\end{align*}
Finally, for $(3)$, for all $x \in \mathbb{R}$ with $x \ne 0$ and all $R\geq 1$,
\begin{align*}
(3) \leq \int_0^{+\infty} R \exp\left( -\frac{R^2 \xi^2}{4}\right) d\xi \leq \int_0^{+\infty} \exp\left( -\frac{\xi^2}{4}\right) d\xi <+\infty. 
\end{align*}
This implies that there exists a positive constant $C$ such that, for all $x \in \mathbb{R}$ with $x \ne 0$ and all $R \geq 1$, 
\begin{align*}
\left| \int_0^{+\infty} G_R'(\xi) \sin(x \xi) d\xi \right| \leq C. 
\end{align*} 
A similar analysis can be performed for the integral with $\xi G''_R(\xi)$ and so one has an uniform decay at the rate $\frac{1}{|x|^2}$ when $|x|$ tends to $+\infty$. Formula \eqref{eq:representation_1D_exact} is a direct consequence of Equation \eqref{eq:two_ipp} and of the definition of the Dawson function together with integration by parts.  
\end{proof}
\noindent
In the next lemma, let us compute integrals related to the Dawson's function $F$ and find an integral representation of the Dawson function $F$. 

\begin{lem}\label{lem:technical_dawson}
Let $F$ be the Dawson's integral function defined, for all $y \geq 0$, by
\begin{align*}
F(y) = \exp(-y^2) \int_0^y e^{t^2} dt, 
\end{align*}
and extended to $\mathbb{R}$ by oddity. Then, 
\begin{align*}
\int_{-\infty}^{+\infty} F(y) \dfrac{e^{-y^2}dy}{y} = \frac{\pi^{\frac{3}{2}}}{4},  \quad \int_{-\infty}^{+\infty} x F(x) \exp\left(-x^2\right)dx = \dfrac{\sqrt{\pi}}{4}.
\end{align*}
Moreover, for all $x \in \mathbb{R}$, 
\begin{align}\label{eq:alternative_representation_Dawson_integral_func}
F(x) = \frac{1}{2} \int_0^{+\infty} e^{-\frac{\xi^2}{4}} \sin \left(x \xi\right) d\xi. 
\end{align}
\end{lem}

\begin{proof}
The proof is a direct application of Abel's continuity theorem together with a series expansion of the function $F$. Combining the representation \cite[$7.5.1$ page $162$]{Olver_10} together with the series expansion \cite[7.6.3 page $162$]{Olver_10}, one gets, for all $y \geq 0$, 
\begin{align}\label{eq:series_expansion_Dawson_Function}
F(y) = \dfrac{\sqrt{\pi}}{2} \sum_{p=0}^{+\infty} \dfrac{(-1)^p y^{2p+1}}{\Gamma\left(p+1+\frac{1}{2}\right)}. 
\end{align}
Now, let $\varepsilon \in (1/2,1)$ and let $F_\varepsilon$ be defined, for all $y \geq 0$, by
\begin{align*}
F_{\varepsilon}(y) = F(\varepsilon y). 
\end{align*}
Then, 
\begin{align*}
\int_{-\infty}^{+\infty} F_\varepsilon(y) \dfrac{e^{-y^2}dy}{y} & = 2 \int_{0}^{+\infty} F_\varepsilon(y) \dfrac{e^{-y^2}dy}{y}  , \\
& = \sqrt{\pi} \int_{0}^{+\infty} \sum_{p=0}^{+\infty} \dfrac{(-1)^p y^{2p+1} \varepsilon^{2p+1}}{\Gamma\left(p+1+\frac{1}{2}\right)} \dfrac{e^{-y^2}dy}{y}.
\end{align*}
But, for all $p \geq 0$, 
\begin{align*}
\int_0^{+\infty} y^{2p} e^{-y^2} dy = \frac{1}{2} \Gamma(p + \frac{1}{2}), 
\end{align*}
and, 
\begin{align*}
\sum_{p = 0}^{+\infty} \dfrac{\Gamma( p + \frac{1}{2})}{\Gamma(p + \frac{1}{2} + 1)} \varepsilon^{2p + 1} < +\infty,
\end{align*}
since $\varepsilon \in (1/2,1)$. 
Thus, for all $\varepsilon \in (1/2,1)$, 
\begin{align*}
\int_{-\infty}^{+\infty} F_\varepsilon(y) \dfrac{e^{-y^2}dy}{y}  = \frac{\sqrt{\pi}}{2} \sum_{p = 0}^{+\infty} \dfrac{(-1)^p \varepsilon^{2p+1}}{p+\frac{1}{2}}.
\end{align*}
Letting $\varepsilon$ tends to $1^-$, 
\begin{align*}
\int_{-\infty}^{+\infty} F(y) \dfrac{e^{-y^2}dy}{y} =  \frac{\sqrt{\pi}}{2} \sum_{p = 0}^{+\infty} \dfrac{(-1)^p}{p+\frac{1}{2}}.
\end{align*}
Finally, using the series expansion of the derivative of the arctan function in $]-1,1[$, one gets,
\begin{align*}
\sum_{p=0}^{+\infty} \dfrac{(-1)^p}{p + \frac{1}{2}} = \frac{\pi}{2}. 
\end{align*}  
The exact theoretical value of the other integral can be proved similarly. Finally, to prove equation \eqref{eq:alternative_representation_Dawson_integral_func}, let us expand in series the right-hand side of equation \eqref{eq:alternative_representation_Dawson_integral_func}. For this purpose, recall that, for all $\theta \in \mathbb{R}$, 
\begin{align}\label{eq:sinus_serie}
\sin \left(\theta\right) = \sum_{p = 0}^{+\infty} \dfrac{(-1)^p \theta^{2p+1}}{(2p+1)!}. 
\end{align} 
Then, for all $x \geq 0$, 
\begin{align*}
\int_0^{+\infty} e^{- \frac{\xi^2}{4}} \sin(x\xi) d\xi & = \sum_{p = 0}^{+\infty} \dfrac{(-1)^p x^{2p+1}}{(2p+1)!} \int_0^{+\infty} e^{- \frac{\xi^2}{4}} \xi^{2p+1}d\xi , \\
& =\sum_{p = 0}^{+\infty} \dfrac{(-1)^p x^{2p+1}2^{2p+2}}{(2p+1)!} \int_0^{+\infty} e^{- \omega^2} \omega^{2p+1}d\omega. 
\end{align*}
But, for all $p \geq 0$,
\begin{align*}
\int_0^{+\infty} e^{-\omega^2} \omega^{2p+1}d \omega = \frac{1}{2}\int_0^{+\infty}e^{-u} u^{p+\frac{1}{2}} \frac{du}{\sqrt{u}} = \frac{p!}{2}. 
\end{align*}
Thus, for all $x \geq 0$, 
\begin{align*}
\frac{1}{2}\int_0^{+\infty} e^{- \frac{\xi^2}{4}} \sin(x\xi) d\xi = \sum_{p = 0}^{+\infty} \dfrac{(-1)^p x^{2p+1}p! 2^{2p}}{(2p+1)!}. 
\end{align*}
Now, by the multiplication formula for the Gamma function (see, e.g., \cite[formula $5.5.6$ page $138$]{Olver_10}), for all $p \geq 0$,
\begin{align*}
\Gamma(p+1) \Gamma\left(p+1+\frac{1}{2}\right) = \sqrt{\pi} 2^{-(2p+1)} \Gamma\left(2(p+1)\right).
\end{align*}
This concludes the proof of the lemma. 
\end{proof}

\begin{lem}\label{lem:reduction_cc}
Let $d \geq 1$ be an integer, let $\mu$ be a non-degenerate probability measure on $\mathbb{R}^d$ and let $\nu$ be a non-degenerate L\'evy measure on $\mathbb{R}^d$ such that $\mu \ast \nu << \mu$. Let $\|\cdot\|_{H+E}$ be the norm defined, for all $f \in \mathcal{C}_b^1(\mathbb{R}^d)$, by
\begin{align}\label{eq:norm_sobolev_nl}
\| f\|^2_{H+E} := \|f\|^2_{L^2(\mu)} + \mathcal{E}_{\nu,\mu}(f,f), 
\end{align}
where $\mathcal{E}_{\nu,\mu}$ is given, for all $f ,g \in \mathcal{C}_b^1(\mathbb{R}^d)$, by 
\begin{align}\label{eq:seminorm_sobolev_nl}
\mathcal{E}_{\nu,\mu}(f,g) = \int_{\mathbb{R}^d} \int_{\mathbb{R}^d} ( f(x+u) - f(x) )( g(x+u) - g(x) ) \nu(du) \mu(dx). 
\end{align}
Then, 
\begin{align}\label{eq:reduction_cc}
\overline{\mathcal{C}_c^{\infty}(\mathbb{R}^d)}^{\|\cdot\|_{H+E}} = \overline{\mathcal{C}^1_b(\mathbb{R}^d)}^{\|\cdot\|_{H+E}}.
\end{align}
\end{lem}

\begin{proof}
\textit{Step 1}. In the first part of the proof, let us reduce the problem to functions which are infinitely differentiable on $\mathbb{R}^d$ such that $\|f\|_{\infty} <+ \infty$ and $\| \nabla(f)\|_{\infty} < +\infty$. Take $f \in \mathcal{C}^1_b(\mathbb{R}^d)$ and let $(\rho_\varepsilon)_{\varepsilon >0}$ be a sequence of standard mollifiers, namely, $\rho$ is an infinitely differentiable function on $\mathbb{R}^d$ with compact support such that $\rho(-x) = \rho(x)$, for all $x \in \mathbb{R}^d$, such that $\operatorname{supp}(\rho) \subset \overline{\mathcal{B}(0,1)}$ and such that 
\begin{align}\label{eq:pm_property}
\int_{\mathbb{R}^d} \rho(x) dx = 1.
\end{align}
Then, for all $\varepsilon>0$, let $\rho_\varepsilon$ be defined, for all $x \in \mathbb{R}^d$, by
\begin{align}\label{def:mollifier_scale_e}
\rho_{\varepsilon}(x) = \frac{1}{\varepsilon^d} \rho\left(\frac{x}{\varepsilon}\right). 
\end{align}
Now, set $f_\varepsilon = f \ast \rho_{\varepsilon}$, where $\ast$ is the standard convolution product of functions. Since $f$ belongs to $\mathcal{C}_b^1(\mathbb{R}^d)$, it is clear that 
\begin{align}\label{eq:uniform_convergence}
\|f_\varepsilon - f \|_{\infty} \longrightarrow 0,
\end{align}
as $\varepsilon$ tends to $0$, which clearly implies that $(f_\varepsilon)_{\varepsilon >0}$ converges to $f$ in $L^2(\mu)$. Moreover, for all $\varepsilon>0$, 
\begin{align*}
\mathcal{E}_{\nu, \mu}(f_\varepsilon - f , f_\varepsilon - f) & = \int_{\mathbb{R}^d} \int_{\mathcal{B}(0,1)} \left| \Delta_u(f_\varepsilon -f)(x) \right|^2 \nu(du) \mu(dx) \\
&\quad\quad\quad +  \int_{\mathbb{R}^d} \int_{\mathcal{B}(0,1)^c} \left| \Delta_u(f_\varepsilon -f)(x) \right|^2 \nu(du) \mu(dx).
\end{align*}
For the second term on the right-hand side of the previous equality, 
\begin{align*}
 \int_{\mathbb{R}^d} \int_{\mathcal{B}(0,1)^c} \left| \Delta_u(f_\varepsilon -f)(x) \right|^2 \nu(du) \mu(dx) \leq 4 \|f_\varepsilon - f\|^2_\infty \nu\left(\mathcal{B}(0,1)^c\right),
\end{align*}
which clearly converges to $0$ as $\varepsilon$ tends to $0$. For the other term, by Taylor's formula, 
\begin{align*}
 \int_{\mathbb{R}^d} \int_{\mathcal{B}(0,1)} \left| \Delta_u(f_\varepsilon -f)(x) \right|^2 \nu(du) \mu(dx) & \leq \int_{\mathbb{R}^d} \int_{\mathcal{B}(0,1)} \|u\|^2 \\
 &\quad \times \left(\int_0^1 \|\nabla(f_\varepsilon -f)(x+tu)\|^2 dt\right) \nu(du) \mu(dx).
\end{align*}
Now, by the Lebesgue dominated convergence theorem, for all $x \in \mathbb{R}^d$, 
\begin{align*}
\underset{\varepsilon \longrightarrow 0}{\lim} \|\nabla(f_\varepsilon)(x) - \nabla(f)(x)\| = 0.
\end{align*}
But, $\|\nabla(f_\varepsilon)(x) - \nabla(f)(x)\| \leq 2 \|\nabla(f)\|_{\infty}$, for all $x \in \mathbb{R}^d$. So, the Lebesgue dominated convergence theorem ensures that
\begin{align*}
 \int_{\mathbb{R}^d} \int_{\mathcal{B}(0,1)} \left| \Delta_u(f_\varepsilon -f)(x) \right|^2 \nu(du) \mu(dx) \longrightarrow 0, 
\end{align*}
as $\varepsilon$ tends to $0$. This concludes the first step of the proof. \\
\textit{Step 2}: Let us assume that $f$ is an infinitely differentiable function on $\mathbb{R}^d$ such that $\|f\|_{\infty}< +\infty$ and such that $\|\nabla(f)\|_{\infty} <+\infty$. Let $\chi$ be an infinitely differentiable even function on $\mathbb{R}^d$ with compact support, such that $\chi(x) = 1$, for all $x \in \overline{\mathcal{B}(0,1)}$, such that $0 \leq \chi \leq 1$ and such that $\chi(x) = 0$, for all $x \in \mathcal{B}(0,2)^c$. Then, for all $R \geq 1$, let $f_R$ be defined, for all $x \in \mathbb{R}^d$, by
\begin{align}\label{eq:smooth_truncation}
f_R(x) = \chi\left(\frac{x}{R}\right) f(x). 
\end{align} 
Clearly, for all $R\geq 1$, $f_R \in \mathcal{C}_c^{\infty}(\mathbb{R}^d)$. Then, 
\begin{align*}
\|f_R - f\|^2_{L^2(\mu)} = \int_{\mathbb{R}^d} \left(1-\chi\left(\frac{x}{R}\right)\right)^2 |f(x)|^2 \mu(dx) \longrightarrow 0, \quad R \longrightarrow +\infty.
\end{align*}
Moreover, for all $R \geq 1$, 
\begin{align*}
\mathcal{E}_{\nu,\mu}(f_R -f, f_R -f) & = \int_{\mathbb{R}^d} \int_{\mathcal{B}(0,1)} |\Delta_u(f_R-f)(x)|^2 \nu(du)\mu(dx) \\
&\quad\quad +\int_{\mathbb{R}^d} \int_{\mathcal{B}(0,1)^c} |\Delta_u(f_R-f)(x)|^2 \nu(du)\mu(dx).
\end{align*}
For the second term on the right-hand side of the previous equality, observe that 
\begin{align*}
\underset{R \longrightarrow +\infty}{\lim} |\Delta_u(f_R-f)(x)|^2 = 0, \quad (x,u) \in \mathbb{R}^{2d}, \quad
|\Delta_u(f_R-f)(x)| \leq 4 \|f\|_{\infty}. 
\end{align*}
Then, the Lebesgue dominated convergence theorem ensures that
\begin{align*}
\underset{R \longrightarrow +\infty}{\lim} \int_{\mathbb{R}^d} \int_{\mathcal{B}(0,1)^c} |\Delta_u(f_R-f)(x)|^2 \nu(du)\mu(dx) = 0.
\end{align*} 
Finally, for all $x \in \mathbb{R}^d$, 
\begin{align*}
\underset{R\rightarrow +\infty}{\lim}\|\nabla(f_R)(x) - \nabla(f)(x)\| = 0, \quad \| \nabla(f_R) - \nabla(f)\|_{\infty} \leq \| \nabla(\chi)\|_{\infty} \|f\|_{\infty} + 2\|\nabla(f)\|_{\infty},
\end{align*}
which ensures that 
\begin{align*}
\underset{R \longrightarrow +\infty}{\lim} \int_{\mathbb{R}^d} \int_{\mathcal{B}(0,1)} |\Delta_u(f_R-f)(x)|^2 \nu(du)\mu(dx) = 0.
\end{align*}
This concludes the proof of the lemma. 
\end{proof}

\begin{lem}\label{lem:reduction_csb}
Let $d \geq 1$ be an integer, let $\mu$ be a non-degenerate probability measure on $\mathbb{R}^d$ and let $\nu$ be a non-degenerate L\'evy measure on $\mathbb{R}^d$ such that $\mu \ast \nu << \mu$. Let $D(\mathcal{E}_{\nu,\mu})$ be defined by 
\begin{align}\label{eq:sobolev_space}
D(\mathcal{E}_{\nu,\mu}) = \{f \in L^2(\mu): \, \mathcal{E}_{\nu,\mu}(f,f) < +\infty\},
\end{align}
where, for all $f,g \in D(\mathcal{E}_{\nu,\mu})$, 
\begin{align*}
\mathcal{E}_{\nu,\mu}(f,g) = \int_{\mathbb{R}^d} \int_{\mathbb{R}^d} (g(x+u)-g(x))(f(x+u)-f(x)) \nu(du) \mu(dx). 
\end{align*}
Then, the set of bounded Borel measurable and compactly supported functions of $D(\mathcal{E}_{\nu,\mu})$ is dense in $D(\mathcal{E}_{\nu,\mu})$ with respect to the norm $\|\cdot\|_{H+E}$ given by \eqref{eq:norm_sobolev_nl}.
\end{lem}

\begin{proof}
\textit{Step 1}: Let $(G_R)_{R>0}$ be defined, for all $R >0$ and all $x \in \mathbb{R}$, by
\begin{align*}
G_R(x) = (-R) \vee (x \wedge R).
\end{align*}
Then, for all $R>0$, $|G_R(x)| \leq R$, for all $x,y \in \mathbb{R}$ and all $R>0$, 
\begin{align*}
|G_R(x)| \leq |x|, \quad |G_R(x) - G_R(y)| \leq |x-y|,
\end{align*}
and, for all $x \in \mathbb{R}$, $\underset{R \longrightarrow +\infty}{\lim} G_R(x) = x$. Let $f \in D(\mathcal{E}_{\nu,\mu})$ and let us consider $(G_R(f))_{R>0}$. Since the form is Markovian, it is clear that $G_R(f) \in D\left(\mathcal{E}_{\nu,\mu}\right)$, for all $R>0$. Now, for all $R>0$
\begin{align*}
\|G_R(f)-f\|^2_{L^2(\mu)} = \int_{\mathbb{R}^d} |G_R(f(x)) -f(x)|^2 \mu(dx).
\end{align*}
But, for $\mu$-a.e. $x \in \mathbb{R}^d$,
\begin{align}
\underset{R \longrightarrow +\infty}{\lim}  |G_R(f(x)) -f(x)|^2 = 0 , \quad  |G_R(f(x)) -f(x) | \leq 2|f(x)|. 
\end{align}
Then, the Lebesgue dominated convergence theorem ensures that 
\begin{align*}
\underset{R \longrightarrow +\infty}{\lim} \|G_R(f)-f\|_{L^2(\mu)} = 0.
\end{align*}
Moreover, since $\mu \ast \nu <<\mu$, $\mu \otimes \nu$-a.e. $(x,u)\in \mathbb{R}^d \times \mathbb{R}^d$, 
\begin{align*}
\underset{R \longrightarrow +\infty}{\lim} \left| \Delta_u(G_R(f))(x) - \Delta_u(f)(x) \right|^2 = 0.
\end{align*} 
Finally, for all $R>0$ and $\mu \otimes \nu$-a.e. $(x,u) \in \mathbb{R}^d \times \mathbb{R}^d$, 
\begin{align*}
|\Delta_u(G_R(f))(x) - \Delta_u(f)(x)| \leq 2 |f(x+u) -f(x)|.
\end{align*}
Thus, the Lebesgue dominated convergence theorem ensures that, 
\begin{align*}
\underset{R \longrightarrow +\infty}{\lim} \mathcal{E}_{\nu,\mu}(G_R(f)-f , G_R(f)-f) = 0,
\end{align*} 
and so we are done with the first step.\\
\textit{Step 2}: Now, let $f \in D\left(\mathcal{E}_{\nu,\mu}\right)$ be such that $\|f\|_\infty <+\infty$. Let $\chi$ be an infinitely differentiable even function on $\mathbb{R}^d$ with compact support, such that $\chi(x) = 1$, for all $x \in \overline{\mathcal{B}(0,1)}$, such that $0 \leq \chi \leq 1$ and such that $\chi(x) = 0$, for all $x \in \mathcal{B}(0,2)^c$. Let $(f_R)_{R \geq 1}$ be defined, for all $R \geq 1$ and $\mu$-a.e. $x \in \mathbb{R}^d$, by 
\begin{align}\label{eq:smooth_trunc_sobolev_nl}
f_R(x) = \chi\left(\frac{x}{R}\right) f(x).  
\end{align}
Clearly, for all $R \geq 1$, $f_R$ is a bounded Borel measurable and compactly supported function with $f_R \in L^2(\mu)$. Moreover, 
\begin{align*}
\|f_R -f\|_{L^2(\mu)} \longrightarrow 0,
\end{align*}
as $R$ tends to $+\infty$. Now, for all $R \geq 1$ and $\mu \otimes \nu$-a.e. $(x,u) \in \mathbb{R}^d \times \mathbb{R}^d$, 
\begin{align*}
\left| \Delta_u(f_R)(x)\right| \leq \left| f(x+u) - f(x) \right| + |f(x)|\left| \chi\left(\frac{x+u}{R}\right) - \chi\left(\frac{x}{R}\right)\right|.
\end{align*}
Thus, for $u \in \mathcal{B}(0,1)$, the previous bound boils down to
\begin{align}\label{ineq:small_u}
\left| \Delta_u(f_R)(x)\right| \leq \left| f(x+u) - f(x) \right| + \frac{\|u\|}{R} \|\nabla(\chi)\|_{\infty} |f(x)|,
\end{align}
whereas, for $u \in \mathcal{B}(0,1)^c$, 
\begin{align}\label{ineq:big_u}
\left| \Delta_u(f_R)(x)\right| \leq \left| f(x+u) - f(x) \right| + 2 |f(x)|. 
\end{align}
Then, \eqref{ineq:small_u} and \eqref{ineq:big_u} ensure that $\mathcal{E}_{\nu,\mu}(f_R,f_R)<+\infty$, for all $R \geq 1$. Finally, for $\mu\otimes\nu$-a.e. $(x,u) \in \mathbb{R}^d \times \mathbb{R}^d$, 
\begin{align*}
\underset{R \longrightarrow +\infty}{\lim} \left| \Delta_u(f_R)(x) - \Delta_u(f)(x)\right| = 0, 
\end{align*}
and, 
\begin{align*}
\left| \Delta_u(f_R)(x) - \Delta_u(f)(x) \right| \leq 2\left| f(x+u) - f(x) \right| + \|u\|\|\nabla(\chi)\|_{\infty} |f(x)|\bbone_{\|u\|\leq1} + 2 |f(x)|\bbone_{\|u\|>1}. 
\end{align*}
Then, the Lebesgue dominated convergence theorem ensures that 
\begin{align*}
\underset{R\longrightarrow +\infty}\lim \mathcal{E}_{\nu,\mu}(f_R-f,f_R-f) = 0. 
\end{align*} 
This concludes the proof of the lemma.
\end{proof}

\begin{prop}\label{prop:markov_uniqueness_rot_inv}
Let $d \geq 1$ be an integer, let $\alpha \in (0,2)$, let $\mu_\alpha^{\operatorname{rot}}$ be the rotationally invariant $\alpha$-stable probability measure on $\mathbb{R}^d$ which Fourier transform is given by \eqref{eq:characteristic_function_rot_inv} with L\'evy measure denoted by $\nu_\alpha^{\operatorname{rot}}$. Let $D\left(\mathcal{E}_{\nu_\alpha^{\operatorname{rot}}, \mu_\alpha^{\operatorname{rot}}}\right)$ be defined by 
\begin{align}\label{eq:def_sobolev_space_rot_inv}
D\left(\mathcal{E}_{\nu_\alpha^{\operatorname{rot}}, \mu_\alpha^{\operatorname{rot}}}\right) = \{f \in L^2(\mu_\alpha^{\operatorname{rot}}) :\, \mathcal{E}_{\nu_\alpha^{\operatorname{rot}}, \mu_\alpha^{\operatorname{rot}}}(f,f)<+\infty\},
\end{align}
where, for all $f,g \in D\left(\mathcal{E}_{\nu_\alpha^{\operatorname{rot}}, \mu_\alpha^{\operatorname{rot}}}\right)$, 
\begin{align}
\mathcal{E}_{\nu_\alpha^{\operatorname{rot}}, \mu_\alpha^{\operatorname{rot}}}(f,g) = \int_{\mathbb{R}^d} \int_{\mathbb{R}^d} (f(x+u)-f(x))(g(x+u)-g(x))\nu_{\alpha}^{\operatorname{rot}}(du) \mu_\alpha^{\operatorname{rot}}(dx).
\end{align}
Then, 
\begin{align}\label{eq:markov_uniqueness_rot_inv}
D\left(\mathcal{E}_{\nu_\alpha^{\operatorname{rot}}, \mu_\alpha^{\operatorname{rot}}}\right) = \overline{\mathcal{C}_c^{\infty}(\mathbb{R}^d)}^{\| \cdot \|_{H+E}},
\end{align}
where, for all $f \in D\left(\mathcal{E}_{\nu_\alpha^{\operatorname{rot}}, \mu_\alpha^{\operatorname{rot}}}\right)$,
\begin{align}\label{eq:norm_sobolev_rot_invariant}
\| f \|^2_{H+E} = \|f\|^2_{L^2(\mu_\alpha^{\operatorname{rot}})} + \mathcal{E}_{\nu_\alpha^{\operatorname{rot}}, \mu_\alpha^{\operatorname{rot}}}(f,f). 
\end{align}
\end{prop}

\begin{proof}
Thanks to Lemma \ref{lem:reduction_csb}, let us consider $f \in D\left(\mathcal{E}_{\nu_\alpha^{\operatorname{rot}}, \mu_\alpha^{\operatorname{rot}}}\right)$  bounded and with compact support which is denoted by $K_f$ in the sequel. Let $(\rho_\varepsilon)_{\varepsilon>0}$ be a sequence of standard mollifiers as in the proof of lemma \ref{lem:reduction_cc} and let $(f_\varepsilon)_{\varepsilon>0}$ be defined, for all $\varepsilon>0$ and all $x \in \mathbb{R}^d$, by
\begin{align*}
f_\varepsilon(x) = \int_{\mathbb{R}^d} \rho_\varepsilon(x-z)f(z)dz. 
\end{align*}
Since $f$ is bounded on $\mathbb{R}^d$ and with compact support, it is clear that $f \in L^2(\mathbb{R}^d,dx)$. Moreover, for all $\varepsilon>0$
\begin{align*}
f_\varepsilon \in \mathcal{C}_c^{\infty}(\mathbb{R}^d), \quad \|f_\varepsilon\|_\infty \leq \|f\|_{\infty}, \quad f_\varepsilon \longrightarrow f, 
\end{align*}
in $L^2(\mathbb{R}^d,dx)$, as $\varepsilon$ tends to $0$. Thus, along some subsequence $(\varepsilon_n)_{n \geq 1}$, $\ell$-a.e. $x \in \mathbb{R}^d$
\begin{align*}
f_{\varepsilon_n}(x) \longrightarrow f(x), \quad n \longrightarrow +\infty,
\end{align*}
where $\ell$ denotes the $d$-dimensional Lebesgue measure on $\mathbb{R}^d$. Since $\mu_\alpha^{\operatorname{rot}} << \ell$, $\mu_\alpha^{\operatorname{rot}}$-a.e. $x \in \mathbb{R}^d$, 
\begin{align}\label{eq:convergence_almost_everywhere}
f_{\varepsilon_n}(x) \longrightarrow f(x), \quad n \longrightarrow +\infty. 
\end{align}
Moreover, for all $n \geq 1$ and $\mu_\alpha^{\operatorname{rot}}$-a.e. $x \in \mathbb{R}^d$, 
\begin{align}\label{eq:domination}
\left| f_{\varepsilon_n}(x) - f(x) \right| \leq 2\|f\|_{\infty}.
\end{align}
Then, thanks to \eqref{eq:convergence_almost_everywhere} and to \eqref{eq:domination}, the Lebesgue dominated convergence theorem ensures that
\begin{align}\label{eq:L2_convergence}
\|f_{\varepsilon_n} - f\|_{L^2(\mu_\alpha^{\operatorname{rot}})} \longrightarrow 0,
\end{align} 
 as $n$ tends to $+\infty$. Now, for all $n \geq 1$, 
 \begin{align*}
 \mathcal{E}_{\nu_\alpha^{\operatorname{rot}}, \mu_\alpha^{\operatorname{rot}}} (f_{\varepsilon_n} , f_{\varepsilon_n}) & = \int_{\mathbb{R}^d} \int_{\mathcal{B}(0,1)^c} \left| f_{\varepsilon_n}(x+u) - f_{\varepsilon_n}(x)\right|^2 \nu_{\alpha}^{\operatorname{rot}}(du) \mu_\alpha^{\operatorname{rot}}(dx) \\
 &\quad\quad + \int_{\mathbb{R}^d} \int_{\mathcal{B}(0,1)} \left| f_{\varepsilon_n}(x+u) - f_{\varepsilon_n}(x)\right|^2 \nu_{\alpha}^{\operatorname{rot}}(du) \mu_\alpha^{\operatorname{rot}}(dx). 
 \end{align*}
 For the first term on the right-hand side of the previous equality, for all $n \geq 1$, 
 \begin{align*}
 \int_{\mathbb{R}^d} \int_{\mathcal{B}(0,1)^c} \left| f_{\varepsilon_n}(x+u) - f_{\varepsilon_n}(x)\right|^2 \nu_{\alpha}^{\operatorname{rot}}(du) \mu_\alpha^{\operatorname{rot}}(dx) \leq 4 \|f\|^2_{\infty} \int_{\mathcal{B}(0,1)^c} \nu_\alpha^{\operatorname{rot}}(du) <+\infty. 
 \end{align*}
 For the second term, using the definition of $f_{\varepsilon_n}$ and Jensen's inequality, for all $n \geq 1$, 
 \begin{align*}
 \int_{\mathbb{R}^d} \int_{\mathcal{B}(0,1)} \left| f_{\varepsilon_n}(x+u) - f_{\varepsilon_n}(x)\right|^2 \nu_{\alpha}^{\operatorname{rot}}(du) \mu_\alpha^{\operatorname{rot}}(dx) & \leq \int_{\mathbb{R}^d} \int_{\mathcal{B}(0,1)} \nu_{\alpha}^{\operatorname{rot}}(du) \mu_\alpha^{\operatorname{rot}}(dx) \\
 &\quad \times \left| \int_{\mathbb{R}^d} (f(x+u-z)-f(x-z)) \rho_{\varepsilon_n}(z)dz\right|^2 ,  \\
 & \leq \int_{\mathbb{R}^d} \int_{\mathcal{B}(0,1)} \nu_{\alpha}^{\operatorname{rot}}(du) p_\alpha^{\operatorname{rot}}(x)dx \\
 &\quad \times \int_{\mathbb{R}^d} (f(x+u-z)-f(x-z))^2 \rho_{\varepsilon_n}(z)dz , \\
 & \leq \int_{\mathbb{R}^d} \int_{\mathcal{B}(0,1)} \nu_{\alpha}^{\operatorname{rot}}(du) \frac{\mu_\alpha^{\operatorname{rot}}(dy)}{p_\alpha^{\operatorname{rot}}(y)}  \\
 &\quad \times \int_{\mathbb{R}^d} (f(y+u)-f(y))^2 p_\alpha^{\operatorname{rot}}(y+z) \rho_{\varepsilon_n}(z)dz.
 \end{align*}
Now, for all $x \in \mathbb{R}^d$ and all $z \in \mathbb{R}^d$, 
\begin{align*}
\frac{p_\alpha^{\operatorname{rot}}(x+z)}{p_\alpha^{\operatorname{rot}}(x)} \leq C_{\alpha,d} \left(1+\|z\|\right)^{\alpha+d},
\end{align*}
for some positive constant $C_{\alpha,d}>0$ depending only on $\alpha$ and on $d$. Thus, for all $n \geq 1$, 
\begin{align*}
 \int_{\mathbb{R}^d} \int_{\mathcal{B}(0,1)} \left| f_{\varepsilon_n}(x+u) - f_{\varepsilon_n}(x)\right|^2 \nu_{\alpha}^{\operatorname{rot}}(du) \mu_\alpha^{\operatorname{rot}}(dx) & \leq \int_{\mathbb{R}^d} \int_{\mathcal{B}(0,1)} (f(y+u)-f(y))^2 \nu_{\alpha}^{\operatorname{rot}}(du)\mu_\alpha^{\operatorname{rot}}(dy) \\
 &\quad \quad \times C_{\alpha,d} \int_{\mathbb{R}^d} \left(1+\|z\|\right)^{\alpha+d} \rho_{\varepsilon_n}(z)dz \\
 & \leq \int_{\mathbb{R}^d} \int_{\mathcal{B}(0,1)} (f(y+u)-f(y))^2 \nu_{\alpha}^{\operatorname{rot}}(du)\mu_\alpha^{\operatorname{rot}}(dy) \\
 &\quad \quad \times C_{\alpha,d} \int_{\mathbb{R}^d} \left(1+ \varepsilon_n \|z\|\right)^{\alpha+d} \rho(z)dz \\
 & \leq C_{\alpha,d,\rho} \mathcal{E}_{\nu_\alpha^{\operatorname{rot}}, \mu_\alpha^{\operatorname{rot}}}(f,f),
\end{align*}
for some $C_{\alpha,d,\rho}>0$ depending only on $\alpha$, on $d$ and on $\rho$. Thus, 
\begin{align}\label{eq:uniform_boundedness_seminorm}
\underset{n \geq 1}{\sup}\, \left(\mathcal{E}_{\nu_\alpha^{\operatorname{rot}}, \mu_\alpha^{\operatorname{rot}}}(f_{\varepsilon_n}, f_{\varepsilon_n}) \right) < + \infty. 
\end{align} 
Then, by the Banach-Saks property, there exists a further subsequence $(\eta_n)_{n\geq 1}$ of $(\varepsilon_n)_{n \geq 1}$ such that, 
\begin{align*}
\frac{1}{N} \sum_{n =1}^N f_{\eta_n} \longrightarrow \tilde{f}, \quad N \longrightarrow +\infty,
\end{align*}
with respect to the norm $\|\cdot\|_{H+E}$, where $\tilde{f} \in D\left(\mathcal{E}_{\nu_\alpha^{\operatorname{rot}} , \mu_\alpha^{\operatorname{rot}}}\right)$. In particular, 
\begin{align*}
\frac{1}{N} \sum_{n =1}^N f_{\eta_n} \longrightarrow \tilde{f}, \quad N \longrightarrow +\infty,
\end{align*}
in $L^2(\mu_\alpha^{\operatorname{rot}})$. Thanks to \eqref{eq:L2_convergence}, $f = \tilde{f}$. This concludes the proof of the proposition. 
\end{proof}
\noindent
The next result proves that the ``carr\'e de Mehler" semigroup defined in \eqref{eq:carre_mehler_sg} is indeed the semigroup generated by the closure of the bilinear form $(\mathcal{E}_{\nu_\alpha, \mu_\alpha}, \mathcal{C}_b^1(\mathbb{R}^d))$ under an uniform boundedness condition on the logarithmic derivative of the positive Lebesgue density $p_\alpha$.

\begin{thm}\label{thm:thesame1}
Let $d \geq 1$ be an integer and let $\alpha \in (1,2)$.~Let $\mu_\alpha$ be a non-degenerate symmetric $\alpha$-stable probability measure on $\mathbb{R}^d$ with L\'evy measure $\nu_\alpha$ and with positive Lebesgue density $p_\alpha$ such that
\begin{align}\label{eq:cond_uniform_boundedness}
\left\|  \frac{\nabla(p_\alpha)}{p_\alpha} \right\|_{\infty} <+\infty.
\end{align}
Let $\left(\mathcal{P}_t\right)_{t\geq 0}$ and $(P_t)_{t\geq 0}$ be as in Remark \ref{rem:discussion_condition_constant_weight_function}. Then, both semigroups coincide. 
\end{thm}

\begin{proof}
The proof is an adaptation of the proof of \cite[Theorem $3.6$]{Fuhrman_95} for the non-degenerate symmetric $\alpha$-stable probability measure on $\mathbb{R}^d$ with $\alpha \in (1,2)$.~First, recall that, by \cite[Lemma $2.1$]{AH22_5}, for all $f \in \mathcal{S}(\mathbb{R}^d)$, all $t>0$ and all $x \in \mathbb{R}^d$, 
\begin{align}\label{eq:integral_represen_dual}
(P^{\nu_\alpha}_t)^*(f)(x) = \frac{1}{p_\alpha(x)} \int_{\mathbb{R}^d} f(y)p_\alpha(y) p_\alpha \left(\dfrac{x - y e^{-t}}{\left(1-e^{-\alpha t}\right)^{\frac{1}{\alpha}}}\right) \frac{dy}{\left(1-e^{-\alpha t}\right)^{\frac{d}{\alpha}}}
\end{align}
and, for all $x \in \mathbb{R}^d$ and all $t>0$,  
\begin{align*}
\frac{1}{p_\alpha(x)} \int_{\mathbb{R}^d} p_\alpha(y) p_\alpha \left(\dfrac{x - y e^{-t}}{\left(1-e^{-\alpha t}\right)^{\frac{1}{\alpha}}}\right) \frac{dy}{\left(1-e^{-\alpha t}\right)^{\frac{d}{\alpha}}} = 1.
\end{align*}
Now, for all $\xi \in \mathbb{R}^d$, let $\varphi^1_\xi$ and $\varphi^2_\xi$ be the two functions defined, for all $x \in \mathbb{R}^d$, by
\begin{align*}
\varphi^1_\xi(x) = \cos\left(\langle \xi ; x \rangle\right) , \quad \varphi^2_\xi(x) = \sin \left( \langle \xi ;x \rangle \right). 
\end{align*} 
Next, let $f \in L^2(\mu_\alpha)$ be such that $\langle f ; \varphi^j_\xi \rangle_{L^2(\mu_\alpha)} = 0$, for all $j \in \{1,2\}$ and all $\xi \in \mathbb{R}^d$. Then, 
\begin{align*}
\int_{\mathbb{R}^d} f(x) e^{i \langle x ; \xi \rangle} \mu_\alpha(dx) = 0. 
\end{align*} 
This implies that $f = 0$ $\mu_\alpha$-a.e. Let us denote by $\mathcal{L}_\alpha$ and by $L_\alpha$ the respective $L^2(\mu_\alpha)$-generators of the semigroups $(\mathcal{P}_t)_{t \geq 0}$ and $(P_t)_{t \geq 0}$. Since the family of functions $\{\varphi^1_\xi , \varphi^2_\xi ,\, \xi \in \mathbb{R}^d\}$ is total in $L^2(\mu_\alpha)$, in order to prove the theorem it is sufficient to prove that: for all $\beta >0$, all $\xi \in \mathbb{R}^d$ and all $j \in \{1,2\}$, 
\begin{align}\label{eq:resolvent_equality}
\left(\beta E - \mathcal{L}_\alpha\right)^{-1} \varphi^j_\xi = \left(\beta E - L_\alpha\right)^{-1} \varphi^j_\xi,
\end{align}
where $E$ is the identity operator. In the sequel, let us restrict to the case $j = 1$ since the other case can be treated similarly. Note that $\left| \left(P_t^{\nu_\alpha}\right)^* \circ \left( P_t^{\nu_\alpha} \right) \left(\varphi_\xi^1\right)(x) \right| \leq 1$, for all $x \in \mathbb{R}^d$, all $\xi \in \mathbb{R}^d$ and all $t>0$. At this point, let us consider the function $v_\alpha$ defined, for all $x \in \mathbb{R}^d$, all $\xi \in \mathbb{R}^d$ and all $\beta >0$, by
\begin{align}\label{eq:resolvent_fourier}
v_\alpha(x) = \int_0^{+\infty} e^{- \beta t} \left(P_{\frac{t}{\alpha}}^{\nu_\alpha}\right)^* \circ \left( P_{\frac{t}{\alpha}}^{\nu_\alpha} \right) \left(\varphi_\xi^1\right)(x) dt.
\end{align} 
Then, by definition of the resolvent operator, for all $x \in \mathbb{R}^d$, 
\begin{align*}
v_\alpha(x) = \left(\beta E - \mathcal{L}_\alpha\right)^{-1}( \varphi^1_\xi)(x). 
\end{align*}
Next, let $u_\alpha$ be the function defined by 
\begin{align}
u_\alpha =  \left(\beta E - L_\alpha\right)^{-1}( \varphi^1_\xi).
\end{align}
Again, by definition of the resolvent operator, $u_\alpha$ is the unique element of $D \left(\overline{\mathcal{E}_{\nu_\alpha, \mu_\alpha}} \right)$ such that, for all $\psi \in \mathcal{C}_c^{\infty}\left(\mathbb{R}^d\right)$, 
\begin{align}\label{eq:uniqueness_u_alpha}
\overline{\mathcal{E}_{\nu_\alpha, \mu_\alpha}}\left(u_\alpha ; \psi \right) + \beta \langle u_\alpha; \psi \rangle_{L^2(\mu_\alpha)} = \langle \varphi^1_{\xi} ; \psi \rangle_{L^2(\mu_\alpha)}. 
\end{align}
So, it is sufficient to prove that $v_ \alpha \in D \left(\overline{\mathcal{E}_{\nu_\alpha, \mu_\alpha}} \right)$ and that, for all $\psi \in \mathcal{C}_c^{\infty}\left(\mathbb{R}^d\right)$, 
\begin{align}\label{eq:variationnal_problem}
\overline{\mathcal{E}_{\nu_\alpha, \mu_\alpha}}\left(v_\alpha ; \psi \right) + \beta \langle v_\alpha; \psi \rangle_{L^2(\mu_\alpha)} = \langle \varphi^1_{\xi} ; \psi \rangle_{L^2(\mu_\alpha)}. 
\end{align}
Recall that, by construction, $D \left(\overline{\mathcal{E}_{\nu_\alpha, \mu_\alpha}} \right) = \overline{\mathcal{C}_b^{1}(\mathbb{R}^d)}^{\| \cdot \|_{H+E}}$, where $\mathcal{C}_b^1(\mathbb{R}^d)$ is the set of continuously differentiable functions which are bounded on $\mathbb{R}^d$ together with their first partial derivatives. Now, for all $x \in \mathbb{R}^d$, 
\begin{align*}
\left| v_\alpha(x) \right| \leq \int_0^{+\infty} e^{- \beta t} \left| \left(P_{\frac{t}{\alpha}}^{\nu_\alpha}\right)^* \circ \left( P_{\frac{t}{\alpha}}^{\nu_\alpha} \right) \left(\varphi_\xi^1\right)(x) \right| dt \leq \frac{1}{\beta}. 
\end{align*}
Moreover, for all $f \in \mathcal{C}_b^1(\mathbb{R}^d)$, all $x \in \mathbb{R}^d$ and all $t\geq 0$, 
\begin{align*}
\nabla P_t^{\nu_\alpha}(f)(x) = e^{-t} P_t^{\nu_\alpha} \left(\nabla(f)\right)(x). 
\end{align*}
Finally, for all $t>0$ and all $x \in \mathbb{R}^d$, 
\begin{align*}
\nabla \left((P_t^{\nu_\alpha})^*(f)(x)\right) & = - \frac{\nabla(p_\alpha)(x)}{p_\alpha(x)^2} \int_{\mathbb{R}^d} f(y) p_\alpha(y) p_\alpha\left(\dfrac{x-ye^{-t}}{\left(1-e^{-\alpha t}\right)^{\frac{1}{\alpha}}}\right) \frac{dy}{\left(1-e^{-\alpha t}\right)^{\frac{d}{\alpha}}} \\
&\quad\quad + \frac{1}{(1-e^{-\alpha t})^{\frac{1}{\alpha}}p_\alpha(x)} \int_{\mathbb{R}^d} f(y)p_\alpha(y) \nabla\left(p_\alpha\right)\left(\dfrac{x-ye^{-t}}{\left(1-e^{-\alpha t}\right)^{\frac{1}{\alpha}}}\right) \frac{dy}{\left(1-e^{-\alpha t}\right)^{\frac{d}{\alpha}}}.
\end{align*}
Thus, denoting by $A$ and $B$ the terms on the right-hand side of the previous equality, for all $x \in \mathbb{R}^d$, 
\begin{align*}
\left\| A(x) \right\| & \leq \left\| \frac{\nabla(p_\alpha)(x)}{p_\alpha(x)} \right\| \frac{\|f\|_{\infty}}{p_\alpha(x)} \int_{\mathbb{R}^d} p_\alpha(y) p_\alpha\left(\dfrac{x-ye^{-t}}{\left(1-e^{-\alpha t}\right)^{\frac{1}{\alpha}}}\right) \frac{dy}{\left(1-e^{-\alpha t}\right)^{\frac{d}{\alpha}}} , \\
& \leq \left\| \frac{\nabla(p_\alpha)}{p_\alpha} \right\|_{\infty} \|f\|_{\infty},
\end{align*}
and
\begin{align*}
\| B(x) \| & \leq  \frac{ \|f\|_{\infty}}{\left(1-e^{-\alpha t}\right)^{\frac{1}{\alpha}}}\frac{1}{p_\alpha(x)} \int_{\mathbb{R}^d} p_\alpha(y) \bigg\|\nabla(p_\alpha) \left( \frac{x-ye^{-t}}{\left(1-e^{-\alpha t}\right)^{\frac{1}{\alpha}}}\right) \bigg\| \frac{dy}{\left(1-e^{-\alpha t}\right)^{\frac{d}{\alpha}}} , \\
& \leq \frac{ \|f\|_{\infty}}{\left(1-e^{-\alpha t}\right)^{\frac{1}{\alpha}}} \left\| \frac{\nabla(p_\alpha)}{p_\alpha} \right\|_{\infty}.
\end{align*}
Thus, for all $f \in \mathcal{C}_b^1(\mathbb{R}^d)$ and all $t>0$, 
\begin{align}\label{ineq:supremum_subcommutation}
\left\| \nabla (P_t^{\nu_\alpha})^*(f) \right\|_{\infty} \leq \left\| \frac{\nabla(p_\alpha)}{p_\alpha} \right\|_{\infty} \|f\|_{\infty} \left(1+\frac{1}{\left(1-e^{-\alpha t}\right)^{\frac{1}{\alpha}}} \right). 
\end{align}
From the previous estimate, for all $t >0$, all $\xi \in \mathbb{R}^d$ and all $x \in \mathbb{R}^d$, 
\begin{align*}
\left\| \nabla  \left(\left(P_t^{\nu_\alpha}\right)^* \circ P_t^{\nu_\alpha}(\varphi^1_\xi)\right)(x)\right\| & \leq \left\| \frac{\nabla(p_\alpha)}{p_\alpha} \right\|_{\infty} \| P_t^{\nu_\alpha}(\varphi^1_\xi) \|_{\infty} \left(1+\frac{1}{\left(1-e^{-\alpha t}\right)^{\frac{1}{\alpha}}} \right) , \\
& \leq \left\| \frac{\nabla(p_\alpha)}{p_\alpha} \right\|_{\infty}\left(1+\frac{1}{\left(1-e^{-\alpha t}\right)^{\frac{1}{\alpha}}} \right). 
\end{align*}
The previous estimate implies that $v_\alpha$ is continuously differentiable on $\mathbb{R}^d$ and that, for all $x \in \mathbb{R}^d$, 
\begin{align*}
\left\| \nabla(v_\alpha)(x) \right\| & = \left\| \int_0^{+\infty} e^{- \beta t}\nabla \left(\left(P_{\frac{t}{\alpha}}^{\nu_\alpha}\right)^* \circ P_{\frac{t}{\alpha}}^{\nu_\alpha}(\varphi^1_\xi)\right)(x) dt \right\| , \\
& \leq \left\| \frac{\nabla(p_\alpha)}{p_\alpha} \right\|_{\infty} \int_{0}^{+\infty} e^{- \beta t} \left(1+\frac{1}{\left(1-e^{-\alpha t}\right)^{\frac{1}{\alpha}}} \right) dt < +\infty,
\end{align*}
since $\alpha \in (1,2)$ and $\beta>0$.~Then, $v_\alpha \in \mathcal{C}_b^1(\mathbb{R}^d) \subset D \left(\overline{\mathcal{E}_{\nu_\alpha, \mu_\alpha}} \right)$.~It remains to prove \eqref{eq:variationnal_problem}.~Thanks to \eqref{eq:closure_compact_support_rep}, let $(v_{\alpha,n})_{n \geq 1}$ be a sequence of functions in $\mathcal{C}_c^{\infty}(\mathbb{R}^d)$ such that $v_{\alpha,n} \longrightarrow v_\alpha$ in the norm $\| \cdot \|_{H+E}$ as $n$ tends to $+\infty$. Then, for all $\psi \in \mathcal{C}_c^{\infty}(\mathbb{R}^d)$,
\begin{align*}
\overline{\mathcal{E}_{\nu_\alpha, \mu_\alpha}}\left(v_\alpha, \psi\right) & = \mathcal{E}_{\nu_\alpha,\mu_\alpha}\left(v_\alpha, \psi\right) , \\
& = \underset{n \rightarrow +\infty}{\lim} \mathcal{E}_{\nu_\alpha,\mu_\alpha}\left(v_{\alpha,n}, \psi\right). 
\end{align*}
Now, by polarization and using the results contained in \cite[Appendix]{AH20_4}, for all $n \geq 1$ and all $\psi \in \mathcal{C}_c^{\infty}(\mathbb{R}^d)$, 
\begin{align*}
 \mathcal{E}_{\nu_\alpha, \mu_\alpha}\left(v_{\alpha,n}, \psi\right) = - \frac{1}{\alpha} \langle v_{\alpha,n}  ; \left(\mathcal{L}^{\nu_\alpha} + (\mathcal{L}^{\nu_\alpha})^* \right)(\psi) \rangle_{L^2(\mu_\alpha)}, 
\end{align*}
which gives, at the limit,  
\begin{align}\label{eq:after_ipp_limit}
 \mathcal{E}_{\nu_\alpha, \mu_\alpha}\left(v_{\alpha}, \psi\right) = - \frac{1}{\alpha} \langle v_{\alpha}  ; \left(\mathcal{L}^{\nu_\alpha} + (\mathcal{L}^{\nu_\alpha})^* \right)(\psi) \rangle_{L^2(\mu_\alpha)}, 
\end{align}
where $\mathcal{L}^{\nu_\alpha}$ is the generator of the semigroup $\left(P_t^{\nu_\alpha}\right)_{t \geq 0}$ and where $(\mathcal{L}^{\nu_\alpha})^*$ is the generator of the dual semigroup $\left(\left(P_t^{\nu_\alpha}\right)^*\right)_{t \geq 0}$. Recall that, for all $\psi \in \mathcal{C}_c^{\infty}(\mathbb{R}^d)$ and all $x \in \mathbb{R}^d$, 
\begin{align*}
&\mathcal{L}^{\nu_\alpha}(\psi)(x) = - \langle x ; \nabla(\psi)(x) \rangle + \int_{\mathbb{R}^d} \langle u ; \nabla(\psi)(x+u) - \nabla(\psi)(x) \rangle \nu_\alpha(du), \\
&(\mathcal{L}^{\nu_\alpha})^*(\psi)(x) = \left(M_\alpha^{-1} \circ A_\alpha \circ M_\alpha\right) (\psi)(x),
\end{align*}
where
\begin{align*}
&A_\alpha(\psi)(x) = d\psi(x) + \langle x ; \nabla(\psi)(x) \rangle + \int_{\mathbb{R}^d} \langle \nabla(\psi)(x+u) - \nabla(\psi)(x) ; u \rangle \nu_\alpha(du) , \\
& M_\alpha(\psi)(x) = \psi(x)p_\alpha(x) , \quad M_\alpha^{-1}(\psi)(x) = \frac{\psi(x)}{p_\alpha(x)}. 
\end{align*}
Now, by definition of $v_\alpha$ and since $(\beta E - \mathcal{L}_\alpha)$ is a self-adjoint operator on $L^2(\mu_\alpha)$, for all $\psi \in \mathcal{C}^{\infty}_{c}(\mathbb{R}^d)$, 
\begin{align*}
\left \langle v_\alpha  ; \left(\beta E - \frac{1}{\alpha}\left(\mathcal{L}^{\nu_\alpha} +(\mathcal{L}^{\nu_\alpha})^*\right)\right) \psi \right\rangle_{L^2(\mu_\alpha)} = \langle \varphi_\xi^1 ; \psi \rangle_{L^2(\mu_\alpha)}.
\end{align*}
This concludes the proof of the theorem. 
\end{proof}
\noindent
\textbf{Acknowledgements}: I thank Christian Houdr\'e for discussions on a preliminary version of this manuscript and for pointing out to me the references \cite{HK_07} and \cite{RS_10}.~I also thank the school of mathematics at the Georgia institute of technology for its hospitality where part of this research was carried out and presented. The author's research was supported by a PEPS JCJC.

\end{document}